\documentclass[reqno]{amsart}
\usepackage[top=1.1in, bottom=1in, left=1in, right=1in]{geometry}
\usepackage{amsfonts}
\usepackage{amssymb}
\usepackage{extarrows}
\usepackage{amsthm}
\usepackage{amsmath}
\usepackage{mathtools}
\usepackage{mathrsfs}
\usepackage{color}
\usepackage{enumerate}
\usepackage[numbers,sort&compress]{natbib}
\usepackage{pdfsync}
\usepackage{esint}
\usepackage{float}
\usepackage{caption}
\usepackage{subfigure}
\usepackage{enumitem}
\usepackage{appendix}
\usepackage{bbm}
\usepackage{graphicx}
\usepackage{array}
\usepackage{booktabs}

\allowdisplaybreaks

\usepackage{tikz} 
\usetikzlibrary{calc}
\usetikzlibrary{intersections}
\usetikzlibrary{patterns}

\usepackage[colorlinks,
linkcolor=black,       
anchorcolor=blue,      
citecolor=blue,        
]{hyperref}

\makeatother

\numberwithin{equation}{section}

\newtheorem{theorem}{Theorem}[section]
\newtheorem{definition}[theorem]{Definition}

\newtheorem{lemma}[theorem]{Lemma}

\newtheorem{proposition}[theorem]{Proposition}
\newtheorem{corollary}[theorem]{Corollary}
\numberwithin{equation}{section}

\theoremstyle{remark}
\newtheorem{remark}[theorem]{Remark}

\newcommand{\R}{{\mathbb R}}

\def\wt{\widetilde}
\def\wh{\widehat}
\def\9{{\infty}}

\def\na{{\nabla}}

\def\bbr{{\mathbb{R}}}

\def\bbp{{\mathbb{P}}}
\def\({\left(}
\def\){\right)}
\def\<{\left<}
\def\>{\right>}

\newcommand{\imu}{\mathrm{i}}
\newcommand{\dd}{\,\mathrm{d}}
\renewcommand{\epsilon}{\varepsilon}

\renewcommand{\Im}{\operatorname{Im}}
\renewcommand{\Re}{\operatorname{Re}}

\newcommand{\N}{{\mathbb{N}}}

\newcommand{\C}{{\mathbb{C}}}
\newcommand{\Z}{{\mathbb{Z}}}

\newcommand{\PP}{{\mathbb{P}}}
\newcommand{\G}{{\mathbb{G}}}

\newcommand{\bS}{{\mathbb{S}}}

\newcommand{\X}{{\mathbb{X}}}
\newcommand{\Y}{{\mathbb{Y}}}

\newcommand{\cF}{{\mathcal{F}}}

\newcommand{\cI}{{\mathcal{I}}}
\newcommand{\cJ}{{\mathcal{J}}}

\newcommand{\cP}{{\mathcal{P}}}

\newcommand{\cT}{{\mathcal{T}}}
\newcommand{\cU}{{\mathcal{U}}}

\newcommand{\Sp}{{\mathbb S}}
\newcommand{\vece}{{\textnormal{\textbf{e}}}}

\newcommand{\ModN}{C_{\leq (\frac{\lambda}{2^8})^2}}

\newcommand{\TempN}{P^{(t)}_{\leq (\frac{\lambda}{2^8})^2}}

\newcommand{\SOne}{S^{s,a,0}}
\newcommand{\Ssa}{S^{s,a,0}}
\newcommand{\NOne}{N^{s,a,0}}
\newcommand{\Nsa}{N^{s,a,0}}
\newcommand{\XS}{\mathbb{X}^{s,a}}
\newcommand{\Xs}{\mathbb{X}^{s,a}}
\newcommand{\GS}{\mathbb{G}^{s,a}}
\newcommand{\Gs}{\mathbb{G}^{s,a}}
\newcommand{\tS}{\tilde{S}^{s,a,b}}
\usepackage{setspace}
\AddToHook{begindocument}{%
	\setlength{\jot}{1pt} 
	\setlength{\abovedisplayskip}{2pt plus 2pt minus 2pt}%
	\setlength{\belowdisplayskip}{1pt plus 2pt minus 2pt}%
	\setlength{\abovedisplayshortskip}{0pt plus 1pt minus 1pt}%
	\setlength{\belowdisplayshortskip}{0pt plus 1pt minus 1pt}%
}
\makeatletter
\def\thm@space@setup{%
	\thm@preskip=6pt %
	\thm@postskip=3pt %
}
\makeatother

\usepackage{etoolbox}
\makeatletter
\patchcmd\subsection
{.5\linespacing\@plus.7\linespacing} 
{.3\linespacing\@plus.5\linespacing}             
{}{}
\patchcmd\subsubsection
{.3\linespacing\@plus.5\linespacing} 
{0.2\linespacing\@plus.3\linespacing}            
{}{}

\patchcmd\section
{.7\linespacing\@plus\linespacing} 
{.3\linespacing\@plus0.5\linespacing}                   
{}{}

\patchcmd\section
{.5\linespacing}                    
{.3\linespacing}
\makeatother

\makeatletter
\renewenvironment{proof}[1][\proofname]{\par
	\pushQED{\qed}%
	\normalfont \topsep2pt\@plus1pt\@minus1pt\relax   %
	\trivlist
	\item[\hskip\labelsep\itshape#1\@addpunct{.}]\ignorespaces
}{%
	\popQED\endtrivlist\@endpefalse
}
\makeatother

\begin{document}
	
	\title[]{The stochastic Zakharov system in dimension $ d \geq 4$: Local well-posedness and regularization by noise for scattering}
	
	\author{Martin Spitz}
	\address[Martin Spitz]{Fakult\"at f\"ur Mathematik,
		Universit\"at Bielefeld, D-33501 Bielefeld, Germany}
	\email{mspitz@math.uni-bielefeld.de}
	\thanks{}
	
	\author{Deng Zhang}
	\address[Deng Zhang]{School of Mathematical Sciences, MOE-LSC,
		CMA-Shanghai, Shanghai Jiao Tong University, China.}
	\email{dzhang@sjtu.edu.cn}
	\thanks{}
	
	\author{Zhenqi Zhao}
	\address[Zhenqi Zhao]{School of Mathematical Sciences, Shanghai Jiao Tong University, China.}
	\email{semimartingale@sjtu.edu.cn}
	\thanks{}

	\keywords{}
	%
	
	\begin{abstract}
		In this paper, we develop the well-posedness theory and uncover the noise-regularization effect on scattering for the stochastic Zakharov system in dimensions $d \geq 4$ and beyond the energy space. Our focus is particularly directed at the large data regime, where the global existence and long-time dynamics of the deterministic Zakharov system remain largely open.
We prove the local well-posedness of the stochastic system in the full deterministic regularity regime and establish a blow-up alternative at the endpoint regularity, which implies the persistence of regularity in the full well-posedness regime. Furthermore, we prove that for any large initial data, with high probability, non-conservative noise yields global and scattering solutions.
 
		Our proof introduces a tailored functional framework. To establish local well-posedness, we employ a refinement of adapted Fourier restriction and lateral Strichartz spaces, which allows us to control both the nonlinear interactions and the critical first-order derivative perturbations arising from rescaling transforms. To achieve the noise-regularization effect, we augment this setting with maximal function spaces. We derive new trilinear estimates for the stochastic wave nonlinearity that are crucial for the global dynamics by fully exploiting the temporal regularity of geometric Brownian motions in scaling-(sub)critical Besov spaces.

	\end{abstract}
	
	\maketitle

	\tableofcontents

	
	\section{Introduction}   \label{Sec:Intro}
	
	We study the local well-posedness and the noise-regularization effect of the stochastic Zakharov system
	\begin{equation}   \label{eq:IntroStoZak}
		\left\{\aligned
		\imu \dd X + \Delta X \dd t &= \Re(Y) X \dd t
		- \imu \mu X \dd t + \imu X \dd W_1(t),  \\
		\frac{1}{\alpha} \imu \dd Y + |\na|Y \dd t &= -|\na| |X|^2 \dd t + \dd W_2(t), 
		\endaligned
		\right.
	\end{equation}
	where $W_1$ and $W_2$ denote the driving noise, and $\mu$ represents an associated drift correction.
	The deterministic Zakharov system, i.e.~\eqref{eq:IntroStoZak} with $W_1 = W_2 = 0$ and $\mu = 0$, was  introduced by Zakharov~\cite{Za72} in 1972 to model rapid oscillations of the electric field in a non- or weakly magnetized plasma. 
	
	One of the intriguing features of the Zakharov system is
	its close connection to the focusing cubic nonlinear Schr{\"o}dinger equation (NLS)
	\begin{equation}
		\label{eq:NLS}
		\imu \partial_t u + \Delta u = - |u|^2 u,
	\end{equation}
	which arises as the subsonic limit $\alpha \rightarrow \infty$ of the deterministic Zakharov system.
	We refer to~\cite{KPV95, MN08, OT92, SW86} for rigorous results in this direction.
	
	The quadratic
	coupling between
	the Schr{\"o}dinger and the wave equation in the Zakharov system
	has led to a rich local and global well-posedness theory,
	requiring sophisticated methods from harmonic analysis and the theory of dispersive equations.
	We refer to~\cite{BC96, GTV97, BGHN15, CHN23} for some milestones in the deterministic local well-posedness theory and
	the references therein.
	In particular,
	the optimal regularity range of $(s,l) \in \R \times \R$
	for the local well-posedness theory
	in $H^s(\R^d) \times H^l(\R^d)$ has recently been established
	 in~\cite{CHN23}.
	
	In recent years,
	tremendous progress towards the understanding of the long-time dynamics of the  Zakharov system
	in the energy space
	$H^1(\R^d)\times L^2(\R^d)$ has been achieved.
	In the energy space,
	the Zakharov system
	conserves mass and energy
	\begin{align*}
		m(u(t)) = \frac12 \int_{\R^d} |u(t)|^2 \dd x, \qquad e_Z(u(t), v(t)) = \int_{\R^d} \frac12 |\nabla u(t)|^2  + \frac14 |v(t)|^2 + \frac12 \Re(v(t)) |u(t)|^2 \dd x.
	\end{align*}
	Although it does not possess a scaling invariance, it is called energy critical in dimension $d = 4$, as the closely related
	cubic NLS~\eqref{eq:NLS}
	is energy critical in this dimension. This close connection is further evidenced by the fact that the ground state solution $Q$ of the cubic NLS~\eqref{eq:NLS} gives rise to a stationary, non-dispersing solution $(Q, -Q^2)$ of the Zakharov system. In fact, it is conjectured that the NLS ground state $Q$ plays the same role for the dynamic dichotomy of the Zakharov system as for the focusing cubic NLS in the energy critical case:
	On the one hand,
	it is conjectured that all solutions originating below the ground state exist globally and scatter. 
	This conjecture was proved in the radial case by Guo and Nakanishi~\cite{GN21}. 
	In the non-radial case, global well-posedness below the ground state was established by Candy, Herr, and Nakanishi~\cite{CHN21},
	and the first step of the concentration compactness--rigidity method was performed by Candy~\cite{Ca24}. 
	Completing the second step and, consequently, proving scattering below the ground state, remains an open problem.
	On the other hand,
	slightly above the ground state, singularities can form and finite-time blow-up solutions were constructed 
	by Krieger and Schmid~\cite{KS24a, KS24b}.
	We also refer to~\cite{GM94a, GM94b, Me96, BC96, CHT08, GNW13, HPS13, GLNW14,  Sp22} for global well-posedness, scattering, and singularity formation in lower dimensions.
	
	In the {\it large data} regime,
	the dynamics of the Zakharov system are largely open.
	Moreover, much less is known outside the energy space. In~\cite{CHN23}, global existence and scattering were proved for initial data from the well-posedness regime in dimensions $d \geq 4$, provided the Schr{\"o}dinger data is sufficiently small. Beyond this small data result, however, no mechanism yielding global solutions for large data in this regime is known.

	\vspace*{4pt plus 2pt minus 2pt}%
	It is widely believed 
	in the community of stochastic partial differential equations that 
	noise is able to improve
	the well-posedness
	and dynamics of deterministic equations.
	This kind of phenomenon is commonly referred to as {\it regularization by noise}.
	It has been
	investigated in several settings, in particular for various stochastic fluid models.
	For example, transport noise can
	improve the uniqueness for transport equations \cite{FGP10} and can
	prevent blowup
	for the 3D stochastic vorticity
	Navier-Stokes equations \cite{FL21}. 
	
	\vspace*{4pt plus 2pt minus 2pt}%
	The stochastic Zakharov system~\eqref{eq:IntroStoZak} arises naturally
	if one accounts for model uncertainty and random fluctuations in the ion
	density and the temperature of the plasma.
	In dimension one,
	global well-posedness of the stochastic Zakharov system
	was proved by Tsutsumi~\cite{Ts22},
	and the subsonic limit  $\alpha\to \infty$
	of the stochastic Zakharov system to a stochastic cubic NLS was proved by Barru\'e, de Bouard, and Debussche~\cite{BBD24}.
	In the energy space,
	the global well-posedness below the  ground state was addressed
	in dimensions three and four 
	by Herr, R\"ockner, and the first two authors~\cite{HRSZ25, HRSZ24}. 
	A crucial phenomenon uncovered in~\cite{HRSZ24} is 
	that non-conservative
	noise has the ability to prevent blow-up
	and even lead to scattering of solutions to the stochastic Zakharov system. This result ensures global existence and scattering with high probability for any data from the energy space, particularly in regimes where the deterministic counterpart either exhibits finite-time blow-up (above the ground state) or where scattering remains a conjecture (below the ground state).
	This intriguing regularization-by-noise effect
	thus sheds some light
	on the long-time dynamics of the
	Zakharov system in the energy space,
	particularly
	in the large data regime,
	where the long-time dynamics of the deterministic Zakharov system are largely open. 
	
	\vspace*{4pt plus 2pt minus 2pt}%
	The objective of the present work is to develop 
	the well-posedness theory 
	and explore the noise-regularization effect 
	for the stochastic Zakharov system 
	in dimensions $d\geq 4$ 
	and beyond the energy space. 
	The main results of this paper are summarized as follows:  
	\begin{enumerate}
		\item A comprehensive local well-posedness theory for the stochastic Zakharov system in dimensions $d \geq 4$.
		The local well-posedness in
		$H_x^s \times H_x^l$
		is proved in the full deterministic regularity regime.
		This result is optimal in the sense that it contains the deterministic result by taking $W_1 = W_2 = 0$, for which the range is sharp \cite{CHN23}.
		
		\item
		A blow-up alternative in terms of the endpoint regularity. This result particularly implies the  persistence of regularity in the full local well-posedness regime.
		
		\item
		The noise-regularization effect on the long-time behavior of solutions for a large subset of the local well-posedness regime.
		We show that for any initial data -- including large data where deterministic global dynamics are entirely unknown --
		with high probability sufficiently strong non-conservative noise yields global and scattering solutions.
		This result is proved in the well-posedness regime up to
		a borderline
		dictated by the regularity threshold of the noise. 
	\end{enumerate}			
	
	We emphasize that these are the first results for the stochastic Zakharov system in high dimensions and beyond
	the energy space, 
	which in particular 
	extend the local well-posedness and regularization-by-noise results and sharpen the blow-up alternative from the 4D energy-critical case in~\cite{HRSZ24}.
	
	Indeed, the approach to local well-posedness in~\cite{HRSZ24}
	relied on a specific class of Fourier restriction spaces tailored to the energy space $H^1(\R^4) \times L^2(\R^4)$.
	The previous methods extend straightforwardly to certain limited regularities, but face a barrier when applied to the full well-posedness regime. Specifically, the previously employed adapted Fourier restriction spaces are too restrictive to accommodate the noise terms in this broader regime. Overcoming this structural limitation is the motivation for the new functional framework we develop for the local well-posedness theory in this paper.
	
	Establishing the regularization-by-noise effect in the broad regularity regime considered in this paper requires a substantial departure from the techniques developed for the energy-critical case~\cite{HRSZ24}. To achieve this, we establish a new set of trilinear estimates for the stochastic wave nonlinearity, which crucially rely on novel global bounds for geometric Brownian motions in scaling-(sub)critical Besov spaces derived in this work.

	\subsection{Formulation of main results}

	Consider
	the stochastic Zakharov system in dimension $d\geq 4$
	\begin{equation}   \label{eq:StoZak}
		\left\{\aligned
		\imu \dd X + \Delta X \dd t &= \Re(Y) X \dd t
		- \imu \mu X \dd t + \imu X \dd W_1(t),  \\
		\frac{1}{\alpha} \imu \dd Y + |\na|Y \dd t &= -|\na| |X|^2 \dd t + \dd W_2(t), \\ 
		(X(0), Y(0)) &= (X_0, Y_0) \in H_x^s \times H_x^l. 
		\endaligned
		\right.
	\end{equation} 
	Here $H_x^s$ and $H_x^l$ denote the standard $L_x^2$-based Sobolev spaces.
	The stochastic Zakharov system is a model from plasma physics,
	the unknown variable $X \colon \R \times \R^d \rightarrow \C$ denotes the complex amplitude of the electric field, $Y \colon \R \times \R^d \rightarrow \C$ describes the fluctuation of the ion density, and $\alpha>0$ represents the ion sound speed.
	The noise $W_1$ and $W_2$
	are independent Wiener processes
	\begin{align} \label{Noise}
			W_1(t,x) = \sum\limits_{k=1}^\infty \imu \phi^{(1)}_k(x) \beta^{(1)}_k(t),\quad W_2(t,x) = \sum\limits_{k=1}^\infty  \phi^{(2)}_k(x) \beta^{(2)}_k(t)
		\end{align}
		for $(t,x)\in \R_+\times \R^d$, where
	$\{\phi^{(1)}_k\}_k \subseteq H_x^{\frac{d}{2}+2+(s-1)_+}$
	and $\{\phi^{(2)}_k\}_k \subseteq H_x^{\frac{d}{2}+s-1}$ are real-valued functions,
	$\{\beta^{(j)}_k\}$ are real-valued independent Brownian motions
	on a stochastic basis $(\Omega, \mathscr{F}, \{\mathscr{F}_t\}, \mathbb{P})$, and $r_+ = \max\{r,0\}$ denotes the positive part for any real number $r$.
	The multiplicative noise $X \dd W_1(t)$ is taken in the sense of It{\^o},
	and
	\begin{align*}
		\mu = \frac 12 \sum\limits_{k=1}^\infty
		|\phi_k^{(1)}|^2 <\infty.
	\end{align*}
	For real-valued $\{\phi^{(1)}_k\}_k$,
	$\mu$ is exactly the Stratonovich correction term,
	and $- \imu \mu X \dd t + \imu X \dd W_1(t)$ is the Stratonovich differential $\imu X \circ \dd W_1(t)$,
	so that the  mass of the Schr{\"o}dinger component is conserved.
	In the plasma model,
	$W_1$ and $W_2$ model
	fluctuations in the ion density and the
	temperature of the plasma, respectively.
	We refer to Subsection~2.1 in~\cite{HRSZ25} for a heuristic derivation of the stochastic model~\eqref{eq:StoZak}.
	
	Without loss of generality, we take $\alpha=1$ in the following.

	\vspace*{4pt plus 2pt minus 2pt}%
	
	Throughout this paper the spatial coefficients of the noise \eqref{Noise}
	satisfy the following hypothesis.
	
	\vspace*{4pt plus 2pt minus 2pt}%
	\paragraph{\bf Hypothesis (H)}
	The spatial functions $\{\phi^{(j)}_k\}$, $j=1,2$, satisfy
	the summability conditions
	\begin{align} \label{phik-condition}
		\sum\limits_{k=1}^\infty
		\|\phi^{(1)}_k\|_{H_x^{\frac{d}{2}+2+(s-1)_+}}^2 +
		\sum\limits_{l=1}^d  \sum\limits_{k=1}^\infty \int  \sup_{y\in \cP_{\vece_l}} |\nabla \phi^{(1)}_k(r \vece_l+y)|\dd r  <\infty, \ \
		\sum\limits_{k=1}^\infty
		\|\phi^{(2)}_k\|_{H_x^{\frac{d}{2}+s-1}}^2 <\infty,
	\end{align}
	where $\vece_1, \vece_2, \dots, \vece_{d}$ denote the standard orthonormal basis of $\bbr^d$ and $\cP_{\vece} = \{\xi \in \R^d \colon \xi \cdot \vece = 0\}$.

	\vspace*{4pt plus 2pt minus 2pt}%
	We understand solutions to~\eqref{eq:StoZak}
	 in the following sense. 
	
	\begin{definition}   \label{def:Solution}
		Let $(s,l)\in \R^2$ and $T\in (0,\infty)$. 
		We say that $(X,Y)$ is a probabilistic strong solution to \eqref{eq:StoZak}
		on $[0,\tau]$,
		where $\tau\in(0,T]$ is an $\{\mathscr{F}_t\}$-stopping time,
		if $(X,Y)$ is an $H_x^s\times H_x^l$-valued
		$\{\mathscr{F}_t\}$-adapted process which belongs to $C([0,\tau],H_x^s \times H_x^l)$
		and satisfies $\mathbb{P}$-a.s.
		for any $t\in [0,\tau]$,
		\begin{equation}   \label{equa-stoZak-def}
			\left\{\aligned
			X(t) &= X_0 + \int_0^t \imu \Delta X \dd s - \int_0^t  \imu \Re(Y) X \dd s - \int_0^t \mu X \dd s + \int_0^t X \dd W_1(s),    \\
			Y(t) &= Y_0 + \int_0^t \imu |\na|Y \dd s + \int_0^t \imu |\na| |X|^2 \dd s - \imu W_2(t),
			\endaligned
			\right.
		\end{equation}
		as equations in $H_x^{s-2} \times H_x^{l-1}$.
		
		For any given $\{\mathscr{F}_t\}$-stopping time $\tau^*$, we also call $(X,Y)$ a probabilistic strong solution to \eqref{eq:StoZak}
		on $[0,\tau^*)$ if $(X,Y)$ is an $\{\mathscr{F}_t\}$-adapted process belonging to $C([0,\tau^*),H_x^s \times H_x^l)$ such that for any $T \in (0,\infty)$ and any $\{\mathscr{F}_t\}$-stopping time $\tau < \tau^*$, $(X,Y)$ is a probabilistic strong solution to \eqref{eq:StoZak} on $[0,\tau \wedge T]$.
	\end{definition}

	\vspace*{4pt plus 2pt minus 2pt}%
	\paragraph{\bf I. 
		Local well-posedness and blow-up alternative}
	The following conditions on the regularities $(s,l) \in \R^2$ determine the regime of local well-posedness, see Figure~\ref{picture1.1}:
	\begin{align}\label{IniReg-condition}
		l\geq \frac{d}{2}-2,\quad \max \Big\{l-1, \frac{l}{2}+\frac{d-2}{4} \Big\}\leq s\leq l+2,\quad (s,l)\neq \Big(\frac{d}{2}, \frac{d}{2}-2\Big), \Big(\frac{d}{2}, \frac{d}{2}+1\Big),
	\end{align}
	where $d\geq 4$. We note that the lowest regularities satisfying these conditions, i.e. the lower left corner of the well-posedness regime in Figure~\ref{picture1.1}, is given by $(s,l) = (\frac{d-3}{2}, \frac{d-4}{2})$. We refer to this regularity as the
	{\it endpoint regularity}.
	As we shall see later,
	the endpoint regularity plays a distinctive role in the blow-up alternative for solutions of the stochastic Zakharov system.

	\begin{figure}[ht]
		\centering
		\includegraphics[scale= 0.5]{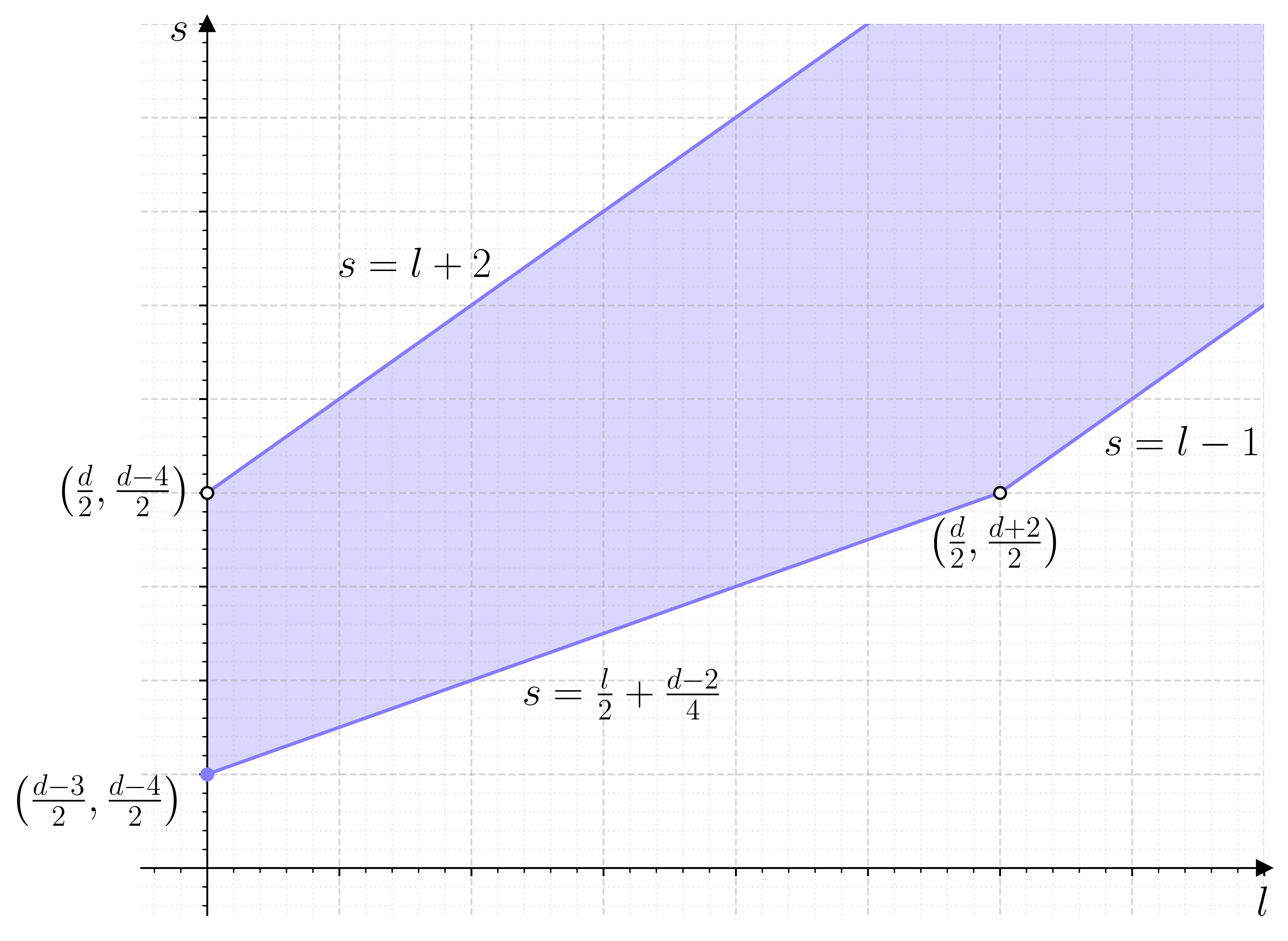}
		\vspace{-10pt}
		\caption{Regularity regime 
			for local well-posedness in $d=4$.}
		\label{picture1.1}
	\end{figure}

	The first main result of the present work
	is the local well-posedness and blow-up alternative
	of the stochastic Zakharov system in dimensions $d\geq 4$ in the above regularity regime.

	\begin{theorem}  [Local well-posedness and blow-up alternative] \label{Thm-LWPNEP}
		\label{thm:LocalWP}
		Let $d \geq 4$ and $(s,l)$ satisfy \eqref{IniReg-condition}. 
		Assume Hypothesis (H).
		Then, for any deterministic initial data  $(X_0,Y_0) \in H_x^s \times H_x^l$,
		there exists a stopping time \linebreak $\tau^* > 0$ such that the
		stochastic Zakharov system~\eqref{eq:StoZak} has a unique $\{\cF_t\}$-adapted solution $(X,Y)$ in $C([0,\tau^*), H_x^s \times H_x^l)$,
		$\mathbb{P}$-a.s.
		Moreover, $\PP$-a.s. if $\tau^* <\infty$,  then
		\begin{align*}
			(i)\quad  \limsup_{t \rightarrow \tau^*} (\|X(t)\|_{H_x^{\frac{d-3}{2}}} + \|Y(t)\|_{H_x^{\frac{d-4}{2}}}) = \infty, 
		\end{align*} 
		or 
		\begin{align*}  
			(ii)\quad \|X\|_{L^2 ([0,\tau^*) ; W_x^{\frac{d-3}{2},\frac{2d}{d-2}})} = \infty.
		\end{align*}
	\end{theorem}
	
	\begin{remark}
		\label{rem:Uniqueness}
		The uniqueness in the statement of Theorem~\ref{Thm-LWPNEP} means that for any $T \in (0,\infty)$ and any $\{\cF_t\}$-adapted stopping time $\tau < \tau^*$ the process $(X,Y)$ is the unique solution of~\eqref{eq:StoZak} in the sense of Definition~\ref{def:Solution} satisfying
		\begin{align*}
			(X,Y) \in C([0,\tau],H_x^s \times H_x^l), \qquad X \in \X^{\frac{d-3}{2},0}([0, \tau \wedge T]),
		\end{align*}
		where the space $\X^{\frac{d-3}{2},0}$ is introduced in~\eqref{eq:DefRestrictionNorm} below.
	\end{remark}
	
	\begin{remark} \label{rem:LWP}
		\begin{enumerate}
			\item
			Theorem \ref{Thm-LWPNEP} provides  the first local well-posedness result for the stochastic Zakharov system in dimensions $d > 4$ and in $d = 4$ beyond the energy space.
			It is proved in the same regularity regime \eqref{IniReg-condition} as in the deterministic case.
			The regularity regime for local well-posedness is known to be optimal in the deterministic case, as the flow map fails to be $C^2$ outside of it~\cite{CHN23}.

			\item
			We emphasize that the above blow-up alternative 
			for a solution in $H_x^s \times H_x^l$ only depends on 
			the endpoint regularity and is independent of $s$ and $l$.
			Consequently, it implies the persistence of regularity
			in the full well-posedness regime~\eqref{IniReg-condition}:
			if the initial datum possesses higher regularity, i.e., it belongs to $H_x^{s'} \times H_x^{l'}$ with $s' \geq s$ and $l' \geq l$, then the solution remains in this more regular space throughout its interval of existence.   
			In contrast to the 4D energy-critical case~\cite{HRSZ24}, where the blow-up criterion was tied to the energy norm, our result establishes a universal endpoint criterion. Our proof introduces a {\it bootstrap type argument}
			for the persistence of regularity for
			the Schr\"odinger and wave components, 
			see Subsection \ref{Subsec:ImproBlowUpCon} below. 
			
			\item
			The state space of stochastic solutions combines a refinement of the adapted Fourier restriction spaces from~\cite{CHN23} 
			(see \eqref{eq:DefTildeSsablambda} below)
			and the lateral Strichartz spaces
			employed in the context of Schr\"odinger maps~\cite{BIKT11}. 
			The latter are used to control the problematic first order derivative term arising from the noise via the rescaling transform (see equation \eqref{eq:RanZakbc} below). Crucially, the two types of function spaces are compatible to form a suitable functional framework for the stochastic Zakharov system in the full well-posedness regime.
			As a result, we sharpen the control of the nonlinearity
			in the weaker Fourier restriction spaces.
			See 
			Subsection \ref{subsec:IdeaProof} for more explanations. 
		\end{enumerate}
	\end{remark}

	\paragraph{\bf II.
		Noise-regularization-effect on scattering}
	
	Our second main result shows the regularization effect of the noise on global well-posedness and scattering	for general initial data. We consider the following conditions on the regularities $(s,l) \in \R^2$:
	\begin{align}\label{IniReg-conditionNoiseReg}
		l\geq \frac{d}{2}-2,\quad  s>l-\frac{1}{2},\quad l+2\geq s\geq \frac{l}{2}+\frac{d-2}{4},\quad (s,l)\neq \Big(\frac{d}{2}, \frac{d}{2}-2\Big),
	\end{align}
	where $d\geq 4$. We further classify this noise-regularization regime into three subregimes, see Table~\ref{table:regularity-regimes} below. Each subregime presents distinct analytical challenges, requiring tailored techniques. In particular, we derive different trilinear estimates for the stochastic wave nonlinearity in each of them, see Theorem~\ref{Prop-trilinear}.
	
	Figure~\ref{picture1.2}
	illustrates the noise-regularization regime and the three subregimes:

	\begin{figure}[h]
		\centering
		\includegraphics[scale= 0.5]{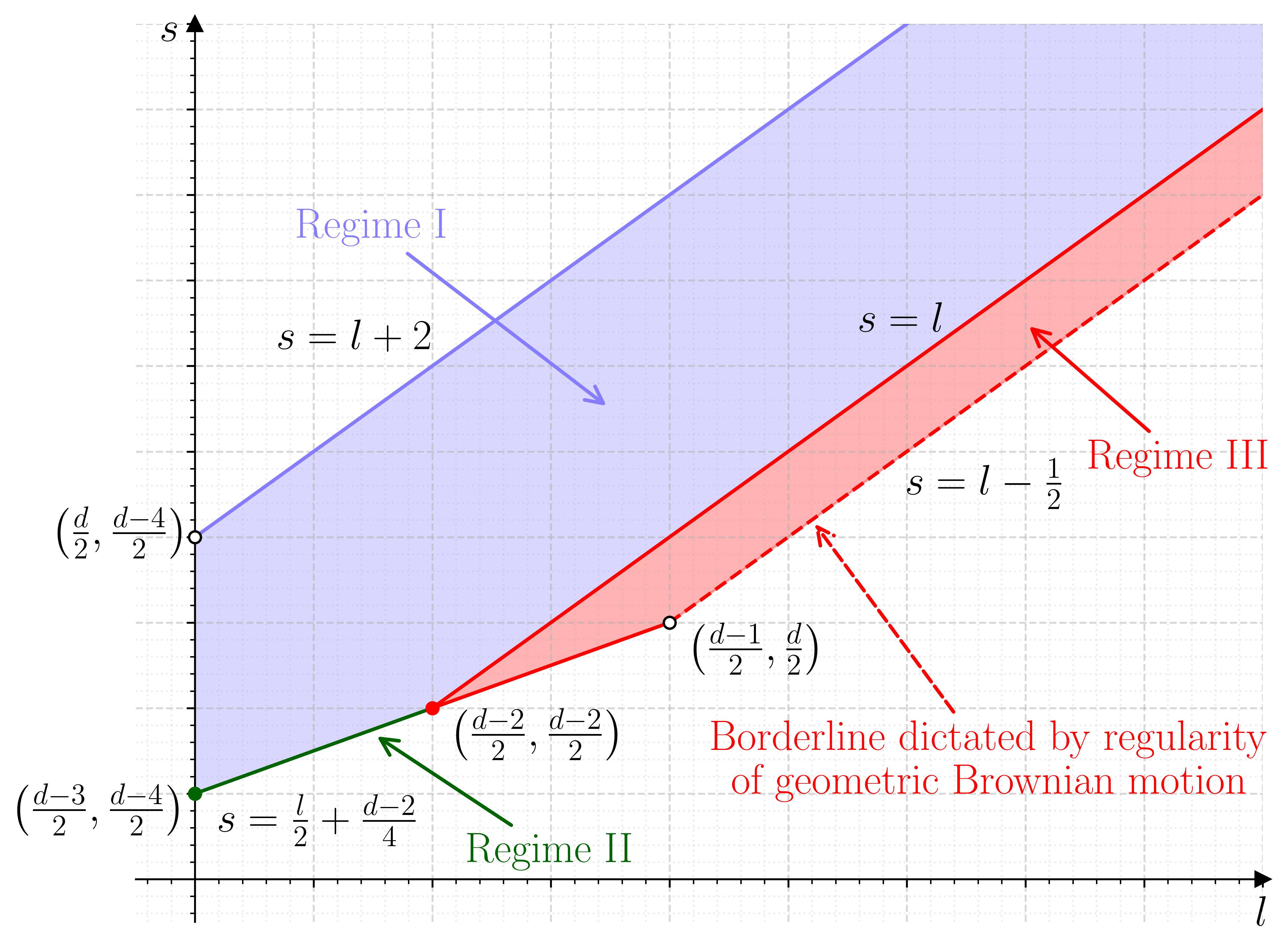}
		\vspace{-10pt}
		\caption{
			Regularity regime  
			for noise-regularization on GWP and scattering in $d=4$.}
		\label{picture1.2}
	\end{figure}

	\noindent The three different  
	noise-regularization regimes  
	are as follows:
	\begin{table}[h]
		\centering
		\renewcommand{\arraystretch}{1.2}
		\scalebox{0.8}{
			\begin{tabular}{c|c}
				\hline
				& Range of  $(s,l)$  \\
				\hline
				Noise-regularization regime I &  $l\geq \tfrac{d}{2}-2,\quad l+2\geq s>\max\{ l, \tfrac{l}{2}+\tfrac{d-2}{4} \},\quad (s,l)\neq (\tfrac{d}{2}, \tfrac{d}{2}-2)$\\
				Noise-regularization regime  II &  $\tfrac{d}{2}-1>l\geq \tfrac{d}{2}-2,\quad s= \tfrac{l}{2}+\tfrac{d-2}{4}$\\
				Noise-regularization regime  III &  $l\geq s>l-\tfrac{1}{2},\quad s\geq \tfrac{l}{2}+\tfrac{d-2}{4}$\\
				\hline
		\end{tabular}}
		\caption{Noise-regularization regimes I, II, and III.}
		\label{table:regularity-regimes}
	\end{table}

	\begin{theorem}[Global well-posedness and scattering via regularization by noise]
		\label{thm:RegNoise}
		Consider the stochastic Zakharov system \eqref{eq:StoZak}
		with  an infinite dimensional
		Wiener process $W_1$ 
		as in \eqref{Noise} with spatially constant, complex-valued coefficients
		$\{\phi_k^{(1)}\}\subseteq l^2(\mathbb{C})$
		satisfying 
		$(\sum_{k=1}^\infty |\Im(\phi_k^{(1)})|^2)^{\frac12} > 0$, 
		and $W_2 \equiv 0$. Set $c_k:= \Im \phi_k^{(1)}$ and $\mathbf{c}:=\{c_k\}\in l^2(\mathbb{R})$.
		Let $d \geq 4$ and $(s,l)$ satisfy \eqref{IniReg-conditionNoiseReg}.
		Then, for any deterministic initial data
		$(X_0, Y_0) \in H^s_x \times H^l_x$,
		we have
		\begin{align}\label{prop:Scatter}
			\bbp ((X(t), Y(t)) \text{ scatters as } t \rightarrow \infty) \longrightarrow 1,\ \ \text{as} \ \|\mathbf{c}\|_{l^2}\to \infty,
		\end{align}
		where $(X,Y)$ denotes the solution of~\eqref{eq:StoZak},
		and ``scatters'' means that there exists $(z_+,v_+)\in H^s_x \times H^l_x$ such that
		\begin{align}  \label{scatter-XY}
			\lim_{t \rightarrow \infty} \|e^{-\imu t \Delta} e^{\wh \mu t - W_1(t)} X(t) - z_+\|_{H_x^s} = 0\ \  \text{and} \ \ \lim_{t \rightarrow \infty} \|e^{-\imu t |\nabla|} Y(t) - v_+\|_{H_x^l} = 0,
		\end{align}
		where
		\begin{align} \label{def-whmu}
			\wh \mu := \frac{1}{2} \Big(\sum_{k=1}^\infty |\phi_k^{(1)}|^2 - \sum_{k=1}^\infty (\phi_k^{(1)})^2\Big).
		\end{align}
	\end{theorem}

	\begin{remark} \label{rem:RegNoise}
		\begin{enumerate}
			\item
			Theorem~\ref{thm:RegNoise} provides the first global existence and scattering results for large data in the energy-supercritical setting $d>4$ and in the energy-critical setting $d = 4$ with data outside the energy space. 
			We note that in these regimes, 
			the global existence and long-time dynamics are largely open in the deterministic case. 
			In the 4D energy-critical case,
			the noise-regularization-effect on scattering for the Zakharov system with data in the energy space  $(s,l)=(1,0)$ was first proved in~\cite{HRSZ24}.

			\item
			The noise-regularization effect 
				in Theorem~\ref{thm:RegNoise} 
				is proved up to the local well-posedness borderlines
				(i.e., $s=l+2$,
				$l = \frac{d-4}{2}$, and
				$s=\frac l2 + \frac{d-2}{4}$), 
				as well as the borderline 
				$s=l-\frac 12$ 
				dictated by the regularity threshold of the noise.

			\item 
				The noise-regularization regime in Theorem \ref{thm:RegNoise} 
				is divided into three subregimes (see Figure \ref{picture1.2}), 
				reflecting the distinct analytical challenges and tailored techniques required for each. 
				In order to reach the borderlines,
				we fully exploit the global-in-time $B^s_{p,\infty}$-Besov regularity of geometric Brownian motions derived in this work.  
				Moreover, exploiting the fast decay of geometric Brownian motions in these norms, 	we use the two-temporal-regimes argument very recently developed in the context of critical stochastic NLS~\cite{SZZ25} to analyze the solution trajectory in the
				small and large time regimes separately. This approach is a major departure from the energy space in dimension four~\cite{HRSZ24}, but necessary to push the regularization-by-noise regime to~\eqref{IniReg-conditionNoiseReg}. 
				We refer to  Subsection~\ref{subsec:RegNoise}
				for more details.
		\end{enumerate}	
	\end{remark}

	The following subsections provide the necessary background, contextualize our main results, and outline the primary difficulties and our proof strategy.
	We treat the local well-posedness
	and noise-regularization effect separately
	in Subsections \ref{subsec:lwp}
	and \ref{subsec:RegNoise},
	respectively.

	\subsection{Local well-posedness and blow-up alternative}
	\label{subsec:lwp}
	
	\subsubsection{Deterministic Zakharov system}
	\label{subsubsec:DetZak}
	
	We concentrate on the development of the local well-posedness theory in this paragraph. 
	
	In \cite{BC96},
	Bourgain and Colliander proved the local well-posedness in the energy space in dimensions two and three. 
	To obtain a local well-posedness result in such a low regularity setting, it is crucial to exploit the resonance structure of the nonlinear interaction,
	for which the $X^{s,b}$-spaces were used in \cite{BC96}.
	This approach was extended by Ginibre, Tsutsumi, and Velo \cite{GTV97} in all space dimensions,
	where the admissible range  for the local well-posedness
	in $H_x^s \times H_x^l$
	was studied.
	The local well-posedness regime there was significantly extended by Bejenaru, Guo, Herr, and Nakanishi~\cite{BGHN15},
	by using a normal form approach to control the nonlinear interactions. However, in part due to boundary terms arising from the normal form transform,  which are difficult to control in a low regularity setting, the local well-posedness regime
	in \cite{BGHN15} was still not sharp.
	
	The sharp range of exponents $(s,l) \in \R^2$ for which local well-posedness holds in dimension $d \geq 4$ was finally determined by Candy, Herr, and Nakanishi in~\cite{CHN23}.
	It was proved that the deterministic Zakharov system is locally well-posed with analytic flow map if $(s,l)$ satisfies~\eqref{IniReg-condition}, while the flow map is not $C^2$ otherwise. The key for obtaining the sharp local well-posedness regime
	is a new functional framework including adapted Fourier restriction spaces $S^{s,a,b}$ which contain additional weighted temporal derivatives.

	Considering a high-low wave-Schr{\"o}dinger interaction in the small temporal frequency regime, it was realized in~\cite{CHN23} that it seems unlikely that solutions to the Zakharov system can be constructed in the full well-posedness regime by iterating in the endpoint Strichartz space. For an insightful explanation of this idea, see the introduction of~\cite{CHN23}. The solution to this problem was to introduce suitable weights of power $a$ in space-time frequency, where $a$ measures a loss of regularity in the small temporal frequency regime, roughly speaking. To close the nonlinear estimates, a second parameter $b$ was necessary, which allows to gain regularity in the high modulation regime. See Section~\ref{sec:FunctFramework} below for the precise definition of the function spaces with the parameters $a$ and $b$.

	\subsubsection{Stochastic Zakharov system}
	The Zakharov system
	in the stochastic setting is far less studied than its deterministic counterpart.
	Tsutsumi~\cite{Ts22}
	first considered the Zakharov system with additive noise in
	dimension one and proved global well-posedness. His approach relied on $X^{s,b}$-spaces with $b < \frac12$, a difference to the deterministic case enforced by the limited $C^{\frac{1}{2}-}$-H{\"o}lder regularity of Wiener noise.
	Barru{\'e}, de Bouard, and Debussche \cite{BBD24} addressed the question of the subsonic limit from the stochastic Zakharov system to the stochastic NLS.
	It was proved that
	the Zakharov system with damping and additive noise in the wave equation converges to the stochastic cubic NLS with multiplicative noise \cite{BBD24}.
	
	In dimension three,
	the stochastic Zakharov system
	was studied in the energy space
	by
	Herr, R{\"o}ckner, and the first two authors~\cite{HRSZ25}.
	The proof of local well-posedness there relied on the normal form approach and  refined rescaling transforms.
	The normal form approach fails in the energy space in dimension four.
	The 4D energy-critical case
	was recently addressed in \cite{HRSZ24},
	by replacing the normal form approach with
	the adapted Fourier restriction spaces
	$S^{s,a,b}$ from~\cite{CHN23} with  parameters $a = \frac14$ and $b = 0$.

	As indicated above,
	the proof of \cite{HRSZ24} can be extended straightforwardly to other regularities where $b = 0$ is admissible. However, this functional framework breaks down in the case $b \neq 0$, 
	as the adapted spaces $S^{s,0,b}$ are too restrictive to absorb the noise-induced perturbations. This limitation prevents the previous methods from yielding the full well-posedness regime~\eqref{IniReg-condition}.

	\subsubsection{Idea of the proof}
	\label{subsec:IdeaProof}
	The main difficulty in obtaining local well-posedness results in the sharp regime~\eqref{IniReg-condition}
	including $b\not =0$
	is to control the nonlinear interactions
	between the Schr\"odinger component, the wave component, and the noise in the Zakharov system.

	To be precise, 
	in a first step,	
	we reduce the stochastic Zakharov system~\eqref{eq:StoZak} to a random system by the rescaling (or Doss-Sussman type) transforms
	\begin{align}
		& u(t):= e^{-W_1(t)}X(t),   \label{rescal.1}  \\
		& v(t):= Y(t) - \cT_{t}(W_2) \ \qquad
		\text{with} \ \cT_t(W_2):= - \imu \int_0^t e^{\imu (t-s)|\na|} \dd W_2(s).   \label{rescal.2}
	\end{align}
	The resulting random system reads
	\begin{equation}   \label{eq:RanZakW1W2}
		\left\{\aligned
		\imu \partial_t u + e^{-W_1}\Delta (e^{W_1} u) &= \Re(v) u + \Re(\cT_t(W_2)) u,  \\
		\imu \partial_t v + |\na |v  &= - |\na||u|^2, \\
		(u(0), v(0)) &= (X_0, Y_0),
		\endaligned
		\right.
	\end{equation}
	or equivalently,
	\begin{equation}   \label{eq:RanZakbc}
		\left\{\aligned
		\imu \partial_t u + \Delta u
		&= \Re(v) u - b\cdot \na u - cu + \Re(\cT_t(W_2)) u,  \\
		\imu \partial_t v + |\na |v  &= - |\na||u|^2, \\
		(u(0), v(0)) &= (X_0, Y_0),
		\endaligned
		\right.
	\end{equation}
	where the coefficients $b$\footnote{To maintain consistency with the established literature, we use the letter $b$ to denote both the noise coefficient and the parameter in the adapted spaces $S^{s,a,b}$. There is no ambiguity, see Remark~\ref{rem:Notationb} for details.} and $c$ are of the form
	\begin{align}
		b &= 2 \na W_1 = 2 \imu \sum\limits_{k=1}^\infty \na \phi^{(1)}_k \beta^{(1)}_k,  \label{eq:Defb}  \\
		c &= -|\na W_1|^2 + \Delta W_1
		= - \sum\limits_{j=1}^d \Big(\sum\limits_{k=1}^\infty
		\partial_j \phi_k^{(1)} \beta^{(1)}_k \Big)^2
		+ \imu \sum\limits_{k=1}^\infty \Delta \phi^{(1)}_k \beta^{(1)}_k.   \label{eq:Defc}
	\end{align}
	The equivalence between systems \eqref{eq:StoZak} and~\eqref{eq:RanZakbc}
	was first proved in Theorem~3.1 of~\cite{HRSZ25} in dimension three and for $(s,l)=(1,0)$.
	The arguments there apply to the case $d \geq 4$ and general $(s,l)$.  
	
	The rescaling transform builds on earlier applications in the theory of stochastic NLS (see, e.g., \cite{BRZ14, BRZ16, BRZ17a}).
	It 
	allows us to analyze the solution trajectories
	in a sharp pathwise manner,
	by using
	advanced techniques from harmonic analysis and dispersive equations that are not directly applicable to stochastic convolutions. 
	
	However, 
	the rescaling transform also comes at a cost:
	it generates the problematic {\it first-order derivative} term 
	$b\cdot \na u$
	in the Schr{\"o}dinger part of the system with coefficients depending on the noise.
	The delicate point is that
	the first-order derivative term is
	at the {\it critical regularity} for  perturbative analysis,
	and thus introduces substantial difficulties absent from the deterministic case.

	\vspace*{4pt plus 2pt minus 2pt}%
	\paragraph{\bf $\bullet$  Refinement of functional framework} 
	
	To overcome the above problem,
	a functional framework was developed in \cite{HRSZ24} 
	that combines the adapted Fourier restriction spaces from~\cite{CHN23} 
	with lateral Strichartz spaces as used in~\cite{BIKT11} for the Schr{\"o}dinger  maps problem. 
	The functional framework 
	yields control of both the nonlinear interactions and the first-order  derivative perturbations. 
	The key to this combination is the {\it compatibility}
	between these two types of spaces, which is established by proving linear estimates  between each space and the dual of the other,
	see Lemma~\ref{lem:StrichartzLocalSmooth}.
	Another crucial ingredient is the compatibility between these function spaces and  rescaling transforms,
	as proved in~\cite{HRSZ24}.

	As mentioned above,
	the method of \cite{HRSZ24} extends straightforwardly
	to adapted spaces with $b=0$ but
	fails in the case $b \neq 0$.
	One of the new contributions here is a refinement of the adapted Fourier restriction spaces 
		in which the original $S^{s,0,b}$-norm controlling
		the Schr{\"o}dinger component is replaced by a weaker $\wt{S}^{s,0,b}$-norm.
	This refined Fourier restriction norm
	is strong enough to control the nonlinear interactions (see Lemma~\ref{le:Bilinearnablauv}) but weak enough to contain the Schr{\"o}dinger evolution of the noise terms (see Lemma~\ref{le:PersisReg}).

	This new functional framework 
	enables us to construct unique local solutions by fixed point arguments 
	both at non-endpoint and endpoint regularities. The endpoint case, i.e. the lowest admissible regularity, requires a different argument than the non-endpoint case as in the deterministic setting~\cite{CHN23}.

	\vspace*{4pt plus 2pt minus 2pt}%
	\paragraph{\bf  $\bullet$  Blow-up alternative and persistence of regularity}
	
		To extend a local solution to its maximal existence time and provide a blow-up alternative, we employ the refined rescaling transforms developed in \cite{Zh20, HRSZ25, HRSZ24} in order to exploit the smallness of the increments of the noise.	
	
	In this work, 
	we derive a blow-up condition that only involves norms at the endpoint regularity. This implies in particular persistence of regularity in the full well-posedness regime, see Remark~\ref{rem:LWP}.
	
	In order to obtain the blow-up alternative
	with optimal regularity,
	we use a bootstrap type argument
	for both the regularity
	of the Schr\"odinger and the wave component.
	We refer to Subsection~\ref{Subsec:ImproBlowUpCon} for more detailed explanations.

	\subsection{Regularization by noise}
	\label{subsec:RegNoise}

	\subsubsection{Background and motivation}
	
	The phenomenon of regularization by noise describes how
	stochastic noise can lead to solutions that are better behaved than their deterministic counterparts.  
	This effect is well-known for finite-dimensional stochastic differential equations (SDEs), 
	see, e.g., \cite{KR05}. 
	In the infinite dimensional setting, 
	it has been investigated for various models, including  
	infinite-dimensional SDEs~\cite{DFRV16}, transport equations~\cite{FGP10}, 
	Navier-Stokes equations~\cite{FL21}, 
	and stochastic Hamilton-Jacobi equations \cite{GG19}.  
	We also mention the vast literature 
	on dispersive equations 
	where random initial data improve the well-posedness, see e.g.
	\cite{Bo96, DNY24, OST22, BT08, BDNY24}.
	
	Concerning stochastic dispersive models,
	it was first observed by
	numerical experiments
	in~\cite{DM02a, DM02b}
	that multiplicative noise can delay blowup
	for stochastic NLS (SNLS),
	while white noise can even prevent blowup.
	This phenomenon was explored rigorously
	for non-conservative noise.
	In the energy-subcritical SNLS setting, it was initially shown that non-conservative noise prevents finite-time blowup with high probability~\cite{BRZ17a}. Afterwards,
	it was sharpened in~\cite{HRZ19}
	to guarantee the existence of global and scattering solutions
	to subcritical SNLS, again with high probability.
	The more difficult critical case was very recently studied in \cite{SZZ25}. 
	Recently, 
	it was proved 
	in \cite{BFMZ24} that 
	adding suitable superlinear noise 
	leads to global well-posedness 
	and hence prevents blowup 
	for NLS with polynomial type nonlinearity. 
	See also \cite{DT11, Ro24} and the references therein
	for the regularization effect of random dispersion.

	The study of regularization by noise
	for the stochastic Zakharov system~\eqref{eq:StoZak} began with~\cite{HRSZ25}, 
	which demonstrated that non-conservative noise can prevent blowup on bounded time intervals with high probability.
	Afterwards, 
	a much stronger conclusion for the energy-critical case was proved in \cite{HRSZ24}: non-conservative noise guarantees global existence and scattering with high probability.
	We will discuss this result in more detail along with the strategy to prove Theorem~\ref{thm:RegNoise} in the next subsection.

	\subsubsection{Idea of the proof}
	
	The analysis of the non-conservative case is structurally different from the conservative one and requires a further rescaling transform. We now define
	\begin{align}
		\label{Rescal-noncons}
		z := e^{\wh \mu t - W_1(t)} X,  \qquad	 v := Y
	\end{align}
	with $\wh \mu$ given by \eqref{def-whmu}.
	The key distinction lies in the parameter $\wh \mu$: in the non-conservative setting of Theorem~\ref{thm:RegNoise}, its real part is positive
	\begin{align}  \label{wtmu-Imphi2}
		\Re \wh \mu
		=  \sum_{k=1}^\infty (\Im \phi_k^{(1)})^2 >0,
	\end{align}
	while $\Re \wh \mu = 0$ in the conservative case considered in Theorem~\ref{thm:LocalWP}.
	The transform~\eqref{Rescal-noncons} converts the stochastic Zakharov system~\eqref{eq:StoZak} into the equivalent random system
	\begin{equation}   \label{eq:RanZakNoncons-intro}
		\left\{\aligned
		\imu \partial_t z + \Delta z &= \Re(v)z,  \\
		\imu \partial_t v + |\na |v  &= - h_{\mathbf{c}}|\na||z|^2, \\
		(z(0), v(0)) &= (X_0, Y_0) 
		\in H_x^s \times H_x^l,
		\endaligned
		\right.
	\end{equation}
	where $h_{\mathbf{c}}$ is the geometric Brownian motion
	\begin{align}  \label{h-W1-def}
		h_{\mathbf{c}}(t):= e^{2 \Re (W_1(t)- \wh \mu t)}
		= e^{-2 \sum_{k=1}^\infty \Im \phi_k^{(1)} \beta^{(1)}_k(t) - 2 \|\mathbf{c}\|^2 t}
	\end{align}
	with $\|\mathbf{c}\|^2 = \sum_{k=1}^\infty |\Im\phi_k^{(1)}|^2$.
	Note that,
	by the law of the iterated logarithm of Brownian motions,
	the geometric Brownian motion $h_{\mathbf{c}}$ decays exponentially fast as $t \rightarrow \infty$. Heuristically, this decay is expected to weaken the nonlinearity in the wave equation and thereby stabilize the system.
	
	\vspace*{4pt plus 2pt minus 2pt}%
	To streamline the description of our proof strategy, we will assume $c_k=0$ for all $2 \leq k < \infty$ and set $c= c_1= \Im(\phi^{(1)}_1)$ in the following discussion, thereby easing the notation. The strategy readily extends to the general case in Theorem~\ref{thm:RegNoise}.
	
	\vspace*{4pt plus 2pt minus 2pt}%
	\paragraph{\bf  $\bullet$ Two-temporal-regimes argument}
	
	The above heuristic arguments can be worked out
	directly on $\R_+$ 
	for subcritical SNLS, see \cite{BRZ17a,HRZ19},
	and for the energy-critical Zakharov 
	system in the energy space~\cite{HRSZ24}. 
	However, 
	the arguments there are insufficient to 
	achieve the full noise-regularizagion regime 
	in Theorem \ref{thm:RegNoise}. 
	
	Instead,  
	we use the two-temporal-regimes argument
	very recently introduced by the authors \cite{SZZ25} for critical SNLS. 
	A key observation is that  
	the geometric Brownian motion decays rapidly and uniformly after the short time $c^{-1}$ with high probability. This allows us to solve~\eqref{eq:RanZakNoncons-intro} in a two-step procedure. 
	Intuitively, 
	in the short time regime $[0,c^{-1})$ 
	one can solve  \eqref{eq:RanZakNoncons-intro} 
	by using the local well-posedness theory, while in the large time regime $[c^{-1},\infty)$ 
	the geometric Brownian motion 
	is very small, and so, 
	the system shall behave close to the linear propagation 
	which exists globally and scatters at infinity.
	
	In order to work out the above strategy, 
	we fully exploit the temporal regularity of  geometric Brownian motions. 
	One delicate point is to ensure the 
	uniform bound with respect to $c$ for the noise term. 
	This issue relates to the 
	scaling-(sub)criticality of geometric Brownian motions.

	\vspace*{4pt plus 2pt minus 2pt}%
	\paragraph{\bf  $\bullet$ Scaling issue}
	
	A naive choice for the space to control the noise 
	is the usual H\"older space  $C^{\frac{1}{2}-}(\R_+)$. 
	However, 
	it is scaling-supercritical with respect to $c$, 
	that is, 
	\begin{align}\label{intro:ScaHolder}
		\|h_c\|_{\dot{C}^{\alpha}([0,c^{-1}])} \overset{d}{=} c^{2\alpha} \|h_1\|_{\dot{C}^{\alpha}([0,c])}\longrightarrow +\infty\quad \text{ as } c\to\infty,
	\end{align} 
	which amplifies the nonlinear effect and derails the scattering of solutions.  
	Here, $\overset{d}{=}$ means equality in probability distribution.
	
	Proceeding differently, 
	we measure the regularity of 
	geometric Brownian motions 
	in the Besov spaces
	\begin{align}\label{intro:ScaBesov} 
		\|h_c\|_{\dot{B}^{s}_{p,\infty}([0,c^{-1}])} \overset{d}{=} c^{2s-\frac{2}{p}} \|h_1\|_{\dot{B}^{s}_{p,\infty}([0,c])},
	\end{align}
	which yields 
	the desired uniform bound of the noise term (with respect to $c$)
	if $s\leq \frac{1}{p}$. 
	
	In the 4D paper \cite{HRSZ24}, the regularization-by-noise result relied on control of the $L_t^6$- and $B_{6,\infty}^{\frac{1}{8}}$-norm of the noise, which are scaling-subcritical from this view point. The global $B_{6,\infty}^{\frac{1}{8}}$-bound was obtained by interpolation and global $V^p$-control. 
	
	Here, we directly prove and exploit global control of geometric Brownian motions in the full range of scaling-subcritical  $B^{\frac{1}{2}-}_{2,\infty}$-norms and in the scaling-critical $B^{\frac{1}{p}}_{p,\infty}$-norms
		with $p>2$ close to $2$ to establish global well-posedness and scattering for the full regime of regularities $(s,l)$ in \eqref{IniReg-conditionNoiseReg}.
	We further prove and use that these norms decay rapidly and uniformly after a short time, see Lemma \ref{prop:GeoBMFastDecay}.

	\vspace*{4pt plus 2pt minus 2pt}%
	\paragraph{\bf  $\bullet$  Trilinear estimates}
	
	The crucial step
	to explore the noise-regularization-effect
	is to prove a suitable {\it trilinear estimate} for the wave nonlinearity in~\eqref{eq:RanZakNoncons-intro}. In contrast to earlier work on SNLS and to the
	3D Zakharov system \cite{HRSZ25}, it is not possible to estimate the geometric Brownian motion simply by an $L^p$-norm in time. Although the noise is spatially independent, the Fourier restriction type norms yield true trilinear interactions involving the geometric Brownian motion. We also point out that these Fourier restriction type norms were tailor made to control the bilinear interactions in the Zakharov nonlinearity in the full regularity regime~\eqref{IniReg-condition}. With the additional interaction of the geometric Brownian motion, 
	it is a priori unclear which regularities one can recover in the trilinear estimate. Even the energy-space case in dimension four~\cite{HRSZ24} is highly nontrivial.
	
	The key observation in~\cite{HRSZ24} was that,
	in the most problematic interaction,
	there is a subtle {\it non-resonance identity} which allows to recover
	{\it spatial regularity} from the
	{\it temporal regularity}
	of the geometric Brownian motion. To exploit this, a global $V^p$-bound (where $V^p$ is the space of bounded $p$-variation) was derived in~\cite{HRSZ24},
	which was combined with its Besov embedding and interpolation with a suitable $L^q$-space.
	
	While this approach was sufficient for the energy-critical case, it is insufficient to obtain the
	noise-regularization-effect in the 
	full 
	noise-regularization regime. 
	In the proof of Theorem~\ref{thm:RegNoise}, we thus extend the strategy from~\cite{HRSZ24} in several directions.

	As mentioned above, we classify the noise-regularization regime in three subregimes with distinct analytical challenges.

	\vspace*{4pt plus 2pt minus 2pt}%
	{\bf Noise-regularization regime I: 
		scaling-subcritical Besov regularity 
		and non-resonance identity.} 
		The primary difficulty in this regime arises from high-low Schr{\"o}dinger-Schr{\"o}dinger frequency interactions in the stochastic wave nonlinearity $- h_c |\nabla| |z|^2$ where the high-frequency component has low modulation. Here we exploit subtle \emph{non-resonance identities} to uncover lower bounds for the temporal frequency either for the low-frequency Schr{\"o}dinger component or the geometric Brownian motion $h_c$ (see e.g.~\eqref{HLLM-decom} in Subsection~\ref{subsec:TrilinearRegimeI}). In the first case, we can recover two spatial derivatives from the Schr{\"o}dinger solution. In the second one, we translate the temporal regularity of the geometric Brownian motion to $1- = 2 \cdot \frac{1}{2}-$ spatial derivatives. In view of the derivative in the nonlinearity, this explains the condition $s > l$ in regime I. Similarly, one misses the borderline $s = \frac{l}{2} + \frac{d-2}{4}$ in the low wave-regularity regime. This heuristic discussion also shows that the noise is the regularity-restricting term -- in the deterministic setting, the second case does not appear and one immediately gains two spatial derivatives.

	\vspace*{4pt plus 2pt minus 2pt}%
	{\bf Noise-regularization regime II: scaling-critical Besov regularity 
		and interpolation argument.} 
		To push the theory to the borderline $s = \frac{l}{2} + \frac{d-2}{4}$, the inherent regularity loss in the scaling-subcritical Besov spaces $B^{\frac{1}{2}-}_{2,\infty}$ is insufficient. Instead, we employ the \emph{scaling-critical} spaces $B^{\frac{1}{p}}_{p,\infty}$ for $p > 2$. The switch in the temporal integrability index combined with an intricate interpolation argument relying on the algebraic structure of regime II allows us to prove trilinear estimates sharp enough to recover the missing fractional derivative that was lost in regime I (see Subsection~\ref{subsec:TrilinRegimeII}). We note that this scale-invariant argument applies to the endpoint regularity $(\frac{d-3}{2}, \frac{d-4}{2})$, which is contained in regime II.

	\vspace*{4pt plus 2pt minus 2pt}%
	{\bf Noise-regularization regime III: 
		local smoothing and maximal function estimates.} 
	In order to treat the noise-regularization regime III up to the borderline $s>l-\frac{1}{2}$, 
	we further make use of the local smoothing effect 
	of the Schr\"odinger component to gain 
	an extra $\frac{1}{2}$ derivative.  
	This in turn requires augmenting our functional framework with maximal function norms as in the context of the Schr\"odinger maps problem~\cite{BIKT11}. 
	
	More precisely, 
	we introduce a new functional framework 
	combining the adapted Fourier restriction norm, 
	the local smoothing norm, 
	and the maximal function norm. 
	The latter two spaces 
	are chosen in the non-endpoint case 
	$L^{p, q}_{\vece_j}$ and $L^{q, p}_{\vece_j}$, 
	with $(p, q)$ close to $(\infty, 2)$, 
	see Section \ref{Sec-MaximalFunctionEstimate} below for the precise definitions. 
	The local smoothing component enables us to gain almost $\frac{1}{2}$ derivatives. Moreover,  
	avoiding the endpoint in the 
	maximal function component allows us to invoke a variant of the Christ-Kiselev lemma in order to bound the Duhamel integral in the maximal function norm by the dual endpoint Strichartz norm, see Lemma~\ref{lem:MaxFctDualEndptStrichartz} below. 
	We note that, similar to the heuristic argument in regime I above, the restriction $s > l - \frac{1}{2}$ in Theorem~\ref{thm:RegNoise} is already dictated by the regularity of the noise. Consequently, moving away from the endpoint in the local smoothing and maximal function components does not further constrain the theorem's validity.

	\section{Functional framework: adapted and lateral Strichartz spaces} \label{sec:FunctFramework}
	
	In this section we introduce the functional framework to solve
	the stochastic Zakharov system. We also provide the essential estimates for the linear flows of the Schr{\"o}dinger and the wave equation in this framework.
	
	\subsection{Notation} We begin by introducing some notation we will use throughout this article. In particular, we use the symbols of a standard Littlewood-Paley decomposition in (spatial) frequency space to introduce decomposition operators in temporal frequency and modulation. The latter is taken with respect to the Schr{\"o}dinger evolution, i.e., we restrict the space-time Fourier transform with respect to the distance to the paraboloid.
	
	\paragraph{\textbf{$\bullet$ Decomposition in frequency and modulation:}}\label{notation:LPdecomposition}
	Take an even function $\eta_0 \in C_c^\infty(\R)$ such that $0 \leq \eta_0 \leq 1$, $\eta_0(r) = 1$ for $|r| \leq \frac{5}{4}$, and $\eta_0(r) = 0$ for $|r| \geq \frac{8}{5}$. For every dyadic number $\lambda \in 2^\Z$ we set
	\begin{align*}
		\eta_\lambda(r) = \eta_0(r/\lambda) - \eta_0(2r/\lambda), \qquad \eta_{\leq \lambda}(r) = \eta_0(r/\lambda)
	\end{align*}
	for all $r \in \R$. The standard Littlewood-Paley projectors are defined as the spatial Fourier multipliers
	\begin{align*}
		P_\lambda = \eta_\lambda(|\nabla|) \quad \text{if } \lambda > 1, \qquad P_1 = \eta_{\leq 1}(|\nabla|).
	\end{align*}
	Hence, $P_\lambda$ localizes the spatial Fourier support to the set $\{\lambda/2 < |\xi| < 2 \lambda\}$ if $\lambda > 1$,
	and to the set $\{|\xi| < 2\}$ if $\lambda = 1$.
	
	Further, we define the temporal frequency and modulation projectors by
	\begin{align*}
		P^{(t)}_\lambda = \eta_\lambda(|\partial_t|), \qquad C_\lambda = \eta_\lambda(| \imu \partial_t + \Delta|), \qquad  \lambda \in 2^\Z.
	\end{align*}
	Consequently, $P^{(t)}_\lambda$ localizes temporal frequencies around $\lambda$,
	while $C_\lambda$ localizes the space-time Fourier support to distances of size $\lambda$ from the paraboloid. We also set
	\begin{align*}
		P_{\leq \lambda} = \eta_{\leq \lambda}(|\nabla|), \qquad P^{(t)}_{\leq \lambda} = \eta_{\leq \lambda}(|\partial_t|), \qquad C_{\leq \lambda} = \eta_{\leq \lambda}(| \imu \partial_t + \Delta|),
	\end{align*}
	as well as $P_{> \lambda} = I - P_{\leq \lambda}$, $P^{(t)}_{> \lambda} = I - P^{(t)}_{\leq \lambda}$, and $C_{> \lambda} = I - C_{\leq \lambda}$.
	Let
	\begin{align*}
		\tilde{P}_\lambda = P_{\frac{\lambda}{2}} + P_\lambda + P_{2 \lambda}
	\end{align*}
	denote the fattened Littlewood-Paley projectors, and correspondingly for the temporal frequency and the modulation projectors. Sometimes, we also use the shorthand $f_\lambda = P_\lambda f$ for the sake of brevity.
	
	\paragraph{\textbf{$\bullet$ Sobolev and Besov spaces:}}
	We will employ the standard Besov and Sobolev spaces,
	which are defined as the sets of tempered distributions such that the following norms are finite.
	The inhomogeneous and homogeneous Sobolev spaces $W^{s,p}$ and $\dot{W}^{s,p}$
	are defined, respectively, via the norms
	\begin{align*}
		\| f \|_{W^{s,p}} = \| \langle \nabla \rangle^s f \|_{L^p} \qquad \text{and} \qquad \| f \|_{\dot{W}^{s,p}} = \| | \nabla |^s f \|_{L^p}.
	\end{align*}
	The inhomogeneous and homogeneous Besov spaces $B^{s}_{p, q}$ and $\dot{B}^{s}_{p,q}$ are
	defined, respectively, via the norms
	\begin{align*}
		\| f \|_{B^s_{p,q}} = \Big(\sum_{\lambda \in 2^{\N_0}} \lambda^{s q} \| P_\lambda f\|_{L^p}^q \Big)^{\frac{1}{q}} \qquad \text{and} \qquad
		\| f \|_{\dot{B}^s_{p,q}} = \Big(\sum_{\lambda \in 2^{\Z}} \lambda^{s q} \| \dot{P}_\lambda f\|_{L^p}^q \Big)^{\frac{1}{q}},
	\end{align*}
	where  $\dot{P}_\lambda = \eta_\lambda(|\nabla|)$ $(\lambda \in 2^\Z)$ denote the homogeneous Littlewood-Paley projectors. As usual we write $H^s = W^{s,2}$ and $\dot{H}^s = \dot{W}^{s,2}$.
	
	We also note that $C_\lambda P_\lambda$, $C_{\leq \lambda} P_\lambda$, etc.\ are convolution operators with kernels bounded in $L^1( \R^{1+d})$
	independent of $\lambda$, and thus these operators are bounded on all $L^q_t L^p_x$, $L^q_t W^{s,p}_x$, and $L^q_t B^s_{p,r}$ spaces uniformly in $\lambda$.
	
	\paragraph{\textbf{$\bullet$ Paraproduct decomposition:}}
	To distinguish different frequency interactions, we use the standard paraproduct decomposition
	\begin{align*}
		f g = (f g)_{LH} + (f g)_{HH} + (f g)_{HL},
	\end{align*}
	where the low-high, high-high, and high-low interactions are defined, respectively, as
	\begin{align*}
		(f g)_{LH} = \sum_{\lambda} P_{\leq \frac{\lambda}{2^8}} f P_\lambda g, \qquad
		(f g)_{HH} = \sum_{\lambda_1 \sim \lambda_2} P_{\lambda_1} f P_{\lambda_2} g, \qquad
		(f g)_{HL} = (g f)_{LH}.
	\end{align*}
	Here, we sum over $\lambda \in 2^{\N}$ with $\lambda \geq 2^8$ in the first sum and over all $\lambda_1, \lambda_2 \in 2^{\N_0}$ with $| \log(\lambda_1/\lambda_2) | \leq 7$ in the second sum. 
	
	We will also use the suggestive notation
		\begin{align*}
			P_{\ll \lambda}: = P_{\leq \frac{\lambda}{2^8}},\quad P_{\ll \lambda^2} := P_{\leq (\frac{\lambda}{2^8})^2},
		\end{align*}
		and $P_{\gtrsim \lambda}$ (resp. $P_{\lesssim \lambda}$), where the latter denotes the corresponding localization operator with respect to the Fourier support $\{\xi \in \R^d \colon |\xi| \geq c\lambda\}$ (resp. $\{\xi \in \R^d \colon |\xi|\leq c\lambda\}$) for some $c>0$. Moreover, $P_{\sim \lambda}:= P_{\gtrsim \lambda}P_{\lesssim \lambda}$.
	
	\paragraph{\textbf{$\bullet$ Homogeneous and inhomogeneous Schr{\"o}dinger and wave propagators:}}
	We write $\cI_0[g]$ for the solution of the inhomogeneous Schr\"odinger equation
	\begin{align*}
		(\imu \partial_t + \Delta) u = g, \qquad u(t_0) = 0,
	\end{align*}
	and $\cJ_0[h]$ for the solution of the inhomogeneous wave equation
	\begin{align*}
		(\imu \partial_t + |\nabla|) v = h, \qquad v(t_0) = 0,
	\end{align*}
	i.e., in the Duhamel form
	\begin{equation}
		\label{eq:DefPropOp}
		\cI_0[g](t) = -\imu \int_{t_0}^t e^{\imu (t-s) \Delta} g(s) \dd s, \qquad
		\cJ_0[h](t) = -\imu \int_{t_0}^t e^{\imu (t-s)|\nabla|} h(s) \dd s.
	\end{equation}
	We further consider the Schrödinger equation with potential $V$
	\begin{align}
		\label{eq:SchrPotentialIntro}
		(\imu \partial_t + \Delta - V) u = g, \qquad u(t_0) = f.
	\end{align}
	If unique solutions to this equation exist, we will denote the homogeneous and inhomogeneous propagators by $\cU_V[f]$ and $\cI_V[g]$, respectively, i.e., the former denotes the solution of~\eqref{eq:SchrPotentialIntro} with $g = 0$ and the latter the solution of~\eqref{eq:SchrPotentialIntro} with $f = 0$.
	We omit the dependence on $t_0$ in the labeling of these operators. The considered $t_0$ will always be clear from the context.

	\subsection{Function spaces for the Schr\"odinger component} \label{Subsec:FunctFrame}
	The function spaces for the Schr\"odinger component
	in the Zakharov system consists of two parts:
	the lateral Strichartz spaces
	and the adapted spaces. The latter are Fourier restriction spaces adapted to the structure of the Zakharov system.

	\vspace*{4pt plus 2pt minus 2pt}%
	\paragraph{\bf $\bullet$ Lateral Strichartz spaces}
	We first introduce the lateral Strichartz spaces in dimension $d\geq 4$.
	The lateral Strichartz spaces capture the local smoothing effect of the Schr{\"o}dinger flow and are crucial to control the problematic first-order perturbation caused by the noise in~\eqref{eq:RanZakbc}.
	
	Let  $\textbf{e} \in \Sp^{d-1}$ and $\mathcal{P}_{\vece}=\{\xi \in \R^{d} \, |\, \xi \cdot \textbf{e}=0 \}$
	with the induced Euclidean measure.
	Set
	\begin{equation} \label{Lepq-def}
		\| f \|_{L_{\textbf{e}}^{p,q}(I \times \R^d)}
		: = \Bigl( \int_{\R} \Bigl( \int_{I \times \mathcal{P}_{\textbf{e}}} |f(t, r \textbf{e} + y)|^q \dd t \dd y \Bigr)^{\frac{p}{q}} \dd r \Bigr)^{\frac{1}{p}},
	\end{equation}
	where $p, q \in [1,\infty]$,
	with the usual adaptions if $p = \infty$ or $q = \infty$.
	
	Let $\phi \in C_c^\infty(\R)$
	be a nonnegative and symmetric function
	such that $\phi(r)=0$ if $|r|\leq \frac{1}{3\sqrt{d}}$ or $|r|\geq 3$ and $\phi(r)=1$ if $\frac{1}{2\sqrt{d}} \leq |r| \leq 2$, and set $\phi_N(r)=\phi(r/N)$.
	Then,
	\begin{equation} \label{eq:DecompositionIdentity}
		\prod_{j=1}^{d}(1-\phi_N(\xi_j))=0
	\end{equation}
	for all $\xi \in \R^d$ with $N/2 < |\xi| < 2 N$.
	Set  $P_{N,\vece} :=\cF_{x}^{-1} \phi_N(\xi \cdot \vece) \cF_{x}$.
	By \eqref{eq:DecompositionIdentity},
	we have the decomposition
	\begin{equation}   \label{eq:DecompositionAngular}
		P_N f = \sum_{j=1}^{d} P_{N,\vece_j} \Big[\prod_{l=1}^{j-1} (1-P_{N,\vece_l})\Big]P_N f,
	\end{equation}
	where $\vece_1, \ldots, \vece_d$ denotes the standard orthonormal basis of $\R^d$.

	\begin{remark}
		The local smoothing estimates for the Schr{\"o}dinger flow in lateral Strichartz spaces are collected in Lemma~\ref{lem:StrichartzLocalSmooth} along with the Strichartz estimates. In particular, we note that for the linear inhomogeneous problem the solution can be controlled in a lateral Strichartz space by the right-hand side in a dual Strichartz space and vice versa.
	\end{remark}

	\paragraph{\bf $\bullet$ Adapted spaces}
	We will use the following adapted function spaces to control the Schr\"odinger-wave interactions:
	for
	$s, a, b \in \R$, $0\leq a,b\leq 1$, and $\lambda \in 2^{\N_0}$ we define
	\begin{align}
		{\|u\|_{\tilde{S}^{s,a,b}_\lambda}}& := \lambda^s \|u\|_{L^\infty_t L^2_x} + \lambda^{s - 2a} \|(\lambda + |\partial_t|)^a u\|_{L^2_t L^{2^*}_x}\nonumber \\
		&\qquad + \lambda^{s-1+b}\|C_{\geq \frac{\lambda^2}{2^3}} (\imu \partial_t + \Delta) u\|_{L_{t,x}^2} + \lambda^{s-1} \Big\| \Big( \frac{\lambda + |\partial_t|}{\lambda^2 + |\partial_t|}\Big)^a (\imu \partial_t + \Delta) u \Big\|_{L^2_{t,x}}\label{eq:DefTildeSsablambda}
	\end{align}
	and
	\begin{align}
		\label{eq:DefNsablambda}
		\|F\|_{N^{s,a,b}_\lambda} := \lambda^{s-2} \|P_{\leq(\frac{\lambda}{2^8})^2}^{(t)} F \|_{L^\infty_t L^2_x} + \lambda^s \|C_{\leq (\frac{\lambda}{2^8})^2} F \|_{L^2_t L^{2_*}_x} + \lambda^{s-1+b} \Big\| \Big( \frac{\lambda + |\partial_t|}{\lambda^2 + |\partial_t|}\Big)^a F \Big\|_{L^2_{t,x}},
	\end{align}
	where $2^*=\frac{2d}{d-2}$, $2_*=\frac{2d}{d+2}$, and $d\geq 4$.

	We note that the space $\tilde{S}^{s,a,b}_{\lambda}$ is inspired by the space $S^{s,a,b}_{\lambda}$ introduced in \cite{CHN23}
	\begin{align}
		\label{eq:DefSsablambda}
		\|u\|_{S^{s,a,b}_\lambda} &:= \lambda^s \|u\|_{L^\infty_t L^2_x} + \lambda^{s - 2a} \|(\lambda + |\partial_t|)^a u\|_{L^2_t L^{2^*}_x} + \lambda^{s-1 + b} \Big\| \Big( \frac{\lambda + |\partial_t|}{\lambda^2 + |\partial_t|}\Big)^a (\imu \partial_t + \Delta) u \Big\|_{L^2_{t,x}}.
	\end{align}

	The corresponding $\tilde{S}^{s,a,b}$- and $N^{s,a,b}$-norms are defined by the $l^2$-sum of the dyadic pieces
	$\|P_\lambda u\|_{\tilde{S}^{s,a,b}_\lambda}$ and $\|P_\lambda F\|_{N^{s,a,b}_\lambda}$, respectively, i.e.
	\begin{align*}
		\|u\|_{\tilde{S}^{s,a,b}}:= \Big( \sum_{\lambda\in 2^{\N_0}} \|u_\lambda\|_{\tilde{S}^{s,a,b}_\lambda}^2\Big)^{\frac12},\quad
		\|F\|_{N^{s,a,b}}:= \Big( \sum_{\lambda\in 2^{\N_0}} \|F_\lambda\|_{N^{s,a,b}_\lambda}^2 \Big)^{\frac12}.
	\end{align*}
	Finally, we set
	\begin{align*}
		\tilde{S}^{s,a,b}(\R) := \{u \in C(\R, H_x^s) \colon \|u\|_{\tilde{S}^{s,a,b}} < \infty\}, \qquad 
		S^{s,a,b}(\R) := \{u \in C(\R, H_x^s) \colon \|u\|_{S^{s,a,b}} < \infty\},
	\end{align*}
	while $N^{s,a,b}(\R)$ is the space of tempered distributions with finite $\| \cdot \|_{N^{s,a,b}}$-norm.
	
	As in \cite{CHN23}, we set
	\begin{align}
		a=a^*(s,l) := \begin{cases}
			\frac{3}{4}(s-l)-\frac{1}{2}& \text{if } s-l\geq 1,\\
			0& \text{if } s-l<1,
		\end{cases}
		\qquad
		b= b^*(s,l):=\begin{cases}
			0& \text{if } s-l>0,\\
			\frac{1}{2}(l-s)+\frac{1}{2}& \text{if } s-l\leq 0.
		\end{cases}\label{con:ab}
	\end{align}
	
	\begin{remark} \label{rem:Notationb}
		We employ the letter $b$ to denote both the parameter in the adapted spaces and the noise term $b(t,x)$ defined in~\eqref{eq:Defb}. We adopt this abuse of notation to maintain consistency with standard conventions in the literature: see~\cite{CHN23} for the adapted spaces and~\cite{HRSZ25, HRSZ24, BRZ14, BRZ16} for the rescaling transform. Given the distinct nature of these objects, there is no risk of ambiguity, and the meaning will be clear from the context.
	\end{remark}
	
	In the case where $a\in [0,1]$,
	an application of Bernstein's inequality yields
	\begin{align}
		\label{eq:CharSsablambda}
		\|u_\lambda\|_{S^{s,a,b}_\lambda} \sim \lambda^s(\|u_\lambda\|_{L^\infty_t L^2_x} + \|C_{\leq (\frac{\lambda}{2^8})^2} u_\lambda\|_{L^2_t L^{2^*}_x}) + \lambda^{s-1+b}\Big\| \Big( \frac{\lambda + |\partial_t|}{\lambda^2 + |\partial_t|}\Big)^a (\imu \partial_t + \Delta) u_\lambda \Big\|_{L^2_{t,x}},
	\end{align}	
	see Remark~2.1 in~\cite{CHN23}.
	
	The relationship between the $S^{s,a,b}$- and $\tilde{S}^{s,a,b}$-norms in the special cases $b = 0$ and $a = 0$ is given respectively by the inequalities
	\begin{align}\label{prop:RelationBetweenTwoS}
		\|u\|_{\tilde{S}^{s,0,b}}\lesssim \|u\|_{S^{s,0,b}}, \quad
		\|u\|_{S^{s,a,0}}\lesssim \|u\|_{\tilde{S}^{s,a,0}}.
	\end{align}
	Both estimates immediately follow from the definition of the norms.

	\begin{remark}
		$(i)$  	
		In the 4D energy-critical case \cite{HRSZ24} only the energy space $(s,l)=(1,0)$ is considered, where $b=0$.
		Here we consider the sharp regularity regime given by \eqref{IniReg-condition}, which in particular requires to treat regularities with $b \neq 0$.
		
		$(ii)$
		Compared to the $S^{s,a,b}_\lambda$-norm, the $\tilde{S}_{\lambda}^{s,a,b}$-norm splits the $L^2_{t,x}$-contribution. The part gaining regularity via the $b$-parameter contains an additional restriction to high modulation. This restriction allows us to control the most difficult low-high interaction of $b\cdot\nabla u$ by translating temporal regularity of the noise to spatial regularity in the case $a = 0$ and $b \neq 0$, see Lemma~\ref{le:PersisReg}, which ultimately results in an improvement of the wave regularity.
		 
		More precisely,
		since the linear estimate in Lemma~\ref{lem:LinEstimates} only applies in the case $b = 0$,
		we first solve the Zakharov system \eqref{eq:RanZakbc}
		in a suitable space $\X^{s,a}$
		and then apply the improvement of regularity
		result from Section~\ref{Sec-Persi} to upgrade the local well-posedness to regularities with $b\neq0$.

	\end{remark}

	\begin{remark}  \label{rem:NormComp}
		We note that
		\begin{align*}
			\|u_{\lambda}\|_{L_t^2 L_x^{2^*}}\lesssim \lambda^{a-s}\lambda^{s-2a}\|(\lambda+|\partial_t|)^a u_{\lambda}\|_{L_t^2 L_x^{2^*}}\lesssim \lambda^{a-s} \|u_{\lambda}\|_{S_{\lambda}^{s,a,0}}
		\end{align*}
		and
		\begin{align*}
			\lambda^{\frac{d-3}{2}}\|u_{\lambda}\|_{L_t^2 L_x^{2^*}}\lesssim \lambda^{\frac{d-3}{2}-s} \|u_{\lambda}\|_{S_{\lambda}^{\frac{d-3}{2},0,0}},
		\end{align*}	
		together with the fact $s>a$ and the Littlewood-Paley theorem imply that
		\begin{align}
			\label{eq:DispersiveBdAdapted}
			\|u\|_{L_t^2 L_x^{2^*}}\lesssim \|u\|_{S^{s,a,0}},\quad \|u\|_{L_t^2 W_x^{\frac{d-3}{2},2^*}}\lesssim \|u\|_{S^{\frac{d-3}{2},0,0}}.
		\end{align}
		We will employ these two inequalities throughout this paper.
	\end{remark}
	
	Finally, we will also use the variant
		\begin{align}
			&\|u\|_{S_w^{s,0,0}}:= \|u\|_{L_t^\infty H_x^s} + \|u\|_{L_t^2 W_x^{s,2^*}} + \|(\imu \partial_t +\Delta) u\|_{L_t^2 H^{s-1}_x} \label{eq:WeakNormS}
		\end{align}
		of the $S^{s,a,b}$-norm, which has slightly weaker summability properties. This weaker norm is needed to close the bilinear estimates in the endpoint case, see Lemma~\ref{lem:BilinearEstimates}~\ref{it:BilinearEstEndpoint1} and~\ref{it:BilinearEstEndpoint2} below.
	
	\paragraph{\bf $\bullet$ Space for the Schr\"odinger component}
	We control the Schr{\"o}dinger component of the stochastic Zakharov system
	in the following $\Xs$ space, combining the adapted and lateral Strichartz spaces
	\begin{align*}
		& \|u\|_{\Xs_\lambda} := \|u\|_{\Ssa_\lambda} + \sum_{j = 1}^d \lambda^{s + \frac{1}{2}} \|P_{\lambda, \vece_j} \ModN u \|_{L^{\infty,2}_{\vece_j}} \quad \text{if } \lambda > 1, \qquad \|u\|_{\Xs_\lambda} := \|u\|_{\Ssa_\lambda} \quad \text{if } \lambda = 1,
	\end{align*}
	and
	\begin{align*}
		\|u\|_{\Xs} := \Big( \sum_{\lambda \in 2^{\N_0}} \|u_\lambda\|_{\Xs_\lambda}^2\Big)^{\frac{1}{2}}.
	\end{align*}
	The right-hand side of the Schr{\"o}dinger equation is controlled via
	\begin{align*}
		\|F\|_{\Gs} := \Big(\| P_1 F\|_{\Nsa_1}^2 + \inf_{F = F_1 + F_2} \Big(\sum_{\lambda \in 2^\N} \|P_\lambda F_1\|_{\Nsa_\lambda}^2 + \sum_{j = 1}^d \sum_{\lambda \in 2^{\N}} \lambda^{2s - 1} \|P_\lambda F_2\|_{L^{1,2}_{\vece_j}}^2 \Big)\Big)^{\frac{1}{2}}.
	\end{align*}
	Let 
	\begin{align*}
		\Xs(\R) := \{ u \in C(\R, H_x^s) \colon \| u \|_{\Xs} < \infty\},
	\end{align*}
	and $\Gs(\R)$ be the set of tempered distributions with finite $\| \cdot \|_{\Gs}$-norm.
	
	We also localize the function spaces above to intervals $I \subseteq \R$ via restriction. For example, we define
	\begin{equation}
		\label{eq:DefRestrictionNorm}
		\|u\|_{\Xs(I)} = \inf_{u' \in \Xs(\R), u'_{|I} = u} \|u'\|_{\Xs(\R)}.
	\end{equation}

	\begin{remark}
		We distinguish between high and low frequencies in the $\Xs_\lambda$-norm
		since the local smoothing estimate is only available for high frequencies,
		which is sufficient to control the problematic derivative terms caused by the noise in~\eqref{eq:RanZakbc}.
	\end{remark}

	The following lemma provides the crucial compatibility-type estimates
	between the lateral Strichartz and adapted spaces, and allows us to control the linear Schr{\"o}dinger flow in the $\Xs$-space. The proof follows analogous arguments as in the case $(s,a) = (1,\frac14)$ treated in \cite[Lemma~3.3]{HRSZ24}
	and is thus omitted.
	
	\begin{lemma}[Control of linear Schr\"odinger flows in $\Xs$-spaces]
		\label{lem:LinEstimates}
		Let $s\in \R$, $0\leq a\leq 1$, $f \in H_x^s$, $g \in \Gs$, and $u$ solve the linear Schr{\"o}dinger equation
		\begin{align*}
			\imu \partial_t u + \Delta u = g, \qquad u(t_0) = f.
		\end{align*}
		Then
		\begin{align*}
			\|u\|_{\Xs} \lesssim \|f\|_{H_x^s} +  \|g \|_{\Gs}.
		\end{align*}
	\end{lemma}
	
	\vspace*{4pt plus 2pt minus 2pt}%
	At last,
	Lemma \ref{lem:DecompX} below shows the decomposability of the $\XS$-space,
	which is used in the gluing procedure
	when extending local solutions to the maximal existence time.
	\begin{lemma}[Decomposability]
		\label{lem:DecompX}
		Let $I_1, I_2 \subseteq \R$ be open intervals such that $I_1 \cap I_2 \neq \emptyset$. If $u$ belongs to $\XS(I_1) \cap \XS(I_2)$, then $u \in \XS(I_1 \cup I_2)$ and
		\begin{align*}
			\|u\|_{\XS(I_1 \cup I_2)} \lesssim (1 + |I_1 \cap I_2|^{-\frac{1}{2}}) (\|u\|_{\XS(I_1)} + \|u\|_{\XS(I_2)}).
		\end{align*}
	\end{lemma}
	The lemma follows in the same way as the special case $(s,a) = (1,\frac14)$ treated in~\cite[Lemma~A.1]{HRSZ24}.

	\subsection{Function space for the wave component}
	The wave component is controlled with the norm
	\begin{align*}
		& \|v\|_{W^{l,\alpha,\beta}_\lambda} = \lambda^l \|v\|_{L^\infty_t L^2_x}
		+ \lambda^{l - \alpha} \|(\lambda + |\partial_t|)^{\alpha} \TempN v\|_{L^\infty_t L^2_x}
		+ \lambda^{\beta - 1} \|(\imu \partial_t + |\nabla|) v\|_{L^2_{t,x}}.
	\end{align*}
	Eventually, we will choose $\alpha = a$ and $\beta = s - \frac{1}{2}$ as in the deterministic case~\cite{CHN23}.
	Let
	\begin{align*}
		W^{l,\alpha,\beta}(\R) := \{v \in C(\R, H_x^l) \colon \| v \|_{W^{l,\alpha, \beta}} < \infty\}.
	\end{align*}
	
	One has the following bound for the linear half-wave flow in the $W^{l,\alpha,\beta}$-space.
	
	\begin{lemma}   [Control of linear wave flows in $W_\lambda^{l,\alpha,\beta}$-space,
		\cite{CHN23} Lemma~2.6]
		\label{lem:LinearEstimateHalfWave}
		Let $0\leq \alpha \leq 1$, $\beta,l\in \R$. Then,
		\begin{align*}
			\|e^{\imu t |\nabla|} g_\lambda \|_{W_\lambda^{l,\alpha,\beta}} \lesssim \lambda^l\|g_\lambda\|_{L^2}
		\end{align*}
		for all $\lambda \in 2^{\N_0}$ and $g\in L^2(\R^d)$.
	\end{lemma}
	
	Also for the wave component we need the weaker variant
	\begin{align}
		\label{eq:WeakNormW}
		&\|v\|_{W_w^{l,0,\beta}}:= \|v\|_{L_t^\infty H_x^l} + \|(\imu \partial_t + |\nabla|)v\|_{L_t^2 H_x^{\beta-1}}
	\end{align}
	of the $W^{l,0,\beta}$-norm to close the nonlinear estimates in the endpoint case, cf.~\eqref{eq:WeakNormS}.

	We conclude this section with the continuity in the endpoint of the time interval of the adapted norm restricted to said interval. This continuity property will be used in the proof of the local well-posedness in Theorem \ref{Thm-LWPNEP}. The proof follows from an analogous argument as in the proof of~\cite[Lemma~C.1]{HRSZ24}.
	\begin{lemma}[Continuity of adapted space ]
		\label{it:ContSumY}
		Let $T>0$, $(s,l)$ satisfy \eqref{IniReg-condition},  $a$ be as in \eqref{con:ab}, and $ \beta =s-\frac12$. If $v\in W^{l,a,\beta}([0, T))$, then the map
		\begin{align*}
			t\mapsto \|v\|_{W^{l,a,\beta}([0,t])+L^2 ([0,t]; W_x^{s,d})}
		\end{align*} is continuous on $(0,T)$.
	\end{lemma}

	\section{Control of nonlinearity and noise}\label{Sec-Esti-Nonl-Noise}

	\subsection{Control of the wave nonlinearity}  \label{Subsec-Esti-Nonl}

	The next lemma shows that the wave nonlinearity is controlled in the adapted space $W^{l,a,\beta}$ by the Schr{\"o}dinger component in the $\tilde{S}^{s,a,b}$-norm. To obtain local well-posedness for regularities from~\eqref{IniReg-condition} with $l \geq s$, i.e., where the parameters $b \neq 0$ and $a = 0$, it is crucial to control the nonlinearity by the weaker $\tilde{S}^{s,0,b}$-norm (compared to the $S^{s,0,b}$-norm in the deterministic setting) as the stochastic terms only allow us to bound the Schr{\"o}dinger component in the $\tilde{S}^{s,0,b}$-norm.
	
	\begin{lemma}\label{le:Bilinearnablauv}
		Let $d\geq 4, s,l,\beta\geq 0$ and $0\leq a, b \leq 1$ satisfy
		\begin{align*}
			\beta\leq \min \{s, 2s-\tfrac{d-2}{2}-a \},\quad 2a\leq 2s-l-\tfrac{d-2}{2},\quad a-b\leq s-l
		\end{align*}
		and
		\begin{align*}
			(s,l)\neq (\tfrac{d}{2}, \tfrac{d+2}{2}), (\tfrac{d-2}{2}+a, \tfrac{d-2}{2}+b),\quad (s,\beta)\neq (\tfrac{d-2}{2}+a, \tfrac{d-2}{2}+a).
		\end{align*}
		Then for any $\varphi, \psi\in \tilde{S}^{s,a,b}$, we have
		\begin{align}\label{eq:Bilinearnablauv}
			\|\cJ_0(|\nabla|(\overline{\varphi}\psi))\|_{W^{l,a,\beta}}\lesssim \|\varphi\|_{\tilde{S}^{s,a,b}}\|\psi\|_{\tilde{S}^{s,a,b}}.
		\end{align}
		
	\end{lemma}

	\begin{remark}
		As explained above, control by the $\tilde{S}^{s,a,b}$-norm is key to treat the regularity regime where $b \neq 0$. In this regime, the $\tilde{S}^{s,a,b}$-norm is weaker than the $S^{s,a,b}$-norm so that Lemma~\ref{le:Bilinearnablauv} and~\eqref{prop:RelationBetweenTwoS} yield
		\begin{align*}
			\|\cJ_0(|\nabla|(\overline{\varphi}\psi))\|_{W^{l,a,\beta}}\lesssim \|\varphi\|_{{S}^{s,a,b}}\|\psi\|_{{S}^{s,a,b}},
		\end{align*}
		i.e., the statement of Theorem~4.1 of \cite{CHN23}.
		Lemma~\ref{le:Bilinearnablauv} thus sharpens Theorem $4.1$ of \cite{CHN23} in the regime $b \not = 0$.
	\end{remark}

	\begin{proof}[Proof of Lemma \ref{le:Bilinearnablauv}]		
		Analogous to the proof of Theorem~4.1 in~\cite{CHN23}, we first apply the energy inequality in Lemma~2.6 of~\cite{CHN23} to reduce the assertion to showing the bounds
		\begin{align}
			\|P_{\leq 2^{16}}(\overline{\varphi}\psi)\|_{L_t^1 L_x^2} \lesssim \|\varphi\|_{\tS}\|\psi\|_{\tS}\label{prop:temesti1},\\
			\Big( \sum_{\lambda\in 2^\N} \lambda^{2(l-1)} \|P_\lambda(\overline{\varphi}\psi)\|_{L_t^\infty L_x^2}^2 \Big)^{\frac12} \lesssim \|\varphi\|_{\tS}\|\psi\|_{\tS}\label{prop:temesti2},\\
			\Big( \sum_{\lambda\in 2^\N} \lambda^{2\beta} \|P_\lambda(\overline{\varphi}\psi)\|_{L_{t,x}^2}^2 \Big)^{\frac12} \lesssim \|\varphi\|_{\tS}\|\psi\|_{\tS}\label{prop:temesti3},\\
			\Big( \sum_{\lambda\in 2^\N} \lambda^{2(l-a+1)} \|(\lambda+|\partial_t|)^a P_\lambda P^{(t)}_{\ll \lambda^2}(\overline{\varphi}\psi )\|_{L_t^1 L_x^2}^2 \Big)^{\frac12} \lesssim \|\varphi\|_{\tS}\|\psi\|_{\tS}\label{prop:temesti4}.
		\end{align}
		
		For \eqref{prop:temesti1}, a simple application of Bernstein's and H{\"o}lder's inequalities yields
		\begin{align*}
			\|P_{\leq 2^{16}}(\overline{\varphi}\psi)\|_{L_t^1 L_x^2} \lesssim \|\overline{\varphi}\psi\|_{L_t^1 L_x^{\frac{d}{d-2}}}\lesssim \|\varphi\|_{L_t^2 L_x^{2^*}} \|\psi\|_{L_t^2 L_x^{2^*}}.
		\end{align*}
		Since $s>a$, the estimate
		\begin{align*}
			\|\varphi_\lambda\|_{L_t^2 L_x^{2^*}} \lesssim \lambda^{-a} \|(\lambda+|\partial_t|)^a \varphi_\lambda\|_{L_t^2 L_x^{2^*}} \lesssim \lambda^{a-s} \|\varphi_\lambda\|_{\tS_\lambda}
		\end{align*}
		implies \eqref{prop:temesti1}.
		
		Moreover, the proof of Theorem $4.1$ in \cite{CHN23} gives that
		\begin{align*}
			\Big( \sum_{\lambda\in 2^\N} \lambda^{2(l-1)} \|P_\lambda(\overline{\varphi}\psi)\|_{L_t^\infty L_x^2}^2\Big)^{\frac12} + \Big( \sum_{\lambda\in 2^\N} \lambda^{2\beta} \|P_\lambda(\overline{\varphi}\psi)\|_{L_{t,x}^2}^2\Big)^{\frac12}  \lesssim \|\varphi\|_{S^{s,a,0}}\|\psi\|_{S^{s,a,0}},
		\end{align*}
		which along with the embedding $\tilde{S}^{s,a,0}\hookrightarrow S^{s,a,0}$ from~\eqref{prop:RelationBetweenTwoS}
		yields \eqref{prop:temesti2} and \eqref{prop:temesti3}.
		
		It remains to show \eqref{prop:temesti4}.
		For this purpose, we decompose
		\begin{align*}
			\overline{\varphi}\psi= (\overline{\varphi}\psi)_{HL} + (\overline{\varphi} \psi)_{HH} + (\overline{\varphi}\psi)_{LH}.
		\end{align*}
		For the $HL$ part, we have $\lambda\gg 1$. We further decompose by modulation to obtain
		\begin{align*}
			P^{(t)}_{\ll \lambda^2} P_{\lambda}(\overline{\varphi}\psi)_{HL} &= \sum_{\mu\sim\lambda} P^{(t)}_{\ll \lambda^2} (P_\mu \overline{\varphi} P_{\leq \frac{\mu}{2^8}}\psi)\\
			& = \sum_{\mu\sim\lambda} P^{(t)}_{\ll \lambda^2} (\overline{C_{<\frac{\mu^2}{2^3}}P_{\mu}\varphi} \, P^{(t)}_{\sim\mu^2} P_{\leq \frac{\mu}{2^8}}\psi)+ \sum_{\mu\sim\lambda} P^{(t)}_{\ll \lambda^2} (\overline{C_{\geq\frac{\mu^2}{2^3}}P_{\mu}\varphi} \, P_{\leq \frac{\mu}{2^8}}\psi) =: A^{HL}_{LM}+ A^{HL}_{HM},
		\end{align*}
		In the $A^{HL}_{LM}$-contribution we exploited an important non-resonant identity: Since $C_{<\frac{\mu^2}{2^3}}P_{\mu}\overline{\varphi}$ has temporal frequency of size $\mu^2$, the temporal frequency of the second factor $P_{\leq \frac{\mu}{2^8}}\psi$ must also be of size $\mu^2$ so that the product can have temproral frequency much smaller than $\lambda^2$.
		
		For the low modulation term $A^{HL}_{LM}$, we estimate
		\begin{align*}
			&\sum_{\mu\sim\lambda} \lambda^{l+1-a} \|(\lambda +|\partial_t|)^a P_{\ll \lambda^2}^{(t)}(\overline{C_{<{\frac{\mu^2}{2^3}}} P_\mu \varphi}P^{(t)}_{\sim \lambda^2} P_{\leq \frac{\mu}{2^8}}\psi) \|_{L_t^1 L_x^2}\\
			&\lesssim \sum_{\mu\sim\lambda} \mu^{l+1+a} \|C_{<{\frac{\mu^2}{2^3}}} P_\mu \varphi \|_{L_t^2 L_x^{2^*}} \|P^{(t)}_{\sim \lambda^2} P_{\leq \frac{\mu}{2^8}}\psi\|_{L_t^2 L_x^d}\\
			&\lesssim \sum_{\mu\sim\lambda} \mu^{l+1+a}\mu^{-2a} \|(\mu+|\partial_t|)^a \varphi_{\mu }\|_{L_t^2 L_x^{2^*}} \mu^{-2} \|(\imu\partial_t+\Delta) P^{(t)}_{\sim \lambda^2} P_{\leq \frac{\mu}{2^8}}\psi\|_{L_t^2 H_x^{\frac{d-2}{2}}}\\
			&\lesssim \sum_{\mu\sim\lambda} \mu^{l-s-1+a+(\frac{d}{2}-s)_+} \|\varphi_\mu\|_{\tilde{S}^{s,a,0}} \|\psi\|_{\tilde{S}^{s,a,0}},
		\end{align*}
		which is $l^2$-summable over $\lambda\gg 1$ since $s-l\geq a-1$ and $2s-l-\frac{d-2}{2}\geq a$.
		Moreover, we infer for the high modulation term $A^{HL}_{HM}$
		\begin{align*}
			&\sum_{\mu\sim\lambda} \lambda^{l+1-a} \|(\lambda+|\partial_t|)^a P^{(t)}_{\ll\lambda^2} (\overline{C_{\geq \frac{\mu^2}{2^3}}P_{\mu}\varphi} P_{\leq \frac{\mu}{2^8}}\psi) \|_{L_t^1 L_x^2}\\
			&\lesssim \sum_{\mu\sim\lambda} \mu^{l+1+a} \| C_{\geq \frac{\mu^2}{2^3}}\varphi_{ \mu} \|_{L_{t,x}^2} \| P_{\leq \frac{\mu}{2^8}} \psi \|_{L_t^2 L_x^\infty}\\
			& \lesssim \sum_{\mu\sim\lambda} \mu^{l-1+a} \Big\|  (\imu\partial_t+ \Delta) C_{\geq \frac{\mu^2}{2^3}} \varphi_{\mu} \Big\|_{L_{t,x}^2} \sum_{\nu\leq \frac{\mu}{2^8}} \nu^{-a} \nu^{\frac{d-2}{2}} \|(\nu+|\partial_t|)^a \psi_{\nu}\|_{L_t^2 L_x^{2^*}}\\
			&\lesssim \sum_{\mu\sim\lambda}  \mu^{l-s-b+a} \|\varphi_\mu\|_{\tilde{S}^{s,a,b}} \sum_{\nu \leq \frac{\mu}{2^8}} \nu^{\frac{d-2}{2}+a-s}\|\psi_{\nu}\|_{\tS_{\nu}},
		\end{align*}
		which is $l^2$-summable over $\lambda \gg 1$ since $s-l\geq a-b, 2s-l-\frac{d-2}{2}\geq 2a-b$ and $(s,l)\neq (\frac{d-2}{2}+a,\frac{d-2}{2}+b)$.
		Then, summing up over $\lambda \gg 1$, we obtain
		\begin{align*}
			\Big( \sum_{\lambda\in 2^\N} \lambda^{2(l-a+1)} \|(\lambda+|\partial_t|)^a P_\lambda P^{(t)}_{\ll \lambda^2}(\overline{\varphi}\psi )_{HL}\|_{L_t^1 L_x^2}^2\Big)^{\frac12} \lesssim \|\varphi\|_{\tilde{S}^{s,a,b}}\|\psi\|_{\tilde{S}^{s,a,b}}.
		\end{align*}
		
		The estimate for low-high interaction $LH$ follows
		similarly as the high-low interaction $HL$. The estimate for the remaining high-high part $HH$ is contained in the proof of in \cite[Theorem 4.1]{CHN23}, which shows that
		\begin{align*}
			\Big( \sum_{\lambda \gg 1} \lambda^{2(l-a+1)} \|(\lambda+|\partial_t|)^a P_\lambda P^{(t)}_{\ll \lambda^2}(\overline{\varphi}\psi )_{HH}\|_{L_t^1 L_x^2}^2\Big)^{\frac12} \lesssim \|\varphi\|_{S^{s,a,0}}\|\psi\|_{S^{s,a,0}}.
		\end{align*}
		The embedding $\tilde{S}^{s,a,0}\hookrightarrow S^{s,a,0}$ from~\eqref{prop:RelationBetweenTwoS} thus yields the bound for the high-high component.
		Therefore, combining the above estimates, we obtain \eqref{prop:temesti4}
		and complete the proof of \eqref{eq:Bilinearnablauv}.	
	\end{proof}

	\subsection{Control of the noise terms}  \label{Subsec-Esti-Noise}
	
	We first record H{\"o}lder continuity properties of the noise terms. In contrast to the energy space case in~\cite{HRSZ24}, we have to exploit the H{\"o}lder continuity of the noise in order to control the contributions from the noise in the full well-posedness regime in Lemma~\ref{lem:BilinearLowerOrder} below.
	
	\begin{lemma} \label{lem:PropNoise}
		Let $T\in (0,\infty)$ and $\kappa\in (0,\frac{1}{2})$.
		Then, $W_1$ is $C^\kappa$-H\"older continuous in $H_x^{\frac{d}{2}+2+(s-1)_+}$,
		$W_2$ and the process $t \mapsto \int_0^t e^{-\imu s|\nabla|} \dd W_2(s)$ are $C^\kappa$-H\"older
		continuous in $H_x^{\frac{d}{2}+s-1}$.
		Moreover, for every $j \in \{1, \ldots, d\}$ and
		for $\bbp$-a.e. $\omega\in \Omega$,
		there exists a sequence $(n_l(\omega))_{l \in \N}$ in $\N$ with $n_l(\omega) \rightarrow \infty$ as $l \rightarrow \infty$ such that
		\begin{align}  \label{eq:limphi1kbeta1k}
			\sum\limits_{k = n_l}^\infty & \int  \sup_{y\in \R^{d-1}} |\nabla \phi^{(1)}_k(r \vece_j+y)| \dd r  \| \beta_k^{(1)}(t,\omega)\|_{C^{\frac14} ([0, T])}  
			\longrightarrow 0,\ \ \text{as}\ l \to \infty.
		\end{align}
	\end{lemma}
	The proof uses standard arguments based on the Kolmogorov continuity criterion, cf.~\cite[Lemma~6.1]{HRSZ25}, and is thus omitted here.
	
	The following estimates control the terms arising from the noise in our functional framework.
	\begin{lemma} [Control of noise terms I]  \label{lem:BilinearLowerOrder}
		Let $s\in \R, 0\leq a \leq 1$, and $I \subseteq \R$ be a finite interval.
		\label{it:ControlNoiseG1}
		Then the following estimates hold.
		\begin{align}
			\Big\| \int_{t_0}^t e^{\imu (t-t')\Delta} (b\cdot\nabla u)(t') \dd t'\Big\|_{L^2 (I; L_x^{2^*})}
			&\lesssim |I|^{\frac12}\Big( \|b\|_{L_{t,x}^\infty}^{\frac12}\sum_{j=1}^d \|b\|_{L_{\vece_j}^{1,\infty}}^{\frac12} +\|b\|_{L_t^\infty H_x^{\frac{d}{2}}}\Big)\|u\|_{S^{\frac{1}{2},0,0}}\label{eq:BlowUpEstNEP},\\
			\Big\| \int_{t_0}^{t} e^{\imu (t-t')\Delta} (b\cdot \nabla u)(t')\dd t' \Big\|_{\X^{s,a}(I)}
			&\lesssim \sum_{j=1}^d \sum_{k=1}^\infty  \|\beta^{(1)}_k\|_{C_t^{\frac14}}  \|\nabla \phi_k^{(1)}\|_{L_{\tilde{\vece}_j}^{1,\infty}}  \|u\|_{\XS(I)} + \|b\|_{C^\frac14_t L^\infty_x} \|u\|_{\XS(I)}\notag\\
			&\qquad + |I|^{\frac{1}{2}} \|b\|_{L^\infty_t H^{\frac{d}{2}+(s-1)_+}_x} \|u\|_{\XS(I)}, \label{eq:Bilinbnablau}
		\end{align}
		where $c_+$ means the positive part of $c$, i.e., $c_+ = \max\{c,0\}$,  $\|\cdot\|_{C^\kappa_t L^\infty_x}$ denotes the norm of the space $C^\kappa(I; L_x^\infty)$,
		and the $L_{\tilde{\vece}_j}^{1,\infty}$-norm is defined as
		\begin{align*}
			\|f\|_{L_{\tilde{\vece}}^{1,\infty}(\R^d)}:= \int_{\R} \sup_{y\in \mathcal{P}_\vece} |f(r \vece+y)|\dd r
		\end{align*}
		for measurable function $f: \R^d\to\R$, cf.~\eqref{Lepq-def}.
	\end{lemma}

	\begin{remark}
		On the finite interval $I$, the hypothesis~\eqref{phik-condition} guarantees 
		\begin{align*}
			\sum_{k=1}^\infty  \|\beta^{(1)}_k\|_{C_t^{\frac14}}  \|\nabla \phi_k^{(1)}\|_{L_{\tilde{\vece}_j}^{1,\infty}} < \infty \quad \PP\text{-}a.s..
		\end{align*} 
		In fact, due to the finiteness of $\sup_k \mathbb{E} \|\beta^{(1)}_k\|_{C_t^{\frac14}}$ and condition~\eqref{phik-condition}, we have
		\begin{align*}
			\mathbb{E} \sum_{k=1}^\infty  \|\beta^{(1)}_k\|_{C_t^{\frac14}}  \|\nabla \phi_k^{(1)}\|_{L_{\tilde{\vece}_j}^{1,\infty}} \lesssim \sup_k \mathbb{E}\|\beta_k^{(1)}\|_{C_t^{\frac{1}{4}}} \sum_{k=1}^\infty \|\nabla \phi_k^{(1)}\|_{L_{\tilde{\vece}_j}^{1,\infty}}<\infty.
		\end{align*}
		Moreover, by 
		Sobolev's embedding $H_x^{\frac{d}{2}+1}\hookrightarrow L_x^\infty$ 
		and  $b = 2 \nabla W_1$, Lemma~\ref{lem:PropNoise}  yields $\|b\|_{C_t^{\frac14} L^\infty_x} < \infty$ $\PP$-a.s..
	\end{remark}
	
	\begin{proof}	
		We first show \eqref{eq:BlowUpEstNEP}. Write
		\begin{align*}
			b \cdot \nabla u = (b \cdot \nabla u)_{HL} + (b \cdot \nabla u)_{HH} + (b \cdot \nabla u)_{LH}
		\end{align*}
		and extend $b$ and $u$ by $0$ from $\R_+$ respectively $I$ to $\R$.
		Using Strichartz estimates,
		we infer
		\begin{align*}
			&\Big\| \int_{t_0}^t e^{\imu (t-t')\Delta} P_\lambda(b\cdot\nabla u)_{HL}(t') \dd t' \Big\|_{L_t^2 L_x^{2^*}}\lesssim \sum_{\mu\sim \lambda}  \|P_\mu b P_{\leq \frac{\mu}{2^8} }(\nabla u)\|_{L_t^1 L_x^2}\\
			&\lesssim \sum_{\mu\sim \lambda} \|P_\mu b\|_{L_t^2 L_x^d} \mu\|u\|_{L_t^2 L_x^{2^*}}\lesssim \sum_{\mu\sim\lambda} \mu^{\frac{d}{2}} \|P_\mu b \|_{L_t^2 L_x^2} \|u\|_{S^{\frac{1}{2},0,0}}.
		\end{align*}
		Summing over $\lambda\in 2^\N$
		and using the usual adaption for the $HH$-part, we obtain
		\begin{align}
			\Big(\sum_{\lambda\in 2^{\N}}\Big\| \int_{t_0}^t e^{\imu (t-t')\Delta} P_\lambda(b\cdot\nabla u)_{HL+HH}(t') \dd t' \Big\|^2_{L_t^2 L_x^{2^*}}   \Big)^{\frac{1}{2}}	\lesssim \|b\|_{L_t^2 H_x^{\frac{d}{2}}} \|u\|_{S^{\frac{1}{2},0,0}} \lesssim |I|^{\frac12}\|b\|_{L_t^\infty H_x^{\frac{d}{2}}} \|u\|_{S^{\frac{1}{2},0,0}}. \label{eq:Duhamelbnablau1}
		\end{align}
		For the remaining low-high interaction, we use the decomposition~\eqref{eq:DecompositionAngular} and Lemma~\ref{lem:StrichartzLocalSmooth}~\ref{it:InhomStrichartzLocalSmooth} to estimate
		\begin{align}
			&\Big\| \int_{t_0}^t e^{\imu (t-t')\Delta} P_\lambda(b\cdot\nabla u)_{LH}(t') \dd t' \Big\|_{L_t^2 L_x^{2^*}}\lesssim \sum_{j=1}^d \Big\| \int_{t_0}^t e^{\imu (t-t')\Delta} P_{\lambda,\vece_j}P_\lambda(b\cdot\nabla u)_{LH}(t') \dd t' \Big\|_{L_t^2 L_x^{2^*}}\notag\\
			&\lesssim \sum_{j=1}^d\lambda^{-\frac12} \|P_\lambda (b\cdot \nabla u)_{LH}\|_{L_{\vece_j}^{1,2}}\lesssim  \sum_{j=1}^d \sum_{\mu\sim\lambda} \mu^{\frac12}\| P_{\leq \frac{\mu}{2^8}} b P_\mu u\|_{L_{\vece_j}^{1,2}}\notag\\
			&\lesssim \sum_{j=1}^d \sum_{\mu\sim\lambda} \mu^{\frac12} \||P_{\leq \frac{\mu}{2^8}}b|^{\frac12}\|_{L_{\vece_j}^{2,\infty}} \||P_{\leq \frac{\mu}{2^8}}b|^{\frac12}\|_{L_t^2 L_x^\infty} \|P_\mu u\|_{L_t^\infty L_x^2}\notag\\
			&\lesssim |I|^{\frac12} \|b\|_{L_{t,x}^\infty}^{\frac12} \sum_{j=1}^d \|b\|_{L_{\vece_j}^{1,\infty}}^{\frac12} \sum_{\mu\sim\lambda} \|P_\mu u\|_{S^{\frac{1}{2},0,0}_\mu},
		\end{align}
		which yields that
		\begin{align}
			\Big(\sum_{\lambda\in 2^{\N}}\Big\| \int_{t_0}^t e^{\imu (t-t')\Delta} P_\lambda(b\cdot\nabla u)_{LH}(t') \dd t' \Big\|^2_{L_t^2 L_x^{2^*}}\Big)^{\frac{1}{2}}  \lesssim |I|^{\frac12} \|b\|_{L_{t,x}^\infty}^{\frac12} \sum_{j=1}^d \|b\|_{L_{\vece_j}^{1,\infty}}^{\frac12} \|u\|_{S^{\frac{1}{2},0,0}}.\label{eq:Duhamelbnablau2}
		\end{align}
		Hence, combining \eqref{eq:Duhamelbnablau1} and \eqref{eq:Duhamelbnablau2}, by the Littlewood-Paley theorem and Minkowski's inequality, we obtain~\eqref{eq:BlowUpEstNEP}.

		Next, we show the most difficult estimate \eqref{eq:Bilinbnablau}.
		For the \textit{HL} and \textit{HH} parts we use Lemma \ref{lem:LinEstimates} and the definition of the $\G^{s,a}$ norm to get
		
		\begin{align}\label{eq:estimateG1}
			&\Big\| \int_{t_0}^{t} e^{\imu (t-t')\Delta} (b\cdot \nabla u)_{HL+HH}(t')\dd t' \Big\|_{\X^{s,a}(I)}\lesssim \| (b\cdot \nabla u)_{HL+HH} \|_{\G^{s,a}(I)}\nonumber\\
			&\lesssim \| P_1(b \cdot \nabla u)_{HL+HH}\|_{\NOne_1} + \Big(\sum_{\lambda \in 2^\N} \|P_\lambda (b \cdot \nabla u)_{HL + HH}\|_{\NOne_\lambda}^2\Big)^{\frac{1}{2}}.
		\end{align}
		Since by Bernstein's inequality,
		\begin{align*}
			\|P_\lambda g\|_{N^{s,a,0}_\lambda} \lesssim \lambda^{s} \|P_\lambda g\|_{L^2_t L^{2_*}_x},
		\end{align*}
		we obtain
		\begin{align*}
			\|P_\lambda (b \cdot \nabla u)_{HL}\|_{\NOne_\lambda}
			&\lesssim \sum_{\mu \sim \lambda} \mu^s \| P_\mu b\|_{L^2_t L^d_x} \|P_{\leq \frac{\mu}{2^8}} \nabla u\|_{L^\infty_t L^2_x} \\
			&\lesssim \|u\|_{\XS} \sum_{\mu \sim \lambda} \mu^{(s-1)_+ +1} \|P_\mu b\|_{L^2_t L^d_x}\\
			&\lesssim \|u\|_{\XS} \sum_{\mu \sim \lambda} \mu^{\frac{d}{2}+(s-1)_+ } \|P_\mu b\|_{L^2_t L^2_x}.
		\end{align*}
		The usual adaptions yield the same estimate for the $HH$-component of $b \cdot \nabla u$. We also obtain this estimate for the first term on the right-hand side of~\eqref{eq:estimateG1} in this way as $P_1 (b \cdot \nabla u) = P_1 (b \cdot \nabla u)_{HL + HH} = P_1 (b \cdot \nabla u)_{HH}$. Summing up over $\lambda\in 2^\N$ we get
		\begin{align}
			\Big(\sum_{\lambda \in 2^\N} \|P_\lambda (b \cdot \nabla u)_{HL+HH}\|_{\NOne_\lambda}^2\Big)^{\frac{1}{2}}
			\lesssim |I|^{\frac{1}{2}} \|b\|_{L^\infty_t H^{\frac{d}{2}+(s-1)_+}_x} \|u\|_{\XS}. \label{eq:bnablauNl}
		\end{align}

		For the remaining \textit{LH} part, we write $P_\lambda(b\cdot \nabla u)_{LH}= P_\lambda \sum_{\mu\sim\lambda} P_{\leq \frac{\mu}{2^8}} b P_\mu \nabla u$ and further decompose by modulation
		\begin{align*}
			P_{\leq \frac{\mu}{2^8}} b P_\mu \nabla u &=  	P_{\leq \frac{\mu}{2^8}} b C_{\leq (\frac{\mu}{2^8})^2}P_\mu \nabla u+ P_{\leq \frac{\mu}{2^8}} b C_{> (\frac{\mu}{2^8})^2}P_\mu \nabla u =: I_{LM} + I_{HM}.
		\end{align*}
		
		For the low modulation part $I_{LM}$, we apply Lemma~\ref{lem:LinEstimates} and then bound the $\G^{s,a}$ norm by the lateral Strichartz norm to obtain
		\begin{align}
			\label{eq:ControlNoiseLHHM}
			\Big\| \int_{t_0}^{t} e^{\imu (t-t')\Delta} P_\lambda I_{LM}(t')\dd t' \Big\|_{\XS_\lambda(I)}
			&\lesssim \Big(\sum_{j = 1}^d \sum_{\nu \in 2^\N} \nu^{2s - 1} \| P_\nu P_\lambda I_{LM}\|_{L^{1,2}_{\vece_j}}^2 \Big)^{\frac12} \nonumber\\
			&\lesssim \lambda^{s-\frac{1}{2}} \sum_{j = 1}^d \sum_{\mu \sim \lambda}\|P_{\leq \frac{\mu}{2^8}} b C_{\leq (\frac{\mu}{2^8})^2} P_\mu \nabla u\|_{L^{1,2}_{\vece_j}}.
		\end{align}
		
		Using the decomposition~\eqref{eq:DecompositionAngular}, we infer
		\begin{align}
			&\lambda^{s-\frac{1}{2}} \sum_{\mu \sim \lambda}\|P_{\leq \frac{\mu}{2^8}} b C_{\leq (\frac{\mu}{2^8})^2} P_\mu \nabla u\|_{L^{1,2}_{\vece_j}} \notag \\
			&\lesssim \|b\|_{L^{1,\infty}_{\vece_j}}^{\frac{1}{2}} \sum_{\mu \sim \lambda} \mu^{s-\frac{1}{2}} \||P_{\leq \frac{\mu}{2^8}} b|^{\frac{1}{2}} |C_{\leq (\frac{\mu}{2^8})^2}P_\mu \nabla u| \|_{L^2_{t,x}}  \notag  \\
			&\lesssim \|b\|_{L^{1,\infty}_{\vece_j}}^{\frac{1}{2}} \sum_{\mu \sim \lambda} \sum_{l = 1}^d \mu^{s-\frac{1}{2}} \Big\||P_{\leq \frac{\mu}{2^8}} b|^{\frac{1}{2}} \Big|P_{\mu, \vece_l} \Big[\prod_{k = 1}^{l-1} (I - P_{\mu, \vece_k})\Big] C_{\leq (\frac{\mu}{2^8})^2}P_\mu \nabla u \Big| \Big\|_{L^2_{t,x}} \nonumber\\
			&\lesssim \|b\|_{L^{1,\infty}_{\vece_j}}^{\frac{1}{2}} \sum_{\mu \sim \lambda} \sum_{l = 1}^d \mu^{s-\frac{1}{2}} \||P_{\leq \frac{\mu}{2^8}} b|^{\frac{1}{2}}\|_{L^{2,\infty}_{\vece_l}} \|P_{\mu, \vece_l} C_{\leq (\frac{\mu}{2^8})^2}P_\mu \nabla u \|_{L^{\infty,2}_{\vece_l}} \nonumber\\
			&\lesssim \max_{j = 1, \ldots, d} \|b\|_{L^{1,\infty}_{\vece_j}}  \sum_{\mu \sim \lambda} \sum_{l = 1}^d \mu^{s+\frac{1}{2}} \|P_{\mu, \vece_l} C_{\leq (\frac{\mu}{2^8})^2}P_\mu u \|_{L^{\infty,2}_{\vece_l}}
			\lesssim \sum_{j = 1}^d \|b\|_{L^{1,\infty}_{\vece_j}} \sum_{\mu \sim \lambda} \|u_\mu\|_{\XS_\mu}. \label{eq:bnablauLS1}
		\end{align}
		
		For the high modulation part $I_{HM}$, we further decompose via the temporal frequency to get
		\begin{align*}
			I_{HM} = P_{\leq \frac{\mu}{2^8}} b C_{> (\frac{\mu}{2^8})^2} P^{(t)}_{\ll \mu^2} P_\mu \nabla u+ P_{\leq \frac{\mu}{2^8}} b C_{> (\frac{\mu}{2^8})^2} P^{(t)}_{\gtrsim \mu^2} P_\mu \nabla u := I_{HM}^{LT} + I_{HM}^{HT}.
		\end{align*}
		To estimate the $I_{HM}^{HT}$ contribution, arguing as in~\eqref{eq:ControlNoiseLHHM} and~\eqref{eq:bnablauLS1}, it is sufficient to bound
		\begin{align}
			&\|b\|_{L^{1,\infty}_{\vece_j}}^{\frac{1}{2}} \sum_{\mu \sim \lambda} \mu^{s-\frac{1}{2}}  \||P_{\leq \frac{\mu}{2^8}} b|^{\frac{1}{2}} |C_{> (\frac{\mu}{2^8})^2} P^{(t)}_{\gtrsim \mu^2} P_\mu \nabla u| \|_{L^2_{t,x}}   \notag \\
			&\lesssim \|b\|_{L^{1,\infty}_{\vece_j}}^{\frac{1}{2}} \|b\|_{L^{\infty}_{t,x}}^{\frac{1}{2}} \sum_{\mu \sim \lambda} \mu^{s+\frac{1}{2}}  \|C_{> (\frac{\mu}{2^8})^2} P^{(t)}_{\gtrsim \mu^2} P_\mu  u \|_{L^2_{t,x}} \nonumber\\
			&\lesssim \|b\|_{L^{1,\infty}_{\vece_j}}^{\frac{1}{2}} \|b\|_{L^{\infty}_{t,x}}^{\frac{1}{2}} \sum_{\mu \sim \lambda}   \mu^{s-\frac32}\Big\|\Big(\frac{\mu + |\partial_t|}{\mu^2 + |\partial_t|}\Big)^{a}(\imu \partial_t + \Delta) P^{(t)}_{\gtrsim \mu^2} P_\mu  u \Big\|_{L^2_{t,x}}   \notag \\
			&\lesssim \|b\|_{L^{1,\infty}_{\vece_j}}^{\frac{1}{2}} \|b\|_{L^{\infty}_{t,x}}^{\frac{1}{2}} \sum_{\mu \sim \lambda}   \|u_\mu\|_{\XS_\mu} \label{eq:estimateLIHT2}.
		\end{align}
		
		Now we turn to the most delicate part $I_{HM}^{LT}$, where we further decompose the noise term depending on its temporal frequency to get
		\begin{align*}
			I_{HM}^{LT} = P^{(t)}_{\ll \mu^2} P_{\leq \frac{\mu}{2^8}} b C_{> (\frac{\mu}{2^8})^2} P^{(t)}_{\ll \mu^2} P_\mu \nabla u + P^{(t)}_{\gtrsim \mu^2} P_{\leq \frac{\mu}{2^8}} b C_{> (\frac{\mu}{2^8})^2} P^{(t)}_{\ll \mu^2} P_\mu \nabla u := I_{HM}^{LLT} + I_{HM}^{HLT}.
		\end{align*}
		
		In order to estimate $I_{HM}^{HLT}$,
		arguing as in~\eqref{eq:ControlNoiseLHHM}, \eqref{eq:bnablauLS1}, and~\eqref{eq:estimateLIHT2}, it suffices to bound
		\begin{align}
			&\sum_{\mu\sim \lambda} \|P^{(t)}_{\gtrsim \mu^2} P_{\leq \frac{\mu}{2^8}} b\|_{L_{\vece_j}^{1,\infty}}^{\frac12}  \mu^{s-\frac12} \| |P^{(t)}_{\gtrsim \mu^2} P_{\leq \frac{\mu}{2^8}} b|^{\frac12} C_{> (\frac{\mu}{2^8})^2} P^{(t)}_{\ll \mu^2} P_\mu \nabla u \|_{L_{t,x}^2}\notag\\
			&\lesssim  \sum_{\mu\sim \lambda} \|P^{(t)}_{\gtrsim \mu^2} P_{\leq \frac{\mu}{2^8}} b\|_{L_{\vece_j}^{1,\infty}}^{\frac12} \|P^{(t)}_{\gtrsim \mu^2} P_{\leq \frac{\mu}{2^8}} b\|_{L_{t,x}^\infty}^{\frac12}   \mu^{a-\frac12} \mu^{s-1} \Big\| \Big(\frac{\mu+|\partial_t|}{\mu^2+|\partial_t|}\Big)^a (\imu\partial_t+\Delta)  C_{> (\frac{\mu}{2^8})^2} P^{(t)}_{\ll \mu^2} P_\mu u \Big\|_{L_{t,x}^2}\notag\\
			&\lesssim \sum_{\mu \sim \lambda}  \|P^{(t)}_{\gtrsim \mu^2} P_{\leq \frac{\mu}{2^8}}b\|_{L^{1,\infty}_{\vece_j}}^{\frac{1}{2}} \|P^{(t)}_{\gtrsim \mu^2}P_{\leq \frac{\mu}{2^8}} b\|_{L^{\infty}_{t,x}}^{\frac{1}{2}}   \mu^{a-\frac12}\|u_\mu\|_{\XS_\mu}.\label{eq:targetIHMHLT}
		\end{align}
		We point out that we need to gain some regularity here if we want to avoid a restriction on the parameter $a$, which would ultimately prohibit us from proving local well-posedness in the full regularity regime~\eqref{IniReg-condition}. For that purpose, we exploit the temporal regularity of Brownian motions. By \eqref{eq:Defb} we have 
		\begin{align*}
			P^{(t)}_{\gtrsim \mu^2} P_{\leq \frac{\mu}{2^8}}b = 2\imu \sum_{k=1}^\infty 	P^{(t)}_{\gtrsim \mu^2} \beta^{(1)}_k(t)  P_{\leq \frac{\mu}{2^8}} \nabla \phi_k^{(1)}(x).
		\end{align*}
		Note that, by the temporal regularity of Brownian motions, we have
		\begin{align}
			\| P^{(t)}_{\gtrsim \mu^2} \beta^{(1)}_k\|_{L_t^\infty} \lesssim \sum_{\nu\gtrsim \mu^2} \|P_\nu^{(t)} \beta^{(1)}_k\|_{L_t^\infty}
			\lesssim \sum_{\nu\gtrsim \mu^2} \nu^{-\frac14} \sup_{\nu\in 2^{\N}} \nu^{\frac14} \|P_\nu^{(t) }\beta^{(1)}_k\|_{L_t^\infty} \lesssim \mu^{-\frac12} \|\beta^{(1)}_k\|_{C_t^{\frac14}},   \label{eq:estimatetemporalBrown}
		\end{align}
		where we used that $B^{\frac14}_{\infty,\infty}(\R) = C^{\frac14}(\R)$. 
		As this characterization is also valid in the Banach space-valued case, see e.g.~\cite[Corollary~14.4.26]{HvNVW23}, we use the equivalence $B^{\frac14}_{\infty,\infty}(\R; L_x^\infty) =  C^{\frac14}(\R; L_x^\infty)$
		to infer that
		\begin{align*}
			\| P^{(t)}_{\gtrsim \mu^2} b\|_{L_t^\infty L^\infty_x} 
			\lesssim \sum_{\nu\gtrsim \mu^2} \nu^{-\frac14} \sup_{\nu\in 2^{\N}} \nu^{\frac14} \|P_\nu^{(t) } b\|_{L_t^\infty L^\infty_x} \lesssim \mu^{-\frac12} \|b\|_{C_t^{\frac14} L^\infty_x}.
		\end{align*}
		Hence, we have
		\begin{align}
			\label{eq:EstbwHoeldReg}
			\|P^{(t)}_{\gtrsim \mu^2} P_{\leq \frac{\mu}{2^8}}b\|_{L^{1,\infty}_{\vece_j}}^{\frac12}  \|P^{(t)}_{\gtrsim \mu^2}P_{\leq \frac{\mu}{2^8}} b\|_{L^{\infty}_{t,x}}^{\frac{1}{2}} 
			&\lesssim  \|P^{(t)}_{\gtrsim \mu^2} P_{\leq \frac{\mu}{2^8}}b\|_{L^{1,\infty}_{\vece_j}} + \|P^{(t)}_{\gtrsim \mu^2}P_{\leq \frac{\mu}{2^8}} b\|_{L^{\infty}_{t,x}} \nonumber\\
			&\lesssim  \sum_{k=1}^\infty  \| P^{(t)}_{\gtrsim \mu^2}\beta^{(1)}_k\|_{L_t^{\infty}}  \| \nabla \phi_k^{(1)} \|_{L_{\tilde{\vece}_j}^{1,\infty}} + \mu^{-\frac12} \|b\|_{C_t^{\frac14}L_x^\infty} \nonumber\\
			&\lesssim  \mu^{-\frac12} \Big(\sum_{k=1}^\infty  \|\beta^{(1)}_k\|_{C_t^{\frac14}}  \| \nabla \phi_k^{(1)} \|_{L_{\tilde{\vece}_j}^{1,\infty}} + \|b\|_{C_t^{\frac14}L_x^\infty}\Big).
		\end{align}
		Plugging this estimate into~\eqref{eq:targetIHMHLT}, the gain of $\mu^{-\frac12}$ from the regularity of Brownian motions and the fact that $a\leq 1$ yield
		\begin{align*}
			\text{R.H.S. of } \eqref{eq:targetIHMHLT} \lesssim \Big(\sum_{k=1}^\infty  \|\beta^{(1)}_k\|_{C_t^{\frac14}}  \| \nabla \phi_k^{(1)} \|_{L_{\tilde{\vece}_j}^{1,\infty}} + \|b\|_{C_t^{\frac14}L_x^\infty} \Big)
			\sum_{\mu\sim\lambda}\|u_\mu\|_{\X^{s,a}}.
		\end{align*}
		
		\vspace{-2pt}
		
		Finally, it remains to estimate $I_{HM}^{LLT}$. Note that one has the identity $I_{HM}^{LLT} = C_{> (\frac{\mu}{2^6})^2} I_{HM}^{LLT}$.
		By the definition of the $\X^{s,a}$-norm and \eqref{eq:CharSsablambda}, it suffices to show the following three estimates:
		\begin{align}
			\mu^{s-1} \Big\| \Big( \frac{\mu+|\partial_t|}{\mu^2+|\partial_t|} \Big)^a (\imu \partial_t+ \Delta) \cI_0 (C_{>(\frac{\mu}{2^6})^2} I_{HM}^{LLT}) \Big\|_{L_{t,x}^2}
			\lesssim& \|b\|_{L_{t,x}^\infty} \|u\|_{\X_{\mu}^{s,a}}, \label{eq:IHMLLT1}\\
			\mu^s \| C_{>(\frac{\mu}{2^8})^2} \cI_0(C_{>(\frac{\mu}{2^6})^2} I_{HM}^{LLT} ) \|_{L_t^\infty L_x^2}
			\lesssim&  \|b\|_{L_{t,x}^\infty} \|u\|_{\X_{\mu}^{s,a}}, \label{eq:IHMLLT2}\\
			\mu^s \| C_{\leq (\frac{\mu}{2^8})^2} \cI_0 (C_{> (\frac{\mu}{2^6})^2} I_{HM}^{LLT}) \|_{L_t^\infty L_x^2 \cap L_t^2 L_x^{2^*}}
			+ \mu^{s+\frac12} \|  P_{\lambda,\vece_j} C_{\leq (\frac{\mu}{2^8})^2}
			&  \cI_0 (C_{> (\frac{\mu}{2^6})^2} I_{HM}^{LLT}) \|_{L_{\vece_j}^{\infty, 2}} \lesssim  \|b\|_{L_{t,x}^\infty} \|u\|_{\X_{\mu}^{s,a}},  \label{eq:IHMLLT3}
		\end{align}
		for every $j \in \{1, \ldots, d\}$, where $\cI_0$ denotes the Schr\"odinger Duhamel integral in~\eqref{eq:DefPropOp} and the $L_t^\infty L_x^2 \cap L_t^2 L_x^{2^*}$-norm is defined as the sum of the $L_t^\infty L_x^2 $-norm and $L_t^2 L_x^{2^*}$-norm.
		
		For the estimate \eqref{eq:IHMLLT1},
		we have
		\begin{align*}
			{\rm L.H.S.\ of\ } \eqref{eq:IHMLLT1}
			&\lesssim \mu^{s-1} \| C_{>(\frac{\mu}{2^6})^2} I_{HM}^{LLT} \|_{L_{t,x}^2}\\
			&\lesssim \|b\|_{L_{t,x}^\infty} \mu^s \| C_{> (\frac{\mu}{2^8})^2} P^{(t)}_{\ll \mu^2} P_\mu  u\|_{L_{t,x}^2}\\
			&\lesssim \|b\|_{L_{t,x}^\infty} \mu^{s-2+a} \Big\| \Big(\frac{\mu+|\partial_t|}{\mu^2+|\partial_t|}\Big)^a (\imu\partial_t+\Delta) P_\mu u\Big\|_{L_{t,x}^2}\lesssim \|b\|_{L_{t,x}^\infty} \|u\|_{\X_{\mu}^{s,a}}.
		\end{align*}
		Moreover, we have
		\begin{align*}
			{\rm L.H.S.\ of\ }  \eqref{eq:IHMLLT2}
			&\lesssim \mu^{s-2} \| (\imu\partial_t+\Delta) C_{>(\frac{\mu}{2^8})^2} \cI_0(C_{>(\frac{\mu}{2^6})^2} I_{HM}^{LLT} ) \|_{L_t^\infty L_x^2}\\
			&\lesssim \mu^{s-2} \| P^{(t)}_{\ll \mu^2} P_{\leq \frac{\mu}{2^8}} b C_{> (\frac{\mu}{2^8})^2} P^{(t)}_{\ll \mu^2} P_\mu (\nabla u) \|_{L_{t}^\infty L_x^2}\\
			&\lesssim \|b\|_{L_{t,x}^\infty} \mu^{s-1} \|u_\mu\|_{L_t^\infty L_x^2} \lesssim \|b\|_{L_{t,x}^\infty} \|u\|_{\X_{\mu}^{s,a}}.
		\end{align*}
		Regarding estimate \eqref{eq:IHMLLT3}, we use the commutation relation
		\begin{align*}
			e^{-\imu t \Delta} C_{>(\frac{\mu}{2^8})^2 } = P^{(t)}_{>(\frac{\mu}{2^8})^2 } e^{-\imu t \Delta}
		\end{align*}
		to write
		\begin{align*}
			C_{\leq (\frac{\mu}{2^8})^2} \cI_0 (C_{>(\frac{\mu}{2^6})^2 }I_{HM}^{LLT}) = e^{\imu (\cdot) \Delta} P^{(t)}_{\leq (\frac{\mu}{2^8})^2} (H(\cdot)- H(t_0)),
		\end{align*}
		where $H(t):= \partial_t^{-1} P^{(t)}_{>(\frac{\mu}{2^6})^2} (e^{-\imu (\cdot) \Delta} I_{HM}^{LLT})$, cf.~(3.24) in~\cite{HRSZ24}.
		Since $P^{(t)}_{\leq (\frac{\mu}{2^8})^2} H = 0$, we have
		\begin{align*}
			C_{\leq(\frac{\mu}{2^8})^2} \cI_0 (C_{>(\frac{\mu}{2^6})^2 }I_{HM}^{LLT}) = -e^{\imu (\cdot) \Delta} H(t_0).
		\end{align*}
		Thus, by Strichartz estimates and the homogeneous local smoothing estimate in Lemma~\ref{lem:StrichartzLocalSmooth}, we have
		\begin{align*}
			{\rm L.H.S.\ of\ }  \eqref{eq:IHMLLT3}
			&\lesssim \mu^s\|H(t_0)\|_{L_x^2}\lesssim \mu^s\|H\|_{L_t^\infty L_x^2}\\
			&\lesssim \mu^{s-2} \| e^{-\imu (\cdot) \Delta} I_{HM}^{LLT}\|_{L_t^\infty L_x^2 }\\
			& \lesssim \mu^{s-1} \| P^{(t)}_{\ll \mu^2} P_{\leq \frac{\mu}{2^8}} b C_{> (\frac{\mu}{2^8})^2} P^{(t)}_{\ll \mu^2} P_\mu \nabla u\|_{L_t^\infty L_x^2}
			\lesssim  \|b\|_{L_{t,x}^\infty} \|u\|_{\X_{\mu}^{s,a}}.
		\end{align*}
		Therefore, combining the above estimates and employing~\eqref{eq:Defb} once more, we obtain~\eqref{eq:Bilinbnablau}.
	\end{proof}
	
	\begin{lemma} [Control of noise terms II] \label{lem:BilinearLowerOrder2}
		Let $s\in \R, 0\leq a \leq 1$, and $I \subseteq \R$ be a finite interval. Then, we have the estimates
		\begin{align}
			\Big\| \int_{t_0}^t e^{\imu (t-s)\Delta} (c u) \dd s\Big\|_{L^2 (I; L_x^{2^*})}
			&\lesssim |I|^{\frac{1}{2}} \|c\|_{L_t^\infty H_x^{\frac{d-3}{2}}}\|u\|_{S^{\frac{1}{2},0,0}}\label{eq:BlowUpEstNEP2},\\
			\Big\| \int_{t_0}^t e^{\imu (t-s)\Delta} ( \cT_{\cdot}(W_2) u) \dd s\Big\|_{L^2 (I; L_x^{2^*})}
			&\lesssim |I|^{\frac{1}{2}} \|\cT_{\cdot}(W_2)\|_{L_t^\infty H_x^{\frac{d-3}{2}}}\|u\|_{S^{\frac{1}{2},0,0}}\label{eq:BlowUpEstNEP3},\\
			\| c u\|_{\GS(I)}
			&\lesssim |I|^{\frac{1}{2}} \|c\|_{L^\infty_t H^{\frac{d}{2}+s-1}_x} \|u\|_{\XS(I)}, \label{eq:Bilincu} \\
			\| \cT_{\cdot}(W_2) u \|_{\GS(I)}
			&\lesssim |I|^{\frac{1}{2}} \|\cT_{\cdot}(W_2)\|_{L^\infty_t H^{\frac{d}{2}+s-1}_x} \|u\|_{\XS(I)}\label{eq:BilincTu}.
		\end{align}
	\end{lemma}
	
	\begin{proof}
		The estimates~\eqref{eq:Bilincu} and~\eqref{eq:BilincTu} can be proved as \eqref{eq:bnablauNl}. It suffices to show the estimate \eqref{eq:BlowUpEstNEP2} since the same argument applies to the estimate \eqref{eq:BlowUpEstNEP3}.
		
		We extend $c$ and $u$ by $0$ from $I$ to $\R$. Then Strichartz estimates yield
		\begin{align*}
			\Big\| \int_{t_0}^t e^{\imu (t-s)\Delta} (c u) \dd s\Big\|_{L^2 (I; L_x^{2^*})}
			\lesssim \|cu\|_{L_t^2 L_x^{2_*}} \lesssim \|c\|_{L_t^2 L_x^{\frac{2d}{3}}} \|u\|_{L_t^\infty L_x^{\frac{2d}{d-1}}} \lesssim |I|^{\frac{1}{2}} \|c\|_{L_t^\infty H_x^{\frac{d-3}{2}}} \|u\|_{L_t^\infty H_x^{\frac{1}{2}}}.
		\end{align*}
		Then \eqref{eq:BlowUpEstNEP2} follows from the embedding $S^{\frac{1}{2},0,0} \hookrightarrow L_t^\infty H_x^{\frac{1}{2}}$.
	\end{proof}

	\section{Improved regularity for solutions}\label{Sec-Persi}
	
	In this section, we establish two results concerning the improved regularity of solutions to the stochastic Zakharov system. Both results are key components in the proof of local well-posedness, particularly for characterizing the maximal existence time.
	
	Furthermore, the analysis in this section is conducted pathwise. Although the solutions and coefficients depend on the sample path $\omega$, we suppress this explicit dependence to simplify the exposition. Consequently, all results presented herein should be understood to hold almost surely.

	\subsection{Improved regularity of the wave component}
	This subsection is devoted to establishing improved regularity for the wave component under the condition $l > s - \frac12$. This regularity refinement is a crucial step in the proof of local well-posedness in the non-endpoint case presented in Subsection~\ref{Subsec-Max-Exist} as it allows us to transition from spaces with $b = 0$ to spaces with $b \neq 0$. The main result is stated in the following lemma.

	\begin{lemma}[Improved wave regularity]\label{le:PersisReg}
		Assume that $(s,l)$ satisfies \eqref{IniReg-condition} and $\tilde{l}=\min\{s-\frac12,l\}$.
		Let $(u, v)\in C(I,H_x^s \times H_x^{\tilde{l}})$ be a solution
		to the stochastic Zakharov system $\eqref{eq:RanZakbc}$ with
		$(u(0), v(0))= (u_0, v_0) \in H_x^s \times H_x^l$ where  $I\subseteq \R_+$ is a finite interval. If  $u\in \mathbb{X}^{s,0}(I)$ and $l> s-\frac{1}{2}$, then $(u,v)\in \tilde{S}^{s,0,b}(I)\times W^{l,0,s-\frac12}(I)$ with $b$ as in \eqref{con:ab}. In particular,
		$(u,v)\in C(I, H_x^s \times H_x^l)$.	
	\end{lemma}

	\begin{proof}
		We first write
		\begin{align}\label{eq:DuhamelSch}
			u(t)=e^{\imu t\Delta}u_0 + \mathcal{I}_0[\Re(v)u - b\cdot \nabla u -cu +\Re(\cT_t(W_2))u]
		\end{align}
		and
		\begin{align*}
			v(t)= e^{\imu t|\nabla|}v_0 - \cJ_0[|\nabla||u|^2].
		\end{align*}
		It suffices to prove that
		\begin{align}  \label{u-wtS}
			\|u\|_{\tilde{S}^{s,0,b}(I)}<\infty.
		\end{align}
		Then an application of Lemma~\ref{lem:LinearEstimateHalfWave} and Lemma~\ref{le:Bilinearnablauv} gives the regularity of the wave component
		\begin{align*}
			\|v\|_{W^{l,0,s-\frac12}(I)}\lesssim \|v_0\|_{H_x^l} + \|u\|_{\tilde{S}^{s,0,b}(I)}^2<\infty.
		\end{align*}
		
		Now we focus on the proof of \eqref{u-wtS}.
		Let $\kappa= l $ if $s>l$
		and $\kappa = s-\frac{1}{2} (1-b)$ if $s\leq l$.
		Applying Lemma~\ref{lem:LinFlowAdaptedSpaces}, \eqref{prop:RelationBetweenTwoS}, Lemma~\ref{lem:BilinearEstimates}~\ref{it:BilinearEstNonendpointbNzero}, and Lemma~\ref{le:Bilinearnablauv}, we obtain
		\begin{align}
			\|e^{\imu t\Delta}u_0\|_{\tilde{S}^{s,0,b}(I)}&\lesssim \|u_0\|_{H_x^s}<\infty,  \\
			\|\mathcal{I}_0[\Re(v)u]\|_{\tilde{S}^{s,0,b}(I)}&\lesssim \|\mathcal{I}_0[\Re(v)u]\|_{{S}^{s,0,b}(I)} \lesssim \|v\|_{W^{\kappa, 0, s-\frac12}(I)} \|u\|_{S^{s,0,0}(I)}\notag\\
			&\lesssim \|v_0\|_{H_x^{\kappa}} \|u\|_{S^{s,0,0}(I)}+ \|u\|^3_{S^{s,0,0}(I)}<\infty,
		\end{align}
		where we also used the embedding $\X^{s,0}\hookrightarrow S^{s,0,0}$ and $\|u\|_{\mathbb{X}^{s,0}(I)}<\infty$.
		
		Concerning the noise terms,
		by the definition of $\tS$, it suffices to consider the estimate for the high modulation part as the other terms can be bounded by the $S^{s,0,0}(I)$-norm as follows
		\begin{align*}
			&\|\cI_0[b\cdot \nabla u+cu - \Re(\cT_t(W_2))u]\|_{S^{s,0,0}(I)}= \|u-e^{\imu t \Delta}u_0 - \cI_0[\Re(v)u] \|_{S^{s,0,0}(I)} \\
			&\lesssim \|u\|_{\X^{s,0}(I)} + \|e^{\imu t\Delta}u_0\|_{\tilde{S}^{s,0,b}(I)}+ \|\cI_0[\Re(v)u]\|_{\tilde{S}^{s,0,b}(I)}<\infty.
		\end{align*}
		
		For that purpose,
		we extend $b$ and $u$ by $0$ from $\R_+$ respectively $I$ to $\R$ and decompose
		\begin{align*}
			b\cdot \nabla u =\sum_{\lambda } P_\lambda b P_{\leq \frac{\lambda}{2^8}}(\nabla u) + \sum_{\lambda} \sum_{\mu\sim\lambda} P_\lambda b P_\mu (\nabla u) +\sum_{\lambda} P_{\leq \frac{\lambda}{2^8}} b P_\lambda (\nabla u).
		\end{align*}
		We first bound the high-low interaction part by
		\begin{align*}
			\lambda^{s-1+b} \|C_{\geq \frac{\lambda^2}{2^3}} (P_\lambda b\cdot P_{\leq \frac{\lambda}{2^8}} \nabla u)\|_{L_{t,x}^2}
			&\lesssim \lambda^{s+b+1} \|P_\lambda b\cdot P_{\leq \frac{\lambda}{2^8}}  u\|_{L_t^2 L_x^{2_*}}  \\
			&\lesssim \lambda^{s+b+1} \|P_\lambda b\|_{L_t^2 L_x^d} \|P_{\leq \frac{\lambda}{2^8}}  u\|_{L^\infty_t L_x^2}    \\
			&\lesssim  \lambda^{s+b+\frac{d}{2}} \|P_\lambda b\|_{L_{t,x}^2 } \|u\|_{\mathbb{X}^{s,0}}.
		\end{align*}
		The usual adaption applies to the high-high interaction part.
		For the low-high interaction part, we further decompose by modulation to get
		\begin{align}\label{eq:PersistenceDecom}
			&\lambda^{s-1+b}\|C_{\geq \frac{\lambda^2}{2^3}} (P_{\leq \frac{\lambda}{2^8}} b\cdot  P_\lambda \nabla u)\|_{L_{t,x}^2}\notag\\
			&\lesssim \lambda^{s+b} \|C_{\geq \frac{\lambda^2}{2^3} } (P_{\leq \frac{\lambda}{2^8}} b\cdot  C_{>(\frac{\lambda}{2^8})^2}P_\lambda  u)\|_{L_{t,x}^2} + \lambda^{s+b} \|C_{\geq \frac{\lambda^2}{2^3} } (P^{(t)}_{\gtrsim \lambda^2}P_{\leq \frac{\lambda}{2^8}} b\cdot  C_{\leq(\frac{\lambda}{2^8})^2}P_\lambda  u)\|_{L_{t,x}^2} \notag \\
			&=: I^{LH}_{HM} + I^{LH}_{LM}.
		\end{align}		
		For the $I^{LH}_{HM}$ contribution, we estimate
		\begin{align*}
			I^{LH}_{HM}&\lesssim \lambda^{s+b} \|P_{\leq \frac{\lambda}{2^8}} b\|_{L_{t,x}^\infty} \|C_{>(\frac{\lambda}{2^8})^2}P_\lambda  u\|_{L_{t,x}^2}\\
			&\lesssim \lambda^{s+b-2} \|P_{\leq \frac{\lambda}{2^8}} b\|_{L_{t,x}^\infty}  \|  (\imu \partial_t +\Delta)P_\lambda u \|_{L_{t,x}^2}
			\lesssim \|b\|_{L_t^\infty H^{\frac{d}{2}+1}_x} \|u_\lambda\|_{\X^{s,0}_\lambda}.
		\end{align*}
		Moreover, for the $I^{LH}_{LM}$ contribution, we have
		\begin{align*}
			I^{LH}_{LM}&\lesssim \lambda^{s+b} \|P^{(t)}_{\gtrsim \lambda^2}P_{\leq \frac{\lambda}{2^8}} b\cdot  C_{\leq(\frac{\lambda}{2^8})^2}P_\lambda  u\|_{L_{t,x}^2}\\
			&\lesssim \sum_{j=1}^d \lambda^{s+b} \|P^{(t)}_{\gtrsim \lambda^2}P_{\leq \frac{\lambda}{2^8}} b\|_{L_{\vece_j}^{2,\infty}} \|C_{\leq(\frac{\lambda}{2^8})^2} P_{\lambda,\vece_j}  u_\lambda\|_{L_{\vece_j}^{\infty,2}}\\
			&\lesssim \sum_{j=1}^d\lambda^{s+\frac12} \|C_{\leq(\frac{\lambda}{2^8})^2}P_{\lambda,\vece_j}  u_\lambda\|_{L_{\vece_j}^{\infty,2}} \lambda^{b-\frac12} \|P^{(t)}_{\gtrsim \lambda^2}P_{\leq \frac{\lambda}{2^8}} b\|_{L_{\vece_j}^{1,\infty}}^{\frac12} \|P^{(t)}_{\gtrsim \lambda^2}P_{\leq \frac{\lambda}{2^8}} b\|_{L^\infty_{t,x}}^{\frac12} \\
			&\lesssim \|u\|_{\X^{s,0}} \lambda^{b - 1 -2\varepsilon}  \Big(\sum_{j=1}^d \sum_{k=1}^\infty \|\beta_k^{(1)}\|_{C_t^{\frac{1}{4}+\varepsilon}} \|\nabla \phi_k^{(1)}\|_{L_{\tilde{\vece}_j}^{1,\infty}} + \|b\|_{C_t^{\frac14+\varepsilon} L_x^\infty} \Big),
		\end{align*}
		where we applied~\eqref{eq:EstbwHoeldReg} with temporal regularity $\frac14 + \varepsilon$ for a small constant $\varepsilon>0$ in the last step.
		Hence, summing over $\lambda\in 2^{\N_0}$ and using that $b \leq 1$, the Hypothesis~\eqref{phik-condition}, and Lemma~\ref{lem:PropNoise}, we obtain
		\begin{align*}
			\sum_{\lambda} \lambda^{2(s-1+b)} \|(\imu \partial_t+\Delta) \cI_0(b\cdot \nabla u) \|_{L_{t,x}^2}<\infty.
		\end{align*}
		Analogous arguments also apply to the other two noise terms of \eqref{eq:DuhamelSch}. Therefore, combining the estimates above, we obtain \eqref{u-wtS} and finish the proof.
	\end{proof}

	\subsection{Improved properties of the Schr{\"o}dinger component} \label{Sec:UniformBoundContinuousExtension}
	
	In this subsection, we establish a uniform bound for the $\X^{s,a}$-norm of Schrödinger component. We show that if a solution to the Zakharov system is bounded in Strichartz spaces (up to its maximal existence time) and its Schrödinger component belongs to $\X^{s,0}$ on every compact subinterval, then the $\X^{s,0}$-norms are uniformly bounded over those subintervals. Consequently, the wave component can be continuously extended beyond the interval of maximal existence. This result, while technical, is essential in deriving the blow-up condition for the local well-posedness theorem.
	
	\begin{lemma}[Uniform bound and continuous extension]
		\label{le:unibouep}
		Suppose that $(s,l)$ satisfies \eqref{IniReg-condition} and $s=l+\frac12$. Let $\{\sigma_n\}\subseteq \R_+$ converge to $\tau^*\in (0,\infty)$ and $(u,v)\in C([0,\tau^*);H_x^s\times H_x^l)$ be a solution of system \eqref{eq:RanZakbc} such that $u\in \X^{s,0}([0,\sigma_n])$ for any $n\in \N $. If 
		\begin{align}\label{prop:UniBonCon}
			\|u\|_{L^\infty ([0,\tau^*); H_x^s)}+\|u\|_{L^2 ([0,\tau^*); W_x^{s, 2^*})}+\|v\|_{L^\infty ([0,\tau^*); H_x^l)}< \infty,
		\end{align}
		then we have the uniform bound
		\begin{align}\label{prop:unibouep}
			\sup_{n\in \N} \|u\|_{\mathbb{X}^{s,0}([0,\sigma_n])}<\infty.
		\end{align}
		Moreover, $v$ can be continuously extended to $C([0,\tau^*] , H_x^l)$.
	\end{lemma}
	
	\begin{remark}
		In the 4D energy-critical case, a continuous extension of the wave component at the regularity $(s,l) = (\frac12,0)$ was previously established in Lemma~B.1 of~\cite{HRSZ24}, but it relied on a uniform a priori bound on the $H_x^1$-norm of the Schr{\"o}dinger component. This approach is inapplicable in the current setting since 
			the endpoint regularity $s=\frac{d-3}{2}$ is strictly less than one when $d=4$. Lemma~\ref{le:unibouep} overcomes this by providing the required uniform bound and the resulting continuous extension without a loss of regularity.
		
		To overcome the reduced regularity assumption, we leverage the local smoothing component of the $\X^{s,0}$-norm. The rescaled Brownian motion $b_{\sigma_n}$ provides the necessary smallness to control this component for sufficiently large $n$. These results -- both the uniform bound and the continuous extension without loss of regularity -- are crucial for establishing the blow-up condition at the endpoint regularity, see Subsection~\ref{Subsec-Blowup}.
	\end{remark}

	\begin{proof} [Proof of Lemma \ref{le:unibouep}]
		Let $b_{\sigma_n}, c_{\sigma_n}$ and $\cT_{\sigma_n+\cdot,\sigma_n}(W_2)$
		be the rescaled noise terms from ~\eqref{eq:Defbsigmacsigma} and~\eqref{eq:DefTsigmaW2}.
		Since $\sigma_n\to \tau^*$ as $n\to \infty$,  the H\"older continuity of the noise in Lemma \ref{lem:PropNoise} yields that for any $\varepsilon > 0$, there exists $N\geq 1$ such that for any $n\geq N$,
		\begin{align}
			\sup_{n\geq N}\|b_{\sigma_n}\|_{L^{1,\infty}_{\vece_j}([0,\tau^*-\sigma_n))}+ \sup_{n\geq N} (\|b_{\sigma_n}\|_{L^\infty ([0,\tau^*-\sigma_n);H_x^{\frac{d}{2}+1+(s-1)_+})} \leq \varepsilon\label{prop:holdernoise},
		\end{align}	
		and
		\begin{align}
			\|c_{\sigma_n}\|_{L^\infty ([0,\tau^*-\sigma_n);H_x^{\frac{d}{2}+s-1})}
			+\|\cT_{\sigma_n+\cdot, \sigma_n}(W_2)\|_{L^\infty ([0,\tau^*-\sigma_n);H_x^{\frac{d}{2}+s-1})}  \leq \varepsilon .\label{prop:holdernoiseb2}
		\end{align}
		In the sequel we consider $n\geq N$ and assume that $\tau^*-\sigma_n<1$ without loss of generality.

		We rewrite the rescaled system \eqref{eq:RanZakbsigmacsigma} as
		\begin{align}  \label{usign-Duhamel}
			u_{\sigma_n}(t)= e^{\imu t \Delta}u(\sigma_n)+\mathcal{I}_0[\Re(v_{\sigma_n})u_{\sigma_n}-b_{\sigma_n}\cdot \nabla u_{\sigma_n}-c_{\sigma_n}u_{\sigma_n}+\Re(\cT_{\sigma_n+\cdot, \sigma_n}(W_2))u_{\sigma_n}],
		\end{align}
		where $u_{\sigma_n}(t) = e^{W_1(\sigma_n)}u(\sigma_n+t)$ is the rescaled solution given by \eqref{eq:u-sigma-v-sigma-rescal.1}.
		We first show that
		\begin{align} \label{u-Sw-finit}
			\sup_{n\geq N}\|u_{\sigma_n}\|_{S_w^{s,0,0}([0, \tau^*-\sigma_n))}<\infty.
		\end{align}
		
		To this end, by the definition of the weak norm $S_w^{s,0,0}$ in~\eqref{eq:WeakNormS},
		\begin{align*}
			\|u_{\sigma_n}\|_{S_w^{s,0,0}([0,\tau^*-\sigma_n))}
			&\lesssim \|u_{\sigma_n}\|_{L^\infty([0,\tau^*-\sigma_n); H_x^s)} + \|u_{\sigma_n}\|_{L^2 (0,\tau^*-\sigma_n; W_x^{s,2^*})}\\
			&+ \|\Re(v_{\sigma_n})u_{\sigma_n}-b_{\sigma_n}\cdot \nabla u_{\sigma_n}-c_{\sigma_n}u_{\sigma_n}+\Re(\cT_{\sigma_n+\cdot, \sigma_n}(W_2))u_{\sigma_n} \|_{L^2 (0,\tau^*-\sigma_n; H_x^{s-1})}.
		\end{align*}
		It is clear that
		\begin{align}
				\sup_{n\geq N}(\|u_{\sigma_n}\|_{L^\infty ([0,\tau^*-\sigma_n); H_x^s)}+ \|u_{\sigma_n}\|_{L^2 (0,\tau^*-\sigma_n; W_x^{s,2^*})})
				&\lesssim \|u\|_{L^\infty ([0,\tau^*); H_x^s)}+ \|u\|_{L^2 (0,\tau^*; W_x^{s,2^*})}<\infty.\label{es:usigman1}
		\end{align}
		Moreover, by classical product estimates and the Sobolev embedding $W_x^{l,\frac{2d}{d+1}} \hookrightarrow H_x^{s-1}$, we estimate
		\begin{align}\label{es:usigman4}
			\|\Re(v_{\sigma_n})u_{\sigma_n}\|_{L^2 (0,\tau^*-\sigma_n; H_x^{s-1} )}\notag
			&\lesssim   \|u_{\sigma_n}\|_{L^2 (0,\tau^*-\sigma_n; W_x^{l,\frac{2d}{d-3}} )}\|v_{\sigma_n}\|_{L^\infty ([0,\tau^*-\sigma_n); L_x^{\frac{d}{2}} )}\notag\\
			&\qquad + \|u_{\sigma_n}\|_{L^2 (0,\tau^*-\sigma_n; L_x^{2d})} \|v_{\sigma_n}\|_{L^\infty ([0,\tau^*-\sigma_n); H_x^l )}  \nonumber\\
			&\lesssim \|v\|_{L^\infty ([0,\tau^*); H_x^l )} \|u\|_{L^2 (0,\tau^*; W_x^{s, 2^*} )}<\infty.
		\end{align}
		Hence, $\Re(v_{\sigma_n})u_{\sigma_n}$ is uniformly bounded in $L_t^2 H_x^{s-1}$.	
		
		In order to estimate the delicate derivative term caused by the noise,
		we decompose
		\begin{align*}
			b_{\sigma_n}\cdot \nabla u_{\sigma_n} = \sum_\lambda P_\lambda b_{\sigma_n} P_{\leq \frac{\lambda}{2^8}} (\nabla u_{\sigma_n}) + \sum_{\lambda_1\sim\lambda_2} P_{\lambda_1}b_{\sigma_n} P_{\lambda_2} (\nabla u_{\sigma_n}) + \sum_\lambda  P_{\leq \frac{\lambda}{2^8}} b_{\sigma_n} P_\lambda  (\nabla u_{\sigma_n}).
		\end{align*} 	
		For the high-low part,
		\begin{align*}
			&\lambda^{s-1}\|P_{\lambda}b_{\sigma_n} P_{\leq \frac{\lambda}{2^8}} (\nabla u_{\sigma_n})\|_{L^2(0,\tau^*-\sigma_n; L_x^2)}\\
			&\lesssim \lambda^{s-1} \|P_{\lambda}b_{\sigma_n}\|_{L^2 (0,\tau^*-\sigma_n; L_x^\infty)} \| P_{\leq \frac{\lambda}{2^8}} (\nabla u_{\sigma_n})\|_{L^\infty ([0,\tau^*-\sigma_n); L_x^2)}\\
			&\lesssim  \lambda^{\frac{d}{2}+(s-1)_+} \|P_{\lambda}b_{\sigma_n}\|_{L^2(0,\tau^*-\sigma_n; L_x^2)} \|u_{\sigma_n}\|_{L^\infty  ([0,\tau^*-\sigma_n); H_x^s)}.
		\end{align*}
		The usual adaption to the high-high part and summing over $\lambda\in 2^{\N_0}$ implies
		\begin{align}
			\label{eq:EstbsigmanHLHH}
			\|(b_{\sigma_n}\cdot \nabla u_{\sigma_n})_{HL+HH} \|_{L^2 (0,\tau^*-\sigma_n; H_x^{s-1})}
			\lesssim& \|b_{\sigma_n}\|_{L^2 (0,\tau^*-\sigma_n; H_x^{\frac{d}{2}+(s-1)_+})} \|u_{\sigma_n}\|_{L^\infty ([0,\tau^*-\sigma_n); H_x^s)}\nonumber\\
			\lesssim& \varepsilon \|u\|_{L^\infty ([0,\tau^*); H_x^s )},
		\end{align}
		where we also used H\"older's inequality and $\tau^*-\sigma_n<1$.
		
		For the low-high part, for $\mu \sim \lambda$,  we have
		\begin{align*}
			&\mu^{s-1} \|P_{\leq \frac{\mu}{2^8}} b_{\sigma_n} P_\mu(\nabla u_{\sigma_n})\|_{L_{t,x}^2}\lesssim \|b_{\sigma_n}\|_{L_{t,x}^\infty} \mu^{s}\|P_\mu u_{\sigma_n}\|_{L_{t,x}^2}
			\lesssim \|b_{\sigma_n}\|_{L_{t}^\infty H^{\frac{d+2}{2}}_x} \|P_\mu u_{\sigma_n}\|_{L_{t}^2 H_x^s}.
		\end{align*}
		Summing up over $\lambda$ and combining with the previous estimate, we thus obtain
		\begin{align}\label{es:usigma22}
			\|b_{\sigma_n}\cdot \nabla u_{\sigma_n} \|_{L^2 (0,\tau^*-\sigma_n; H_x^{s-1})}
			\lesssim \varepsilon\|u\|_{L^\infty ([0,\tau^*); H_x^s )}.
		\end{align}
		
		The other noise terms are easily controlled by
		\begin{align}\label{es:usigman3}
			&\| c_{\sigma_n}u_{\sigma_n} - \Re(\cT_{\sigma_n+\cdot, \sigma_n}(W_2))u_{\sigma_n} \|_{L^2 (0,\tau^*-\sigma_n; H_x^{s-1})}\nonumber\\
			&\lesssim (\tau^*-\sigma_n)^{\frac12}
			[\|c_{\sigma_n}\|_{L^\infty ([0,\tau^*-\sigma_n); H_x^{\frac{d}{2}+s-1})}
			+ \|\cT_{\sigma_n+\cdot, \sigma_n}(W_2)\|_{L^\infty ([0,\tau^*-\sigma_n); H_x^{\frac{d}{2}+s-1})}] \|u_{\sigma_n}\|_{L^\infty ([0,\tau^*-\sigma_n); H_x^s)}\notag \\
			&\lesssim \varepsilon\|u\|_{L^\infty ([0,\tau^*); H_x^s )}.
		\end{align}
		Thus, combining \eqref{es:usigman1} to \eqref{es:usigman3}, we obtain~\eqref{u-Sw-finit}, as claimed.
		
		\vspace*{4pt plus 2pt minus 2pt}%
		Next, we show the uniform bound
		\begin{align}\label{prop:boundedvu}
			\sup_{n\geq N} \|\Re(v_{\sigma_n})u_{\sigma_n}\|_{N^{s,0,0}([0,\tau^*-\sigma_n))}<\infty.
		\end{align}
		We write
		\begin{align*}
			v_{\sigma_n}(t)=e^{\imu t |\nabla|}v(\sigma_n)-\mathcal{J}_0[|\nabla||u_{\sigma_n}|^2].
		\end{align*}
		Note that by Lemma~\ref{lem:BilinearEstimates}~\ref{it:BilinearEstEndpoint2} and Remark~\ref{re:extendsl} in the Appendix, we have
		\begin{align}\label{prop:boundedv}
			\|v_{\sigma_n}\|_{W^{l,0,l}([0,\tau^*-\sigma_n))}
			& \lesssim\|v({\sigma_n})\|_{H_x^l} + \|u_{\sigma_n}\|_{L^2 (0,\tau^*-\sigma_n; W_x^{s,2^*})}\|u_{\sigma_n}\|_{S_w^{s,0,0}([0,\tau^*-\sigma_n))}\nonumber\\
			&\lesssim \|v\|_{L^\infty ([0,\tau^*); H_x^l )}+ \|u\|_{L^2 ([0,\tau^*); W_x^{s,2^*} )}\sup_{n\geq N}\|u_{\sigma_n}\|_{S_w^{s,0,0}([0,\tau^*-\sigma_n))}.
		\end{align}
		Moreover, by Lemma~\ref{lem:BilinearEstimates}~\ref{it:BilinearEstEndpoint1}, we have
		\begin{align}\label{prop:boundedvu1}
			\|\Re(v_{\sigma_n})u_{\sigma_n}\|_{N^{s,0,0}([0,\tau^*-\sigma_n))}\notag
			&\lesssim \|v_{\sigma_n}\|_{W^{l,0,l}([0,\tau^*-\sigma_n))} \|u_{\sigma_n}\|^{\frac12}_{L^2 (0,\tau^*-\sigma_n; W_x^{s,2^*})}\|u_{\sigma_n}\|^{\frac12}_{S_w^{s,0,0}([0,\tau^*-\sigma_n))}\notag\\
			&\lesssim \|v_{\sigma_n}\|_{W^{l,0,l}([0,\tau^*-\sigma_n))}\|u\|_{L^2 (0,\tau^*;  W_x^{s,2^*} )}^{\frac12}\sup_{n\geq N}\|u_{\sigma_n}\|^{\frac12}_{S_w^{s,0,0}([0,\tau^*-\sigma_n))}.
		\end{align}
		Thus, \eqref{prop:boundedvu} follows from \eqref{u-Sw-finit}, \eqref{prop:boundedv} and \eqref{prop:boundedvu1}.
		
		\vspace*{4pt plus 2pt minus 2pt}%
		Now we turn to the proof of the uniform bound~\eqref{prop:unibouep}. Let $\sigma_n<t<\tau^*$. The Bernstein inequality implies $\|F_\lambda\|_{N_\lambda^{s,0,0}} \lesssim \lambda^{s} \|F_\lambda\|_{L_t^2 L_x^{2_*}}$. By the definition of $\G^{s,0}$ and \eqref{prop:holdernoiseb2}, arguing as in~\eqref{eq:EstbsigmanHLHH} and~\eqref{es:usigman3}, we obtain
		\begin{align}\label{prop:GnormNoise}
			&\| (b_{\sigma_n}\cdot \nabla u_{\sigma_n})_{HL+HH}+c_{\sigma_n}u_{\sigma_n}-\Re(\cT_{\sigma_n+\cdot, \sigma_n}(W_2))u_{\sigma_n}\|_{\G^{s,0}([0,t-\sigma_n))}\nonumber\\
			&\lesssim (\tau^*-\sigma_n)^{\frac12} \big(\|b_{\sigma_n}\|_{L^\infty(0,t-\sigma_n;  H_x^{\frac{d}{2}+1+(s-1)_+})}
			+\|c_{\sigma_n}\|_{L^\infty(0,t-\sigma_n; H_x^{\frac{d}{2}+s-1})}
			\notag \\
			&\qquad
			+ \|\cT_{\sigma_n+\cdot, \sigma_n}(W_2)\|_{L^\infty(0,t-\sigma_n; H_x^{\frac{d}{2}+s-1})} \big) \|u_{\sigma_n}\|_{L^\infty ([0,\tau^*); H_x^s)} \nonumber\\
			&\lesssim \varepsilon \|u\|_{L^\infty ([0,\tau^*); H_x^s)}<\infty.
		\end{align}
		
		Regarding the $LH$-part of $(b_{\sigma_n}\cdot \nabla u_{\sigma_n})$, we decompose by modulation into
		\begin{align*}
			P_{\leq \frac{\lambda}{2^8}}b_{\sigma_n} P_\lambda(\nabla u_{\sigma_n}) = P_{\leq \frac{\lambda}{2^8}}b_{\sigma_n} C_{\leq (\frac{\lambda}{2^8})^2}P_\lambda(\nabla u_{\sigma_n}) + P_{\leq \frac{\lambda}{2^8}}b_{\sigma_n} C_{> (\frac{\lambda}{2^8})^2}P_\lambda(\nabla u_{\sigma_n}):= I_{LM} + I_{HM}.
		\end{align*}
		The $\G^{s,0}_\lambda$ norm of $I_{HM}$ is bounded by its $N^{s,0,0}_\lambda$ norm and hence is controlled by
		\begin{align}
			\lambda^s \| P_{\leq \frac{\lambda}{2^8}}b_{\sigma_n} C_{> (\frac{\lambda}{2^8})^2} P_{\lambda}(\nabla u_{\sigma_n})\|_{L_t^2 L_x^{2_*}}\notag
			&\lesssim  \|b_{\sigma_n}\|_{L_t^\infty L_x^d} \, \lambda^{s-1}\|(\imu \partial_t +\Delta) C_{> (\frac{\lambda}{2^8})^2} P_{\lambda}u_{\sigma_n}\|_{L_{t,x}^2}\notag\\
			&\lesssim  \|b_{\sigma_n}\|_{L_t^\infty H_x^{\frac{d}{2}-1}} \|P_{\lambda}u_{\sigma_n}\|_{\X_{\lambda}^{s,0}}\lesssim \varepsilon \|P_{\lambda}u_{\sigma_n}\|_{\X_{\lambda}^{s,0}}.\label{prop:NoiseHMestimate}
		\end{align}
		For the $I_{LM}$ part, we control its $\G^{s,0}_\lambda$-norm via the local smoothing component to infer
		\begin{align}
			&\lambda^{s-\frac12} \| P_{\leq \frac{\lambda}{2^8}}b_{\sigma_n} C_{\leq (\frac{\lambda}{2^8})^2}P_\lambda(\nabla u_{\sigma_n}) \|_{L_{\vece_j}^{1,2}}\notag\\
			&\lesssim \lambda^{s+\frac12} \||P_{\leq \frac{\lambda}{2^8}}b_{\sigma_n}|^{\frac12}\|_{L_{\vece_j}^{2,\infty}} \||P_{\leq \frac{\lambda}{2^8}}b_{\sigma_n}|^{\frac12} C_{\leq (\frac{\lambda}{2^8})^2} P_{\lambda}u_{\sigma_n}\|_{L_{t,x}^2}\notag\\
			&\lesssim \|b_{\sigma_n}\|^{\frac12}_{L_{\vece_j}^{1,\infty}} \lambda^{s+\frac12} \sum_{j=1}^d \||P_{\leq \frac{\lambda}{2^8}}b_{\sigma_n}|^{\frac12}\|_{L_{\vece_j}^{2,\infty}} \|P_{\lambda,\vece_j} C_{\leq (\frac{\lambda}{2^8})^2} P_{\lambda}u_{\sigma_n}\|_{L_{\vece_j}^{\infty,2}}\notag\\
			&\lesssim \sum_{j=1}^d \|b_{\sigma_n}\|_{L_{\vece_j}^{1,\infty}}  \sum_{j=1}^d \lambda^{s+\frac12} \| P_{\lambda,\vece_j} C_{\leq (\frac{\lambda}{2^8})^2} P_{\lambda}u_{\sigma_n}\|_{L_{\vece_j}^{\infty,2}}\lesssim \varepsilon  \|P_{\lambda} u_{\sigma_n}\|_{\X_{\lambda}^{s,0}}. \label{prop:NoiseLMestimate}
		\end{align}
		
		Combining the estimates~\eqref{prop:NoiseHMestimate} and~\eqref{prop:NoiseLMestimate}, the Duhamel expression~\eqref{usign-Duhamel} with Lemma~\ref{lem:LinEstimates}, we arrive at
		\begin{align*}
			\|P_{\lambda} u_{\sigma_n}\|_{\X_\lambda^{s,0}([0,t-\sigma_n])} 
			&\lesssim \lambda^{s}\|P_\lambda (u(\sigma_n))\|_{L_x^2} + \|P_{\lambda}(\Re(v_{\sigma_n})u_{\sigma_n})\|_{N_\lambda^{s,0,0}([0,\tau^*-\sigma_n))} + \sum_{\mu\sim\lambda} \varepsilon \|P_\mu u_{\sigma_n}\|_{\X_\mu^{s,0}([0,t-\sigma_n))} \nonumber\\
			&\qquad +\|P_{\lambda}[ (b_{\sigma_n}\cdot \nabla u_{\sigma_n})_{HL+HH}+c_{\sigma_n}u_{\sigma_n}-\Re(\cT_{\sigma_n+\cdot, \sigma_n}(W_2))u_{\sigma_n}]\|_{\G_\lambda^{s,0}([0,\tau^*-\sigma_n))} ,
		\end{align*}
		which, by summing over $\lambda\in 2^{\N_0}$, yields that
		\begin{align}
			\|u_{\sigma_n}\|_{\mathbb{X}^{s,0}([0,t-\sigma_n])}\nonumber
			&\lesssim \| u\|_{L^\infty ([0,\tau^*);H_x^s)} + \varepsilon \|u_{\sigma_n}\|_{\mathbb{X}^{s,0}([0,t-\sigma_n))}+ \|\Re(v_{\sigma_n})u_{\sigma_n}\|_{N^{s,0,0}([0,\tau^*-\sigma_n))},
		\end{align}
		where we also used~\eqref{prop:GnormNoise}. We point out that the implicit constant is independent of $n$ and $\varepsilon$.
		Since $t<\tau^*$ we can take $m$ large enough such that $[0,t) \subseteq [0,\sigma_m]$ and $u\in \mathbb{X}^{s,0}([0,\sigma_m])$, thus, $u_{\sigma_n} \in \mathbb{X}^{s,0}([0,t-\sigma_n))$.
		Hence, taking $\varepsilon$ small enough, we infer from the above estimate that
		\begin{align*}
			\|u_{\sigma_n}\|_{\mathbb{X}^{s,0}([0,t-\sigma_n])}
			&\lesssim \| u\|_{L^\infty ([0,\tau^*);H_x^s)} + \sup_{n\geq N}\|\Re(v_{\sigma_n})u_{\sigma_n}\|_{N^{s,0,0}([0,\tau^*-\sigma_n))}.
		\end{align*}
		Combining Lemma~\ref{lem:ProductNoiseInX} and Lemma~\ref{lem:DecompX}, we conclude that
		\begin{align*}
			\sup_{n\geq N}\sup_{\sigma_{n}<t<\tau^*}\|u\|_{\mathbb{X}^{s,0}([\sigma_{n},t))}<\infty,
		\end{align*}
		which then immediately yields a uniform $\mathbb{X}^{s,0}$-bound on $[0,\tau^*]$ and hence \eqref{prop:unibouep} follows.

		Finally, we turn to the continuous extension of $v$. For any $t_1<t_2<\tau^*$, we have
		\begin{align*}
			&\|e^{-\imu t_1|\nabla|}v(t_1)-e^{-\imu t_2|\nabla|}v(t_2)\|_{H_x^{l}}=\Big\| \int_{t_1}^{t_2} e^{-\imu s|\nabla|} (|\nabla||u|^2)(s)d s  \Big\|_{W^{l,0,l}([t_1,t_2])}\\
			&\lesssim \|u\|_{L^2 (t_1,t_2; W_x^{s,2^*})} \|u\|_{S_w^{s,0,0}([t_1,t_2])}\lesssim \|u\|_{L^2 (t_1,t_2; W_x^{s,2^*})} \sup_n\|u\|_{\mathbb{X}^{s,0}([0,\sigma_n])}\rightarrow 0 \quad \text{as}\quad  t_1,t_2\rightarrow \tau^*,
		\end{align*}
		where we exploited the uniform bound \eqref{prop:UniBonCon} and \eqref{prop:unibouep}.
		Hence, $\{e^{-\imu t |\nabla|}v(t)\}$ is a Cauchy sequence in $H_x^l$,
		and converges to some $v_* \in H^l_x$ as $t\to \tau^*$.
		Letting
		\begin{align*}
			v^\prime(t):=\mathbbm{1}_{(-\infty,0)}(t)e^{\imu t|\nabla|}v_0 + \mathbbm{1}_{[0,\tau^*)}(t)v(t) +\mathbbm{1}_{(\tau^*,\infty)}(t)e^{\imu (t-\tau^*)|\nabla|}v_*
		\end{align*}
		we see that $v^\prime\in C(\R, H_x^l)$ and $v$ can be continuously extended to $[0,\tau^*]$.
		The proof is complete.
	\end{proof}

	\section{LWP and blow-up alternative: non-endpoint case and persistence of regularity}
	\label{Sec-LWP}
	
	In this section we prove Theorem \ref{Thm-LWPNEP}, i.e. local well-posedness and the blow-up alternative, in the non-endpoint case $(s,l)\neq (\frac{d-3}{2},\frac{d-4}{2})$.
	
	First we establish the well-posedness of the linear Schr\"odinger equation with potentials in $W^{l,a,\beta}$, where the potential is a perturbation of a free wave. With the estimates in Section~\ref{sec:FunctFramework} and Section~\ref{Sec-Esti-Nonl-Noise},
		the argument follows along the same lines as the proof of~\cite[Theorem~7.1]{CHN23}.
	
	\begin{lemma}[Linear Schr\"odinger estimate perturbed by free wave I]
		\label{lem:LinSchrPotSmallTime}
		Let $0\leq s\leq l+2$ and $l\geq \frac{d}{2}-2$ with $(s,l)\neq (\frac{d}{2},\frac{d}{2}-2)$, $\beta=s-\frac12$ and $a=a^*(s,l)$
		given by \eqref{con:ab}.
		Then, there exist $\varepsilon$ and $C > 0$ such that for any interval $I \subset \R$, $t_0 \in I$, $f \in H_x^s$,  $F\in L_t^2 L_x^{2_*}$ with $\cI_0[F] \in \X^{s,a}(I)$, and $V \in W^{l,a,\beta}(I)$ satisfying
		\begin{align*}
			\|V\|_{W^{l,a,\beta}(I) + L^2 (I ; W_x^{s,d})} < \varepsilon,
		\end{align*}
		the Cauchy problem
		\begin{align}
			\label{eq:SchroedingerPotential}
			(\imu \partial_t  + \Delta - \Re(V))u = F, \qquad u(t_0) = f
		\end{align}
		has a unique solution $u \in C(I; H_x^s) \cap L^2 (I ; L_x^{2^*})$, which satisfies
		\begin{align*}
			\|u\|_{\XS(I)} \leq C(\|f\|_{H_x^s} + \|\cI_0[F]\|_{\X^{s,a}(I)}),
		\end{align*}
		and
		\begin{align*}
			\|u\|_{L^2 (I ; L_x^{2^*})} \leq C( \|f\|_{L_x^2} + \|F\|_{L^2 (I ; L_x^{2_*})}).
		\end{align*}
	\end{lemma}
	
	If the potential is a linear wave,
	we recall
	from Lemma 7.5 of \cite{CHN23} that we can decompose $\R$ into a finite collection of intervals such that the smallness condition in Lemma~\ref{lem:LinSchrPotSmallTime} is satisfied on each of these intervals.
	
	\begin{lemma}[\cite{CHN23}, Lemma~7.5]
		\label{lem:SmallnessWavePotential}
		Let $l, s, a \geq 0, \varepsilon >0$, and $V_L=e^{\imu t |\nabla|}g\in L_t^\infty H_x^l$. Then there
		exists a finite collection of open intervals $\{I_j\}_{j=1}^N$,
		such that $\R =\cup_{j=1}^N I_j$ with $\min|I_j\cap I_{j+1}|>0$, and
		\begin{align*}
			\sup_{j=1,\dots,N} \|V_L\|_{W^{l,a,\beta}(I_j)+L^2 (I_j; W_x^{s,d})}<\varepsilon
		\end{align*}
	\end{lemma}
	
	\begin{remark}
		\label{rem:PropagOPWellDef}
		The combination of Lemma~\ref{lem:LinSchrPotSmallTime}, Lemma~\ref{lem:SmallnessWavePotential}, and Lemma~\ref{lem:DecompX} shows that~\eqref{eq:SchroedingerPotential} has a unique solution in $C(I; H_x^s) \cap L^2(I; L_x^{2^*})$ on any interval $I$ with $t_0 \in I$, cf. the proof of Theorem~7.1 in~\cite{CHN23}. This solution also satisfies the estimates in Lemma~\ref{lem:LinSchrPotSmallTime}. In particular, the homogeneous and inhomogeneous propagation operators $\cU_{v_L}$ and $\cI_{v_L}$ of~\eqref{eq:SchroedingerPotential} are well-defined and satisfy the estimates in Lemma~\ref{lem:LinSchrPotSmallTime} for free wave potentials $v_L$.
	\end{remark}
	
	We divide the proof of Theorem \ref{thm:LocalWP} for the non-endpoint case $(s,l) \neq (\frac{d-3}{2},\frac{d-4}{2})$ into four steps, presented in Subsections~\ref{Subsec-LWPNEP} through~\ref{Subsec:ImproBlowUpCon}.
	
	First, utilizing the rescaling transform, we address the initial value problem for~\eqref{eq:RanZakbc} with data in $H_x^s \times H_x^l$. In Subsection~\ref{Subsec-LWPNEP}, we construct a unique local solution in the auxiliary space $H_x^s \times H_x^{\tilde{l}}$, where $\tilde{l} := \min\{s-\frac12,l\}$. Next, we employ the refined rescaling transforms (see Appendix~\ref{Sec-Rescaling}) to extend this solution to its maximal existence time in Subsection~\ref{Subsec-Max-Exist}.
	
	If $l + \frac12 \leq s$ (implying $\tilde{l} = l$), this immediately yields the unique maximal solution in $H_x^s \times H_x^l$. However, if $l > s - \frac12$, we invoke the regularity improvement result from Lemma~\ref{le:PersisReg} to upgrade the wave regularity from $H_x^{\tilde{l}}$ to $H_x^l$. Consequently, we obtain unique maximal solutions in the sense of Definition~\ref{def:Solution}.
	
	Finally, in Subsection~\ref{Subsec-Blowup}, we characterize a finite maximal existence time via a blow-up condition at regularity $(s, \tilde{l})$. We then refine this to the endpoint regularity condition stated in Theorem~\ref{thm:LocalWP} using two persistence of regularity results in Subsection~\ref{Subsec:ImproBlowUpCon}.

	\subsection{Local well-posedness}  \label{Subsec-LWPNEP}
	
	In this subsection, we first prove that system \eqref{eq:StoZak} is locally well-posed up to some stopping time.
	In view of the rescaling approach, see Subsection~\ref{subsec:IdeaProof},
	it is equivalent to solve
	the following Zakharov system with random coefficients
	\begin{equation}  \label{eq:RanZakWavePot}
		\left\{\aligned
		(\imu \partial_t + \Delta - \Re(v_L)) u &= \Re(\rho) u - b\cdot \na u - cu + \Re(\cT_t(W_2)) u, \qquad &u(0) &= u_0, \\
		(\imu \partial_t + |\nabla|) \rho &= - |\nabla| |u|^2, &\rho(0) &= 0,
		\endaligned
		\right.
	\end{equation}
	with 	
	\begin{align*}
		u_0  = X_0, \ \ \rho := v - v_L,
		\ \ {\rm and }\quad  v_L :=  e^{\imu t |\nabla|} v_0.
	\end{align*}
	Recall that $\cU_{v_L}[f]$ and $\cI_{v_L}[g]$ are the homogeneous and inhomogeneous propagation operators of the Schr\"odinger equation with potential $\Re(v_L)$,
	\begin{align*}
		(\imu \partial_t +\Delta - \Re(v_L))u = g, \qquad u(t_0) = f,
	\end{align*}
	respectively, and $\cJ_0[h]$ is the inhomogeneous propagation operator of the wave equation
	\begin{align*}
		(\imu \partial_t + |\nabla|) v = h, \qquad v(t_0) = 0.
	\end{align*}
	Then $\rho = - \cJ_0[|\nabla| |u|^2]$ and
	it suffices to solve the equation
	\begin{align}
		u(t) = \cU_{v_L} [u_0](t) - \cI_{v_L}[\Re(\cJ_0[|\nabla| |u|^2])u](t) - \cI_{v_L}[b \cdot \nabla u + c u - \Re(\cT_{\cdot}(W_2)) u](t).
		\label{eq:LWPFixedPoint}
	\end{align}
	
	For this purpose,
	we define the fixed-point operator $\Phi(u_0, v_0; u)$ by the right-hand side of~\eqref{eq:LWPFixedPoint},
	and we will prove that it is a contractive self-mapping on the ball
	\begin{align*}
		B_{R, \delta}(\tau) := \{u \in \XS([0,\tau]) \colon \
		\|u\|_{L^2 (0,\tau ; L_x^{2^*})} \leq \delta, \  \|u\|_{\XS([0,\tau])} \leq R\}
	\end{align*}	
	equipped with the metric induced by $\| \cdot \|_{\XS([0,\tau])}$,
	for some $\delta, R > 0$
	and some positive stopping time $\tau$.
	Note that $B_{R, \delta}(\tau)$ is a complete metric space.
	
	We sometimes drop the domain of the space-time norms below if it is clear from the context.

	\vspace*{4pt plus 2pt minus 2pt}%
	\paragraph{\bf $\bullet$ Self-mapping}

	Let $\epsilon > 0$ be as in Lemma~\ref{lem:LinSchrPotSmallTime}, $T > 0$, and
	set
	\begin{equation}
		\label{eq:DefTauTilde}
		\tau_0 := \inf\{t \in [0,T] \colon \|v_L\|_{W^{\tilde{l},a,s-\frac{1}{2}}([0,t]) + L^2 (0,t ; W_x^{s,d})}
		\geq \varepsilon\} \wedge \min\{2,T\},
	\end{equation}
	which is strictly positive, due to Lemma~\ref{lem:SmallnessWavePotential}.
	In the following, let $0 \leq \tau \leq \tau_0$ be a stopping time.
	
	By Lemma~\ref{lem:LinSchrPotSmallTime} and Lemma~\ref{lem:LinEstimates}, we have
	\begin{align}
		\| \Phi(u_0, v_0; u) \|_{\XS([0,\tau])}
		&\leq C \|u_0\|_{H_x^s} + C \|\Re(\cJ_0[|\nabla| |u|^2])u\|_{N^{s,a,0}([0,\tau])} + C \|\cI_0[b \cdot \nabla u]\|_{\X^{s,a}([0,\tau])}\nonumber\\
		&\qquad + C \|c u\|_{\GS([0,\tau])}  + C \|\Re(\cT_{\cdot}(W_2)) u\|_{\GS([0,\tau])}. \label{eq:EstFixedPointOp1}
	\end{align}
	To estimate the right-hand side above,
	we note that, by  Lemma~\ref{lem:BilinearEstimates}~\ref{it:BilinearEstNonendpointbNzero} and \ref{it:BilinearEstNonendpointW},
	\begin{align}
		\| \Re(\cJ_0[|\nabla| |u|^2])u\|_{\NOne([0,\tau])}   & \lesssim \| \cJ_0[|\nabla| |u|^2] \|_{W^{\tilde{l},a,s-\frac{1}{2}}([0,\tau])} \|u\|_{\SOne([0,\tau])} \nonumber\\
		&\lesssim \|u\|_{L^2_t L^{2^*}_x}^{2 \theta} \|u\|_{\XS([0,\tau])}^{3 - 2 \theta} \label{eq:FixedPointOpNonlin}
	\end{align}
	for some $\theta \in (0,1)$. Let
	\begin{align}
		\label{eq:DefWstar}
		W_*(t)&:=\| W_1(t)\|_{H_x^{\frac{d}{2}+2+(s-1)_+}} + \sum_{j = 1}^d \sum_{k = 1}^\infty \int \sup_{y \in \R^{d-1}} |\nabla \phi^{(1)}_k(r \vece_j + y)| \dd r \|\beta^{(1)}_k\|_{C^{\frac14}([0,t])}  \nonumber\\
		& \qquad +\|\cT_t(W_2)\|_{{H_x^{\frac{d}{2}+s-1}}} + \|\nabla W_1(t)\|_{C^{\frac14}([0,t]; L_x^\infty)} + \|W_1(t)\|_{{H_x^{\frac{d}{2}+s+1}}}^2,
	\end{align}
	and then $W^*(t) := \sup_{s\in [0,t]} W_*(s)$. By Lemma~\ref{lem:BilinearLowerOrder}, Lemma~\ref{lem:BilinearLowerOrder2}, and the identities~\eqref{eq:Defb} and~\eqref{eq:Defc}, the noise terms in \eqref{eq:EstFixedPointOp1} can be controlled by $(1+\tau^{\frac12})W^*(\tau) \|u\|_{\XS([0,\tau])}$.
	Thus, we infer
	\begin{align}
		\| \Phi(u_0, v_0; u) \|_{\XS([0,\tau])}
		\leq C \|u_0\|_{H_x^s} +  C \|u\|_{L^2_t L^{2^*}_x}^{2 \theta} \|u\|_{\XS([0,\tau])}^{3 - 2 \theta}
		+ C (1+\tau^{\frac{1}{2}})W^*(\tau) \|u\|_{\XS([0,\tau])}. \label{eq:EstFixedPointOp2}
	\end{align}
	
	Concerning the estimate of the $L^2_tL^{2^*}_x$-norm,
	we write the linear propagator $\cU_{v_L}$ as
	\begin{align}  \label{UvL-express}
		\cU_{v_L}[u_0](t) = e^{\imu t \Delta} u_0 + \cI_{v_L}[\Re(v_L) e^{\imu (\cdot) \Delta} u_0](t).
	\end{align}
	By Lemma~\ref{lem:LinSchrPotSmallTime}, Sobolev embedding, and the fact $l\geq \frac{d}{2}-2$, we estimate
	\begin{align*}
		\|\cU_{v_L}[u_0] \|_{L^2_t L^{2^*}_x} &\lesssim \| e^{\imu t \Delta} u_0 \|_{L^2_t L^{2^*}_x} + \|\Re(v_L) e^{\imu t \Delta} u_0\|_{L^2_t L^{2_*}_x} \notag \\
		& \lesssim \| e^{\imu t \Delta} u_0 \|_{L^2_t L^{2^*}_x} + \|v_L\|_{L^\infty_t L^{\frac{d}{2}}_x} \| e^{\imu t \Delta} u_0 \|_{L^2_t L^{2^*}_x} \\
		&\lesssim (1 + \|v_0\|_{H_x^l}) \| e^{\imu t \Delta} u_0 \|_{L^2_t L^{2^*}_x}.
	\end{align*}
	Moreover, Lemma~\ref{lem:LinSchrPotSmallTime}, Lemma~\ref{lem:BilinearEstimates} \ref{it:BilinearEstNonendpointW} and the embedding $\X^{s,a}\hookrightarrow S^{s,a,0}, H_x^{\tilde{l}}\hookrightarrow L_x^{\frac{d}{2}}$ yield
	\begin{align*}
		\| \cI_{v_L}[\cJ_0[|\nabla| |u|^2] u] \|_{L^2_t L^{2^*}_x}
		& \lesssim \|\cJ_0[|\nabla| |u|^2]\|_{L^\infty_t L^{\frac{d}{2}}_x} \|u\|_{L^2_t L^{2^*}_x} \\
		& \lesssim  \|\cJ_0[|\nabla| |u|^2]\|_{W^{\tilde{l},a,s-\frac{1}{2}}([0,\tau])} \|u\|_{L^2_t L^{2^*}_x} \\
		&\lesssim  \|u\|_{\XS([0,\tau])}^{2(1-\theta)} \|u\|_{L^2_t L^{2^*}_x}^{2\theta + 1}.
	\end{align*}
	Taking into account Lemma~\ref{lem:BilinearLowerOrder} we thus arrive at
	\begin{align}
		\|\Phi(u_0, v_0; u)\|_{L^2_t L^{2^*}_x}
		\leq C  \| e^{\imu t \Delta} u_0 \|_{L^2_t L^{2^*}_x} +  C \|u\|_{\XS([0,\tau])}^{2(1-\theta)} \|u\|_{L^2_t L^{2^*}_x}^{2\theta + 1} + C (1+\tau^{\frac{1}{2}})W^*(\tau) \|u\|_{\XS([0,\tau])}. \label{eq:EstFixedPointOp3}
	\end{align}
	
	Fixing now $C = C(\|v_0\|_{H_x^l})$ as the maximum of the generic constants in~\eqref{eq:EstFixedPointOp2} and~\eqref{eq:EstFixedPointOp3} and setting $R = 2 C \|u_0\|_{H_x^s}$, we derive from~\eqref{eq:EstFixedPointOp2} and~\eqref{eq:EstFixedPointOp3} that
	for any  $u \in B_{R, \delta}(\tau)$,
	\begin{align*}
		\| \Phi(u_0, v_0; u) \|_{\XS([0,\tau])}
		&\leq \frac{R}{2} + C( R^{2(1 -  \theta)} \delta^{2 \theta} R  + (1+\tau^{\frac{1}{2}}) R W^*(\tau) ),\\
		\| \Phi(u_0, v_0; u) \|_{L^2_t L^{2^*}_x}
		& \leq   C( \| e^{\imu t \Delta} u_0 \|_{L^2_t L^{2^*}_x} +   R^{2(1-\theta)} \delta^{2\theta+1}   + (1+\tau^{\frac{1}{2}}) R W^*(\tau) ).
	\end{align*}
	These  estimates yield that $\Phi(u_0, v_0; \cdot)$
	is self-mapping on $B_{R,\delta}(\tilde{\tau})$,
	where $\delta \in (0,R)$ is chosen small enough such that
	$8 C R^{2(1 - \theta)} \delta^{2 \theta} \leq 1$ and
	$\tilde{\tau}$ is the stopping time defined by
	\begin{align}
		\label{eq:DefFixedPointStoppingTimeSelfmap}
		\tilde{\tau} := \inf\Big\{t \in [0,T] \colon C  \| e^{\imu (\cdot) \Delta} u_0 \|_{L^2 (0,t ; L_x^{2^*})} +	2 C R  W^*(t) \geq \frac{\delta}{4}\Big\} \wedge \tau_0.
	\end{align}
	Note that $\tilde{\tau} > 0$, $\PP$-a.s.,
	since $\lim_{t \rightarrow 0^+} W^*(t) = 0$, $\PP$-a.s.

	\vspace*{4pt plus 2pt minus 2pt}%
	\paragraph{\bf $\bullet$ Contraction}
	To show that $\Phi(u_0, v_0; \cdot)$ is a contraction, we estimate as in~\eqref{eq:FixedPointOpNonlin} to infer
	\begin{align*}
		&\|\Re(\cJ_0[|\nabla| |u|^2]) u - \Re(\cJ_0[|\nabla| |w|^2]) w\|_{\GS([0,\tau])} \\
		&= \|\Re(\cJ_0[|\nabla|((\overline{u} - \overline{w})u)]) u + \Re(\cJ_0[|\nabla| (\overline{w}(u - w))]) u  + \Re(\cJ_0[|\nabla| |w|^2]) (u - w)\|_{\GS([0,\tau])} \\
		&\lesssim (\|u-w\|_{\XS([0,\tau])} \|u\|_{\XS([0,\tau])} )^{1-\theta} (\|u-w\|_{L^2_t L^{2^*}_x} \|u\|_{L^2_t L^{2^*}_x})^\theta \|u\|_{\XS([0,\tau])}   \\
		&\qquad + (\|w\|_{\XS([0,\tau])} \|u-w\|_{\XS([0,\tau])} )^{1-\theta} (\|w\|_{L^2_t L^{2^*}_x} \|u-w\|_{L^2_t L^{2^*}_x})^\theta \|u\|_{\XS([0,\tau])}  \\
		& \qquad +\|w\|_{L^2_t L^{2^*}_x}^{2\theta} \|w\|_{\XS([0,\tau])}^{2(1-\theta)} \|u - w\|_{\XS([0,\tau])} \\
		&\lesssim \delta^{\theta} R^{2 - \theta} \|u-w\|_{\XS([0,\tau])} 
	\end{align*}
	for all $u,w \in B_{R,\delta}(\tau)$, where we also used Remark~\ref{rem:NormComp} and $\delta \in (0,R)$ in the last inequality. Along the same lines as deriving~\eqref{eq:EstFixedPointOp2}, we get
	\begin{align}
		\|\Phi(u_0, v_0; u) - \Phi(u_0, v_0; w)\|_{\XS([0,\tau])}
		\leq  C  (\delta^{\theta} R^{2 - \theta} +  \delta^{2 \theta} R^{2(1 - \theta)}  +  (1+\tau^{\frac{1}{2}})W^*(\tau) )\|u-w\|_{\XS([0,\tau])}.  \label{eq:EstFixedPointOpDiff}
	\end{align}
	
	This immediately yields that
	$\Phi(u_0, v_0; \cdot)$ is a contractive self-mapping on $B_{R,\delta}(\tilde{\tau}_1)$,
	by taking $\delta > 0$ possibly smaller such that additionally
	\begin{align*}
		C \delta^\theta R^{2- \theta} \leq \frac{1}{4},
	\end{align*}
	and by choosing the stopping time $\tilde{\tau}_1$ defined by
	\begin{align*}
		\tilde{\tau}_1 := \inf\Big\{t \in [0,T]: C W^*(t) \geq \frac{1}{4}\Big\} \wedge \tilde{\tau}.
	\end{align*}			
	Note that $\tilde{\tau}_1 > 0$ $\PP$-a.s.
	
	Since
	the constant $C$ and the radius $R$ are increasing in $\|u_0\|_{H_x^s}$ and $\|v_0\|_{H_x^l}$,
	we can choose a small constant $\delta_*(\|u_0\|_{H_x^s}, \|v_0\|_{H_x^l}) > 0$, which is decreasing in both its arguments, such that
	\begin{align}
		\label{eq:FixedPointDefStoppingTimeFirstIntervalExt}
		\tilde{\tau}_1 = \inf\{t \in [0, T] \colon  \|e^{\imu (\cdot) \Delta} u_0\|_{L^2 (0,t ; L_x^{2^*})} + W^*(t) \geq 4 \delta_*(\|u_0\|_{H_x^s}, \|v_0\|_{H_x^l})\} \wedge \tau_0.
	\end{align}
	Moreover, for the gluing procedure in Subsection~\ref{Subsec-Max-Exist} below
	(see also Proposition~\ref{prop:GluingSolutions} and Lemma~\ref{lem:DecompX}),
	we define another (smaller) stopping time by
	\begin{align}
		\label{eq:FixedPointDefStoppingTimeFirstInterval}
		\tau_1 &= \inf\{t \in [0, T] \colon  \|e^{\imu (\cdot) \Delta} u_0\|_{L^2 (0,t ; L_x^{2^*})} + W^*(t) \geq 2 \delta_*(\|u_0\|_{H_x^s}, \|v_0\|_{H_x^l})\} \nonumber\\
		&\qquad \wedge \inf\Big\{t \in [0,T] \colon \|v_L\|_{W^{\tilde{l},a,s-\frac{1}{2}}([0,t]) + L^2 (0,t ; W_x^{s,d})} \geq  \frac{\epsilon}{2}\Big\} \wedge \min\{1,T\}.
	\end{align}
	Note that, $\tau_1 > 0$ $\PP$-a.s. due to Lemma~\ref{lem:LinSchrPotSmallTime}.
	Moreover, $\tau_1 < \tilde{\tau}_1$ or $\tau_1 = T$ $\PP$-a.s.,
	due to the continuity property in Lemma~\ref{it:ContSumY}.

	\vspace*{4pt plus 2pt minus 2pt}%
	Now, by virtue of the self-mapping and contraction of the map $\Phi(u_0, v_0; \cdot)$
	on $B_{R,\delta}(\tilde{\tau}_1)$,
	we infer from Banach's fixed point theorem that there exists a solution $\tilde{u}_1 \in B_{R,\delta}(\tilde{\tau}_1)$ to~\eqref{eq:LWPFixedPoint}, which is unique in $\XS([0,\tilde{\tau}_1])$ by standard arguments. 
	We then set
	\begin{align*}
		\tilde{v}_1(t) := v_L(t) - \cJ_0[|\nabla| |\tilde{u}_1|^2](t) = e^{\imu t |\nabla|} v_0  - \cJ_0[|\nabla| |\tilde{u}_1|^2](t), \ \
		t\in [0,\tilde{\tau}_1],
	\end{align*}
	which belongs to $W^{\tilde{l},a,s-\frac{1}{2}}([0,\tilde{\tau}_1])$
	due to Lemma~\ref{lem:BilinearEstimates} \ref{it:BilinearEstNonendpointW}. 
	Hence, $(\tilde{u}_1, \tilde{v}_1)$ is the solution to~\eqref{eq:RanZakbc}, where the uniqueness holds in $\XS([0,\tilde{\tau}_1]) \times L^\infty ([0,\tilde{\tau}_1] ; H_x^{\tilde{l}})$.
	Let
	\begin{align*}
		(u_1,v_1)(t) := (\tilde{u}_1(t \wedge \tilde{\tau}_1), \tilde{v}_1(t \wedge \tilde{\tau}_1)),\ \ t\in [0,T].
	\end{align*}
	Then, $(u_1, v_1)$ is an $\{\cF_t\}$-adapted process in $C([0,T], H_x^s \times H_x^{\tilde{l}})$
	and solves~\eqref{eq:RanZakbc} on $[0,\tilde{\tau}_1]$.

	\subsection{Extension to maximal existence time}  \label{Subsec-Max-Exist}
	In this step,
	we extend the solution from Step~1 to its maximal existence time.
	The proof utilizes the inductive application of refined rescaling transforms and the gluing procedure.
	As it is analogous to that in \cite{HRSZ24},
	we briefly sketch the main arguments below.

	\vspace*{4pt plus 2pt minus 2pt}%
	\paragraph{\bf $\bullet$ Inductive arguments}
	For $n \in \N$, let $(u_n, v_n)$ be an $\{\cF_t\}$-adapted continuous process
		in $H_x^s \times H_x^{\tilde{l}}$ and $\sigma_n \leq \tilde{\sigma}_n$ be $\{\cF_t\}$-stopping times,
		such that $\sigma_n < \tilde{\sigma}_n$ or $\sigma_n = T$, $\PP$-a.s.,
		and $(u_n, v_n)$ be the unique solution of~\eqref{eq:RanZakbc} on $[0,\tilde{\sigma}_n]$ in $\XS([0,\tilde{\sigma}_n]) \times W^{\tilde{l},a,s-\frac{1}{2}}([0,\tilde{\sigma}_n])$
		satisfying $(u_n, v_n) \equiv (u_n(\tilde{\sigma}_n), v_n(\tilde{\sigma}_n))$ on $[\tilde{\sigma}_n, T]$.
	
	In view of Propositions \ref{prop:RefinedRescaling} and \ref{prop:GluingSolutions},
	we consider the system
	\begin{equation}   \label{eq:RanZakbsigmacsigmaInProof}
		\left\{\aligned
		\imu \partial_t u_\sigma + \Delta u_\sigma
		&= \Re(v_\sigma) u_\sigma - b_\sigma \cdot \nabla u_\sigma - c_\sigma u_\sigma + \Re(\cT_{\sigma + \cdot, \sigma}(W_2)) u_\sigma,  \\
		\imu \partial_t v_\sigma + |\nabla |v_\sigma  &= - |\nabla||u_\sigma|^2, \\
		(u_\sigma(0), v_\sigma(0)) &= (u_{0,n}, v_{0,n})
		\endaligned
		\right.
	\end{equation}
	with the initial data
	\begin{align}
		\label{eq:DefInitialDataRescaled}
		(u_{0,n}, v_{0,n}) = (e^{W_1(\sigma_n)} u_n(\sigma_n), v_n(\sigma_n) + \cT_{\sigma_n}(W_2)),
	\end{align}
	where
	\begin{align}
		\label{eq:DefbsigmacsigmaInProof}
		&b_\sigma = 2 \nabla W_{1,\sigma}, \quad c_\sigma = |\nabla W_{1,\sigma}|^2 + \Delta W_{1,\sigma}, \quad W_{1,\sigma}(t)= W_1(\sigma+t) - W_1(\sigma), \qquad \text{and } \nonumber\\
		&\cT_{\sigma+t, \sigma} (W_2)
		= -\imu \int_\sigma^{\sigma+t} e^{\imu (\sigma+t-s)|\na|} \dd W_2(s), \quad t \in [0,T].
	\end{align}	
	
	Proceeding as in Subsection~\ref{Subsec-LWPNEP},
	we define the fixed point operator 
	\begin{align}\label{eq:FixedPointusigma}
		&\Phi_\sigma(u_{0,n}, v_{0,n}; u_\sigma)\notag\\
		&:= \cU_{v_{L,n}}[u_{0,n}](t) - \cI_{v_{L,n}}[\Re(\cJ_0[|\nabla| |u_\sigma|^2])u_\sigma](t) - \cI_{v_{L,n}}[b_\sigma \cdot \nabla u_\sigma + c_\sigma u_\sigma - \Re(\cT_{\sigma + \cdot,\sigma}(W_2)) u_\sigma](t),
	\end{align}
	where $v_{L,n}(t) := e^{\imu t |\nabla|} v_{0,n}$,
	and define for the increment of the noise
	\begin{align}
		\label{eq:DefWStarsigman}
		W_{*,\sigma_n}(t) &:= \| W_{1,\sigma_n}(t)\|_{H_x^{\frac{d}{2}+2+(s-1)_+}} + \sum_{j = 1}^d \sum_{k = 1}^\infty \int \sup_{y \in \R^{d-1}} |\nabla \phi^{(1)}_k(r \vece_j + y)| \dd r \|\beta^{(1)}_k (~\cdot+\sigma_n)-\beta^{(1)}_k(\sigma_n) \|_{C^{\frac14}([0,t])} \nonumber\\
		&\qquad + \|\nabla W_{1,\sigma_n}(t)\|_{C^{\frac14}([0,t]; L_x^\infty)}+ \|W_{1,\sigma_n}(t)\|_{{H_x^{\frac{d}{2}+s+1}}}^2 + \|\cT_{\sigma_n + t,\sigma_n}(W_2)\|_{{H_x^{\frac{d}{2}+s-1}}},
	\end{align}
	$W^*_{\sigma_n}(t): = \sup_{s\in [0,t]} W_{*,\sigma_n}(s)$, and
	two $\{\cF_{\sigma_n + t}\}$-stopping times
	\begin{align*}
		\tilde{\tau}_{n+1} &=\inf\{t \in [0, T] \colon  \|e^{\imu (\cdot) \Delta} u_{0,n}\|_{L^2 (0,t ; L_x^{2^*})} + W_{\sigma_n}^*(t) \geq 4 \delta_*(\|u_{0,n}\|_{H_x^s}, \|v_{0,n}\|_{H_x^{\tilde{l}}})\} \\
		&\qquad \wedge \inf\{t \in [0,T] \colon \|v_{L,n}\|_{W^{\tilde{l},a,s-\frac{1}{2}}([0,t]) + L^2 (0,t ; W_x^{s,d})} \geq \epsilon\} \wedge \min\{2,T - \sigma_n\}, \\
		\tau_{n+1} &= \inf\{t \in [0, T] \colon  \|e^{\imu (\cdot) \Delta} u_{0,n}\|_{L^2 (0,t ; L_x^{2^*})} + W_{\sigma_n}^*(t) \geq 2 \delta_*(\|u_{0,n}\|_{H_x^s}, \|v_{0,n}\|_{H_x^{\tilde{l}}})\} \\
		&\qquad \wedge \inf\Big\{t \in [0,T] \colon \|v_{L,n}\|_{W^{\tilde{l},a,s-\frac{1}{2}}([0,t]) + L^2 (0,t ; W_x^{s,d})} \geq  \frac{\epsilon}{2}\Big\} \wedge \min\{1,T - \sigma_n\}
	\end{align*}
	with $\delta_*$ as in Subsection \ref{Subsec-LWPNEP},
	and two $\{\cF_t\}$-stopping times
	\begin{align*}
		\sigma_{n + 1} := \sigma_n + \tau_{n+1}, \qquad \tilde{\sigma}_{n+1} := \sigma_n + \tilde{\tau}_{n+1}.
	\end{align*}
	Note that $\tilde{\tau}_{n+1}$ and $\tau_{n+1}$ are $\{\cF_{\sigma_n + t}\}$-stopping times, since the mapping
	\begin{align*}
		t \mapsto \|v_{L,n}\|_{W^{\tilde{l},a,s-\frac{1}{2}}([0,t]) + L^2 (0,t ; W_x^{s,d})}
	\end{align*}
	is continuous by Lemma~\ref{it:ContSumY}.
	Moreover,  $\sigma_{n+1} < \tilde{\sigma}_{n+1}$,
	or $\sigma_{n+1} = \tilde{\sigma}_{n+1} = T$.
	
	As in the proof of
	\eqref{eq:EstFixedPointOp2}, \eqref{eq:EstFixedPointOp3} and~\eqref{eq:EstFixedPointOpDiff},
	we infer that $\Phi_\sigma(u_{0,n}, v_{0,n}; \cdot)$
	is a contractive self-map on a closed subset of $\XS([0,\tilde{\tau}_{n+1}])$,
	and thus, there exists a unique solution $\tilde{u}_{\sigma_{n+1}}$ of~\eqref{eq:FixedPointusigma}.
	Setting $$\tilde{v}_{\sigma_{n+1}} := v_{L,n} - \cJ_0[|\nabla| |\tilde{u}_{\sigma_{n+1}}|^2],$$
	we thus obtain a solution $(\tilde{u}_{\sigma_{n+1}}, \tilde{v}_{\sigma_{n+1}})$
	of~\eqref{eq:RanZakbsigmacsigmaInProof} in $\XS([0,\tilde{\tau}_{n+1}]) \times W^{\tilde{l},a,s-\frac{1}{2}}([0,\tilde{\tau}_{n+1}])$, which is unique in $\XS([0,\tilde{\tau}_{n+1}]) \times L^\infty ([0,\tilde{\tau}_{n+1}] ; H_x^{\tilde{l}})$. Then
	$$(u_{\sigma_{n+1}}(t), v_{\sigma_{n+1}}(t))
	:= (\tilde{u}_{\sigma_{n+1}}(t \wedge \tilde{\tau}_{n+1}), \tilde{v}_{\sigma_{n+1}}(t \wedge \tilde{\tau}_{n+1})),\quad t \in [0,T]
	$$
	is an $\{\cF_{\sigma_{n} + t}\}$-adapted continuous process in $H_x^s \times H_x^{\tilde{l}}$ and solves~\eqref{eq:RanZakbsigmacsigmaInProof} on $[0,\tilde{\tau}_{n+1}]$.
	
	Finally, we apply Proposition \ref{prop:GluingSolutions} to glue two solutions together
	\begin{align*}
		u_{n+1}(t) &:= u_n(t) \chi_{[0,\sigma_n)}(t) + e^{-W_1(\sigma_n)}u_{\sigma_{n+1}}((t-\sigma_{n}) \wedge \tilde{\tau}_{n+1})\chi_{[\sigma_n, T]}(t), \\
		v_{n+1}(t) &:= v_n(t) \chi_{[0,\sigma_n)}(t) + \Big(v_{\sigma_{n+1}}((t-\sigma_n) \wedge \tilde{\tau}_{n+1}) - e^{\imu ((t - \sigma_n) \wedge \tilde{\tau}_{n+1}) |\nabla|} \cT_{\sigma_n}(W_2)\Big) \chi_{[\sigma_n, T]}(t)
	\end{align*}
	for all $t \in [0,T]$ and derive that $(u_{n+1}, v_{n+1})$ solves~\eqref{eq:RanZakbc} on $[0, \tilde{\sigma}_{n+1}]$.
	Moreover, we have
	\begin{align}\label{prop:ClaimNEP}
		(u_{n+1}, v_{n+1})\in \XS([0,\tilde{\sigma}_{n+1}]) \times W^{\tilde{l},a,s-\frac{1}{2}}([0,\tilde{\sigma}_{n+1}]).
	\end{align}
	We omit the proof of this regularity statement which is derived analogously as in the 4D energy-space case~\cite{HRSZ24}.

	By combining the uniqueness properties of both $(u_n, v_n)$ and $(u_{\sigma_{n+1}}, v_{\sigma_{n+1}})$, together with Proposition~\ref{prop:RefinedRescaling} and Corollary~\ref{cor:RefinedRescalingInX}, we conclude that $(u_{n+1}, v_{n+1})$ is the unique solution to~\eqref{eq:RanZakbc} within the space $\XS([0,\tilde{\sigma}_{n+1}]) \times L^\infty ([0,\tilde{\sigma}_{n+1}] ; H_x^{\tilde{l}})$.
	Furthermore, $(u_{n+1}, v_{n+1})$ is an $\{\cF_t\}$-adapted continuous process in $H_x^s \times H_x^l$ (see, e.g., \cite{BRZ14}),
	coincides with $(u_n, v_n)$ on the interval $[0,\sigma_n]$, and $(u_{n+1}(t), v_{n+1}(t)) = (u_{n+1}(\tilde{\sigma}_{n+1}), v_{n+1}(\tilde{\sigma}_{n+1}))$ for $t \in [\tilde{\sigma}_{n+1}, T]$.

	Thus, by induction arguments,
	we construct an increasing sequence of $\{\cF_t\}$-adapted stopping times $\{\sigma_n\}$
	and corresponding $\{\cF_t\}$-adapted processes $\{(u_n, v_n)\}$,
	such that $(u_n, v_n)$ is the solution of~\eqref{eq:RanZakbc}
	in $\XS([0,\sigma_n]) \times W^{\tilde{l},a,s-\frac{1}{2}}([0,\sigma_n])$, unique in $\XS([0,\tilde{\sigma}_{n+1}]) \times L^\infty ([0,\tilde{\sigma}_{n+1}] ; H_x^{\tilde{l}})$,
	and $(u_{n+1}, v_{n+1})$ coincides with $(u_n, v_n)$ on $[0,\sigma_n], n \geq 1$.
	
	\vspace*{4pt plus 2pt minus 2pt}%
	\paragraph{\bf $\bullet$ Maximal interval of existence}
	Now, let
	\begin{align*}
		\tau_T^* := \lim_{n \rightarrow \infty} \sigma_n \qquad \text{and} \qquad (u^T, v^T) := \lim_{n \rightarrow \infty} (u_n \chi_{[0, \tau_T^*)}, v_n \chi_{[0, \tau_T^*)}).
	\end{align*}
	Then, $(u^T, v^T)$ is an $\{\cF_t\}$-adapted process that solves \eqref{eq:RanZakbc} uniquely
	up to the $\{\cF_t\}$ stopping time $\tau_T^*$.
	Since $\{\tau_T^*\}$ is increasing in $T$, and for $T' > T$,
	$(u^{T'}, v^{T'})$ coincides with $(u^T, v^T)$ on $[0,\tau_T^*)$ by uniqueness,
	we can define
	\begin{align*}
		\tau^* := \lim_{T \rightarrow \infty} \tau_T^* \qquad \text{and} \qquad (u,v) := \lim_{T \rightarrow \infty} (u^T \chi_{[0,\tau^*)}, v^T \chi_{[0,\tau^*)}).
	\end{align*}
	The resulting process $(u,v)\in C([0,\tau^*); H_x^s \times H_x^{\tilde{l}})$ is thus $\{\cF_t\}$-adapted
		and uniquely solves \eqref{eq:RanZakbc}.
	
	If $s \geq l + \frac{1}{2}$, then $\tilde{l}=l$ and hence $(u,v) \in C([0,\tau^*), H_x^s \times H_x^l)$. In the case $l>s-\frac{1}{2}$, Lemma~\ref{le:PersisReg} implies that $(u,v)\in \tilde{S}^{s,0,b}([0,\tau^*)) \times W^{l,0,s-\frac{1}{2}}([0,\tau^*))$, in particular $(u,v) \in C([0,\tau^*), H_x^s \times H_x^l)$. In both cases we thus conclude $(u,v) \in C([0,\tau^*), H_x^s \times H_x^l)$.
	
	Finally, in view of the rescaling transformation in Subsection~\ref{subsec:IdeaProof},
	we infer that
	\begin{align*}
		(X,Y) := (e^{W_1} u, v + \cT_{\cdot}(W_2))
	\end{align*}
	is the unique solution to \eqref{eq:StoZak} on $[0,\tau^*)$ in the sense of Definition~\ref{def:Solution}.

	\subsection{Preliminary blow-up alternative}\label{Subsec-Blowup}
	Before proving the blow-up alternative in Theorem \ref{thm:LocalWP},
	we first provide the following
	more direct preliminary blow-up criterion,
	that is,
	$\PP$-a.s. if $\tau^* <\infty$,  then
	\begin{align} \label{blowup-nonend-preli}
		(i)\quad  \limsup_{t \rightarrow \tau^*} (\|X(t)\|_{H_x^s} + \|Y(t)\|_{H_x^{\tilde{l}}}) = \infty \qquad \text{or}\qquad (ii)\quad \|X\|_{L^2 (0,\tau^* ; W_x^{\tilde{l}+\frac 12,2^*})} = \infty.
	\end{align}
	
	We prove this alternative by contradiction.
	Suppose that it is not true,
	then by Lemma~\ref{lem:PropNoise},
	there exists a set $\Omega'$ of positive measure such that for every $\omega \in \Omega'$ the limit \eqref{eq:limphi1kbeta1k} is satisfied
	and the following holds:
	\begin{enumerate}
		\item \label{it:BlowUpFiniteTime2}$\tau^*(\omega) < \infty$,
		\item \label{it:BlowUpSobNorm2}$\limsup_{t \rightarrow \tau^*(\omega)} (\|X(t,\omega)\|_{H_x^s} + \|Y(t,\omega)\|_{H_x^{\tilde{l}}}) < \infty$,
		\item \label{it:BlowUpDispersiveNorm2}$\|X(\cdot, \omega)\|_{L^2 (0,\tau^*(\omega) ; W_x^{{\tilde{l}}+\frac{1}{2},2^*})} < \infty$.
	\end{enumerate}
	In what follows,
	we fix this $\omega$ and omit the dependence on $\omega$ for simplicity.
	We also use the notations from Subsection \ref{Subsec-Max-Exist}
	and let $T \in (0,\infty)$ be such that $T > \tau^*$.
	
	Since by \eqref{eq:DefInitialDataRescaled},
	$(u_{0,n}, v_{0,n}) = (X(\sigma_n), Y(\sigma_n))$,
	the above condition \ref{it:BlowUpSobNorm2} implies that
	\begin{align*}
		\limsup_{n \rightarrow \infty}(\|u_{0,n}\|_{H_x^s} + \|v_{0,n}\|_{H_x^{\tilde{l}}}) < \infty.
	\end{align*}
	In particular, there exists $r > 0$ such that
	\begin{equation}
		\label{eq:BoundInitialDataRescaled}
		\sup\limits_{n\in \N} (\|u_{0,n}\|_{H_x^s} + \|v_{0,n}\|_{H_x^{\tilde{l}}}) \leq r.
	\end{equation}
	Moreover, since $\tau^* < \infty$, we have
	$\lim_{n \rightarrow \infty} \tau_n = 0$. Without loss of generality, 
	we further assume that $T - \sigma_n < 1$ for all $n \in \N$.
	Since $T - \sigma_n \geq T - (\tau^* - \tau_{n+1}) > \tau_{n+1}$,
	we infer that
	for all $n \in \N$,
	\begin{align}
		\label{eq:StoppingTimeBlowUp}
		\tau_{n+1} &= \inf\{t \in [0, T] \colon  \|e^{\imu (\cdot) \Delta} u_{0,n}\|_{L^2 (0,t ; L_x^{2^*})} + W_{\sigma_n}^*(t) \geq 2 \delta_*(\|u_{0,n}\|_{H_x^s}, \|v_{0,n}\|_{H_x^{\tilde{l}}})\} \nonumber\\
		&\qquad \wedge \inf\Big\{t \in [0,T] \colon \|v_{L,n}\|_{W^{\tilde{l},a,s-\frac12}([0,t]) + L^2 (0,t ; W_x^{s,d})} \geq \frac{\epsilon}{2}\Big\}.
	\end{align}
	Recall that $\delta_*$ is decreasing in both its arguments. Setting $\delta = \delta_*(r,r) > 0$ we obtain from~\eqref{eq:BoundInitialDataRescaled} that
	\begin{align*}
		\delta_*(\|u_{0,n}\|_{H_x^s}, \|v_{0,n}\|_{H_x^{\tilde{l}}}) \geq \delta_*(r,r) = \delta >0
	\end{align*}
	for all $n \in \N$.
	Taking into account the continuity statement in Lemma~\ref{it:ContSumY},
	from~\eqref{eq:StoppingTimeBlowUp} we thus infer that
	\begin{align}
		\label{eq:AlternativeTaunplus1}
		& W_{\sigma_n}^*(\tau_{n+1}) \geq \delta
		\quad \text{or} \quad
		\|e^{\imu (\cdot) \Delta} u_{0,n}\|_{L^2 (0,\tau_{n+1} ; L_x^{2^*})}  \geq \delta  \notag\\
		\text{or} \quad & \|v_{L,n}\|_{W^{\tilde{l},a,s-\frac12}([0,\tau_{n+1}]) + L^2 (0,\tau_{n+1} ; W_x^{s,d})} \geq \frac{\epsilon}{2}\quad \text{for all $n \in \N$}.
	\end{align}
	
	Below, we will show that none of the above three cases can occur
	and thus reach the contradiction.
	
	\vspace*{4pt plus 2pt minus 2pt}%
	For this purpose,
	we first note that
	\begin{equation}
		\label{eq:WStarsigmanLimZero}
		\lim_{n \rightarrow \infty} \sup_{t \in [0, \tau^* - \sigma_n]} W_{\sigma_n}^*(t) = 0.
	\end{equation}
	The proof of \eqref{eq:WStarsigmanLimZero} is analogous to that of~(5.26) in~\cite{HRSZ24} and is therefore omitted. This excludes the first alternative in~\eqref{eq:AlternativeTaunplus1} for large enough $n$.
	
	\vspace*{4pt plus 2pt minus 2pt}%
	We next show that the second alternative in~\eqref{eq:AlternativeTaunplus1} is not satisfied by proving
	\begin{align}\label{eq:SmallFreeWaveNEP}
		\|e^{\imu t \Delta} u_{0,n}\|_{L^2 (0,\tau_{n+1} ; L_x^{2^*})} \longrightarrow 0,\quad {\rm as} \quad  n \rightarrow \infty.
	\end{align}
	
	We recall that
	\begin{align}  \label{usigman-vsigman}
		(u_{\sigma_n}(t), v_{\sigma_n}(t)) = (e^{W_1(\sigma_n)} u(\sigma_n + t), v(\sigma_n + t) + e^{\imu t |\nabla|} \cT_{\sigma_n}(W_2)),   \quad  t \in [0, \tau^* - \sigma_n).
	\end{align}
	By Proposition~\ref{prop:RefinedRescaling},
	$(u_{\sigma_n}, v_{\sigma_n})$ solves~\eqref{eq:RanZakbsigmacsigmaInProof} on $[0, \tau^* - \sigma_n)$ with initial data $(u_{0,n}, v_{0,n})$,
	and hence,
	\begin{align}\label{eq:estimateu0nNEP}
		\|e^{\imu t\Delta}u_{0,n}\|_{L^2 (0,\tau_{n+1}; L_x^{2^*})}
		&\lesssim \|u_{\sigma_n}\|_{L^2 (0,\tau_{n+1}; L_x^{2^*})}+\Big\| \cI_0(\Re(v_{\sigma_n})u_{\sigma_n}) \Big\|_{L^2 (0,\tau_{n+1}; L_x^{2^*})}\notag \\
		&\qquad + \Big\| \cI_0(b_{\sigma_n}\cdot \nabla u_{\sigma_n}+c_{\sigma_n}u_{\sigma_n}-\Re(\cT_{\sigma_n+\cdot,\sigma_n})u_{\sigma_n}) \Big\|_{L^2 (0,\tau_{n+1}; L_x^{2^*})},
	\end{align}
	where $\cI_0$ is the inhomogeneous Schr\"odinger operator with no potential.
	
	From condition~\ref{it:BlowUpDispersiveNorm2} we have the estimate
	\begin{align*}
		\|u_{\sigma_n}\|_{L^2 (0,\tau_{n+1}; L_x^{2^*})} \lesssim \|u\|_{L^2 (\sigma_n, \sigma_n+\tau_{n+1}; L_x^{2^*})}\lesssim \|X\|_{L^2(\sigma_n, \tau^*; W_x^{\tilde{l}+\frac12, 2^*})}<\infty,
	\end{align*}
	which implies that the first term in \eqref{eq:estimateu0nNEP} vanishes as $n\to\infty$ by using the dominated convergence theorem.

	For the second term in \eqref{eq:estimateu0nNEP}, note that by Strichartz estimates,
	\begin{align*}
		\Big\| \cI_0(\Re(v_{\sigma_n})u_{\sigma_n}) \Big\|_{L^2 (0,\tau_{n+1}; L_x^{2^*})}
		&\lesssim \|\Re(v_{\sigma_n})u_{\sigma_n} \|_{L^2 (0,\tau_{n+1}; L_x^{2_*})}\\
		&\lesssim \|v_{\sigma_n}\|_{L^\infty([0,\tau_{n+1}]; H_x^{\tilde{l}})} \|u_{\sigma_n}\|_{L^2 (0,\tau_{n+1}; L_x^{2^*})}\\
		&\lesssim (\|Y\|_{L^\infty  ([\sigma_n, \tau^*); H_x^{\tilde{l}}) } + \|\cT_{\cdot}(W_2)\|_{L^\infty  ([\sigma_n, \tau^*); H_x^{\tilde{l}})}  )\|X\|_{L^2 (\sigma_n, \tau^*; L_x^{2^*}) }.
	\end{align*}
	Hence, this term tends to zero by condition \ref{it:BlowUpSobNorm2}, \ref{it:BlowUpDispersiveNorm2}, \eqref{eq:WStarsigmanLimZero}, and the dominated convergence theorem.
	
	Moreover, for the third term on the right-hand side of \eqref{eq:estimateu0nNEP}, by Strichartz estimates, Lemma \ref{lem:BilinearLowerOrder}, i.e. \eqref{eq:BlowUpEstNEP}, Lemma \ref{lem:BilinearLowerOrder2}, i.e. \eqref{eq:BlowUpEstNEP2} and~\eqref{eq:BlowUpEstNEP3}, and Sobolev embedding, we have
	\begin{align}
		&\Big\| \cI_0 (b_{\sigma_n}\cdot \nabla u_{\sigma_n}+c_{\sigma_n}u_{\sigma_n}-\Re(\cT_{\sigma_n+\cdot,\sigma_n})u_{\sigma_n}) \Big\|_{L^2 (0,\tau_{n+1}; L_x^{2^*})}\notag\\
		&\lesssim (\tau^*-\sigma_n)^{\frac12} \Big( \|c_{\sigma_n}\|_{L_t^\infty H_x^{\frac{d-3}{2}}}  + \|\cT_{\sigma_n+\cdot,\sigma_n}\|_{L_t^\infty H_x^{\frac{d-3}{2}}} \nonumber \\
		&\qquad + \|b_{\sigma_n}\|_{L_t^\infty L^\infty_x}^{\frac12}\sum_{j=1}^d \|b_{\sigma_n}\|_{L_{\vece_j}^{1,\infty}}^{\frac12}+\|b_{\sigma_n}\|_{L^\infty_t H^{\frac{d}{2}}_x}\Big) \|u_{\sigma_n}\|_{S^{\frac{1}{2},0,0}([0,\tau_{n+1}])}\notag \\
		&\lesssim  (\tau^*-\sigma_n)^{\frac12} \sup_{n\in \N}\|u\|_{\X^{\tilde{l} + \frac12,0}_{([0,\sigma_n])}}\sup_{t\in [0,\tau^*-\sigma_n)}W_{\sigma_n}^*(t) ,\label{prop:blowupnoise}
	\end{align}
	where we also used $\tau^*-\sigma_n< T-\sigma_n<1$. 
	Moreover, we note from conditions \ref{it:BlowUpSobNorm2} and~\ref{it:BlowUpDispersiveNorm2} that
	the assumptions of Lemma \ref{le:unibouep} are satisfied provided $\X^{s,a}\hookrightarrow \X^{\tilde{l}+\frac12,0}$ as $u \in \X^{s,a}([0,\sigma_n])$ for all $n \in \N$ by construction.
	
	The embedding $\X^{s,a}\hookrightarrow \X^{\tilde{l}+\frac12,0}$ is trivially true in the case $a=0$ since $s\geq \tilde{l}+\frac12$. If $a>0$, by our choice of the parameter in~\eqref{con:ab}, we have $s-l\geq 1, a = \frac{3}{4}(s-l)-\frac12$ and $\tilde{l}+\frac12 = l+\frac12$. The embedding then follows from~\eqref{eq:CharSsablambda}, the definition of $\X^{s,a}$, and
	\begin{align*}
		\lambda^{l-\frac12} \| (\imu \partial_t +\Delta) u_\lambda \|_{L_{t,x}^2} &\lesssim \lambda^{l-\frac12 + \frac34(s-l)-\frac12} \Big\| \Big(\frac{\lambda+|\partial_t|}{\lambda^2+|\partial_t|}\Big)^a(\imu \partial_t +\Delta) u_\lambda \Big\|_{L_{t,x}^2}\\
		&= \lambda^{\frac14(l-s)}\lambda^{s-1}\Big\| \Big(\frac{\lambda+|\partial_t|}{\lambda^2+|\partial_t|}\Big)^a(\imu \partial_t +\Delta) u_\lambda \Big\|_{L_{t,x}^2}\\
		&\lesssim \lambda^{s-1}\Big\| \Big(\frac{\lambda+|\partial_t|}{\lambda^2+|\partial_t|}\Big)^a(\imu \partial_t +\Delta) u_\lambda \Big\|_{L_{t,x}^2}.
	\end{align*}
	
	Thus, Lemma \ref{le:unibouep} implies the uniform bound $\sup_{n\in \N}\|u\|_{\X^{\tilde{s},0}([0,\sigma_n])}<\infty$, and hence the right-hand side of \eqref{prop:blowupnoise} tends to $0$ as $n\to \infty$, which yields \eqref{eq:SmallFreeWaveNEP}.

	\vspace*{4pt plus 2pt minus 2pt}%
	It remains to show  that the third alternative in~\eqref{eq:AlternativeTaunplus1} cannot hold, i.e.,
	\begin{equation}
		\label{eq:BlowUpAltLinWavePotSmall}
		\|v_{L,n}\|_{W^{{\tilde{l}},a,s-\frac12}([0,\tau_{n+1}]) + L^2 (0,\tau_{n+1} ; W_x^{s,d})} < \frac{\epsilon}{2}
	\end{equation}
	if $n \in \N$ is large enough.
	
	To that purpose we first recall from the construction in Subsection \ref{Subsec-Max-Exist} that
	\begin{align}
		v_{L,n}(t) = e^{\imu t |\nabla|} v_{0,n} = e^{\imu t |\nabla|}(v(\sigma_n) + \cT_{\sigma_n}(W_2)).\label{eq:ReWritreVLNNEP}
	\end{align}
	By Sobolev's embedding we get
	\begin{align}
		\|e^{\imu t |\nabla|}\cT_{\sigma_n}(W_2)\|_{W^{{\tilde{l}},a,s-\frac12}([0,\tau_{n+1}]) + L^2 (0,\tau_{n+1} ; W_x^{s,d})} 
		&\lesssim \|e^{\imu t |\nabla|}\cT_{\sigma_n}(W_2)\|_{L^2 (0,\tau_{n+1} ; W_x^{s,d})} \nonumber\\
		&\lesssim (\tau^* - \sigma_n)^{\frac{1}{2}} \sup_{t \in [0,\tau*]} \|\cT_t(W_2)\|_{H_x^{s+\frac{d-2}{2}}}. \label{eq:BlowUpWaveProfileSecondSummand}
	\end{align}

	Using Lemma~\ref{le:unibouep} once more, we can extend $v$ to a function in $C([0,\tau^*], H_x^{\tilde{l}})$.
	Let $\phi \in C_c^\infty(\R^d)$ with $0 \leq \phi \leq 1$, $\phi = 1$ on $B_1(0)$, $\phi = 0$ on $B_2(0)^c$, and $\|\phi\|_{L^1} = 1$,
	and set $\phi_\nu(x) = \nu^{-d} \phi(\frac{x}{\nu})$ for $x \in \R^d$ and $\nu > 0$.
	Since  $v$ is continuous on the compact interval $[0,\tau^*]$, we have
	\begin{align*}
		\|v - v \ast \phi_\nu\|_{L^\infty ([0,\tau^*] ; H_x^{\tilde{l}})} \longrightarrow 0,\quad \text{as } \nu \rightarrow 0.
	\end{align*}
	Hence, we can fix $\nu > 0$ such that $\|v - v \ast \phi_\nu\|_{L^\infty ([0,\tau^*] ; H_x^{{\tilde{l}}})} < \frac{\epsilon}{4C'}$,
	where $C'$ is the implicit constant in  Lemma~\ref{lem:LinearEstimateHalfWave}.
	Note that
	\begin{align}
		&\|v(\sigma_n)\|_{L_x^2} \lesssim \|v\|_{L^\infty  ([0,\tau^*); L_x^2)} \lesssim \|Y\|_{L^\infty  ([0,\tau^*); L_x^2)} +  \|\cT_{\cdot} (W_2)\|_{L_t^\infty L_x^2}=: R<\infty. \label{eq:BlowUpBoundv}
	\end{align}
	Then we use Lemma~\ref{lem:LinearEstimateHalfWave} and Sobolev's embedding again as in~\eqref{eq:BlowUpWaveProfileSecondSummand} to infer that
	\begin{align} 
		\label{eq:BlowUpWaveProfileFirstSummand}
		&\|e^{\imu t |\nabla|}v(\sigma_n)\|_{W^{{\tilde{l}},a,s-\frac12}([0,\tau_{n+1}]) + L^2 (0,\tau_{n+1} ; W_x^{s,d})} \nonumber\\
		& \leq \|e^{\imu t |\nabla|}(v(\sigma_n) - v \ast \phi_\nu(\sigma_n))\|_{W^{{\tilde{l}},a,s-\frac12}([0,\tau_{n+1}])} + \|e^{\imu t |\nabla|}(v \ast \phi_\nu(\sigma_n))\|_{L^2 (0,\tau_{n+1} ; W_x^{s,d})} \nonumber \\
		&\leq C'\|v(\sigma_n) - v \ast \phi_\nu(\sigma_n)\|_{H_x^{\tilde{l}}} + C\|e^{\imu t |\nabla|}(v(\sigma_n) \ast \phi_\nu)\|_{L^2 (0,\tau_{n+1} ; H_x^{\frac{d}{2}+s-1})} \nonumber \\
		&\leq C'\|v - v \ast \phi_\nu\|_{L^\infty ([0,\tau^*] ; H_x^{\tilde{l}})} + C\tau_{n+1}^{\frac{1}{2}} \|v(\sigma_n) \ast \phi_\nu\|_{H^{\frac{d}{2}+s-1}_x}  \nonumber \\
		&\leq \frac{\epsilon}{4} + C(\tau^* - \sigma_n)^{\frac{1}{2}} \nu^{-(s+\frac{d}{2}-1)} \|v(\sigma_n)\|_{L^2_x} \nonumber \\
		&\leq \frac{\epsilon}{4} + C(\tau^* - \sigma_n)^{\frac{1}{2}} \nu^{-(s+\frac{d}{2}-1)} R.
	\end{align}
	
	Therefore, combining \eqref{eq:BlowUpWaveProfileSecondSummand} and \eqref{eq:BlowUpWaveProfileFirstSummand} and using $\lim_{n \rightarrow \infty} \sigma_n = \tau^*$, we obtain \eqref{eq:BlowUpAltLinWavePotSmall} for all large enough~$n$. This finishes the proof of the blow-up alternative in~\eqref{blowup-nonend-preli}.
	\hfill \qed
	
	\subsection{Bootstrapping  regularity}\label{Subsec:ImproBlowUpCon}
	
	We finally prove the blow-up condition in Theorem~\ref{thm:LocalWP}. We argue by contradiction and assume that the blow-up alternative is not true. Then there exists $\Omega' \subseteq \Omega$ with positive measure such that  \eqref{eq:limphi1kbeta1k} holds and
	\begin{enumerate}
		\item \label{it:BlowUpFiniteTime} $\tau^*(\omega) < \infty$
		\item \label{it:BlowUpSobNorm} $\limsup_{t \rightarrow \tau^*(\omega)} (\|u(t,\omega)\|_{H_x^{\frac{d-3}{2}}} + \|v(t,\omega)\|_{H_x^{\frac{d-4}{2}}}) < \infty$
		\item \label{it:BlowUpDispersiveNorm} $\|u(\cdot, \omega)\|_{L^2 (0,\tau^*(\omega) ; W_x^{\frac{d-3}{2},2^*})} < \infty$
	\end{enumerate}
	for all $\omega \in \Omega'$. Due to~(i) and the previous step, we can also assume that
	\begin{align}\label{prop:ImBlowUpOriCon}
		\limsup_{t \rightarrow \tau^*(\omega)} (\|u(t,\omega)\|_{H_x^s} + \|v(t,\omega)\|_{H_x^{\tilde{l}}}) = \infty \qquad \text{or}\qquad  \|u(\cdot, \omega)\|_{L^2 (0,\tau^*(\omega) ; W_x^{\tilde{l}+\frac{1}{2},2^*})} = \infty
	\end{align}
	holds for every $\omega \in \Omega'$, where $\tilde{l}=\min \{s-\frac{1}{2}, l \}$. 
	Below we fix such an $\omega \in \Omega'$ and drop it for simplicity.

	Define
	\begin{equation}
		\label{eq:Defl0}
		l_w:= \sup\Big\{\frac{d-4}{2} \leq l' \leq \tilde{l}\colon \|v\|_{L^\infty ([0,\tau^*); H_x^{l'})} < \infty \Big\}.
	\end{equation}
	Note that due to~\ref{it:BlowUpSobNorm}
	the above set is not empty.
	For any $l' \geq \frac{d-4}{2}$, we now define the parameter sets
	\begin{align*}
		A_{l'}&:= \Big\{s' \in [l'+\tfrac12, s] \colon (s',l') \text{ satisfies } \eqref{IniReg-condition} \Big\},\\
		S_{l'}&:= \Big\{s'\in A_{l'}: \|u\|_{L^\infty ([0,\tau^*); H_x^{s'})} + \|u\|_{L^2 (0,\tau^*; W_x^{s'-a',2^*})} < \infty\Big\},
	\end{align*}
	where $a'=a^*(s',l')$ from~\eqref{con:ab}.
	
		The proof of the blow-up condition relies on two key lemmas establishing the persistence of regularity: Lemma~\ref{lem:PersRegSchr} for the Schr{\"o}dinger component and Lemma~\ref{lem:PersRegWave} for the wave component.

	The core idea exploits the flexibility in distributing regularities within the nonlinear estimates (Lemma~\ref{le:Bilinearnablauv} and Lemma~\ref{lem:BilinearEstimates}). This ultimately allows us to control the solution at a higher regularity $(s,l)$ using controlling norms at lower regularity. We demonstrate this regularity improvement separately for the Schr{\"o}dinger component (at fixed wave regularity) and the wave component (at fixed Schr{\"o}dinger regularity). Intuitively, these results can be applied iteratively to control the solution at any admissible regularity in~\eqref{IniReg-condition} via a norm at lower regularity. More precisely, we establish a self-improving property of the parameter sets $A_{l'}$ and $S_{l'}$, analogous to a bootstrap argument.
	
	\begin{lemma}[Persistence of Schr{\"o}dinger regularity]
		\label{lem:PersRegSchr}
		Let $\frac{d-4}{2} \leq l' \leq \tilde{l}$ such that
		\begin{align*}
			\|v\|_{L^\infty([0,\tau^*); H_x^{l'})} < \infty.
		\end{align*}
		If $S_{l'} \neq \emptyset$, then $S_{l'} = A_{l'}$.
	\end{lemma}
	
	\begin{proof}
		We argue by contradiction. Suppose that $S_{l'} \neq \emptyset$, but $S_{l'} \neq A_{l'}$. Then there exists $\frac{d-3}{2} \leq s'' \leq s'$ such that $s'' \in S_{l'}$, $s' \in A_{l'} \setminus S_{l'}$, and $s' < s'' + \frac{1}{8}$. We will show that
		\begin{align}\label{prop:ImBlowUpU3}
			\sup_{n\in \N} \|u\|_{\X^{s', a'}([0,\sigma_n])}<\infty,
		\end{align}
		where $a' = a^*(s', l')$ and $\{\sigma_n\}$ is a sequence that converges to $\tau^*$. Provided~\eqref{prop:ImBlowUpU3} is true, we get
		\begin{align*}
			\|u\|_{L^\infty (0,\tau^*; H_x^{s'})} + \|u\|_{L^2 (0,\tau^*; W_x^{s'-a',2^*})} &= \sup_{n\in \N} (\|u\|_{L^\infty ([0,\sigma_n]; H_x^{s'})} + \|u\|_{L^2 ([0,\sigma_n]; W_x^{s'-a',2^*})}) \\
			&\leq \sup_{n\in \N}\|u\|_{\X^{s', a'}([0,\sigma_n])}<\infty,
		\end{align*}
		and hence $s'\in S_{l'}$, which contradicts the fact that $s' \notin S_{l'}$.
		
		So it remains to show \eqref{prop:ImBlowUpU3}. For that purpose, recall the definition of $(u_{\sigma_n}, v_{\sigma_n})$ from~\eqref{eq:u-sigma-v-sigma-rescal.1} and \eqref{eq:u-sigma-v-sigma-rescal.2}, which solves the system \eqref{eq:RanZakbsigmacsigma} with coefficients $b_{\sigma_n}, c_{\sigma_n}$ on the interval $[0,\tau^*-\sigma_n)$. For $N\in\N$ to be chosen below, we write
		\begin{align*}
			u_{\sigma_N}(t) = e^{\imu t\Delta}u(\sigma_N) + \cI_0[\Re(v_{\sigma_N})u_{\sigma_N}] - \cI_0[b_{\sigma_N}\cdot \nabla u_{\sigma_N} + c_{\sigma_N}u_{\sigma_N} - \Re(\cT_{\sigma_N+\cdot, \sigma_N}(W_2))u_{\sigma_N}].
		\end{align*}
		
		For $M>N$, we estimate by Lemma \ref{lem:BilinearLowerOrder} and Lemma~\ref{lem:BilinearLowerOrder2}
		\begin{align*}
			\|u_{\sigma_N}\|_{\X^{s',a'}([0,\sigma_M-\sigma_N])}
			&\lesssim  \|u(\sigma_N)\|_{H_x^{s'}} + \|\cI_0[\Re(v_{\sigma_N})u_{\sigma_N}]\|_{\X^{s',a'}([0,\sigma_M-\sigma_N))} \\
			&\qquad + \sup_{t\in [0,\tau^*-\sigma_N)} (1 + t^{\frac12}) W^*_{\sigma_N}(t)  \|u_{\sigma_N}\|_{\X^{s',a'}([0,\sigma_M-\sigma_N])},
		\end{align*}
		cf.~\eqref{eq:EstFixedPointOp2}.
		As $\sup_{t\in [0,\tau^*-\sigma_N)} W^*_{\sigma_N}(t)$ tends to zero as $N\to \infty$, we find some $N_0\in \N$ such that for any $N\geq N_0$,
		\begin{align}
			\|u_{\sigma_N}\|_{\X^{s',a'}([0,\sigma_M-\sigma_N])}\lesssim \|u(\sigma_N)\|_{H_x^{s'}} + \|\cI_0[\Re(v_{\sigma_N})u_{\sigma_N}]\|_{\X^{s',a'}([0,\sigma_M-\sigma_N])}.
		\end{align}
		Here we also exploited that $u_{\sigma_N}\in \X^{s',a'}([0,\sigma_M-\sigma_N])$ as $u\in \X^{s',a'}([0,\sigma_M])$ by the definition of $\tau^*$ and the fact that $s' \leq s$.
		
		Next, by Lemma~\ref{lem:LinEstimates}, Lemma~\ref{it:BilinearEstNonendpoint} and the embedding $\X^{s',a'}\hookrightarrow S^{s',a',0}$, we have
		\begin{align*}
			\|\cI_0[\Re(v_{\sigma_N})u_{\sigma_N}]\|_{\X^{s',a'}([0,\sigma_M-\sigma_N])}
			\lesssim \|v_{\sigma_N}\|_{W^{l',a',\beta'}+ L_t^2 W_x^{s',d}([0,\sigma_M-\sigma_N])} \|u_{\sigma_N}\|_{\X^{s',a'}([0,\sigma_M-\sigma_N])},
		\end{align*}
		where we set $\beta':= \max\{\frac{d-4}{2}, s'-1\}$.
		
		We write
		\begin{align*}
			v_{\sigma_N}(t) = e^{\imu t|\nabla|} v(\sigma_{N}) - \cJ_0[|\nabla||u_{\sigma_n}|^2]
		\end{align*}
		and estimate for any $M > N \geq N_0$
		\begin{align*}
			&\|v_{\sigma_N}\|_{W^{l',a',\beta'}+ L_t^2  W_x^{s',d} ([0,\sigma_M-\sigma_N])}\\
			&\leq \|e^{\imu t|\nabla|} v(\sigma_{N})\|_{W^{l',a',\beta'}+ L_t^2  W_x^{s',d}([0,\sigma_M-\sigma_N])}+ \|\cJ_0[|\nabla||u_{\sigma_n}|^2]\|_{W^{l',a',\beta'}([0,\sigma_M-\sigma_N])}.
		\end{align*}
		Set $a'' = a^*(s'',l')$. Since $s''\in S_{l'}$ and $l' + \frac12 \leq s'' - a''$, we get
		\begin{align*}
			\|u\|_{L^2 (0,\tau^*; W_x^{l'+\frac12,2^*})}\leq \|u\|_{L^2 (0,\tau^*; W_x^{s''-a'',2^*})}<\infty,
		\end{align*}
		so that the blow-up condition
		\eqref{blowup-nonend-preli} we already established implies that $\tau^{**}>\tau^*$, where $\tau^{**}$ is the maximal existence time of the $H_x^{s''}\times H_x^{l'}$ solution to the system~\eqref{eq:RanZakbc}. In particular, $u\in \X^{s'',a''}([0,\tau^*])$ and then an analogous argument as for Lemma~\ref{le:unibouep} implies that $v$ can be continuously extended in $C([0,\tau^*],H_x^{l'})$. As in the proof of~\eqref{eq:BlowUpAltLinWavePotSmall} we then infer
		\begin{align*}
			\|e^{\imu t|\nabla|} v(\sigma_N)\|_{W^{l',a',\beta'}+ L_t^2  W_x^{s',d} ([0,\tau^*-\sigma_N])} \longrightarrow 0
		\end{align*}
		as $N\to\infty$. Consequently, we find some positive integer $N_1\geq N_0$ such that for any $M>N\geq N_1$
		\begin{align*}
			\|u_{\sigma_N}\|_{\X^{s',a'}([0,\sigma_M-\sigma_N])}\lesssim \|u(\sigma_N)\|_{H_x^{s'}} + \|\cJ_0[|\nabla||u_{\sigma_N}|^2]\|_{W^{l',a',\beta'}([0,\sigma_M-\sigma_N])} \|u_{\sigma_N}\|_{\X^{s',a'}([0,\sigma_M-\sigma_N])}.
		\end{align*}
		
		For the remaining wave nonlinearity, we employ (8.1) from~\cite{CHN23}, where we exploit that $s' - s'' < \frac18$. This estimate shows that there exists some $\theta>0$ such that
		\begin{align}\label{prop:JEMS8_1}
			\|\cJ_0[|\nabla||u_{\sigma_N}|^2]\|_{W^{l',a',\beta'}([0,\sigma_M-\sigma_N])}\lesssim \|u_{\sigma_N}\|^\theta_{L^2 (0,\sigma_M-\sigma_N; W_x^{\frac{d-3}{2},2^*})} \|u_{\sigma_N}\|^{2-\theta}_{S^{s'',a'',0}([0,\sigma_M-\sigma_N])}.
		\end{align}
		Using Lemma~\ref{lem:ProductNoiseInX}, we note that
		\begin{align*}
			\|u_{\sigma_N}\|_{S^{s'',a'',0}([0,\sigma_M-\sigma_N])} \lesssim \|u\|_{\X^{s'',a''}([0,\tau^*])}.
		\end{align*}
		Moreover, we already showed that $\|u\|_{\X^{s'',a''}([0,\tau^*])}$ is finite and by~\ref{it:BlowUpDispersiveNorm}, we get
		\begin{align*}
			\|u_{\sigma_N}\|_{L^2 (0,\sigma_M-\sigma_N; W_x^{\frac{d-3}{2},2^*})} \lesssim \|u\|_{L^2 (\sigma_N,\tau^*; W_x^{\frac{d-3}{2}, 2^*})}\longrightarrow 0
		\end{align*}
		as $N\to\infty$. Combining this with the above, we conclude that there exists $N_2 \geq N_1$ that for any $M>N\geq N_2$,
		\begin{align*}
			\|u\|_{\X^{s',a'}([\sigma_N,\sigma_M])}\lesssim \|u(\sigma_{N})\|_{H_x^{s'}},
		\end{align*}
		where the implicit constant is independent of $M$ and $N$. Now we fix some $N \geq N_2$. Since $u$ belongs to $\X^{s',a'}([0,\sigma_N+\frac{\tau^*-\sigma_N}{2}])$, Lemma~\ref{lem:DecompX} implies that there exists $R>0$ such that
		\begin{align*}
			\|u\|_{\X^{s',a'}([0,\sigma_M])} \leq R
		\end{align*}
		for all $M \in \N$, which yields \eqref{prop:ImBlowUpU3}, thus finishing the proof of the lemma.
	\end{proof}
	
	We continue with the lemma on the persistence of regularity for the wave component. 
	\begin{lemma}[Persistence of wave regularity]
		\label{lem:PersRegWave}
		Let $l'' \geq \frac{d-4}{2}$ such that $\|v\|_{L^\infty([0,\tau^*); H_x^{l''})} < \infty$ and $s' \in S_{l''}$ with $s' > \frac{d-3}{2}$. For every $l' \in (l'', \tilde{l}]$ satisfying $l' \leq s' - \frac{1}{2}$ and $l' < l'' + \frac{1}{2}$, we then have
		\begin{align*}
			\|v\|_{L^\infty([0,\tau^*); H_x^{l'})} < \infty.
		\end{align*}
	\end{lemma}
	
	\begin{proof}
		In the following, we will denote $a'=a^*(s',l')$ and $a''=a^*(s',l'')$ with $a^*$ from~\eqref{con:ab}. Since $s'\in S_{l''}$ and
		\begin{align*}
			\|u\|_{L^2 (0,\tau^*; W_x^{l''+\frac12,2^*})} \leq \|u\|_{L^2 (0,\tau^*; W_x^{s'-a'',2^*})},
		\end{align*}
		we have 
		\begin{align*}
			\|u\|_{L^\infty ([0,\tau^*); H_x^{s'})} + \|u\|_{L^2 (0,\tau^*; W_x^{l''+\frac12,2^*})} + \|v\|_{L^\infty ([0,\tau^*); H_x^{l''})} < \infty.
		\end{align*}
		By~\eqref{prop:ImBlowUpOriCon} and the blow-up condition we already proved, the maximal existence time $\tau^{**}$ of the $H_x^{s'} \times H_x^{l''}$ solution to the system is strictly larger than $\tau^*$. Hence, we infer that $u\in \X^{s',a''}([0,\tau^*])$. Note that because of $l'>l''$, we have $a' \leq a''$. To check the conditions of Lemma~\ref{lem:BilinearEstimates}~\ref{it:BilinearEstNonendpointW}, we note that since $s' - a' > \frac{d-3}{2}$,
		\begin{align*}
			2s'-\frac{d-2}{2}-a' >s'-\frac12.
		\end{align*}
		If $a''=0$, then $s'-l' \geq \frac12 > a''$, and if $a''\neq 0$ then
		\begin{align*}
			s'-l' = \frac{3}{4}(s'-l'')+\frac{1}{4}(s'-l')+\frac{3}{4}(l''-l') = a''+\frac12+ \frac{1}{4}(s'-l') + \frac34 (l''-l') > a'',
		\end{align*}
		where we used that $s' - l' \geq \frac{1}{2}$ and $l'' - l' > -\frac{1}{2}$ by assumption.
		Moreover, if $a'' = 0$, then
		\begin{align*}
			2s'-l'-\frac{d-2}{2} \geq s'-\frac{d-3}{2}>0=2a''
		\end{align*}
		as $s' - l' \geq \frac{1}{2}$,
		and if $a''\neq 0$, then
		\begin{align*}
			2s'-l'-\frac{d-2}{2}
			&= \frac{3}{2}(s' - l'') + \frac{1}{2}(s' - l'') + l'' - l' + l'' - \frac{d-2}{2} \\
			&= 2(a''+\frac12) + \frac12 (s' - l') + \frac12 (l'' - l') +l''-\frac{d-2}{2}>2a'' + l'' - \frac{d-4}{2} \geq 2 a''.
		\end{align*}
		Hence, we can apply Lemma~\ref{lem:BilinearEstimates}~\ref{it:BilinearEstNonendpointW} with parameters $s'$, $l'$, $a'', \beta = s' - \frac12$ and obtain some $\theta>0$ such that
		\begin{align*}
			\|v\|_{W^{l',a'',s'-\frac12}([0,\tau^*))}
			&\lesssim  \|e^{\imu t|\nabla|} v_0\|_{W^{l',a'',s'-\frac12}([0,\tau^*))} + \|\cJ_0[|\nabla||u|^2]\|_{W^{l',a'',s'-\frac12}([0,\tau^*))}\\
			&\lesssim \|v_0\|_{H_x^{l'}} + \|u\|_{L^2 (0,\tau^*; L_x^{2^*})}^\theta \|u\|_{S^{s',a'',0}([0,\tau^*))}^{2-\theta}\\
			&\lesssim  \|v_0\|_{H_x^{l'}} + \|u\|^2_{\X^{s',a''}([0,\tau^*])}<\infty.
		\end{align*}
		Then we complete the proof by the estimate
		\begin{align*}
			\|v\|_{L^\infty ([0,\tau^*); H_x^{l'})}\lesssim \|v\|_{W^{l',a'',s'-\frac12}([0,\tau^*))}<\infty.
		\end{align*}
	\end{proof}

	{\bf Proof of blow-up alternative in Theorem~\ref{thm:LocalWP}.}
	We now return to the proof of the blow-up alternative in Theorem~\ref{thm:LocalWP}. 
	Recall that $l_w$ was defined in~\eqref{eq:Defl0} and that $\frac{d-4}{2} \leq l_w \leq \tilde{l}$.
	We set
	\begin{align*}
		l_s:= \sup\Big\{ \frac{d-4}{2}\leq l' \leq l_w\colon S_{l'} \neq \emptyset \Big\}.
	\end{align*}
	Note that by~\ref{it:BlowUpSobNorm} and~\ref{it:BlowUpDispersiveNorm}, we have $S_{\frac{d-4}{2}} \neq \emptyset$ and hence $l_s \geq \frac{d-4}{2}$.
	
	\vspace*{4pt plus 2pt minus 2pt}%
	We first claim that  $l_s = l_w$.
	Suppose that $l_s<l_w$. We can then choose a small positive constant $0<\varepsilon \ll 1$ such that $l_s+\varepsilon< l_w$. Take some $l' \in [\frac{d-4}{2}, l_s]$ satisfying $l_s-l'<\varepsilon$ and $S_{l'} \neq \emptyset$. By Lemma~\ref{lem:PersRegSchr} we have $S_{l'} = A_{l'}$. In particular, $s':= l'+\frac12 + \varepsilon \in A_{l'} = S_{l'}$, i.e.,
	\begin{align}\label{prop:ImBlowUpU1}
		\|u\|_{L^\infty  ([0,\tau^*); H_x^{s'})} + \|u\|_{L^2 (0,\tau^*; W_x^{s',2^*})}<\infty,
	\end{align}
	where we used that $a' = a^*(s', l') = 0$ and that $s' < \tilde{l} + \frac12 \leq s$. Hence, we also have $s' \in A_{l' + \varepsilon}$, which shows $s' \in S_{l' + \varepsilon}$ in view of~\eqref{prop:ImBlowUpU1}.
	In particular, we infer $S_{l'+\varepsilon}\neq\emptyset$. But this contradicts the definition of $l_s$ since $l_s < l' + \varepsilon < l_w$,
	thereby yielding $l_s = l_w$,
	as claimed.
	
	\vspace*{4pt plus 2pt minus 2pt}%
	Next we aim to show
	that $l_w = \tilde{l}$
	and
	\begin{equation}
		\label{eq:BlowUpBoundWave}
		\|v\|_{L^\infty ([0,\tau^*); H_x^{\tilde{l}})}<\infty.
	\end{equation}
	
	If $\tilde{l}=\frac{d-4}{2}$, the conclusion follows from \ref{it:BlowUpSobNorm}
	and the definition of $l_w$.
	Otherwise, let $0< \varepsilon \ll 1$. We pick some $l' \geq \frac{d-4}{2}$ with $l' \leq l_w$ and $l_w - l' < \varepsilon$ such that $\|v\|_{L^\infty ([0,\tau^*); H_x^{l'})}<\infty$ and $S_{l'} \neq \emptyset$. Here we used the definitions of $l_w$ and $l_s$ and that $l_w = l_s$ (respectively~\ref{it:BlowUpSobNorm} and~\ref{it:BlowUpDispersiveNorm} in the case $l_w = \frac{d-4}{2}$). By Lemma~\ref{lem:PersRegSchr}, we have $S_{l'} = A_{l'}$. Since $s' := \min\{l' + \frac12 + \varepsilon, s\} \in A_{l'}$, we thus get $s' \in S_{l'}$.
	
	As $s' > \frac{d-3}{2}$, we can now apply Lemma~\ref{lem:PersRegWave} with $l'' = \min\{l' + \varepsilon, \tilde{l}\}$ (where the roles of $l'$ and $l''$ are reversed), which yields $\|v\|_{L^\infty ([0,\tau^*); H_x^{l''})}<\infty$. If $l_w < \tilde{l}$, this contradicts the definition of $l_w$. Hence, we have $l_w = \tilde{l}$.
	Consequently,
	$l'' = \tilde{l}$
	and \eqref{eq:BlowUpBoundWave} follows.
	
	\vspace*{4pt plus 2pt minus 2pt}%
	At last,
	we prove that
	\begin{align} \label{u-LHs-L2Wl-finite}
		\|u\|_{L^\infty ([0,\tau^*); H_x^s)}
		+ \|u\|_{L^2 (0,\tau^*; W_x^{\tilde{l} + \frac12,2^*})} <\infty.
	\end{align}
	This together with \eqref{eq:BlowUpBoundWave}
	will contradict \eqref{prop:ImBlowUpOriCon},
	and thus lead to the blow-up alternative result.
	
	In order to prove \eqref{u-LHs-L2Wl-finite},
	we note that if $\tilde{l} = \frac{d-4}{2}$,
	then $S_{\tilde{l}} \neq \emptyset$ by~\ref{it:BlowUpSobNorm} and~\ref{it:BlowUpDispersiveNorm},
	and thus,
	$A_{\tilde{l}} = S_{\tilde{l}}$ by Lemma~\ref{lem:PersRegSchr}. In particular, we have $s \in A_{\tilde{l}} = S_{\tilde{l}}$. If $\tilde{l} > \frac{d-4}{2}$, we choose $0 < \varepsilon \ll 1$ such that $\tilde{l} - \varepsilon > \frac{d-4}{2}$. Since $l_s = l_w = \tilde{l}$, we have $S_{\tilde{l} - \varepsilon} = A_{\tilde{l} - \varepsilon}$ by Lemma~\ref{lem:PersRegSchr} again. In particular, we have $\tilde{l} + \frac{1}{2} \in A_{\tilde{l} - \varepsilon} = S_{\tilde{l} - \varepsilon}$. Noting that $a^*(\tilde{l} + \frac12, \tilde{l} - \varepsilon) = a^*(\tilde{l} + \frac12, \tilde{l}) = 0$, this implies that $\tilde{l} + \frac12 \in S_{\tilde{l}}$. Employing Lemma~\ref{lem:PersRegSchr} once more, we infer $A_{\tilde{l}} = S_{\tilde{l}}$ so that we conclude $s \in A_{\tilde{l}} = S_{\tilde{l}}$ again. Combining the definition of $S_{\tilde{l}}$ with~\eqref{eq:BlowUpBoundWave}, we arrive at
	\begin{align}\label{eq:ImBlowUpU2}
		\|u\|_{L^\infty ([0,\tau^*); H_x^s)} + \|u\|_{L^2 (0,\tau^*; W_x^{s-a,2^*})}+ \|v\|_{L^\infty ([0,\tau^*); H_x^{\tilde{l}})}<\infty,
	\end{align}
	where $a = a^*(s, \tilde{l})$.
	Note that
	if $a=0$, then $s-a = s \geq \tilde{l}+\frac12$.
	Moreover, if $a\neq 0$,
	then $a +\frac{1}{2}= \frac{3}{4}(s-\tilde{l}) < (s-\tilde{l})$, i.e., $\tilde{l}+\frac12 < s- a$. Consequently,
	\eqref{u-LHs-L2Wl-finite}
	follows from \eqref{eq:ImBlowUpU2}.
	The proof of Theorem~\ref{thm:LocalWP}
	is complete.
	\qed

	\section{LWP and blow-up alternative: endpoint case} \label{Sec-LWPEP}
	
	In this section, we turn to the proof of Theorem \ref{thm:LocalWP}
	in the endpoint case $(s,l)=(\tfrac{d-3}{2},\tfrac{d-4}{2})$.
	Unlike in the non-endpoint setting,
	the nonlinear estimates here do not yield the $L_t^2 L_x^{2^*}$-component on the right-hand side leveraged to obtain smallness in the fixed point argument before.
	Instead, we rely on the more delicate estimates~\eqref{eq:Bilinnablauu} and~\eqref{eq:BilinvuEndpoint}, which utilize the weaker $S_w^{s,0,0}$- and $W_w^{l,0,l}$-norms defined in~\eqref{eq:WeakNormS} and~\eqref{eq:WeakNormW}.

	\subsection{Local well-posedness}  \label{Subsec-LWP-Endpt}
	Consider the fixed-point operator
	\begin{align*}
		\Phi(u_0, v_0; u) := \cU_{v_L} [u_0](t) - \cI_{v_L}[\Re(\cJ_0[|\nabla| |u|^2])u](t) - \cI_{v_L}[b \cdot \nabla u + c u - \Re(\cT_{\cdot}(W_2)) u](t),
	\end{align*}
	again, but on the complete metric space
	\begin{align*}
		B_{R, \delta}(\tau) = \{u \in \mathbb{X}^{s,0}([0,\tau]) \colon \|u\|_{L^2 (0,\tau ; W_x^{s,2^*})} \leq \delta, \, \|u\|_{\mathbb{X}^{s,0}([0,\tau])} \leq R\}
	\end{align*}	
	equipped with the metric induced by $\| \cdot \|_{\mathbb{X}^{s,0}([0,\tau])}$,
	where  $\delta, R > 0$ and $\tau > 0$ is a stopping time to be chosen later.
	We will show that $\Phi(u_0, v_0; \cdot)$ is a contractice self-map on $B_{R, \delta}(\tau)$.
	
	\vspace*{4pt plus 2pt minus 2pt}%
	For this purpose, we fix $\epsilon > 0$ as in Lemma~\ref{lem:LinSchrPotSmallTime},
	$T > 0$, and take the stopping time
	\begin{equation}
		\label{eq:DefTauTildeEP}
		\tau_0 := \inf\{t \in [0,T] \colon \|v_L\|_{W^{l,0,l}([0,t]) + L^2 (0,t ; W_x^{s,d})}
		\geq \gamma\} \wedge \min\{2,T\},
	\end{equation}
	where $0<\gamma\leq \varepsilon$ is a small constant to be chosen later.
	Note that $\tau_0 > 0$ due to Lemma~\ref{lem:SmallnessWavePotential}.
	Again we may assume $\tau \leq \tau_0$ by considering $\tau_0\wedge \tau$.

	\vspace*{4pt plus 2pt minus 2pt}%
	We start with the estimates of the $\mathbb{X}^{s,0}([0,\tau])$-norm.
	By
	Lemma~\ref{lem:LinSchrPotSmallTime} and Lemma~\ref{lem:LinEstimates},  we have
	\begin{align}
		\| \Phi(u_0, v_0; u) \|_{\mathbb{X}^{s,0}([0,\tau])}
		&\leq C \|u_0\|_{H_x^s} + C \|\Re(\cJ_0[|\nabla| |u|^2])u\|_{\mathbb{G}^{s,0}([0,\tau])} + C \|\mathcal{I}_0[b \cdot \nabla u]\|_{\mathbb{X}^{s,0}([0,\tau])}  \nonumber\\
		&\qquad + C \|c u\|_{\mathbb{G}^{s,0}([0,\tau])} + C \|\Re(\cT_{\cdot}(W_2)) u\|_{\mathbb{G}^{s,0}([0,\tau])}. \label{eq:EstFixedPointOp1EP}
	\end{align}
	Then, by Lemma~\ref{lem:BilinearEstimates} \ref{it:BilinearEstEndpoint1}, \ref{it:BilinearEstEndpoint2} and the embedding $W^{l,0,l}\hookrightarrow W_w^{l,0,l}$, we have
	\begin{align}
		\|\Re(\cJ_0[|\nabla| |u|^2])u\|_{\mathbb{G}^{s,0}([0,\tau])}
		&\lesssim \| \Re(\cJ_0[|\nabla| |u|^2])u\|_{N^{s,0,0}([0,\tau])}    \nonumber\\
		& \lesssim \| \cJ_0[|\nabla| |u|^2] \|_{W_w^{l,0,l}([0,\tau])} \|u\|^{\frac12}_{L_t^2W_x^{s,2^*}}\|u\|^{\frac12}_{S_w^{s,0,0}([0,\tau])} \nonumber\\
		&\lesssim \|u\|_{L^2_t W^{s,2^*}_x}^{\frac32} \|u\|_{\mathbb{X}^{s,0}([0,\tau])}^{\frac32}. \label{eq:FixedPointOpNonlinEP}
	\end{align}
	Plugging this into~\eqref{eq:EstFixedPointOp1EP} and applying Lemma~\ref{lem:BilinearLowerOrder} and Lemma~\ref{lem:BilinearLowerOrder2}, we get
	\begin{align}
		\| \Phi(u_0, v_0; u) \|_{\mathbb{X}^{s,0}([0,\tau])}
		&\leq C ( \|u_0\|_{H_x^s} +  \|u\|_{L^2_t W^{s,2^*}_x}^{\frac32} \|u\|_{\mathbb{X}^{s,0}([0,\tau])}^{\frac32}
		+ (1+\tau^{\frac{1}{2}}) W^*(\tau) \|u\|_{\mathbb{X}^{s,0}([0,\tau])} ). \label{eq:EstFixedPointOp2EP}
	\end{align}
	
	Regarding the $L^2_tW^{s,2^*}_x$-norm,
	using \eqref{UvL-express}, the embedding $S^{s,0,0}  \hookrightarrow L_t^2W_x^{s,2^*}$, Lemma~\ref{lem:LinSchrPotSmallTime}, Lemma~\ref{lem:LinEstimates}, Lemma~\ref{lem:BilinearEstimates}, and Lemma~\ref{it:BilinearEstNonendpoint}, we derive
	\begin{align*}
		\|\cU_{v_L}[u_0] \|_{L^2_t W_x^{s,2^*}} &\leq \| e^{\imu t \Delta} u_0 \|_{L^2_t W^{s,2^*}_x} + \|\cI_{v_L}[\Re(v_L)\cU_0[u_0]]\|_{\X^{s,0}([0,\tau])} \notag \\
		&\leq \| e^{\imu t \Delta} u_0 \|_{L^2_t W^{s,2^*}_x} + C\|\cI_0[\Re(v_L)\cU_0[u_0]]\|_{\X^{s,0}([0,\tau])} \notag \\
		& \leq \| e^{\imu t \Delta} u_0 \|_{L^2_t W^{s,2^*}_x} + C\|v_L\|_{W^{l,0,l}([0,\tau])+L^2 (0,\tau; W_x^{s,d})}\|\cU_0[u_0]\|_{S^{s,0,0}([0,\tau])}\\
		& \leq \| e^{\imu t \Delta} u_0 \|_{L^2_t W^{s,2^*}_x} + C\|v_L\|_{W^{l,0,l}([0,\tau])+L^2(0,\tau; W_x^{s,d})}\|u_0\|_{H_x^s}
	\end{align*}
	The embedding $ \mathbb{X}^{s,0} \hookrightarrow L_t^2 W_x^{s,2^*}$, Lemma~\ref{lem:LinSchrPotSmallTime},  Lemma~\ref{lem:LinEstimates}, and Lemma \ref{lem:BilinearEstimates} \ref{it:BilinearEstEndpoint1}, \ref{it:BilinearEstEndpoint2}  also show that
	\begin{align*}
		\| \cI_{v_L}[\Re(\cJ_0[|\nabla| |u|^2]) u] \|_{L^2_t W^{s,2^*}_x}
		&\lesssim \| \cI_0[\Re(\cJ_0[|\nabla| |u|^2]) u] \|_{\X^{s,0}([0,\tau])} \lesssim \|\Re(\cJ_0[|\nabla||u|^2])u\|_{N^{s,0,0}([0,\tau])} \\
		& \lesssim  \|\cJ_0[|\nabla| |u|^2]\|_{W^{l,0,l}([0,\tau])} \|u\|^{\frac12}_{L^2_t W^{s,2^*}_x} \|u\|^{\frac12}_{\mathbb{X}^{s,0}([0,\tau])}
		\lesssim  \|u\|^{\frac32}_{L^2_t W^{s,2^*}_x} \|u\|^{\frac32}_{\mathbb{X}^{s,0}([0,\tau])}.
	\end{align*}
	Taking into account Lemma~\ref{lem:BilinearLowerOrder}, Lemma~\ref{lem:BilinearLowerOrder2} and Lemma~\ref{it:BilinearEstNonendpoint}, we thus arrive at
	\begin{align}
		\|\Phi(u_0, v_0; u)\|_{L^2_t W^{s,2^*}_x} &\leq  \| e^{\imu t \Delta} u_0 \|_{L^2_t W^{s,2^*}_x} +  C\|v_L\|_{W^{l,0,l}([0,\tau])+L^2 (0,\tau; W_x^{s,d})}\|u_0\|_{H_x^s}\nonumber\\
		&\qquad + C\|u\|^{\frac32}_{L^2_t W^{s,2^*}_x} \|u\|^{\frac32}_{\mathbb{X}^{s,0}([0,\tau])} + C (1+\tau^{\frac{1}{2}})W^*(\tau) \|u\|_{\mathbb{X}^{s,0}([0,\tau])}
		\label{eq:EstFixedPointOp3EP}
	\end{align}
	
	Now, let $C$ be the maximum of the generic constants in~\eqref{eq:EstFixedPointOp2EP} and~\eqref{eq:EstFixedPointOp3EP}
	and set $R = 2 C \|u_0\|_{H_x^s}$.
	We conclude from \eqref{eq:EstFixedPointOp2EP} and~\eqref{eq:EstFixedPointOp3EP} that
	for any $u \in B_{R, \delta}(\tau)$,
	\begin{align*}
		&\| \Phi(u_0, v_0; u) \|_{\mathbb{X}^{s,0}([0,\tau])} \leq \frac{R}{2}
		+ C (\delta^{\frac32}R^{\frac32}
		+ (1+\tau^{\frac{1}{2}}) R W^*(\tau) ), \\
		&\| \Phi(u_0, v_0; u) \|_{L^2_t W^{s,2^*}_x}
		\leq  \| e^{\imu t \Delta} u_0 \|_{L^2_t W^{s,2^*}_x([0,\tau])}
		+  C  ( \gamma \frac{R}2 + \delta^{\frac32}R^{\frac32}
		+  (1+ \tau^{\frac{1}{2}})  R W^*(\tau) ).
	\end{align*}
	It follows that
	\begin{align*}
		\Phi(u_0, v_0; \cdot) B_{R,\delta}(\tilde{\tau}) \subseteq  B_{R,\delta}(\tilde{\tau}),
	\end{align*}
	where
	$\delta \in (0,R)$ and $\gamma \in (0,\varepsilon)$
	are small enough such that $8C\delta^{\frac32}R^{\frac12}\leq 1$  and $C\gamma \frac{R}{2}+C\delta^{\frac{3}{2}}R^{\frac{3}{2}} \leq \frac{3\delta}{8}$,
	and $\tilde{\tau}$ is the strictly positive stopping time defined by
	\begin{align}
		\label{eq:DefFixedPointStoppingTimeSelfmapEP}
		\tilde{\tau} := \inf\Big\{t \in [0,T] \colon  \| e^{\imu (\cdot) \Delta} u_0 \|_{L^2 (0,t ; W_x^{s,2^*})} +	2 C R  W^*(t) \geq \frac{\delta}{4}\Big\} \wedge \tau_0.
	\end{align}

	\vspace*{4pt plus 2pt minus 2pt}%
	Regarding the contraction of $\Phi(u_0, v_0; \cdot)$, we first argue as in~\eqref{eq:FixedPointOpNonlinEP} to derive
	\begin{align*}
		&\|\Re(\cJ_0[|\nabla| |u|^2]) u - \Re(\cJ_0[|\nabla| |w|^2]) w\|_{\mathbb{G}^{s,0}([0,\tau])} \\
		&\leq \|\Re(\cJ_0[|\nabla|((\overline{u} - \overline{w})u)]) u + \Re(\cJ_0[|\nabla| (\overline{w}(u - w))]) u  + \Re(\cJ_0[|\nabla| |w|^2]) (u - w)\|_{N^{s,0,0}([0,\tau])} \\
		&\lesssim \|u-w\|_{L_t^2 W_x^{s,2^*}}^{\frac12} \|u-w\|_{\mathbb{X}^{s,0}([0,\tau])}^{\frac12} \|u\|_{L_t^2 W_x^{s,2^*}}\|u\|_{\mathbb{X}^{s,0}} \\
		&\qquad + \|w\|^{\frac12}_{L_t^2 W_x^{s,2^*}}\|w\|^{\frac12}_{\mathbb{X}^{s,0}([0,\tau])} \|u-w\|^{\frac12}_{L_t^2 W_x^{s,2^*}} \|u-w\|^{\frac12}_{\mathbb{X}^{s,0}([0,\tau])} \|u\|^{\frac12}_{L_t^2 W_x^{s,2^*}}\|u\|^{\frac12}_{\mathbb{X}^{s,0}([0,\tau])}\\
		&\qquad + \|w\|_{L_t^2 W_x^{s,2^*}}\|w\|_{\mathbb{X}^{s,0}([0,\tau])} \|u-w\|_{L_t^2 W_x^{s,2^*}}^{\frac12} \|u-w\|_{\mathbb{X}^{s,0}([0,\tau])}^{\frac12} \lesssim \delta R \|u-w\|_{\mathbb{X}^{s,0}([0,\tau])}
	\end{align*}
	for all $u,w \in B_{R,\delta}(\tau)$,
	where we also used the embedding $ \mathbb{X}^{s,0} \hookrightarrow L_t^2 W_x^{s,2^*}$ once more in the last step.
	Arguing as in the proof of \eqref{eq:EstFixedPointOp2EP} we get
	\begin{align}
		&\|\Phi(u_0, v_0; u) - \Phi(u_0, v_0; w)\|_{\mathbb{X}^{s,0}([0,\tau])}
		\leq C (\delta R
		+ (1+\tau^{\frac{1}{2}})W^*(\tau)) \|u-w\|_{\mathbb{X}^{s,0}([0,\tau])}.  \label{eq:EstFixedPointOpDiffEP}
	\end{align}
	
	It follows that $\Phi(u_0, v_0; \cdot)$ is a contraction  on $B_{R,\delta}(\tilde{\tau}_1)$
	by taking  $\delta > 0$ possibly smaller such that additionally $C\delta R \leq \frac{1}{4}$  and
	\begin{align*}
		\tilde{\tau}_1 := \inf\Big\{t \in [0,T]: CW^*(t) \geq \frac{1}{4}\Big\} \wedge \tilde{\tau}.
	\end{align*}			
	
	An application of Banach's fixed-point theorem thus
	yields that there exists a unique solution $\tilde{u}_1 \in B_{R,\delta}(\tilde{\tau}_1)$ to~\eqref{eq:LWPFixedPoint}, which is unique in $\mathbb{X}^{s,0}([0,\tilde{\tau}_1])$ by standard arguments.
	Setting
	\begin{align*}
		\tilde{v}_1(t) := v_L(t) - \cJ_0[|\nabla| |\tilde{u}_1|^2](t) = e^{\imu t |\nabla|} v_0  - \cJ_0[|\nabla| |\tilde{u}_1|^2](t), \ \
		t\in [0,\tilde{\tau}_1],
	\end{align*}
	which belongs to $W^{l,0,l}([0,\tilde{\tau}_1])$ due to Lemma~\ref{lem:BilinearEstimates} \ref{it:BilinearEstEndpoint2},
	we obtain that
	$(\tilde{u}_1, \tilde{v}_1)$ is the unique solution to \eqref{eq:RanZakbc}, where the uniqueness holds in $\mathbb{X}^{s,0}([0,\tilde{\tau}_1]) \times L^\infty ([0,\tilde{\tau}_1] ; H_x^l)$.
	Finally,
	\begin{align*}
		(u_1,v_1)(t) := (\tilde{u}_1(t \wedge \tilde{\tau}_1), \tilde{v}_1(t \wedge \tilde{\tau}_1)),\ \ t\in [0,T],
	\end{align*}
	is an $\{\cF_t\}$-adapted process in $C([0,T], H_x^s \times H_x^l)$
	and solves~\eqref{eq:RanZakbc} on $[0,\tilde{\tau}_1]$.

	\subsection{Extension to maximal time}\label{Subsec-Max-ExistEP}
	Employing the refined rescaling strategy, we extend the local solution to its maximal existence time. This procedure is analogous to the one detailed in Subsection~\ref{Subsec-Max-Exist}, but employing the estimates from Subsection~\ref{Subsec-LWP-Endpt}. We therefore limit our exposition to a brief sketch.
	
	Consider system \eqref{eq:RanZakbsigmacsigmaInProof}
	with the initial data $(u_{0,n}, v_{0,n})$ given by \eqref{eq:DefInitialDataRescaled} and the operator $\Phi_\sigma(u_{0,n}, v_{0,n}; u_\sigma)$ as in \eqref{eq:FixedPointusigma}. 
	Define the $\{\cF_{\sigma_n + t}\}$-stopping times by
	\begin{align*}
		\tilde{\tau}_{n+1} &=\inf\{t \in [0, T] \colon  \|e^{\imu (\cdot) \Delta} u_{0,n}\|_{L^2 (0,t ; W_x^{s,2^*})} + W_{\sigma_n}^*(t) \geq 4 \delta_*(\|u_{0,n}\|_{H_x^s})\} \\
		&\qquad \wedge \inf\{t \in [0,T] \colon \|v_{L,n}\|_{W^{l,0,l}([0,t]) + L^2 (0,t ; W_x^{s,d})} \geq \gamma\} \wedge \min\{2,T - \sigma_n\}, \\
		\tau_{n+1} &= \inf\{t \in [0, T] \colon  \|e^{\imu (\cdot) \Delta} u_{0,n}\|_{L^2 (0,t ; W_x^{s,2^*})} + W_{\sigma_n}^*(t) \geq 2 \delta_*(\|u_{0,n}\|_{H_x^s})\} \\
		&\qquad \wedge \inf\Big\{t \in [0,T] \colon \|v_{L,n}\|_{W^{l,0,l}([0,t]) + L^2 (0,t ; W_x^{s,d})} \geq  \frac{\gamma}{2}\Big\} \wedge \min\{1,T - \sigma_n\},
	\end{align*}
	where $\delta^*$ is a decreasing strictly positive function,
	and two $\{\cF_t\}$-stopping times by
	\begin{align*}
		\sigma_{n + 1} := \sigma_n + \tau_{n+1}, \qquad \tilde{\sigma}_{n+1} := \sigma_n + \tilde{\tau}_{n+1}.
	\end{align*}

	Arguing as in Subsection~\ref{Subsec-LWP-Endpt}
	we infer that
	$\Phi_\sigma(u_{0,n}, v_{0,n}; \cdot)$ is a contraction on a closed ball of $\mathbb{X}^{s,0}([0,\tilde{\tau}_{n+1}])$, which leads to a fixed point $u_{\sigma_{n+1}}$
	of $\Phi_\sigma(u_{0,n}, v_{0,n}; \cdot)$.
	Setting $\tilde{v}_{\sigma_{n+1}} := e^{\imu t |\nabla|}v_{0,n} - \cJ_0[|\nabla| |\tilde{u}_{\sigma_{n+1}}|^2]$,
	we thus obtain a solution $(\tilde{u}_{\sigma_{n+1}}, \tilde{v}_{\sigma_{n+1}})$
	of~\eqref{eq:RanZakbsigmacsigmaInProof} in $\mathbb{X}^{s,0}([0,\tilde{\tau}_{n+1}]) \times W^{l,0,l}([0,\tilde{\tau}_{n+1}])$, which is unique in $\mathbb{X}^{s,0}([0,\tilde{\tau}_{n+1}]) \times L^\infty ([0,\tilde{\tau}_{n+1}] ; H_x^l)$.

	Next, following the procedure of Subsection \ref{Subsec-Max-Exist}
	we obtain a sequence of solutions $(u_{\sigma_{n+1}}(t), v_{\sigma_{n+1}}(t))$,
	$n \geq 1$, and then
	glue them to have the unique solution $(u_{n+1}, v_{n+1})\in \mathbb{X}^{s,0}([0,\tilde{\sigma}_{n+1}]) \times W^{l,0,l}([0,\tilde{\sigma}_{n+1}])$ to \eqref{eq:RanZakbc}, where the uniqueness holds in $\mathbb{X}^{s,0}([0,\tilde{\sigma}_{n+1}]) \times L^\infty ([0,\tilde{\sigma}_{n+1}] ; H_x^l)$.
	Moreover, $(u_{n+1}, v_{n+1})$ is $\{\cF_t\}$-adapted and continuous in $H_x^s \times H_x^l$.
	
	We set
	\begin{align*}
		\tau_T^* := \lim_{n \rightarrow \infty} \sigma_n \qquad \text{and} \qquad (u^T, v^T) := \lim_{n \rightarrow \infty} (u_n \chi_{[0, \tau_T^*)}, v_n \chi_{[0, \tau_T^*)}),
	\end{align*}
	and then
	\begin{align*}
		\tau^* := \lim_{T \rightarrow \infty} \tau_T^* \qquad \text{and} \qquad (u,v) := \lim_{T \rightarrow \infty} (u^T \chi_{[0,\tau^*)}, v^T \chi_{[0,\tau^*)}).
	\end{align*}
	The resulting process $(u,v)$ is thus $\{\cF_t\}$-adapted and
	continuous in $H_x^s \times H_x^l$,
	and uniquely solves \eqref{eq:RanZakbc} on $[0,\tau^*)$.
	Now we transfer back via the rescaling
	and obtain that
	\begin{align*}
		(X,Y) := (e^{W_1} u, v + \cT_{\cdot}(W_2)).
	\end{align*}
	is the unique solution of~\eqref{eq:StoZak} on $[0,\tau^*)$ in the sense of Definition~\ref{def:Solution}. Here, we also exploited Lemma~\ref{lem:ProductNoiseInX}, which ensures that on any compact interval $X$ belongs to $\X^{s,0}$ if and only if $u$ does.

	\subsection{Blow-up alternative}
	Suppose that the blow-up alternative in Theorem~\ref{thm:LocalWP}
	fails.
	Then, invoking Lemma~\ref{lem:PropNoise},
	there exists a set of positive measure on which:
	\begin{enumerate}
		\item \label{it:blowupconEP1} $\tau^*(\omega) < \infty$,
		\item \label{it:blowupconEP2} $\limsup_{t \rightarrow \tau^*(\omega)} (\|X(t,\omega)\|_{H_x^s} + \|Y(t,\omega)\|_{H_x^l}) < \infty$,
		\item \label{it:blowupconEP3} $\|X(\cdot, \omega)\|_{L^2 (0,\tau^*(\omega) ; W_x^{s,2^*})} < \infty$,
	\end{enumerate}
	and where the convergence \eqref{eq:limphi1kbeta1k}
	holds.
	In the following
	we fix an $\omega$
	from this set and let $T \in (0,\infty)$ such that $T > \tau^*(\omega)$.
	The dependence of $\omega$ is omitted below.

	As in Subsection \ref{Subsec-Blowup},
	it follows from the above three conditions \ref{it:blowupconEP1}-\ref{it:blowupconEP3} that there exists $\delta>0$ such that for $n$ large enough,
	\begin{align}
		\label{eq:AlternativeTaunplus1EP}
		&W_{\sigma_n}^*(\tau_{n+1}) \geq \delta  \quad
		\text{or} \quad
		\|e^{\imu (\cdot) \Delta} u_{0,n}\|_{L^2 (0,\tau_{n+1} ; W_x^{s,2^*})}  \geq \delta \quad \notag\\
		\text{or} \quad &\|v_{L,n}\|_{W^{l,0,l}([0,\tau_{n+1}]) + L^2 (0,\tau_{n+1} ; W_x^{s,d})} \geq \frac{\gamma}{2},
	\end{align}
	where $\{\sigma_n, \tau_n\}$ is constructed as in Subsection \ref{Subsec-Max-ExistEP}, $W^*_{\sigma_n}$ is as in \eqref{eq:DefWStarsigman}, $v_{L,n}$ denotes the linear wave defined immediately after~\eqref{eq:FixedPointusigma},
	and $T - \sigma_n < 1$ for all $n \in \N$.
	But as the noise terms in the endpoint case are the same as in the non-endpoint case,
	the arguments in Subsection \ref{Subsec-Blowup}
	also show that the first alternative in~\eqref{eq:AlternativeTaunplus1EP} is not valid
	for $n$ large enough.
	
	Next, we show that
	\begin{align}\label{prop:blowupsecEP}
		\|e^{\imu t \Delta} u_{0,n}\|_{L^2 (0,\tau_{n+1} ; W_x^{s,2^*})} \longrightarrow 0\quad \text{as } n\rightarrow \infty,
	\end{align}
	which excludes the second alternative in \eqref{eq:AlternativeTaunplus1EP}.
	
	To that purpose,
	using \eqref{usigman-vsigman} and \eqref{eq:RanZakbsigmacsigmaInProof},
	we derive analogously to~\eqref{eq:estimateu0nNEP}
	\begin{align}\label{eq:estimateu0n}
		\|e^{\imu t\Delta}u_{0,n}\|_{L^2 (0,\tau_{n+1}; W_x^{s,2^*})}
		&\lesssim \|u_{\sigma_n}\|_{L^2 (0,\tau_{n+1}; W_x^{s,2^*})}+\Big\| \cI_0(\Re(v_{\sigma_n})u_{\sigma_n}) \Big\|_{L^2 (0,\tau_{n+1}; W_x^{s,2^*})}\notag \\
		&\qquad + \Big\| \cI_0(b_{\sigma_n}\cdot \nabla u_{\sigma_n}+c_{\sigma_n}u_{\sigma_n}-\Re(\cT_{\sigma_n+\cdot,\sigma_n})u_{\sigma_n}) \Big\|_{L^2 (0,\tau_{n+1}; W_x^{s,2^*})}.
	\end{align}
	The third term on the right-hand side of  \eqref{eq:estimateu0n} can be bounded via Lemma \ref{lem:BilinearLowerOrder} by
	\begin{align}
		\Big\| \cI_0(b_{\sigma_n}\cdot \nabla u_{\sigma_n}+c_{\sigma_n}u_{\sigma_n}-\Re(\cT_{\sigma_n+\cdot,\sigma_n})u_{\sigma_n}) \Big\|_{L^2 (0,\tau_{n+1}; W_x^{s,2^*})}
		\lesssim W^*_{\sigma_n}(t)\sup_n \|u\|_{\X^{s,0}([0,\sigma_n])}.\label{prop:endponitblowupDuhamel}
	\end{align}
	Moreover, for the second term in \eqref{eq:estimateu0n}, we estimate
	\begin{align}
		&\Big\| \cI_0(\Re(v_{\sigma_n})u_{\sigma_n}) \Big\|_{L^2 (0,\tau_{n+1}; W_x^{s,2^*})}\lesssim\Big\|\cI_0(\Re(v_{\sigma_n})u_{\sigma_n}) \Big\|_{S^{s,0,0}([0,\tau_{n+1}])}\notag\\
		&\lesssim \| \Re(v_{\sigma_n})u_{\sigma_n} \|_{N^{s,0,0}([0,\tau_{n+1}])}\lesssim \|v_{\sigma_n}\|_{W_w^{l,0,l}([0,\tau_{n+1}])} \|u_{\sigma_n}\|_{L^2 (([0,\tau_{n+1}]); W_x^{s,2^*})}^{\frac12} \sup_n \|u\|_{\mathbb{X}^{s,0}([0,\sigma_n])}^{\frac12}\notag\\
		&\lesssim   \Big(\|v\|_{L^\infty (0,\tau^*; H_x^l)} + \|\cT_{\sigma_n}(W_2)\|_{L^\infty([0,\tau_{n+1}]; H_x^l)}\Big) \|u\|_{L^2 (\sigma_n,\tau^*; W_x^{s,2^*})}^{\frac12} \sup_n \|u\|_{\mathbb{X}^{s,0}([0,\sigma_n])}^{\frac12} \notag\\
		&\qquad +\|u\|^{\frac32}_{L^2 (\sigma_n,\tau^*; W_x^{s,2^*})} \sup_n \|u\|^{\frac32}_{\mathbb{X}^{s,0}([0,\sigma_n])}. \label{prop:endpointblowupSchrDuhaml}
	\end{align}
	In the last step, we employed $v_{\sigma_n}= e^{\imu t|\nabla|} (v(\sigma_n)+ \mathcal{T}_{\sigma_n}(W_2)) - \cJ_0[|\nabla||u_{\sigma_n}|^2]$, Lemma~\ref{lem:LinearEstimateHalfWave}, Lemma~\ref{lem:BilinearEstimates}~\ref{it:BilinearEstEndpoint2}, and the embeddings $W^{l,0,l}\hookrightarrow W_w^{l,0,l}$ and $\X^{s,0}\hookrightarrow S_w^{s,0,0}$ to infer 
	\begin{align*}
		\|v_{\sigma_n}\|_{W_w^{l,0,l}([0,\tau_{n+1}])}&\lesssim \|(v(\sigma_n)+ \mathcal{T}_{\sigma_n}(W_2))\|_{H_x^l} + \|u_{\sigma_n}\|_{L^2(0,\tau_{n+1}; W_x^{s,2^*})} \|u_{\sigma_n}\|_{S_w^{s,0,0}([0,\tau_{n+1}])}\\
		&\lesssim \|v\|_{L^\infty (0,\tau^*; H_x^l)} + \|\cT_{\sigma_n}(W_2)\|_{L^\infty([0,\tau_{n+1}]; H_x^l)} + \|u\|_{L^2(\sigma_n,\tau^*; W_x^{s,2^*})} \sup_n \|u\|_{\X^{s,0}([0,\sigma_n])}.
	\end{align*}
	
	Note that in view of~\ref{it:blowupconEP2} and~\ref{it:blowupconEP3}, Lemma~\ref{le:unibouep} implies the uniform bound $\sup_n \|u\|_{\X^{s,0}([0,\sigma_n])}<\infty$. The noise term $\|\cT_{\sigma_n}(W_2)\|_{L^\infty([0,\tau_{n+1}]; H_x^l)}\lesssim W^*_{\sigma_n}(\tau_{n+1})$ converges to zero. The condition \ref{it:blowupconEP2} and the dominated convergence theorem imply that $\|u\|_{L^2(\sigma_n,\tau^*; W_x^{s,2^*})}\to 0$ as $n\to \infty$. Thus, combing these properties and the estimate \eqref{prop:endponitblowupDuhamel}, \eqref{prop:endpointblowupSchrDuhaml}, we complete the proof of \eqref{prop:blowupsecEP}.
	
	Finally, we have that
	\begin{equation}
		\label{eq:BlowUpAltLinWavePotSmallEP}
		\|v_{L,n}\|_{W^{l,0,l}([0,\tau_{n+1}]) + L^2 (0,\tau_{n+1} ; W_x^{s,d})} < \frac{\gamma}{2}
	\end{equation}
	if $n$ is large enough.
	The proof follows along the same lines as in the proof of \eqref{eq:BlowUpAltLinWavePotSmall}
	with $l$ replacing $\tilde{l}$.
	This excludes the third alternative in~\eqref{eq:AlternativeTaunplus1EP}.
	
	Therefore, the combination of~\eqref{eq:AlternativeTaunplus1EP}, \eqref{prop:blowupsecEP}, and~\eqref{eq:BlowUpAltLinWavePotSmallEP}
	contradicts \eqref{eq:AlternativeTaunplus1EP},
	and thus proves the blow-up alternative in Theorem~\ref{thm:LocalWP}
	for the endpoint case. \hfill \qed
	
	Finally, observe that for any $(s,l)$ in the local well-posedness regime~\eqref{IniReg-condition}, we have the embedding \linebreak $H^s_x \times H^l_x \hookrightarrow H^{\frac{d-3}{2}}_x \times H^{\frac{d-4}{2}}_x$. Consequently, any solution $(X,Y) \in C([0,\tau^*), H^s_x \times H^l_x)$ that satisfies the auxiliary condition $X \in \X^{\frac{d-3}{2}}([0,\tau \wedge T])$ for any $T > 0$ and $\{ \cF_t \}$-adapted stopping time $\tau < \tau^*$, is a solution at the endpoint regularity $(\frac{d-3}{2}, \frac{d-4}{2})$. The uniqueness established at the endpoint thus immediately applies to this higher-regularity solution. This concludes the proof of uniqueness for Theorem~\ref{thm:LocalWP} in the sense of Remark~\ref{rem:Uniqueness}.

	\section{Extending the functional framework for noise-regularization: maximal function spaces  }\label{Sec-MaximalFunctionEstimate}
	
	In this section we introduce a new functional framework tailored to the 
	noise-regularization problem on scattering. In addition to the Fourier restriction and local smoothing spaces, 
	the new functional framework 
	incorporates maximal function spaces, 
	which is inspired by the analysis of Schr\"odinger maps in~\cite{BIKT11}.
	The inclusion of these maximal function norms is crucial to reach the borderline of the noise-regularization regime dictated by the noise. 
	
	As explained in Subsection \ref{subsec:IdeaProof} before, 
	we shall use the non-endpoint local smoothing 
	and maximal function spaces. 
	Precisely, 
	let $0 < \epsilon \ll 1$, 
	and set $\delta = \delta(\epsilon) = \frac{4}{d} \epsilon$. We then define the exponents $(p_\epsilon, q_\epsilon)$ by
	\begin{align*}
		p_\epsilon = \frac{4}{\delta} = \frac{d}{\epsilon} \qquad \text{and} \qquad q_\epsilon = \frac{4}{2-\delta} = \frac{4 d}{2 d - 4 \epsilon}.
	\end{align*}
	Note that
	\begin{align*}
		\frac{1}{p_\epsilon} + \frac{1}{q_\epsilon} = \frac12.
	\end{align*}
	
	We define the controlling norm $\| \cdot \|_{\Z^{s,b}_\epsilon}$ for the Schr{\"o}dinger component by
	\begin{align}\label{def:Z-norm}
		\| z \|_{\Z^{s,b}_{\epsilon,\lambda}} = \| z \|_{S^{s,0,b}_\lambda} + \sum_{j = 1}^d \lambda^{s + \frac12 - \epsilon} \| P_{\lambda, \vece_j} C_{\leq(\frac{\lambda}{2^8})^2} z \|_{L^{p_\epsilon, q_\epsilon}_{\vece_j}} + \sum_{j = 1}^d \lambda^{s - \frac{d-1}{2} + \epsilon} \| z \|_{L^{q_\epsilon, p_\epsilon}_{\vece_j}}
	\end{align}
	if $\lambda > 1$, $\| z \|_{\Z^{s,b}_{\epsilon,\lambda}} = \| z \|_{S^{s,0,b}}$ if $\lambda = 1$, and
	\begin{align*}
		\| z \|_{\Z^{s,b}_\epsilon} = \Big(\sum_{\lambda \in 2^{\N_0}} \| P_\lambda z \|_{\Z^{s,b}_{\epsilon,\lambda}}^2 \Big)^{\frac{1}{2}}.
	\end{align*}
	
	Let us first recall the maximal function estimate for the linear Schr{\"o}dinger evolution in lateral Strichartz spaces, see Lemma~7.1 in~\cite{BIKT11} or Lemma~4.1 in~\cite{IK07}.
	
	\begin{lemma}[Endpoint homogeneous maximal function estimate]
		\label{lem:MaximalFunctionEstimate}
		Let $\lambda \in 2^{\N}$ and $\vece \in \bS^{d-1}$. Then
		\begin{align*}
			\| e^{\imu t \Delta} P_\lambda f \|_{L^{2,\infty}_{\vece}} \lesssim \lambda^{\frac{d-1}{2}} \| P_\lambda f \|_{L^2_x}.
		\end{align*}
	\end{lemma}
	
	Next we show the maximal function estimate in $L^{q_\epsilon, p_\epsilon}_{\vece}$, both for the linear flow and the Duhamel integral.
	
	\begin{lemma}[Near-endpoint maximal function estimates]
		\label{lem:MaxFctDualEndptStrichartz}
		Let $\lambda \in 2^{\N}$, $\vece \in \bS^{d-1}, t_0\in \R$, and $0 < \epsilon \ll 1$. Then the following estimates hold
		\begin{align}
			&\| e^{\imu t \Delta} P_\lambda f \|_{L^{q_\epsilon, p_\epsilon}_{\vece}} \lesssim \lambda^{\frac{d-1}{2} - \epsilon} \| P_\lambda f \|_{L^2_x}, \label{eq:EstMaxFctHomInterpol}\\
			&\Big\| \int_{t_0}^t e^{\imu(t-s) \Delta} P_\lambda g(s) \dd s \Big\|_{L^{q_\epsilon, p_\epsilon}_{\vece}} \lesssim \lambda^{\frac{d-1}{2} - \epsilon} \| P_\lambda g \|_{L^2_t L^{2_*}_x}. \label{eq:EstMaxFctInhomInterpol}
		\end{align}
	\end{lemma}
	
	\begin{proof}
		Recall that $\delta = \frac{4}{d} \epsilon$. With $\theta = 1 - \delta$ we have
		\begin{align}
			\label{eq:ParameterInterpol}
			\frac{1}{p_\epsilon} = \frac{\theta}{\infty} + \frac{1-\theta}{4} \qquad \text{and} \qquad \frac{1}{q_\epsilon} = \frac{\theta}{2} + \frac{1-\theta}{4}.
		\end{align}
		By interpolation, we thus obtain
		\begin{equation}
			\label{eq:EstInterpolMaxFunctL4}
			\| e^{\imu t \Delta} f_\lambda \|_{L^{q_\epsilon, p_\epsilon}_{\vece}} \lesssim \| e^{\imu t \Delta} f_\lambda \|_{L^{2,\infty}_{\vece}}^\theta\| e^{\imu t \Delta} f_\lambda \|_{L^{4,4}_{\vece}}^{1-\theta}.
		\end{equation}
		Using Bernstein estimates and that $(4, \frac{2d}{d-1})$ is Schr{\"o}dinger admissible we infer
		\begin{align}
			\| e^{\imu t \Delta} f_\lambda \|_{L^{4,4}_{\vece}} = \| e^{\imu t \Delta} f_\lambda \|_{L^4_t L^4_x} \lesssim \lambda^{\frac{d-2}{4}} \| e^{\imu t \Delta} f_\lambda \|_{L^4_t L^{\frac{2d}{d-1}}_x} \lesssim \lambda^{\frac{d-2}{4}} \| f_\lambda \|_{L^2_x}. \label{eq:EstHomL44}
		\end{align}
		Combining this estimate and Lemma~\ref{lem:MaximalFunctionEstimate} with~\eqref{eq:EstInterpolMaxFunctL4}, we arrive at
		\begin{align*}
			\| e^{\imu t \Delta} f_\lambda \|_{L^{q_\epsilon, p_\epsilon}_{\vece}} \lesssim \lambda^{\frac{d-1}{2} (1 - \delta)} \lambda^{\frac{d-2}{4} \delta} \|f_\lambda\|_{L^2_x} \lesssim \lambda^{\frac{d-1}{2} - \frac{d}{4} \delta}  \|f_\lambda\|_{L^2_x} \lesssim \lambda^{\frac{d-1}{2} - \epsilon}  \|f_\lambda\|_{L^2_x},
		\end{align*}
		thereby proving \eqref{eq:EstMaxFctHomInterpol}.
		
		To prove~\eqref{eq:EstMaxFctInhomInterpol}, we first note that~\eqref{eq:EstMaxFctHomInterpol} and the dual endpoint Strichartz estimate (cf. Lemma~\ref{lem:StrichartzLocalSmooth}) imply that for every time interval $J \subseteq \R$
		\begin{align}
			\Big\| \int_J e^{\imu(t-s) \Delta} g_\lambda(s) \dd s \Big\|_{L^{q_\epsilon, p_\epsilon}_{\vece}} \lesssim \lambda^{\frac{d-1}{2} - \epsilon} \Big\| \int_J e^{-\imu s \Delta} g_\lambda(s) \dd s \Big\|_{L^2_x} \lesssim  \lambda^{\frac{d-1}{2} - \epsilon} \| g_\lambda \|_{L^2(J; L_x^{2_*})}. \label{eq:EstMaxFctIntJ}
		\end{align}
		We next apply a variant of the Christ-Kiselev lemma. We first note that due to~\eqref{eq:EstMaxFctIntJ}, showing~\eqref{eq:EstMaxFctInhomInterpol} is equivalent to proving
		\begin{equation}
			\label{eq:EstMaxFctMinusInfty}
			\Big\| \int_{-\infty}^t e^{\imu(t-s) \Delta} g_\lambda(s) \dd s \Big\|_{L^{q_\epsilon, p_\epsilon}_{\vece}} \lesssim \lambda^{\frac{d-1}{2} - \epsilon} \| g_\lambda \|_{L^2_t L^{2_*}_x}.
		\end{equation} 
		We can assume in the following that $g_\lambda(t) \neq 0$ for almost all $t \in \R$.  
		Actually, 
		given $g_\lambda \in L^2_t L^{2_*}_x(\R \times \R^d)$, choose a function $h \colon \R \rightarrow L_x^{2_*}(\R^d)$ such that $\|h(t)\|_{L_x^{2_*}} = 1$ and $h(t) \notin \operatorname{span}\{g_\lambda(t)\}$ for every $t \in \R$. If we prove~\eqref{eq:EstMaxFctMinusInfty} with $g_\lambda$ replaced by $\tilde{g}(t) := g_\lambda(t) + \nu e^{- t^{2}}h(t)$ for every $\nu > 0$, then \eqref{eq:EstMaxFctMinusInfty} follows by taking the limit $\nu \rightarrow 0$. 	After normalization, we can further assume that $\lambda^{\frac{d-1}{2} - \epsilon} \| g_\lambda \|_{L^2_t L^{2_*}_x} = 1$.  Define
		\begin{align*}
			G \colon \R \rightarrow [0,1], \qquad G(t) = \int_{-\infty}^t \lambda^{d-1 - 2 \epsilon} \|g_\lambda(s)\|_{L^{2_*}_x}^2 \dd s.
		\end{align*}
		Then $G$ is continuous and strictly monotonically increasing.
		
		\vspace*{4pt plus 2pt minus 2pt}%
		Let us now introduce a Whitney decomposition. We partition $[0,1]$ into dyadic intervals $[(j-1) 2^{-n}, j 2^{-n}]$ for $j \in \{1, \ldots, 2^n\}$ for every $n \in \N$ and define a relationship $I \sim J$ on these intervals in the following way: We say that $I \sim J$ if and only if $I$ and $J$ have the same size, are not adjacent but have adjacent parents, and $\sup J < \inf I$. It is then easy to check that for almost all $x,y \in [0,1]$ with $x < y$ there is exactly one pair of dyadic intervals $I,J$ such that $x \in J$, $y \in I$, and $I \sim J$. Consequently, we have
		\begin{align*}
			\chi_{x < y}(x,y) = \sum_{I,J \colon I \sim J} \chi_I(y) \chi_J(x)
		\end{align*}
		for almost all $x,y \in [0,1]$. Applying this identity with $x = G(s)$ and $y = G(t)$ and using that $G$ is strictly monotonically increasing, we obtain
		\begin{equation}
			\label{eq:WhitneyG}
			\chi_{G(s) < G(t)}(s,t) = \sum_{I,J \colon I \sim J} \chi_{G^{-1}(I)}(t) \chi_{G^{-1}(J)}(s)
		\end{equation}
		for almost all $s,t \in \R^2$.
		Since $G$ is strictly  increasing, $s < t$ if and only if $G(s) < G(t)$. Hence,
		\begin{align*}
			\int_{-\infty}^t  e^{\imu(t-s) \Delta} g_\lambda(s) \dd s = \int_\R \chi_{G(s) < G(t)}(s,t)  e^{\imu(t-s) \Delta} g_\lambda(s) \dd s 
			= \sum_{I,J \colon I \sim J} \chi_{G^{-1}(I)}(t) \int_{G^{-1}(J)} e^{\imu(t-s) \Delta} g_\lambda(s) \dd s
		\end{align*}
		for almost all $t$, where we employed~\eqref{eq:WhitneyG} in the last step. 
		
		For fixed $n$, all the dyadic intervals $I$ with $|I| = 2^{-n}$ have disjoint interior so that also the intervals $G^{-1}(I)$ have disjoint interior. Writing $\sum_{I,J \colon I \sim J} = \sum_{n \in \N} \sum_{I,J \colon I \sim J, |I| = 2^{-n}}$ and using the above identity we infer
		\begin{align*}
			&\Big\| \int_{-\infty}^t  e^{\imu(t-s) \Delta} g_\lambda(s) \dd s \Big\|_{L^{q_\epsilon, p_\epsilon}_{\vece}}
			\leq \sum_{n \in \N} \Big\| \sum_{I,J \colon I \sim J, |I| = 2^{-n}} \chi_{G^{-1}(I)}(t) \int_{G^{-1}(J)}  e^{\imu(t-s) \Delta} g_\lambda(s) \dd s \Big\|_{L^{q_\epsilon, p_\epsilon}_{\vece}} \\
			&= \sum_{n \in \N} \Big( \int_\R \Big(\sum_{I,J \colon I \sim J, |I| = 2^{-n}} \int_{G^{-1}(I)} \int_{\cP_\vece} \Big| \int_{G^{-1}(J)} e^{\imu(t-s) \Delta} g_\lambda(s) \dd s \Big|^{p_\epsilon} \dd y \dd t \Big)^{\frac{q_\epsilon}{p_\epsilon}} \dd r \Big)^{\frac{1}{q_\epsilon}} \\
			&= \sum_{n \in \N} \Big( \int_\R \Big(\sum_{I,J \colon I \sim J, |I| = 2^{-n}} \Big\| \int_{G^{-1}(J)} e^{\imu(t-s) \Delta} g_\lambda(s) \dd s \Big\|_{L^{p_\epsilon}_{t,y}(G^{-1}(I) \times \cP_\vece)}^{p_\epsilon}\Big)^{\frac{q_\epsilon}{p_\epsilon}} \dd r \Big)^{\frac{1}{q_\epsilon}},
		\end{align*}
		where $\dd y$ denotes the induced Euclidean measure on the hyperplane $\cP_\vece$ perpendicular to $\vece$ and $\dd r$ denotes integration in direction $\vece$, cf.~\eqref{Lepq-def}. Since $q_\varepsilon< p_\varepsilon$, using the embedding $l^{q_\epsilon} \hookrightarrow l^{p_\epsilon}$, we thus deduce
		\begin{align*}
			\Big\| \int_{-\infty}^t  e^{\imu(t-s) \Delta} g_\lambda(s) \dd s \Big\|_{L^{q_\epsilon, p_\epsilon}_{\vece}}
			&\leq \sum_{n \in \N} \Big( \int_\R \sum_{I,J \colon I \sim J, |I| = 2^{-n}} \Big\| \int_{G^{-1}(J)} e^{\imu(t-s) \Delta} g_\lambda(s) \dd s \Big\|_{L^{p_\epsilon}_{t,y}(G^{-1}(I) \times \cP_\vece)}^{q_\epsilon} \dd r \Big)^{\frac{1}{q_\epsilon}} \\
			&\leq \sum_{n \in \N} \Big(\sum_{I,J \colon I \sim J, |I| = 2^{-n}} \Big\| \int_{G^{-1}(J)} e^{\imu(t-s) \Delta} g_\lambda(s) \dd s \Big\|_{L^{q_\epsilon, p_\epsilon}_\vece}^{q_\epsilon} \Big)^{\frac{1}{q_\epsilon}}.
		\end{align*}
		For every dyadic interval $J$ with $|J| = 2^{-n}$ there are at most two dyadic intervals $I$ with $I \sim J$. Combining this with estimate~\eqref{eq:EstMaxFctIntJ}, we obtain
		\begin{align*}
			\Big\| \int_{-\infty}^t  e^{\imu(t-s) \Delta} g_\lambda(s) \dd s \Big\|_{L^{q_\epsilon, p_\epsilon}_{\vece}} 
			\lesssim \sum_{n \in \N} \Big(\sum_{J \colon |J| = 2^{-n}} (\lambda^{\frac{d-1}{2} - \epsilon} \| g_\lambda \|_{L^2(G^{-1}(J); L_x^{2_*})})^{q_\epsilon}\Big)^{\frac{1}{q_\epsilon}}.
		\end{align*}
		Noting that by construction
		\begin{align*}
			\lambda^{\frac{d-1}{2} - \epsilon} \| g_\lambda \|_{L^2_t L^{2_*}_x(G^{-1}(J) \times \R^d)} = 2^{-\frac{n}{2}}
		\end{align*}
		for every dyadic interval $J$ with $|J| = 2^{-n}$, we finally arrive at
		\begin{align*}
			\Big\| \int_{-\infty}^t  e^{\imu(t-s) \Delta} g_\lambda(s) \dd s \Big\|_{L^{q_\epsilon, p_\epsilon}_{\vece}} 
			\lesssim \sum_{n \in \N} \Big( 2^n 2^{-\frac{n}{2} q_\epsilon}\Big)^{\frac{1}{q_\epsilon}} \lesssim \sum_{n \in \N} 2^{-n(\frac{1}{2} - \frac{1}{q_\epsilon})} \lesssim 1,
		\end{align*}
		as $q_\epsilon > 2$. This concludes the proof of~\eqref{eq:EstMaxFctInhomInterpol}.
	\end{proof}
	
	As a consequence, 
	we obtain the following estimate for the Schr{\"o}dinger flow in the $\| \cdot \|_{\Z^{s,b}_\epsilon}$-norm.
	
	\begin{lemma}[Compatibility 
		between the maximal framework 
		and adapted spaces]
		\label{lem:LinSchroedingerZnorm}
		Let $0 < \epsilon \ll 1$ and $t_0\in \R$. Then for any $s\in \R$ and $0\leq b\leq 1$ we have
		\begin{align*}
			&\| e^{\imu t \Delta} f \|_{\Z^{s,b}_\epsilon} \lesssim \| f \|_{H^s_x} \qquad \text{and} \qquad \Big\| \int_{t_0}^t e^{\imu (t-t') \Delta} g(t') \dd t' \Big\|_{\Z^{s,b}_\epsilon} \lesssim \| g \|_{N^{s,0,b}}.
		\end{align*}
	\end{lemma}
	\begin{proof}
		We start with the homogeneous estimate. Lemma~\ref{lem:LinFlowAdaptedSpaces} and Lemma~\ref{lem:MaxFctDualEndptStrichartz} already show that
		\begin{align*}
			\| e^{\imu t \Delta} f_\lambda \|_{S^{s,0,b}_\lambda} \lesssim \lambda^s \| f_\lambda\|_{L^2_x} \quad (\lambda \geq 1) \qquad \text{and} \qquad  \sum_{j = 1}^d \lambda^{s -\frac{d-1}{2} + \epsilon} \| e^{\imu t \Delta} f_\lambda \|_{L^{q_\epsilon, p_\epsilon}_{\vece_j}} \lesssim \lambda^s \| f_\lambda\|_{L^2_x} \quad (\lambda > 1).
		\end{align*} 
		Hence, it suffices to prove 
		for the local smoothing component that 
		\begin{align} \label{max-L2-homo}
			\sum_{j = 1}^d \lambda^{s + \frac{1}{2} - \epsilon} \| P_{\lambda, \vece_j} C_{\leq (\frac{\lambda}{2^8})^2} e^{\imu t \Delta} f_\lambda \|_{L^{p_\epsilon, q_\epsilon}_{\vece_j}} \lesssim \lambda^s \|f_\lambda\|_{L^2_x}.
		\end{align}
		To this end, 
		we argue similarly as in the proof of~\eqref{eq:EstMaxFctHomInterpol} by interpolation with $L^{4,4}_\vece$. Let $\lambda \in 2^\N$. Recalling the identities~\eqref{eq:ParameterInterpol} with $\theta = 1 - \delta$, interpolation yields
		\begin{align}
			\| P_{\lambda, \vece_j} C_{\leq (\frac{\lambda}{2^8})^2} e^{\imu t \Delta} f_\lambda \|_{L^{p_\epsilon, q_\epsilon}_{\vece_j}} \lesssim \| P_{\lambda, \vece_j} C_{\leq (\frac{\lambda}{2^8})^2} e^{\imu t \Delta} f_\lambda \|_{L^{\infty, 2}_{\vece_j}}^\theta \| P_{\lambda, \vece_j} C_{\leq (\frac{\lambda}{2^8})^2} e^{\imu t \Delta} f_\lambda \|_{L^{4, 4}_{\vece_j}}^{1-\theta} \label{eq:EstInterpolLocalSmoothing}
		\end{align}
		for every $j \in \{1, \ldots, d\}$. Combining Lemma~\ref{lem:StrichartzLocalSmooth} and~\eqref{eq:EstHomL44}, we infer
		\begin{align*}
			\| P_{\lambda, \vece_j} C_{\leq (\frac{\lambda}{2^8})^2} e^{\imu t \Delta} f_\lambda \|_{L^{p_\epsilon, q_\epsilon}_{\vece_j}} \lesssim \lambda^{-\frac{1}{2} (1 - \delta)} \lambda^{\frac{d-2}{4} \delta} \| f_\lambda\|_{L^2_x} \lesssim \lambda^{-\frac{1}{2} + \frac{d}{4} \delta}  \| f_\lambda\|_{L^2_x} \lesssim \lambda^{-\frac{1}{2} + \epsilon} \| f_\lambda\|_{L^2_x}, 
		\end{align*} 
		which implies \eqref{max-L2-homo}, 
		thereby yielding the homogeneous estimate in the  $\Z^{s,b}_\epsilon$-norm.

		We now turn to the estimate of the Duhamel integral $\cI_0[g]$. For $\lambda \in 2^{\N_0}$ and the adapted space component we again apply Lemma~\ref{lem:LinFlowAdaptedSpaces} which yields
		\begin{align}
			\| \cI_0[g_\lambda] \|_{S^{s,0,b}} \lesssim \| g_\lambda\|_{N^{s,0,b}}. \label{eq:EstDuhamelZnormAdpatedComp}
		\end{align}
		Let $\lambda > 1$ in the following. For the local smoothing component, we proceed as in~\eqref{eq:EstInterpolLocalSmoothing}  to deduce
		\begin{align}
			\| P_{\lambda, \vece_j} C_{\leq (\frac{\lambda}{2^8})^2} \cI_0[g_\lambda] \|_{L^{p_\epsilon, q_\epsilon}_{\vece_j}} 
			&\lesssim \| P_{\lambda, \vece_j} C_{\leq (\frac{\lambda}{2^8})^2} \cI_0[g_\lambda] \|_{L^{\infty, 2}_{\vece_j}}^\theta \| P_{\lambda, \vece_j} C_{\leq (\frac{\lambda}{2^8})^2} \cI_0[g_\lambda] \|_{L^{4, 4}_{\vece_j}}^{1-\theta} \nonumber \\
			&\lesssim \| P_{\lambda, \vece_j} C_{\leq (\frac{\lambda}{2^8})^2} \cI_0[g_\lambda] \|_{L^{\infty, 2}_{\vece_j}}^\theta (\lambda^{\frac{d-2}{4}} \|C_{\leq (\frac{\lambda}{2^8})^2} \cI_0[g_\lambda] \|_{L^4_t L^{\frac{2d}{d-1}}_x})^{1 - \theta},
		\end{align}
		where we also used Bernstein estimates as in~\eqref{eq:EstInterpolMaxFunctL4} in the last step. As $(4, \frac{2d}{d-1})$ is Schr{\"o}dinger admissible, interpolation between the admissible pairs $(\infty,2)$ and $(2, 2^*)$ shows that the above right-hand side is controlled by the $\X^{s,0}$-norm, where we also used~\eqref{eq:CharSsablambda}. More precisely, we have
		\begin{align*}
			\| P_{\lambda, \vece_j} C_{\leq (\frac{\lambda}{2^8})^2} \cI_0[g_\lambda] \|_{L^{p_\epsilon, q_\epsilon}_{\vece_j}} 
			&\lesssim \lambda^{-(s + \frac{1}{2}) \theta} \| \cI_0[g_\lambda] \|_{\X^{s,0}_\lambda}^\theta \lambda^{\frac{d-2}{4}(1-\theta)} \lambda^{-s (1-\theta)} \| \cI_0[g_\lambda]\|_{\X^{s,0}_\lambda}^{1-\theta} 
			\lesssim \lambda^{-s - \frac{1}{2} + \epsilon } \| \cI_0[g_\lambda] \|_{\X^{s,0}_\lambda},
		\end{align*}
		i.e.,
		\begin{equation}
			\label{eq:EstDuhamelZnormLocalSmoothingComp}
			\sum_{j = 1}^d \lambda^{s + \frac{1}{2}  - \epsilon} \| P_{\lambda, \vece_j} C_{\leq (\frac{\lambda}{2^8})^2} \cI_0[g_\lambda] \|_{L^{p_\epsilon, q_\epsilon}_{\vece_j}} \lesssim \| \cI_0[g_\lambda] \|_{\X^{s,0}_\lambda}.
		\end{equation}
		
		It remains to estimate the maximal function component of the $\Z^{s,b}_\epsilon$-norm. To that purpose, we first split into low and high modulation  of $g_\lambda$
		\begin{align*}
			\cI_0[g_\lambda] = \cI_0[C_{\leq (\frac{\lambda}{2^8})^2} g_\lambda] + \cI_0[C_{> (\frac{\lambda}{2^8})^2}g_\lambda].
		\end{align*}
		For the low modulation contribution, an application of Lemma~\ref{lem:MaxFctDualEndptStrichartz} directly yields
		\begin{align}
			\lambda^{s - \frac{d-1}{2} + \epsilon} \| \cI_0[C_{\leq (\frac{\lambda}{2^8})^2} g_\lambda] \|_{L^{q_\epsilon, p_\epsilon}_{\vece_j}} \lesssim \lambda^s \| C_{\leq (\frac{\lambda}{2^8})^2} g_\lambda \|_{L^2_t L^{2_*}_x} \lesssim \| g_\lambda \|_{N^{s,0,b}_\lambda}. \label{eq:EstMaxFctLowMod}
		\end{align}
		For the high modulation, we further split
		\begin{align*}
			\cI_0[C_{> (\frac{\lambda}{2^8})^2}g_\lambda] = C_{\leq (\frac{\lambda}{2^9})^2}\cI_0[C_{> (\frac{\lambda}{2^8})^2}g_\lambda]  + C_{> (\frac{\lambda}{2^9})^2} \cI_0[C_{> (\frac{\lambda}{2^8})^2}g_\lambda]
		\end{align*}
		and
		\begin{align*}
			C_{> (\frac{\lambda}{2^9})^2} \cI_0[C_{> (\frac{\lambda}{2^8})^2}g_\lambda] = P^{(t)}_{\lesssim \lambda^2}  C_{> (\frac{\lambda}{2^9})^2} \cI_0[C_{> (\frac{\lambda}{2^8})^2}g_\lambda]
			+ \sum_{\nu \gg \lambda^2} P^{(t)}_\nu  C_{> (\frac{\lambda}{2^9})^2} \cI_0[C_{> (\frac{\lambda}{2^8})^2}g_\lambda].
		\end{align*}
		Using the Bernstein estimates we derive
		\begin{align*}
			&\| P^{(t)}_{\lesssim \lambda^2}  C_{> (\frac{\lambda}{2^9})^2} \cI_0[C_{> (\frac{\lambda}{2^8})^2}g_\lambda] \|_{L^{q_\epsilon, p_\epsilon}_{\vece_j}} \lesssim \lambda^{\frac12 - \frac{1}{q_\epsilon}} \lambda^{2(\frac{1}{2} - \frac{1}{p_\epsilon})} \lambda^{(d-1)(\frac{1}{2} - \frac{1}{p_\epsilon})} \| P^{(t)}_{\lesssim \lambda^2}  C_{> (\frac{\lambda}{2^9})^2} \cI_0[C_{> (\frac{\lambda}{2^8})^2}g_\lambda] \|_{L^{2,2}_{\vece_j}} \\
			&\lesssim \lambda^{\frac{\delta}{4} + 1 - \frac{\delta}{2} + \frac{d-1}{2} - (d-1)\frac{\delta}{4}} \| P^{(t)}_{\lesssim \lambda^2}  C_{> (\frac{\lambda}{2^9})^2} \cI_0[C_{> (\frac{\lambda}{2^8})^2}g_\lambda] \|_{L^2_t L^2_x}
			\lesssim \lambda^{\frac{d-1}{2} - \frac{d}{4} \delta + 1} \lambda^{-2} \| C_{> (\frac{\lambda}{2^8})^2}g_\lambda \|_{L^2_t L^2_x} \\
			&\lesssim \lambda^{\frac{d-1}{2} - \epsilon - 1} \| g_\lambda \|_{L^2_t L^2_x},
		\end{align*}
		and, similarly, 
		\begin{align*}
			&\sum_{\nu \gg \lambda^2} \| P^{(t)}_{\nu}  C_{> (\frac{\lambda}{2^9})^2} \cI_0[C_{> (\frac{\lambda}{2^8})^2}g_\lambda] \|_{L^{q_\epsilon, p_\epsilon}_{\vece_j}} 
			\lesssim \sum_{\nu \gg \lambda^2} \lambda^{\frac12 - \frac{1}{q_\epsilon}} \nu^{\frac{1}{2} - \frac{1}{p_\epsilon}} \lambda^{(d-1)(\frac{1}{2} - \frac{1}{p_\epsilon})} \| P^{(t)}_{\nu}  C_{> (\frac{\lambda}{2^9})^2} \cI_0[C_{> (\frac{\lambda}{2^8})^2}g_\lambda] \|_{L^{2,2}_{\vece_j}} \\
			&\lesssim \lambda^{\frac{\delta}{4} + \frac{d-1}{2} - \frac{d-1}{4} \delta} \sum_{\nu \gg \lambda^2} \nu^{\frac{1}{2} - \frac{\delta}{4}} \nu^{-1} \| C_{> (\frac{\lambda}{2^8})^2}g_\lambda \|_{L^2_t L^2_x}
			\lesssim \lambda^{\frac{d-1}{2} - \frac{d}{4} \delta + \frac{\delta}{2}} \lambda^{2(-\frac12 - \frac{\delta}{4})} \| g_\lambda \|_{L^2_t L^2_x} \lesssim \lambda^{\frac{d-1}{2} - \epsilon - 1} \| g_\lambda \|_{L^2_t L^2_x}.
		\end{align*}
		Combining the last two estimates we arrive at 
		\begin{align}
			\sum_{j = 1}^d \lambda^{s - \frac{d-1}{2} + \epsilon} \| C_{> (\frac{\lambda}{2^9})^2} \cI_0[C_{> (\frac{\lambda}{2^8})^2}g_\lambda] \|_{L^{q_\epsilon, p_\epsilon}_{\vece_j}} \lesssim \lambda^{s-1} \|g_\lambda \|_{L^2_t L^2_x} \lesssim \| g_\lambda \|_{N^{s,0,b}_{\lambda}}. \label{eq:EstMaxFctCompHighHighMod}
		\end{align}
		To treat the remaining contribution, we first claim that
		\begin{align}
			\| C_{\leq (\frac{\lambda}{2^9})^2} \cI_0[C_{> (\frac{\lambda}{2^8})^2} g_\lambda] \|_{L^{q_\epsilon, p_\epsilon}_{\vece_j}} &\lesssim \lambda^{\frac{d-1}{2} - \epsilon - 2} \|  g_\lambda \|_{L^\infty_t L^2_x}, \label{eq:EstMaxFctLoHighModLowTemp}\\
			\| C_{\leq (\frac{\lambda}{2^9})^2} \cI_0[C_{\sim \nu} C_{> (\frac{\lambda}{2^8})^2} g_\lambda] \|_{L^{q_\epsilon, p_\epsilon}_{\vece_j}} &\lesssim \lambda^{\frac{d-1}{2} - \epsilon} \nu^{-1} \|g_\lambda \|_{L^\infty_t L^2_x}, \label{eq:EstMaxFctLoHighModHiTemp}
		\end{align}
		for any $\nu \gg \lambda$. Employing these two estimates and Bernstein estimates, we infer
		\begin{align*}
			&\| C_{\leq (\frac{\lambda}{2^9})^2} \cI_0[C_{> (\frac{\lambda}{2^8})^2} g_\lambda] \|_{L^{q_\epsilon, p_\epsilon}_{\vece_j}} 
			\lesssim \| C_{\leq (\frac{\lambda}{2^9})^2} \cI_0[C_{> (\frac{\lambda}{2^8})^2} P^{(t)}_{\lesssim \lambda^2} g_\lambda] \|_{L^{q_\epsilon, p_\epsilon}_{\vece_j}} + \sum_{\nu \gg \lambda^2} \| C_{\leq (\frac{\lambda}{2^9})^2} \cI_0[C_{> (\frac{\lambda}{2^8})^2} P^{(t)}_\nu g_\lambda] \|_{L^{q_\epsilon, p_\epsilon}_{\vece_j}} \\
			&\lesssim \lambda^{\frac{d-1}{2} - \epsilon - 2} \| P^{(t)}_{\lesssim \lambda^2} g_\lambda \|_{L^\infty_t L^2_x} + \sum_{\nu \gg \lambda^2} \lambda^{\frac{d-1}{2} - \epsilon} \nu^{-1} \| P^{(t)}_\nu g_\lambda \|_{L^\infty_t L^2_x} \\
			&\lesssim \lambda^{\frac{d-1}{2} - \epsilon - 1}  \| P^{(t)}_{\lesssim \lambda^2} g_\lambda \|_{L^2_t L^2_x} + \sum_{\nu \gg \lambda^2} \lambda^{\frac{d-1}{2} - \epsilon} \nu^{-\frac{1}{2}} \| P^{(t)}_\nu g_\lambda \|_{L^2_t L^2_x}
			\lesssim \lambda^{\frac{d-1}{2} - \epsilon - 1} \| g_\lambda \|_{L^2_t L^2_x}.
		\end{align*}
		We conclude that
		\begin{align}
			\sum_{j = 1}^d \lambda^{s - \frac{d-1}{2} + \epsilon} \| C_{\leq (\frac{\lambda}{2^9})^2} \cI_0[C_{> (\frac{\lambda}{2^8})^2}g_\lambda] \|_{L^{q_\epsilon, p_\epsilon}_{\vece_j}} \lesssim \lambda^{s-1} \|g_\lambda \|_{L^2_t L^2_x} \lesssim \| g_\lambda \|_{N^{s,0,b}_{\lambda}}.
			\label{eq:EstMaxFctCompLowHighMod}
		\end{align}
		Combining~\eqref{eq:EstDuhamelZnormAdpatedComp}, \eqref{eq:EstDuhamelZnormLocalSmoothingComp}-\eqref{eq:EstMaxFctCompHighHighMod},  and~\eqref{eq:EstMaxFctCompLowHighMod}, taking the square sum  and applying Lemma~\ref{lem:LinEstimates}, we arrive at
		\begin{align*}
			\|\cI_0[g]\|_{\Z^{s,b}_\epsilon} \lesssim \| g \|_{N^{s,0,b}},
		\end{align*}
		proving the second assertion of the lemma.
		
		It only remains to show \eqref{eq:EstMaxFctLoHighModLowTemp} and~\eqref{eq:EstMaxFctLoHighModHiTemp}. Using the commutation relation
		\begin{align*}
			e^{- \imu (\cdot) \Delta} C_{> (\frac{\lambda}{2^8})^2} = P^{(t)}_{> (\frac{\lambda}{2^8})^2} e^{-\imu (\cdot) \Delta},
		\end{align*}
		we write
		\begin{align*}
			C_{\leq (\frac{\lambda}{2^9})^2} \int_{t_0}^t e^{\imu (t - t') \Delta} C_{>(\frac{\lambda}{2^8})^2} g_\lambda \dd t' 
			= e^{\imu t \Delta} P^{(t)}_{\leq (\frac{\lambda}{2^9})^2} \int_{t_0}^t \partial_{t'} \partial_{t'}^{-1} P^{(t')}_{> (\frac{\lambda}{2^8})^2} (e^{-\imu t' \Delta} g_\lambda) \dd t' 
			= e^{\imu t \Delta} P^{(t)}_{\leq (\frac{\lambda}{2^9})^2} (H(t) - H(t_0)),
		\end{align*}
		where $H(t) = \partial_t^{-1} P^{(t)}_{> (\frac{\lambda}{2^8})^2}(e^{-\imu (\cdot) \Delta} g_\lambda)$. Since
		\begin{align*}
			P^{(t)}_{\leq (\frac{\lambda}{2^9})^2} H(t) = \partial_t^{-1} P^{(t)}_{\leq (\frac{\lambda}{2^9})^2} P^{(t)}_{> (\frac{\lambda}{2^8})^2}(e^{-\imu (\cdot) \Delta} g_\lambda) = 0,
		\end{align*}
		the above Duhamel term reduces to the linear Schr{\"o}dinger evolution $-e^{\imu t \Delta} H(t_0)$. An application of the homogeneous maximal function estimate from Lemma~\ref{lem:MaxFctDualEndptStrichartz} yields
		\begin{align}
			\| C_{\leq (\frac{\lambda}{2^9})^2} \cI_0[C_{> (\frac{\lambda}{2^8})^2} g_\lambda] \|_{L^{q_\epsilon, p_\epsilon}_{\vece_j}}
			\lesssim \| e^{\imu t \Delta} H(t_0) \|_{L^{q_\epsilon, p_\epsilon}_{\vece_j}} \lesssim \lambda^{\frac{d-1}{2} - \epsilon} \|H(t_0)\|_{L^2_x} \lesssim \lambda^{\frac{d-1}{2} - \epsilon} \|H\|_{L^\infty_t L^2_x}. \label{eq:EstDuhamelMaxFctLowHighModH}
		\end{align}
		Finally, a computation yields the estimate
		\begin{align*}
			\| H \|_{L^\infty_t L^2_x} \lesssim \lambda^{-2} \| e^{- \imu t  \Delta} C_{>(\frac{\lambda}{2^8})^2} g_\lambda \|_{L^\infty_t L^2_x} \lesssim \lambda^{-2} \| g_\lambda \|_{L^\infty_t L^2_x}.
		\end{align*}
		In combination with~\eqref{eq:EstDuhamelMaxFctLowHighModH}, we conclude~\eqref{eq:EstMaxFctLoHighModLowTemp}. Estimate~\eqref{eq:EstMaxFctLoHighModHiTemp} follows analogously.
	\end{proof}

	We end this section with the following 
	estimates in the $\Z^{s,0}_\epsilon$-spaces
	for the linear Schr\"odinger flow with free wave potential.  
	Recall that $\cU_{v_L}[f]$ and $\mathcal{I}_{v_L}[F]$ denote the homogeneous and inhomogeneous solution operators, 
	respectively, to the linear Schr\"odinger equation with potential $v_L:= e^{\imu t|\nabla|} Y_0$. 
	
	\begin{lemma}[Linear Schr\"odinger with free wave potential II]
		\label{lem:LinSchrPotSmallTimeLocSmo}
		Let $0\leq s\leq l+2$ and $l\geq \frac{d-4}{2}$ with $(s,l)\neq (\frac{d}{2},\frac{d-4}{2})$.
		Then, for any interval $I \subset \R$, $t_0 \in I$, $f \in H_x^s$, and $F\in L_t^2 L_x^{2_*}$ with $F \in N^{s,0,0}(I)$, we have
		\begin{align*}
			\|\cU_{v_L}[f]\|_{\Z_\varepsilon^{s,0}(I)} \lesssim \|f\|_{H^s_x}, \qquad \|\cI_{v_L}[F]\|_{\Z_\varepsilon^{s,0}(I)} \lesssim \|F\|_{N^{s,0,0}(I)},
		\end{align*}
		and
		\begin{align*}
			\|u\|_{L^2 (I ; L_x^{2^*})} \lesssim  \|f\|_{L_x^2} + \|F\|_{L^2 (I ; L_x^{2_*})},
		\end{align*}
		where the implicit constants depend on $v_L$.
	\end{lemma}
	\begin{proof}
		The proof proceeds analogously to the argument in Remark~\ref{rem:PropagOPWellDef}. It relies on decomposing into time intervals where the wave potential is small (Lemma~\ref{lem:SmallnessWavePotential}) and a small free wave potential estimate corresponding to Lemma~\ref{lem:LinSchrPotSmallTime}. This estimate is established by following the proof of Lemma~\ref{lem:LinSchrPotSmallTime}, but replacing the $\X^{s,a}$-norm with the $\Z^{s,0}_{\epsilon}$-norm and invoking Lemma~\ref{lem:LinSchroedingerZnorm} instead of Lemma~\ref{lem:LinEstimates}.
	\end{proof}

	\section{Global control of geometric Brownian motions in Besov spaces}\label{Sec:GRCGBM}
	In this section, we establish the  global-in-time control of the geometric Brownian motion 
	(GBM for short)
	\begin{align}  \label{h-gBM} 
		h_c(t):= e^{-2c\beta(t)-2c^2t} 
	\end{align} 
	in Besov spaces, where $\beta$ is a one dimensional Brownian motion.
	
	We start with proving the global-in-time $C^{\frac{1}{2}-}(\R_+)$-H\"older regularity of GBM. With a slight abuse of notation, for any time interval $I$, we do not interpret $\| \cdot \|_{C^\alpha(I)}$-norm as a restriction norm but set
	\begin{align}
		\label{eq:DefHoelderNorm}
		\| g \|_{C^\alpha(I)} := \| g \|_{L^\infty(I)} + \|g \|_{\dot{C}^\alpha(I)}, \qquad \|g\|_{\dot{C}^\alpha(I)}:= \sup_{\substack{s\neq t\\ s,t\in I}} \frac{|g(t)-g(s)|}{|t-s|^\alpha}.
	\end{align}
	
	\begin{lemma}[Global-in-time $C^{\frac12-}$-regularity of GBM]\label{prop:GeoBMinHolder}
		Let $c \in \R$. For any $0<\alpha<\frac12$,  we have $h_c \in C^\alpha(\R_+)$ with probability one.
	\end{lemma}
	\begin{proof}
		Without loss of generality we assume that $c=1$ and write $h = h_1$. From the law of the iterated logarithm of Brownian motion
		\begin{align}\label{prop:lawitlog}
			\liminf_{t\to\infty} \frac{\beta(t)}{\sqrt{2t\log\log t}}=-1,\qquad \limsup_{t\to\infty} \frac{\beta(t)}{\sqrt{2t\log\log t}}=1, \quad \PP\text{-}a.s.,
		\end{align}
		and the continuity of its paths, we infer $\|h\|_{L^\infty(\R_+)}<\infty$ $\PP$-a.s. Hence, it suffices to show that the homogeneous H\"older norm satisfies
		\begin{align}\label{prop:hCa}
			\|h\|_{\dot{C}^\alpha(\R_+)}<\infty, \quad \PP\text{-}a.s.
		\end{align}
		
		For this purpose,
		we first note that in the proof of  Proposition $7.3$ in \cite{HRSZ24} it was established
		that there exists a full probability set $\Omega'$ such that for any $\omega\in \Omega'$ there is $T=T(\omega)>2$,
		such that for any $t\geq T-2$
		and for any $n\geq T-2$,
		\begin{align}   \label{h-Can-expdecay}
			\|\beta(\cdot,\omega)\|_{\dot{C}^\alpha([n,n+1])}\leq e^{\frac{n}{16}} ,
			\quad |\beta(t,\omega)|\leq 2\sqrt{2 t \log(\log(t))} \leq \frac{t}{16}.
		\end{align}
		Below we fix such an $\omega \in \Omega'$ and the corresponding $T = T(\omega)$, and consider the inner temporal regime $[0,T]$
		and the outer temporal regime $[T-1,\infty)$ separately. The dependence on $\omega$ is omitted for the ease of notation.
		
		(i) The inner temporal regime $[0,T]$: First we note that by the mean value theorem,
		\begin{align}\label{prop:MeanValueGBM}
			|h(t)-h(s)|\leq 2e^{2\xi} (|\beta(t)-\beta(s)|+|t-s|)
		\end{align}
		for some $\xi$ between $-\beta(s)-s$ and $-\beta(t)-t$. Let $M_T:=\sup_{t\in [0,T]}|\beta(t)+t| $, which is finite due to the continuity of Brownian motions. We then have
		\begin{align*}
			|h(t)-h(s)|&\leq 2e^{2M_T}(|\beta(t)-\beta(s)|+|t-s|)\\
			&\leq 2e^{2M_T}(|\beta(t)-\beta(s)|+ (2T)^{1-\alpha} |t-s|^\alpha).
		\end{align*}
		It follows that
		\begin{align}  \label{h-Ca-T}
			\sup_{\substack{s\neq t\\ s,t\in [0,T]}} \frac{|h(t)-h(s)|}{|t-s|^\alpha}\leq 2e^{2M_T} (\|\beta\|_{\dot{C}^\alpha{([0,T])}} +(2T)^{1-\alpha})<\infty.
		\end{align}
		
		(ii) The outer temporal regime $[T-1,\infty)$: We will consider the two cases $|t-s|\geq 10^{-2}$ and $|t-s|< 10^{-2}$. 
		
		For the former case, using \eqref{h-Can-expdecay} we have
		\begin{align*}
			\sup_{\substack{|s-t|\geq 10^{-2}\\ s,t\geq T-1}} \frac{|h(t)-h(s)|}{|t-s|^\alpha} \leq 10^{2\alpha} \cdot 2 \sup_{t\geq T-1} |h(t)|\lesssim  e^{-\frac{T}{3}}<\infty.
		\end{align*}
		
		For the latter case, we first estimate
		\begin{align}\label{prop:GBMO1}
			\sup_{\substack{0<|s-t|< 10^{-2}\\ s,t\geq T-1}} \frac{|h(t)-h(s)|}{|t-s|^\alpha}\leq \sup_{\substack{[t]\leq s<t<s+10^{-2}\\ s,t\geq T-1}} \frac{|h(t)-h(s)|}{|t-s|^\alpha} + \sup_{\substack{s<[t]\leq t<s+10^{-2}\\ s,t\geq T-1}} \frac{|h(t)-h(s)|}{|t-s|^\alpha},
		\end{align}
		where $[t]$ denotes the floor function of $t\in \R$.
		Using \eqref{h-Can-expdecay} and \eqref{prop:MeanValueGBM}, we derive that for any $n\geq T-2$,
		\begin{align*}
			\sup_{\substack{s\neq t\\ s,t\in [n,n+1]}} \frac{|h(t)-h(s)|}{|t-s|^\alpha}
			&\leq
			\sup_{\substack{s\neq t\\ s,t\in [n,n+1]}} \frac{2e^{2\xi} (|\beta(t)-\beta(s)|+|t-s|)}{|t-s|^\alpha} \\
			& \leq 2e^{-\frac{15}{8}n} ( \|\beta(\cdot)\|_{\dot{C}^\alpha([n,n+1])}+1) \notag \\
			& \leq 2e^{-\frac{15}{8}n}(1+e^{\frac{n}{16}})\lesssim e^{-\frac{n}{3}}.
		\end{align*}
		Thus, the first term on the right-hand side of \eqref{prop:GBMO1} is bounded by
		\begin{align*}
			\sup_{\substack{[t]\leq s<t<s+10^{-2}\\ s,t\geq T-1}} \frac{|h(t)-h(s)|}{|t-s|^\alpha} \lesssim e^{-\frac{T}{3}}.
		\end{align*}
		The second term on the right-hand side of \eqref{prop:GBMO1} is bounded by
		\begin{align*}
			\sup_{\substack{s<[t]\leq t<s+10^{-2}\\ s,t\geq T-1}} \frac{|h(t)-h(s)|}{|t-s|^\alpha} &\leq \sup_{\substack{s<[t]\leq t<s+10^{-2}\\ s,t\geq T-1}} 2\frac{e^{2\xi} (|\beta(t)-\beta(s)|+|t-s|)}{|t-s|^\alpha}  \\
			&\leq \sup_{\substack{s<[t]\leq t<s+10^{-2}\\ s,t\geq T-1}} 2e^{2\xi}\Big( \frac{|\beta(t)-\beta([t])|}{|t-[t]|^\alpha}+ \frac{|\beta([t])-\beta(s)|}{|[t]-s|^\alpha}+10^{-2(1-\alpha)}\Big)\\
			&\lesssim \sup_{n \geq T-2} 2e^{-\frac{15}{8}n} (e^{\frac{n}{16}}+e^{\frac{n}{16}}+10^{-2(1-\alpha)}) \lesssim e^{-\frac{n}{3}} .
		\end{align*}
		
		Combining the above estimates, we obtain
		\begin{align*}
			\sup_{\substack{s\neq t\\ s,t\geq T-1}}\frac{|h(t)-h(s)|}{|t-s|^\alpha}
			&\leq  \sup_{\substack{0<|s-t|< 10^{-2}\\ s,t\geq T-1}}\frac{|h(t)-h(s)|}{|t-s|^\alpha}
			+\sup_{\substack{|s-t|\geq 10^{-2}\\ s,t \geq T-1}}\frac{|h(t)-h(s)|}{|t-s|^\alpha} \notag \\
			&
			\lesssim e^{-\frac{T}{3}}+ e^{-\frac{T}{3}}\lesssim e^{-\frac{T}{3}}<\infty.
		\end{align*}
		
		(iii) The crossing case $s< T < t$ : We directly have
		\begin{align*}
			\sup_{\substack{s\neq t\\ s< T < t}}\frac{|h(t)-h(s)|}{|t-s|^\alpha}
			& \leq \sup_{\substack{s\neq t\\ s< T < t}}   \frac{|h(T)-h(s)|}{|T-s|^\alpha} + \sup_{\substack{s\neq t\\ s< T < t}} \frac{|h(t)-h(T)|}{|t-T|^\alpha},
		\end{align*}
		where the first term is bounded by case (i) and the second term is bounded by case (ii). This completes the proof of the lemma.
	\end{proof}
	\begin{remark}
		\label{rem:HoelderEstUnb}
		We note that we actually proved in step (ii) of the above proof that for $\PP$-a.s. $\omega$ there is $C(\omega) > 0$ and $T_0(\omega) > 0$ such that
		\begin{equation}
			\label{prop:fastdecayh}
			\| h(\cdot, \omega) \|_{\dot{C}^\alpha([t,\infty))} = \sup_{\substack{s'\neq t'\\ s',t' \geq t}}\frac{|h(t',\omega)-h(s',\omega)|}{|t'-s'|^\alpha} \leq C(\omega) e^{-\frac{t}{3}}
		\end{equation}
		for all $t \geq T_0(\omega)$, where $\| h(\cdot, \omega) \|_{\dot{C}^\alpha([t,\infty))}$ is still interpreted as in~\eqref{eq:DefHoelderNorm} and we wrote $h$ for $h_1$.
	\end{remark}

	We next extend this regularity statement in the Besov scale to all integrability exponents $p$ in $[1,\infty]$. This \emph{global} control in Besov norms is crucial 
	to solve the scaling issue explained in the introduction 
	and to derive the trilinear estimates in Theorem~\ref{Prop-trilinear} below.

	Since we localize function spaces from the real line to time intervals via restriction (see~\eqref{eq:DefRestrictionNorm}), we need to work with appropriate extensions. We define for any $t_0 \geq 0$ and $c > 0$
	\begin{align}
		\label{eq:DefExtensionhc}
		\psi_{c,t_0} \colon C([0,\infty)) \rightarrow C(\R), \quad \psi_{c,t_0}[g](t) = \begin{cases}
			e^{-2g(t) - 2c^2 t}, \quad &t \geq t_0, \\
			e^{-2g(t_0) - 2c^2 t_0}(c^2(t - t_0) + 1), &-\frac{1}{c^2} + t_0 \leq t < t_0, \\
			0, &t < -\frac{1}{c^2} + t_0,
		\end{cases}
	\end{align}
	which is measurable when $C([0,\infty))$ and $C(\R)$ are equipped with their respective Borel-$\sigma$-algebras. We note that $\psi_{c,t_0}[c\beta]$ is an extension of $h_c$ from $[t_0,\infty)$ to $\R$. We show in the next lemma that this extension belongs to $B^{\alpha}_{p,\infty}(\R)$ $\PP$-a.s., which in particular implies
	that $h_c \in B^{\alpha}_{p,\infty}(I)$, $\PP$-a.s., 
	for any interval $I \subseteq[0,\infty)$, $\alpha \in (0,\frac12)$ and $p \in [1,\infty]$. 
	
	\begin{lemma}[Global-in-time Besov regularity of GBM]\label{le:GBMGloBesov}
		Let $c > 0$ and $t_0 \geq 0$. 
		Then, for any $0<\alpha<\frac12$ and $p\in[1,\infty]$,  we have $\psi_{c,t_0}[c \beta] \in B^{\alpha}_{p,\infty}(\R)$ with probability one.
	\end{lemma}
	
	\begin{proof}
		Without loss of generality we assume $c=1$ and write $h = \psi_{1,t_0}[c \beta]$ for the ease of notation. Clearly, the chosen extension yields $h \in C^{\alpha}(\R) = B^\alpha_{\infty,\infty}(\R)$ $\PP$-a.s. in view of Lemma~\ref{prop:GeoBMinHolder}. It remains to treat the case $p \in [1,\infty)$.
		
		Recall the definition of Littlewood-Paley projectors in the Notation part. By Remark~\ref{rem:HoelderEstUnb}, we can find a full measure set $\Omega'$ such that for any $\omega\in \Omega'$ there exists $T=T(\omega) = \max\{T_0(\omega), t_0\}$, such that for any $t\geq T$ we have the estimate~\eqref{prop:fastdecayh}. We fix such an $\omega \in \Omega'$ and drop it for the ease of notation.
		
		Let $\lambda\in 2^{\N_0}$. 
		Then, 
		$\|P_1 h\|_{L^p} = \|\cF^{-1}({\eta}_{\leq 1})\ast h\|_{L^p}\lesssim \|h\|_{L^p}<\infty$ $\PP$-a.s. Hence, it suffices to consider $\lambda>1$ in the following.
		We write
		\begin{align*}
			P_\lambda h(t) &= \int_{\R} \cF^{-1}(\eta_{\lambda})(t') h(t-t') \dd t'\\
			&= \int_{\R} \cF^{-1}(\eta_{\lambda})(t') (h(t-t')-h(t)) \dd t'\\
			&= \int_{0<|t'|<\frac{|t|}{2}} \cF^{-1}(\eta_{\lambda})(t') (h(t-t')-h(t)) \dd t' + \int_{|t'|\geq \frac{|t|}{2}} \cF^{-1}(\eta_{\lambda})(t') (h(t-t')-h(t)) \dd t' =: I + II.
		\end{align*}
		
		If $|t|\leq 2T$, by Lemma~\ref{prop:GeoBMinHolder} and a change of variables, we directly estimate
		\begin{align*}
			|I|+|II| &\leq \int_{\R} |\cF^{-1}(\eta_{\lambda})(t')|  |h(t-t')-h(t)| \dd t'\\
			&\leq \int_{\R} |t'|^\alpha  |\cF^{-1}(\eta_{\lambda})(t')| \dd t' \|h\|_{C^\alpha(\R)}\\
			& = \lambda^{-\alpha} \int_{\R} |t'|^\alpha  |\cF^{-1}(\eta_{1})(t')|\dd t' \|h\|_{C^\alpha(\R)}\lesssim \lambda^{-\alpha} \|h\|_{C^\alpha(\R)}.
		\end{align*}
		
		If $|t|>2T$, we first estimate $I$ where we have $|t'| < \frac{|t|}{2}$. In the case $t<-2T$ we have $h(t-t')-h(t)=0$. Hence, we only need to focus on the case $t\geq 2T$, where $t-t'\in[\frac{t}{2},\infty)$. Using \eqref{prop:fastdecayh} we have
		\begin{align*}
			|I| \lesssim \int_{0<|t'|<\frac{|t|}{2}} |t'|^\alpha  |\cF^{-1}(\eta_{\lambda})(t')|\dd t' \sup_{0<|t'|<\frac{|t|}{2}} \frac{|h(t-t')-h(t)|}{|t'|^\alpha} \lesssim \lambda^{-\alpha} e^{-\frac{t}{6}}.
		\end{align*}
		In order to estimate $II$, where $|t'| \geq \frac{|t|}{2}$, we fix some $\beta>0$. Then
		\begin{align*}
			|II|&\leq \int_{|t'|\geq \frac{|t|}{2}} |t'|^\beta  |\cF^{-1}(\eta_{\lambda})(t')| \frac{|h(t-t')-t(t')|}{|t'|^\beta} \dd t'\\
			&\leq \lambda^{-\beta} \int_{\R} |t'|^\beta  |\cF^{-1}(\eta_{1})(t')| \dd t'  \frac{2\|h\|_{L^\infty(\R)}}{|t/2|^\beta}\lesssim \lambda^{-\beta}\|h\|_{L^\infty(\R)} |t|^{-\beta}.
		\end{align*}
		Choosing $\beta = 2$, since $\alpha\in (0,\frac12)$, we arrive at $|II|\lesssim \lambda^{-\alpha} \|h\|_{L^\infty(\R)} |t|^{-2}$.
		
		Therefore, 
		from the estimates above we conclude that for any $t\in \R$ 
		and any $\lambda\in 2^\N$, 
		\begin{align*}
			\lambda^\alpha |P_\lambda h(t)| \lesssim \chi_{|t|\leq 2T} \|h\|_{C^\alpha(\R)} + \chi_{|t|>2T} (e^{-\frac{|t|}{6}}+ |t|^{-2})(1+\|h\|_{L^\infty(\R)}),
		\end{align*}
		which yields $\lambda^\alpha \|P_\lambda h\|_{L^p(\R)}\leq C<\infty$ for some constant $C>0$, 
		thereby finishing the proof of the lemma.
	\end{proof}

	We conclude this section by showing that GBMs decay rapidly after short time intervals with high probability and that they are uniformly bounded (in $c$) in $B^{s}_{p,\infty}([0,\infty))$ if $s \leq \frac{1}{p}$. Both properties are 
	crucial to derive the noise-regularization-effect on scattering.

	\begin{lemma}\label{prop:GeoBMFastDecay}
		Let $0\leq s <\frac{1}{2}, 1\leq p\leq \infty$, and $c\geq1$. 
				
		\noindent$(i)$ Fast decay after short time ($c \rightarrow \infty$): For any $\varepsilon>0$, we have
		\begin{align}\label{prop:GBMNormFastDeacy}
			\PP( \|h_c\|_{B_{p,\infty}^s([c^{-1},\infty))}\geq \varepsilon)\longrightarrow 0\quad \text{as}\ c\rightarrow \infty.
		\end{align}
		
		\noindent$(ii)$ Uniform bound (in $c$): 
		If $s\leq \frac{1}{p}$, we have
		\begin{align}\label{prop:GBMNormUniformBound}
			\sup_{c\geq 1} \|h_c\|_{B^s_{p,\infty}([0,\infty))}<\infty,\ \ \mathbb{P}\text{-a.s.} 
		\end{align} 
		
		\noindent$(iii)$ Exponential decay of Besov norms ($c$ fixed): 
		If $p\in(1,\infty]$, then for $\mathbb{P}$-a.s{.} $\omega$
		there exist $C_1(c,\omega)>0$, $C_2 > 0$, and sufficiently large $T_0(\omega) > 0$ 
		such that for any interval $I = [t_0,\infty)$ with 
		$t_0 \geq T_0(\omega)$
		\begin{align}\label{prop:GBMNormExpoDeacy}
			\|h_c(\cdot, \omega)\|_{B^s_{p,\infty}(I)} \leq C_1(c,\omega) e^{-C_2 t_0}. 
		\end{align} 
	\end{lemma}
	
	\begin{proof}
		$(i)$ We start with the proof of \eqref{prop:GBMNormFastDeacy}. We use the extension $\psi_{c,c^{-1}}[c \beta]$ from~\eqref{eq:DefExtensionhc}, which coincides with $h_c$ on $[c^{-1}, \infty)$. In the case $p=\infty$, using $B^{s}_{\infty, \infty}(\R) = C^s(\R)$ with equivalent norms, we infer
		\begin{align*}
			\|h_c\|_{B^{s}_{\infty, \infty}([c^{-1},\infty))} \leq \| \psi_{c,c^{-1}}[c \beta] \|_{B^{s}_{\infty, \infty}(\R)} \leq C \| \psi_{c,c^{-1}}[c \beta] \|_{C^{s}(\R)}.
		\end{align*}
		Using the self-symmetry of Brownian motions
		$\PP\circ (c\beta(\cdot))^{-1}= \PP\circ(\beta (c^2 \cdot))$, we obtain
		\begin{align*}
			\PP (\|\psi_{c,c^{-1}}[c \beta]\|_{C^{s}(\R)} \geq \varepsilon )
			&= \PP ( \|\psi_{c,c^{-1}}[\beta(c^2 \cdot)] \|_{C^{s}(\R)} \geq  \varepsilon )
			\leq \PP ( c^{2s} \|\psi_{1,c}[\beta] \|_{C^{s}(\R)} \geq  \varepsilon),
		\end{align*}
		where we used the definition of the extension in~\eqref{eq:DefExtensionhc} to obtain $\psi_{1,c}[\beta]$ after scaling in the last step, as well as $c \geq 1$.  For the H{\"o}lder norm, by the definition of $\psi_{1,c}$ and using the interpretation~\eqref{eq:DefHoelderNorm}, we infer
		\begin{align*}
			\|\psi_{1,c}[\beta]\|_{C^s(\R)} \leq e^{-2\beta(c) - 2c} + \| h_1 \|_{C^s([c,\infty))}.
		\end{align*} 
		Employing the law of the iterated logarithm and Remark~\ref{rem:HoelderEstUnb}, we obtain that for $\PP$-a.s. $\omega$ there are sufficiently large $c_0(\omega) > 0$ and $C(\omega) > 0$ such that
		\begin{align}
		\label{eq:HoelderNormPsi1c}
			\|\psi_{1,c}[\beta]\|_{C^s(\R)} \leq C(\omega) e^{-\frac{c}{3}}
		\end{align}
		for all $c \geq c_0(\omega)$.
		Hence, we have $\PP$-a.s. for sufficiently large $c$
		\begin{align*}
			c^{2s} \|\psi_{1,c}[\beta] \|_{C^{s}(\R)} \leq C(\omega) c^{2s} e^{-\frac{c}{3}} \longrightarrow 0\quad \text{as}\ c \to \infty
		\end{align*}
		and thus
		\begin{align*}
			\PP (\|h_c\|_{B^{s}_{\infty, \infty}([c^{-1},\infty))} \geq \varepsilon )
			\leq  \PP ( c^{2s} \|\psi_{1,c}[\beta] \|_{C^{s}(\R)} \geq C^{-1} \varepsilon) \longrightarrow 0 \quad \text{as $c \rightarrow \infty$}.
		\end{align*}

		In the case $1\leq p <\infty$, we still use the extension $\psi_{c,c^{-1}}[c \beta]$ from~\eqref{eq:DefExtensionhc}, which coincides with $h_c$ on $[c^{-1}, \infty)$. By the definition of the restriction norm, we thus have
		\begin{align*}
			\|h_c\|_{B_{p,\infty}^s([c^{-1},\infty))} \leq \| \psi_{c,c^{-1}}[c \beta] \|_{B_{p,\infty}^s(\R)}.
		\end{align*}
		Using the self-symmetry of Brownian motions once more and the equivalence of $\| \cdot \|_{B^s_{p,\infty}(\R)}$ and $\|\cdot\|_{L^p(\R)} + \| \cdot \|_{\dot{B}_{p,\infty}^s(\R)}$, the scaling of the $\|\cdot\|_{L^p(\R)}$- and $\| \cdot \|_{\dot{B}_{p,\infty}^s(\R)}$-norms combined with $c \geq 1$ implies
		\begin{align*}
			\PP (\|h_c\|_{B_{p,\infty}^s([c^{-1},\infty))} \geq \varepsilon )
			&\leq \PP (\| \psi_{c,c^{-1}}[c \beta] \|_{B_{p,\infty}^s(\R)} \geq \varepsilon )
			= \PP ( \| \psi_{c,c^{-1}}[\beta(c^2 \cdot)] \|_{B_{p,\infty}^s(\R)} \geq \varepsilon ) \\
			&\leq \PP ( c^{2s-\frac{2}{p}}\|\psi_{1,c}[\beta]\|_{B_{p,\infty}^s(\R)} \geq C^{-1} \varepsilon),
		\end{align*} 
		where we also exploited the definition in~\eqref{eq:DefExtensionhc} to obtain $\psi_{1,c}[\beta]$ after scaling in the last step again.
		Note that by interpolation, we have
		\begin{align*}
			c^{2s-\frac{2}{p}}\|\psi_{1,c}[\beta]\|_{B_{p,\infty}^s(\R)} &\lesssim c^{2s-\frac{2}{p}} \|\psi_{1,c}[\beta]\|^{\frac{1}{p}}_{B_{1,\infty}^s(\R)} \|\psi_{1,c}[\beta]\|^{1-\frac{1}{p}}_{B^s_{\infty, \infty}(\R)} 
			&\lesssim c^{2s-\frac{2}{p}} \|\psi_{1,c}[\beta]\|^{\frac{1}{p}}_{B_{1,\infty}^s(\R)} \|\psi_{1,c}[\beta]\|^{1-\frac{1}{p}}_{C^s(\R)}.
		\end{align*}
		By Lemma~\ref{prop:GeoBMinHolder}, $\|\psi_{1,c}[\beta]\|_{B_{1,\infty}^s(\R)}$ is finite $\PP$-a.s.
		
		Employing~\eqref{eq:HoelderNormPsi1c} for the H{\"o}lder-norm again, we get $\PP$-a.s.
		\begin{align*}
			c^{2s-\frac{2}{p}}\|\psi_{1,c}[\beta]\|_{B_{p,\infty}^s([c,\infty))} 
			&\lesssim c^{2s-\frac{2}{p}} e^{-\frac{c}{3}(1-\frac{1}{p})} \longrightarrow 0\quad \text{as}\ c \to \infty, 
		\end{align*} 
		where we also used $s < \frac12$ in the case $p = 1$. We conclude that
		\begin{align*}
			\PP (\|h_c\|_{B_{p,\infty}^s([c^{-1},\infty))} \geq \varepsilon )
			\longrightarrow 0\quad \text{as}\ c \to \infty,
		\end{align*}  
		thereby finishing the proof of \eqref{prop:GBMNormFastDeacy}. 
		
		\vspace{4pt plus 2pt minus 2pt}
		$(ii)$ In order to prove~\eqref{prop:GBMNormUniformBound}, we argue as in the proof of \eqref{prop:GBMNormFastDeacy} via the self-symmetry of Brownian motions and the definition of the extensions in~\eqref{eq:DefExtensionhc} to infer
		\begin{align*}
			\|h_c\|_{B^s_{p,\infty}([0,\infty))} \leq \|\psi_{c,0}[c\beta]\|_{B^s_{p,\infty}(\R)} \overset{d}{=} \|\psi_{c,0}[\beta(c^2 \cdot)]\|_{B^s_{p,\infty}(\R)} \leq  c^{2(s-\frac{1}{p})}\|\psi_{1,0}[\beta]\|_{B^s_{p,\infty}(\R)} \leq \|\psi_{1,0}[\beta]\|_{B^s_{p,\infty}(\R)}
		\end{align*}
		where $\overset{d}{=}$ means 
		equality in distribution and where we used $\frac{1}{p}\geq s$ in the last step. Since $\psi_{1,0}[\beta] \in B^{s}_{p,\infty}(\R)$ $\PP$-a.s. by Lemma~\ref{le:GBMGloBesov}, we obtain~\eqref{prop:GBMNormUniformBound}.

		\vspace*{4pt plus 2pt minus 2pt}%
		$(iii)$ The proof of~\eqref{prop:GBMNormExpoDeacy} follows similar ideas as above. Using the extension $\psi_{c,t_0}[c \beta]$ of $h_c$ from $[t_0,\infty)$ to $\R$, we infer
		\begin{align} \label{esti-Bsp-Cs}
			\|h_c\|_{B_{p,\infty}^s([t_0,\infty))} 
			&\lesssim \| \psi_{c,t_0}[c \beta] \|_{B^s_{p,\infty}(\R)}
			\lesssim \|\psi_{c,t_0}[c \beta]\|^{\frac{1}{p}}_{B_{1,\infty}^s(\R)} \|\psi_{c,t_0}[c \beta]\|^{1-\frac{1}{p}}_{B^s_{\infty,\infty}(\R)} \nonumber \\
			&\lesssim  \|\psi_{c,t_0}[c \beta]\|^{\frac{1}{p}}_{B_{1,\infty}^s(\R)} \|\psi_{c,t_0}[c \beta]\|^{1-\frac{1}{p}}_{C^s(\R)}.
		\end{align} 
		Arguing as in the proof of \eqref{prop:GBMNormUniformBound}, exploiting the self-symmetry of Brownian motions and~\eqref{eq:DefExtensionhc}, we obtain
		\begin{align*}
			\|\psi_{c,t_0}[c \beta]\|_{B^s_{p,\infty}(\R)} \overset{d}{=} \|\psi_{c,t_0}[\beta(c^2 \cdot)]\|_{B^s_{p,\infty}(\R)} \leq  c^{2(s-\frac{1}{p})}\|\psi_{1,c^2 t_0}[\beta]\|_{B^s_{p,\infty}(\R)} < \infty \qquad \PP\text{-a.s.}
		\end{align*}
		by Lemma~\ref{le:GBMGloBesov}. In view of the definition of $\psi_{c,t_0}$, we directly estimate
		\begin{align*}
			\|\psi_{c,t_0}[c \beta]\|_{C^s(\R)} \leq c^2 e^{- 2c \beta(t_0) - 2c^2 t_0} + \| h_c \|_{C^s([t_0, \infty))},
		\end{align*}
		where we interpret $\| h_c \|_{C^s([t_0, \infty))}$ as in~\eqref{eq:DefHoelderNorm} again. An inspection of the proof of Lemma~\ref{le:GBMGloBesov} shows that
		\begin{align*}
			\| h_c \|_{C^s([t_0, \infty))} \lesssim c^2 e^{- \frac{c}{3} t_0}
		\end{align*}
		for all $t_0 \geq T_0$ $\PP$-a.s., where we also used $c \geq 1$.
		Combining the previous estimates with~\eqref{esti-Bsp-Cs}, we conclude that there exist $C_1(c,\omega) > 0$ and $C_2 > 0$ such that
		\begin{align*}
			\|h_c(\cdot, \omega)\|_{B_{p,\infty}^s([t_0,\infty))} \leq C_1(c,\omega) e^{-C_2 t_0}
		\end{align*}
		for $\PP$-a.s. $\omega$, finishing the proof of \eqref{prop:GBMNormExpoDeacy}.
	\end{proof}
	
	\begin{remark}
	\label{rem:BesovIntersection}
		Observe that in the proof of the preceding lemma, the extension operator is independent of the parameters $0 \leq s < \frac12$ and $1 \leq p \leq \infty $ across cases (i), (ii), and (iii). Consequently, the lemma's statements remain valid if the Besov spaces are replaced by intersections of admissible Besov spaces. In particular, owing to the embedding $B^{\frac{1}{4}}_{\infty,\infty} \hookrightarrow L^\infty$,  the estimates in \eqref{prop:GBMNormFastDeacy}, \eqref{prop:GBMNormUniformBound}, and~\eqref{prop:GBMNormExpoDeacy} hold when the Besov norm $\| h_c \|_{B^s_{p,\infty}}$ is upgraded to the intersection norm $\| h_c \|_{B^s_{p,\infty} \cap L^\infty}$ over the respective domains $[c^{-1},\infty)$, $[0,\infty)$, and $I$. These intersection spaces are precisely those utilized in the proof of Theorem~\ref{thm:RegNoise}.
	\end{remark}
	
	To apply the lemmas above to our system \eqref{eq:RanZakNoncons-intro}, we will use a time-change martingale theorem for Brownian motions (see, e.g. \cite[Theorem 19.4]{Ka21}) to extend previous properties to infinite dimensional noise.
	\begin{corollary}\label{prop:ExtendToInfiniteNoise}
		Let $\mathbf{c}:=\{c_k\}\in l^2, \|\mathbf{c}\|^2:= \sum_{k=1}^\infty |c_k|^2$ and $\{\beta_k\}$ be a sequence of independent standard one dimensional Brownian motions. Define the geometric Brownian motion with infinite dimensional noise that
		\begin{align*}
			\tilde{h}_{\mathbf{c}}(t):= e^{-2(\sum_{k=1}^\infty c_k \beta_k) - 2\|\mathbf{c}\|^2t}.
		\end{align*}
		Then there exists a one dimensional Brownian motion $\tilde{\beta}(t)$ such that
		\begin{align}\label{prop:TimeChangeBM}
			\tilde{h}_{\mathbf{c}}(t) = e^{-2\tilde{\beta}(\|\mathbf{c}\|^2 t)-2\|\mathbf{c}\|^2 t},\quad \PP\text{-a.s.}
		\end{align}
		In particular, replacing $c$ with $\|\mathbf{c}\|$, if we assume the same conditions hold for $\tilde{h}_{\mathbf{c}}(t)$ in Lemma \ref{eq:DefHoelderNorm}, Lemma \ref{le:GBMGloBesov}, and Lemma \ref{prop:GeoBMFastDecay}, then the same conclusion follows.
	\end{corollary}
	
	\begin{proof}
		The equality \eqref{prop:TimeChangeBM} is a direct consequence of the time-change martingale theorem for Brownian motions. For the remaining of the corollary, by the self-symmetry of Brownian motions we have
		\begin{align*}
			e^{-2\tilde{\beta}(\|\mathbf{c}\|^2 t)-2\|\mathbf{c}\|^2 t} \overset{d}{=} e^{-2\|\mathbf{c}\|\tilde{\beta}( t)-2\|\mathbf{c}\|^2 t},
		\end{align*}
		where $\overset{d}{=}$ means equality in distribution. Then all the conclusions in Lemma \ref{eq:DefHoelderNorm}, Lemma \ref{le:GBMGloBesov}, and Lemma \ref{prop:GeoBMFastDecay} hold with $c$ replaced by $\|\mathbf{c}\|$.
	\end{proof}

	\section{Trilinear estimates for stochastic wave nonlinearity}

	This section establishes the crucial trilinear estimates for the stochastic wave nonlinearity. Our analysis relies on the Besov regularity $B^{\frac{1}{2}-}_{p,\infty}$ of the geometric Brownian motion, multilinear Fourier restriction estimates for the interaction between the noise and two Schr{\"o}dinger solutions, and the local smoothing effect of the Schr\"odinger propagation.  
	In particular, by incorporating maximal function norms into our functional framework, we are able to derive trilinear estimates which ultimately allow us to reach the regularity threshold for the regularization by noise phenomenon dictated by the regularity of the noise. 
	
	We recall that the inhomogeneous half-wave propagator $\cJ_0$ is defined in \eqref{eq:DefPropOp}.

	\begin{theorem}[Trilinear estimates involving geometric Brownian motion]\label{Prop-trilinear}
		
		Let $d\geq 4$, $s,l,\beta \geq 0$,  and
		$0\leq a,b \leq 1$. 
		Then the following holds: 
		
		\vspace*{4pt plus 2pt minus 2pt}%
		\paragraph{$(i)$	{\bf Noise-regularization regime I}} 
		If the conditions 
		\begin{align*}
				\begin{aligned}
					&2s-l-\tfrac{d-2}{2}\geq \max\{ 3a-b, 2a\},\quad
					&s-l &\geq  \min\{2a-b, \tfrac{7}{4}-b\},  \\
					&2s-l-\tfrac{d-2}{2} > \max\{a, 2 a - b\},  &s - l &> a,
				\end{aligned}
		\end{align*}
		\begin{align*}
			(s,l) \neq  (\tfrac{d-2}{2}+a, \tfrac{d-2}{2}-a+b ),
		\end{align*}
		and
		\begin{align*}
			\min \{ s, 2s-\tfrac{d-2}{2}-a \}\geq \beta ,\quad(s,\beta)\neq (\tfrac{d-2}{2}+a,\tfrac{d-2}{2}+a)
		\end{align*}
		are satisfied, then 
		there exists $\alpha \in (0,\frac12)$ close to $\frac12$ and $C>0$ such that for any interval $I \subseteq \R$,
		$\varphi, \psi \in S^{s,a,b}(I)$, and $h\in B_{2,\infty}^\alpha\cap L^\infty(I)$, 
		\begin{align}  \label{triesti-improve}
			\|\cJ_0[h|\nabla|(\overline{\varphi}\psi)]\|_{W^{l,a,\beta}(I)}\leq C \|h\|_{B_{2,\infty}^\alpha \cap L^\infty(I)}\|\varphi\|_{S^{s,a,b}(I)}\|\psi\|_{S^{s,a,b}(I)}.
		\end{align}

		Moreover, if the conditions
		\begin{align*}
				2s -l -\tfrac{d-2}{2}> 3a, \quad  s-l> \min\{2a, \tfrac{7}{4}\},\quad \min \{ s, 2s-\tfrac{d-2}{2}-a \}> \beta
		\end{align*}
		are satisfied, 
		then there exist $\alpha \in (0, \frac12)$ close to $\frac12$, $\theta \in (0,1)$, and $C>0$ such that for any interval $I \subseteq \R$, $\varphi, \psi \in S^{s,a,0}(I)$, and $h\in B^\alpha_{2,\infty}\cap L^\infty(I)$, we have
		the following estimate including the dispersive endpoint Strichartz norm
		\begin{align}\label{prop:triest2}
			\|\cJ_0[h|\nabla|(\overline{\varphi}\psi)]\|_{W^{l, a, \beta}(I)}
			\leq C \|h\|_{B_{2,\infty}^\alpha \cap L^\infty(I)} (\|\varphi\|_{L^2 (I; L_x^{2^*})} \|\psi\|_{L^2 (I; L_x^{2^*})})^\theta ( \|\varphi\|_{S^{s,a,0}(I)}\|\psi\|_{S^{s,a,0}(I)})^{1-\theta}.
		\end{align}

		\vspace*{4pt plus 2pt minus 2pt}%
		\paragraph{	$(ii)$
			{\bf Noise-regularization regime II} } 
		If the conditions 
		\begin{align*}
			s = \tfrac{l}{2}+\tfrac{d-2}{4} \quad \text{and} \quad \tfrac{d-4}{2}\leq l< \tfrac{d-2}{2}
		\end{align*}
		are satisfied, then there exist $p \in (2,\infty)$ close to $2$, $\theta\in(0, 1)$, and $C > 0$ such that for any interval $I \subseteq \R$, $\varphi, \psi \in S^{s,0,0}(I)$, and $h\in B^{\frac{1}{p}}_{p,\infty}\cap L^\infty(I)$,
		\begin{align}
			\|\cJ_0[h|\nabla|(\overline{\varphi}\psi)]\|_{W^{l, 0, l}(I)}
			\leq C \|h\|_{B_{p,\infty}^{\frac{1}{p}}\cap L^\infty(I)} (\|\varphi\|_{L^2 (I; W_x^{s,2^*})} \|\psi\|_{L^2 (I; W_x^{s,2^*})})^{\theta} (\|\varphi\|_{S^{s,0,0}(I)}\|\psi\|_{S^{s,0,0}(I)})^{1-\theta}.
			\label{prop:triestiend}
		\end{align}
		
		\paragraph{	$(iii)$
			{\bf Noise-regularization regime III}  }
		If the conditions 
		\begin{align*}
			s-l\geq -b,\quad
			2s-l-\tfrac{d-2}{2}\geq 0,\quad (s,l)\neq (\tfrac{d}{2}, \tfrac{d+2}{2}), (\tfrac{d-2}{2},\tfrac{d-2}{2}+b),
		\end{align*}
		\begin{align*}
			\min\{s, 2s-\tfrac{d-2}{2} \} \geq \beta,\quad (s,\beta)\neq (\tfrac{d-2}{2}, \tfrac{d-2}{2}),
		\end{align*}
		and 
		\begin{align*}
			s - l > - \tfrac{1}{2}
		\end{align*}
		are satisfied, then there exist $p \in (2,\infty)$ close to $2$, $0 < \varepsilon \ll 1$, and $C>0$ such that for any interval $ I \subseteq \R$,
		$\varphi, \psi \in \Z_\varepsilon^{s,b}(I)$, and $h\in B^{\frac{1}{p}}_{p,\infty}\cap L^\infty(I)$,
		\begin{align}  \label{triesti-LocSmo}
			\|\cJ_0[h|\nabla|(\overline{\varphi}\psi)]\|_{W^{l,0,\beta}(I)}\leq C  \|h\|_{B^{\frac{1}{p}}_{p,\infty}\cap L^\infty(I)}\|\varphi\|_{\Z_\varepsilon^{s,b}(I)}\|\psi\|_{\Z_\varepsilon^{s,b}(I)}.
		\end{align}		
		
	\end{theorem}

	\begin{remark}
		Theorem $7.4$ of \cite{HRSZ24} established the trilinear estimate for the specific case $(s,l)=(1,0)$, relying on the global $V^p$-regularity ($p > 2$) of the geometric Brownian motion. 
		In Theorem~\ref{Prop-trilinear} we extend the trilinear estimates to the significantly wider range of regularities $(s,l)$ described in Theorem~\ref{thm:RegNoise}.
		
		In fact, we cover the entire local well-posedness regime, except for a strip where the wave regularity exceeds the Schr{\"o}dinger regularity. This restriction arises because the trilinear estimate is ultimately limited by the low regularity of the noise. In contrast, the bilinear interaction governing the local well-posedness theory recovers a higher amount of regularity.
		
		To derive the trilinear estimates, we exploit both the Besov regularity of the geometric Brownian motion (established in Section~\ref{Sec:GRCGBM}) and the local smoothing effect of the Schr{\"o}dinger equation. We perform a careful analysis of the trilinear interactions between the noise and the Schr{\"o}dinger waves in space-time. Crucially, this analysis allows us to convert temporal regularity of the noise into spatial regularity.
		
		One key feature here is that our trilinear estimates only employ
		Besov norms which respect the scaling of the geometric Brownian motions. 
		This choice is crucial for establishing uniform bounds in the subsequent global existence and scattering arguments. In particular, for regime II and regime III, we rely solely on scale-invariant Besov norms, which enables us to cover the endpoint line $s = \frac{l}{2} + \frac{d-2}{4}$.
	\end{remark}

	\subsection{Trilinear estimates in noise-regularization regime I} 
	\label{subsec:TrilinearRegimeI}
	We start with proving the trilinear estimates \eqref{triesti-improve} and \eqref{prop:triest2} 
	in the noise-regularization regime I. We note that there is some freedom in the choice of the parameter $\frac{7}{4}$ here. Its purpose is to ensure that all the assumptions on the parameters are satisfied in regime~I. In particular, on the boundary line $s = l + 2$ where $s - l = 2 = 2 a$, it guarantees that the assumptions for~\eqref{prop:triest2} hold true.

	\vspace*{4pt plus 2pt minus 2pt}%
	\paragraph{\bf Proof of \eqref{triesti-improve}}  
	The conditions in regime I imply that there exists $0 < \varepsilon \ll 1$ such that
	\begin{align*}
			2s-l-\tfrac{d-2}{2}\geq \max\{ 3a-b, a+2\varepsilon, 2a, 2a-b+2\varepsilon\},\quad
			s-l\geq \max\{a+2\varepsilon, \min\{2a-b, \tfrac{7}{4}-b\}\},
	\end{align*}
	\begin{align*}
		(s,l) \neq  (\tfrac{d-2}{2}+a, \tfrac{d-2}{2}-a+b )
	\end{align*}
	and
	\begin{align*}
		\min \{ s, 2s-\tfrac{d-2}{2}-a \}\geq \beta ,\quad(s,\beta)\neq (\tfrac{d-2}{2}+a,\tfrac{d-2}{2}+a)
	\end{align*}
	are satisfied. We set $\alpha = \frac12 - \varepsilon$.
	
	In view of the definition of the restriction norms, it is sufficient to consider arbitrary extensions from $I$ to $\R$ of $h, \varphi$, and $\psi$ and prove the estimate on $\R$. Moreover, by Lemma $2.6$ of \cite{CHN23}, it suffices to prove 
	the estimates 
	\begin{align}
		\Big( \sum_{\lambda\in 2^\N} \lambda^{2(l-a+1)}\|(\lambda+|\partial_t|)^a P^{(t)}_{\ll \lambda^2} P_{\lambda}(h\overline{\varphi}\psi)\|_{L_t^1L_x^2}^2\Big)^{\frac12}&\lesssim \|h\|_{B^{\alpha}_{2,\infty} \cap L^\infty_t} \|\varphi\|_{S^{s,a,b}}\|\psi\|_{S^{s,a,b}}, \label{eq:NoiseRegE41}\\
		\Big( \sum_{\lambda\in 2^\N} \lambda^{2(l-1)}\|P_\lambda(h\overline{\varphi}\psi)\|_{L_t^\infty L_x^2}^2 \Big)^{\frac12}&\lesssim \|h\|_{L_t^\infty} \|\varphi\|_{S^{s,a,b}}\|\psi\|_{S^{s,a,b}},   \label{eq:NoiseRegE42}\\
		\Big( \sum_{\lambda\in 2^\N} \lambda^{2\beta}\|P_\lambda (h\overline{\varphi}\psi)\|_{L_{t,x}^2}^2\Big)^{\frac12}&\lesssim \|h\|_{L_t^\infty} \|\varphi\|_{S^{s,a,b}}\|\psi\|_{S^{s,a,b}},   \label{eq:NoiseRegE43}\\
		\|P_{\leq 2^{16}} (h\overline{\varphi}\psi)\|_{L_t^1 L_x^2}&\lesssim \|h\|_{L_t^\infty} \|\varphi\|_{S^{s,a,b}}\|\psi\|_{S^{s,a,b}}.    \label{eq:NoiseRegE44}
	\end{align}
	
	For the last three estimates \eqref{eq:NoiseRegE42}, \eqref{eq:NoiseRegE43}, and \eqref{eq:NoiseRegE44}, since $h$ is independent of the spatial variable, we have  $P_\lambda(h\overline{\varphi}\psi) = hP_\lambda(\overline{\varphi}\psi)$ and $P_{\leq 2^{16}}(h\overline{\varphi}\psi) = hP_{\leq 2^{16}}(\overline{\varphi}\psi)$.
	Bounding $h$ on the left-hand side of \eqref{eq:NoiseRegE42}-\eqref{eq:NoiseRegE44} by $\|h\|_{L_t^\infty}$,
	the estimates follow from the proof of Theorem~$4.1$ in~\cite{CHN23}, 
	provided that 
	\begin{align}
		&s+1\geq l,\quad 2s\geq l+\tfrac{d-2}{2},\quad (s,l)\neq(\tfrac{d}{2},\tfrac{d+2}{2})\notag\\
		& s\geq a,\quad s\geq \beta,\quad 2s-\beta-\tfrac{d-2}{2}\geq a,\quad (s,\beta) \neq (\tfrac{d-2}{2}+a,\tfrac{d-2}{2}+a).\label{con:wavetri234}
	\end{align}
	Note that $(s,l)\neq(\tfrac{d}{2},\tfrac{d+2}{2})$ is satisfied as $s > l$ in regime I.
	Below we prove the remaining estimate  \eqref{eq:NoiseRegE41}.
	For this purpose,
	we use the paraproduct decomposition 
	\begin{align}\label{prop:ParaproductDecom}
		P_\lambda(\overline{\varphi}\psi)=\sum_{\lambda/2\leq\mu\leq 2\lambda} P_\lambda(\overline{\varphi}_\mu \psi_{\ll \mu}) +\sum_{\lambda_1\sim \lambda_2\gtrsim \lambda} P_\lambda(\overline{\varphi}_{\lambda_1}\psi_{\lambda_2})+\sum_{\lambda/2\leq\mu\leq 2\lambda} P_\lambda(\overline{\varphi}_{\ll \mu} \psi_\mu) .
	\end{align}
	
	Since the left-hand side of~\eqref{eq:NoiseRegE41} is invariant under complex conjugation, it is sufficient to estimate the high-low and the high-high interaction.

	\vspace*{4pt plus 2pt minus 2pt}%
	\paragraph{\bf $(i)$ High-low interaction:}
	Let $\mu\sim \lambda$ and further decompose by modulation
	\begin{align*}
		P^{(t)}_{\ll \lambda^2}(h \overline{\varphi}_\mu \psi_{\ll\mu})= 
		P^{(t)}_{\ll \lambda^2}(h \overline{C_{\gtrsim \mu^2}\varphi_\mu}\psi_{\ll\mu})
		+ P^{(t)}_{\ll \lambda^2}(h \overline{C_{\ll \mu^2}\varphi_\mu}\psi_{\ll\mu}) 
		=:  I^{HL}_{HM} + I^{HL}_{LM}.
	\end{align*}
	
	\vspace*{4pt plus 2pt minus 2pt}%
	\subparagraph{\bf $(i.1)$ High modulation 
		contribution $I^{HL}_{HM}$:} We further decompose by temporal frequency to get
	\begin{align} \label{HLHM-decom}
		I^{HL}_{HM}= P^{(t)}_{\ll \lambda^2}(P^{(t)}_{\gtrsim\mu^2}h \overline{C_{\gtrsim \mu^2}\varphi_\mu}\psi_{\ll\mu}) +P^{(t)}_{\ll \lambda^2}(P^{(t)}_{\ll\mu^2}h \overline{C_{\gtrsim \mu^2}\varphi_\mu}\psi_{\ll\mu})
		=: I^{HL}_{HM,HT} + I^{HL}_{HM,LT}.
	\end{align}
	
	First we use the H\"older and Bernstein inequalities to estimate
	\begin{align*}
		&\mu^{l-a+1}\|(\lambda+|\partial_t|)^a I^{HL}_{HM,HT}\|_{L_t^1 L_x^2}\\
		&\lesssim  \mu^{l+a+1} \|P^{(t)}_{\gtrsim \mu^2}h \|_{L_t^2} \|C_{\gtrsim \mu^2}\varphi_\mu \|_{L_{t,x}^2} \|\psi_{\ll\mu}\|_{L^\infty_t L_x^\infty}\\
		&\lesssim \mu^{l+a+1}\mu^{-2\alpha} \|h\|_{B^\alpha_{2,\infty}} \mu^{-2+a} \mu^{1-s-b} \|\varphi_\mu\|_{S^{s,a,b}_{\mu}}  \sum_{\nu \ll \mu} \nu^{\frac{d}{2}-s} \nu^s \|\psi_{\nu}\|_{L^\infty_t L^2_x}\\
		&\lesssim \mu^{-s+l+2a-2\alpha-b} \|h\|_{B_{2,\infty}^\alpha} \|\varphi_\mu\|_{S^{s,a,b}_{\mu}} \sum_{\nu \ll \mu }  \nu^{\frac{d}{2}-s} \|\psi_\nu\|_{S^{s,a,b}_\nu},
	\end{align*}
	which is $l^2$-summable and gives
	\begin{align}
		\Big(\sum_{\lambda\in 2^\N} \lambda^{2(l-a+1)} \|(\lambda+|\partial_t|)^a \sum_{\mu\sim\lambda} P_\lambda I_{HM,HT}^{HL}\|^2_{L_t^1 L_x^2}\Big)^{\frac12} \lesssim \|h\|_{B_{2,\infty}^\alpha\cap L_t^\infty} \|\varphi\|_{S^{s,a,b}} \|\psi\|_{S^{s,a,b}}, \label{prop:triwaveestHLHMHT}
	\end{align}
	provided that
	\begin{align*}
		s-l\geq 2a-1-b+2\varepsilon,\quad 2s-l-\tfrac{d-2}{2}\geq 2a-b+2\varepsilon, \quad (s,l) \neq (\tfrac{d}{2}, \tfrac{d+2}{2} - 2a + b - 2 \varepsilon).
	\end{align*}
	Note that the last condition is satisfied as $s - l \geq 2 a - b$ in regime I.
	
	We next estimate the low temporal contribution $I^{HL}_{HM,LT}$. We first consider the case $a\leq\tfrac{7}{8}$, where $\min\{2a - b, \frac{7}{4} - b\} = 2 a - b$.
	We infer
	\begin{align*}
		&\mu^{l-a+1}\|(\lambda+|\partial_t|)^a I^{HL}_{HM,LT} \|_{L_t^1 L_x^2}\\
		&\lesssim \mu^{l+a+1} \|h\|_{L_t^\infty} \|C_{\gtrsim \mu^2} \varphi_\mu\|_{L_{t,x}^2} \|\psi_{\ll \mu}\|_{L_t^2 L_x^\infty}\\
		&\lesssim \mu^{l+a+1} \|h\|_{L_t^\infty} \mu^{-2+a+1-s-b} \|\varphi_\mu\|_{S_\mu^{s,a,b}} \sum_{\nu\ll\mu} \nu^{\frac{d-2}{2}-s+a} \nu^{s-2a} \|(\nu+|\partial_t|)^a \psi_\nu\|_{L_t^2 L_x^{2^*}}\\
		&\lesssim \mu^{l+2a-s-b}\|h\|_{L_t^\infty} \|\varphi_\mu\|_{S_\mu^{s,a,b}} \sum_{\nu\ll\mu} \nu^{\frac{d-2}{2}-s+a}  \|\psi_\nu\|_{S_\nu^{s,a,b}},
	\end{align*}
	which is $l^2$-summable provided that
	\begin{align*}
		s-l\geq 2a-b, \quad 2s-l-\tfrac{d-2}{2}\geq 3a-b,\quad (s,l)\neq (\tfrac{d-2}{2}+a, \tfrac{d-2}{2}-a+b).
	\end{align*}
	It follows that
	\begin{align}
		\Big(\sum_{\lambda\in 2^\N} \lambda^{2(l-a+1)} \|(\lambda+|\partial_t|)^a  \sum_{\mu\sim\lambda} P_\lambda I_{HM,LT}^{HL} \|^2_{L_t^1 L_x^2} \Big)^{\frac12} \lesssim \|h\|_{B_{2,\infty}^\alpha\cap L_t^\infty} \|\varphi\|_{S^{s,a,b}} \|\psi\|_{S^{s,a,b}}. 
		\label{prop:triwaveestHLHMLT}
	\end{align} 
	
	In the case $a>\tfrac{7}{8}$, we have $\min\{2a - b, \frac{7}{4} - b\} = \frac{7}{4} - b$. Here we further split
		\begin{align*}
			I_{HM,LT}^{HL} = P^{(t)}_{\ll \lambda^2}(P^{(t)}_{\ll\mu} h \overline{C_{\gtrsim \mu^2}\varphi_\mu}\psi_{\ll\mu}) + P^{(t)}_{\ll \lambda^2}(P^{(t)}_{\gtrsim\mu} P^{(t)}_{\ll\mu^2} h P^{(t)}_{\ll\mu^2}( \overline{C_{\gtrsim \mu^2}\varphi_\mu}\psi_{\ll\mu})) =: I_{HM,LT.1}^{HL} + I_{HM,LT.2}^{HL}.
	\end{align*}
	
	For the first term $I_{HM,LT.1}^{HL}$, we use Lemma \ref{le:27zak4}, Bernstein estimates, and the relation $\mu\sim\lambda$ to derive
		\begin{align*}
			&\mu^{l-a+1} \|(\lambda+|\partial_t|)^a I_{HM,LT.1}^{HL} \|_{L_t^1 L_x^2}\\
			&\lesssim \mu^{l-3a+1} \|(\mu+|\partial_t|)^a P^{(t)}_{\ll\mu}h\|_{L^\infty} \|(\mu+|\partial_t|)^a C_{\gtrsim \mu^2}\varphi_\mu\|_{L_{t,x}^2} \|(\mu+|\partial_t|)^a\psi_{\ll\mu}\|_{L_t^2 L_x^\infty}\\
			&\lesssim \mu^{l-1}\|h\|_{L_t^\infty} \Big\|\Big(\frac{\mu+|\partial_t|}{\mu^2+|\partial_t|}\Big)^a (\imu \partial_t+\Delta) \varphi_\mu \Big\|_{L_{t,x}^2} \sum_{\nu\ll\mu} \Big(\frac{\mu}{\nu}\Big)^a \nu^{\frac{d-2}{2}-s} \nu^s \|(\nu+|\partial_t|)^a \psi_\nu\|_{L_t^2 L_x^{2^*}}\\
			&\lesssim \mu^{l-s-b+a} \|h\|_{L_t^\infty} \|\varphi_\mu\|_{S_\mu^{s,a,b}} \sum_{\nu\ll\mu}\nu^{\frac{d-2}{2}+a-s} \|\psi_\nu\|_{S_\nu^{s,a,b}},
		\end{align*}
		which is $l^2$-summable provided that 
		\begin{align*}
			s-l\geq a-b,\quad 2s-l-\tfrac{d-2}{2}\geq 2a-b,\quad (s,l)\neq (\tfrac{d-2}{2}+a, \tfrac{d-2}{2}+b).
	\end{align*}
	
	For the second term $I_{HM,LT.2}^{HL}$, we argue similar as for the $I^{HL}_{HM,HT}$ contribution to get
		\begin{align*}
			&\mu^{l-a+1} \|(\lambda+|\partial_t|)^a I_{HM,LT.2}^{HL} \|_{L_t^1 L_x^2}\\
			&\lesssim \mu^{l+a+1}  \| P^{(t)}_{\gtrsim \mu}P^{(t)}_{\ll\mu^2} h\|_{L_t^2} \|  \overline{C_{\gtrsim \mu^2}\varphi_\mu} \|_{L_{t,x}^2} \|\psi_{\ll\mu}\|_{L_{t,x}^\infty}\\
			&\lesssim \mu^{l+a+1-\alpha} \|h\|_{B_{2,\infty}^\alpha} \mu^{a-s-1-b}\|\varphi_\mu\|_{S_\mu^{s,a,b}} \sum_{\nu\ll\mu} \nu^{\frac{d}{2}-s} \|\psi_\nu\|_{S_\nu^{s,a,b}}\\
			&\lesssim \mu^{l+2a-s-b-\alpha} \|h\|_{B_{2,\infty}^\alpha} \|\varphi_\mu\|_{S_\mu^{s,a,b}} \sum_{\nu\ll\mu} \nu^{\frac{d}{2}-s} \|\psi_\nu\|_{S_\nu^{s,a,b}},
		\end{align*}
		which is $l^2$-summable provided that
		\begin{align*}
			s-l\geq 2a - b - \tfrac{1}{2} + \epsilon,\quad 2s-l-\tfrac{d-2}{2}\geq 2 a - b + \tfrac{1}{2} + \epsilon, \quad (s,l) \neq (\tfrac{d}{2}, \tfrac{d}{2} - 2a + b + \tfrac{1}{2} - \epsilon).
	\end{align*}
	Using that $\frac{1}{2} + \epsilon < a \leq 1$, we note that $2 a - b + \frac{1}{2} + \epsilon < 3 a - b$ and $2 a - b - \frac{1}{2} + \epsilon < \frac{7}{4} - b$. The right-hand sides of the bounds for $I^{HL}_{HM,LT.1}$ and $I^{HL}_{HM, LT.2}$ are thus $l^2$-summable.
	
	Hence, estimates \eqref{prop:triwaveestHLHMHT} and \eqref{prop:triwaveestHLHMLT} yield \eqref{eq:NoiseRegE41} for the high modulation part $I^{HL}_{HM}$.

	\subparagraph{\bf $(i.2)$ 
		Low modulation contribution $I^{HL}_{LM}$:}
	We turn to the low modulation part and  further decompose  $I^{HL}_{LM}$
	into 
	\begin{align} \label{HLLM-decom}
		I^{HL}_{LM}
		=P^{(t)}_{\ll \lambda^2}( P^{(t)}_{\ll \mu^2} h \overline{C_{\ll \mu^2}\varphi_\mu}P^{(t)}_{\sim \mu^2}\psi_{\ll\mu})+ P^{(t)}_{\ll \lambda^2}( P^{(t)}_{\gtrsim \mu^2}h \overline{C_{\ll \mu^2}\varphi_\mu}\psi_{\ll\mu})  =: I^{HL}_{LM,LT} + I^{HL}_{LM,HT},
	\end{align}
	where we exploited a \textit{non-resonant identity} in the first summand, which shows that the temporal frequency of $\psi_{\ll \mu}$ must be localized at $\mu^2$.
	
	To estimate $I^{HL}_{LM,LT}$, by H\"older's inequality and Sobolev's embedding $H_x^{\frac{d-2}{2}}\hookrightarrow L^d$, we have
	\begin{align*}
		\mu^{l-a+1} \|(\lambda +|\partial_t|)^a I^{HL}_{LM,LT} \|_{L_t^1 L_x^2} 
		\lesssim~& \mu^{l+1+a} \|h\|_{L_t^\infty} \| C_{\ll \mu^2} \varphi_\mu\|_{L_t^2 L_x^{2^*}} \|P^{(t)}_{\sim \mu^2}\psi_{\ll\mu}\|_{L_t^2 H_x^{\frac{d-2}{2}}}\\
		\lesssim~& \mu^{l-s-1+a+(\frac{d}{2}-s)_+} \|h\|_{L_t^\infty}\mu^{s-2a}  \| (\mu+|\partial_t|)^a  \varphi_{\mu}\|_{L_t^2 L_x^{2^*}} \|\psi\|_{S^{s,a,0}},
	\end{align*}
	which is $l^2$- summable provided that
	$$
	s-l\geq a-1, \quad 2s-l-\tfrac{d-2}{2}\geq a.
	$$ 
	Moreover, for the $I^{HL}_{LM,HT}$ contribution, we have
	\begin{align*}
		\mu^{l-a+1} \|(\lambda+|\partial_t|)^a I^{HL}_{LM,HT} \|_{L_t^1 L_x^2}  
		\lesssim~&  \mu^{l+1+a}\|P^{(t)}_{\gtrsim \mu^2}h \|_{L_t^2} \|C_{\ll \mu^2}\varphi_\mu \|_{L_t^2 L_x^{2^*}} \| \psi_{\ll\mu} \|_{L_t^\infty L_x^d}\\
		\lesssim~& \mu^{l+1+a-2\alpha} \|h\|_{B_{2,\infty}^\alpha} \mu^{-s}  \|\varphi_\mu\|_{S_\mu^{s,a,0}}  \mu^{(\frac{d-2}{2}-s)_+} \|\psi_{\ll\mu}\|_{L^\infty_t H^s_x},
	\end{align*}
	which is $l^2$-summable provided that 
	$$
	s-l\geq a+2\varepsilon,\quad 2s-l-\tfrac{d-2}{2}\geq a+2\varepsilon.
	$$
	Thus, 
	combining the estimates above, we obtain \eqref{eq:NoiseRegE41} 
	for the low modulation part 
	$I^{HL}_{LM}$, 
	and hence the high-low interaction part in \eqref{prop:ParaproductDecom}. 
	
	\vspace*{4pt plus 2pt minus 2pt}%
	It remains to estimate the high-high interaction.
	\vspace*{4pt plus 2pt minus 2pt}%
	
	\paragraph{\bf $(ii)$ High-high interaction:} We first decompose further in temporal frequency
	\begin{align*}
		P^{(t)}_{\ll \lambda^2}(h P_\lambda(\overline{\varphi}_{\lambda_1}\psi_{\lambda_2}) )
		=
		P^{(t)}_{\ll \lambda^2} ( P^{(t)}_{\ll \lambda}h P_\lambda(\overline{\varphi}_{\lambda_1}\psi_{\lambda_2}))
		+ 
		P^{(t)}_{\ll \lambda^2} (  P^{(t)}_{\gtrsim \lambda} h P_\lambda(\overline{\varphi}_{\lambda_1}\psi_{\lambda_2})) 
		=:
		I^{HH}_{LT} + I^{HH}_{HT} .
	\end{align*}
	
	For the low temporal contribution  $I^{HH}_{LT}$, 
	using Lemma \ref{le:27zak4}, $\lambda \lesssim \lambda_1\sim\lambda_2$, and $l-a+1\geq 0$, we have 
	\begin{align*} 
		&\lambda^{l-a+1}  \|(\lambda+|\partial_t|)^a  I^{HH}_{LT} \|_{L_t^1 L_x^2}\\
		&\lesssim  \lambda^{l-2a+1} \|(\lambda+|\partial_t|)^a P^{(t)}_{\ll\lambda} h\|_{L_t^\infty} \|(\lambda_1+|\partial_t|)^a  P_\lambda(\overline{\varphi_{\lambda_1}} \psi_{\lambda_2})\|_{L_t^1 L_x^2}\\
		&\lesssim  \lambda^{l-a+1} \|h\|_{L_t^\infty} \lambda_1^{-a} \|(\lambda_1+|\partial_t|)^a \varphi_{\lambda_1}\|_{L_t^2 L_x^{2^*}} \lambda^{\frac{d-4}{2}} \|(\lambda_2+|\partial_t|)^a \psi_{\lambda_2}\|_{L_t^2 L_x^{2^*}}\\
		&\lesssim  \lambda_1^{l-2s+2a+\frac{d-2}{2}} \|h\|_{L_t^\infty} \|\varphi_{\lambda_1}\|_{S_{\lambda_1}^{s,a,0}} \|\psi_{\lambda_2}\|_{S_{\lambda_2}^{s,a,0}}. 
	\end{align*}
	Then, 
	by the Minkowski inequality, 
	we have 
	\begin{align*}
		\Big(\sum_{\lambda\in 2^\N} \lambda^{2(l-a+1)} \big \|(\lambda+|\partial_t|)^a 
		\sum_{\lambda_1\sim\lambda_2\gtrsim \lambda}
		I^{HH}_{LT}
		\big \|_{L_t^1 L_x^2}^2\Big)^{\frac12}
		&\lesssim  \sum_{\lambda_1\sim\lambda_2} \Big( \sum_{\lambda \lesssim \lambda_1} \lambda^{2(l-a+1)} \Big\|(\lambda+|\partial_t|)^a I^{HH}_{LT} \Big\|^2_{L_t^1 L_x^2}\Big)^{\frac12}\\
		&\lesssim  \sum_{\lambda_1\sim\lambda_2} \lambda_1^{l-2s+2a+\frac{d-2}{2}} \|h\|_{L_t^\infty} \|\varphi_{\lambda_1}\|_{S_{\lambda_1}^{s,a,0}} \|\psi_{\lambda_2}\|_{S_{\lambda_2}^{s,a,0}}\\
		&\lesssim \|h\|_{L_t^\infty} \|\varphi\|_{S^{s,a,0}} \|\psi\|_{S^{s,a,0}},
	\end{align*}
	provided that $2s-l-\tfrac{d-2}{2}\geq 2a$.
	
	Regarding the high temporal contribution $I^{HH}_{HT}$, 
	we further decompose
	\begin{align} \label{HHHT-decom}
		I^{HH}_{HT} 
		&= P^{(t)}_{\ll \lambda^2} (P^{(t)}_{\gtrsim \lambda}h P_\lambda(\overline{C_{\gtrsim \lambda_1^2} \varphi_{\lambda_1}} \psi_{\lambda_2})) 
		+  P^{(t)}_{\ll \lambda^2} ( P^{(t)}_{\ll \lambda_1^2}P^{(t)}_{\gtrsim \lambda}h P_\lambda(\overline{C_{\ll \lambda_1^2} \varphi_{\lambda_1}} P^{(t)}_{\sim \lambda_2^2} \psi_{\lambda_2})) \notag \\
		&\qquad + P^{(t)}_{\ll \lambda^2} ( P^{(t)}_{\gtrsim \lambda_1^2}h P_\lambda(\overline{C_{\ll \lambda_1^2} \varphi_{\lambda_1}} \psi_{\lambda_2})) \notag  \\ 
		&=: I^{HH}_{HT.1} 
		+ I^{HH}_{HT.2}
		+ I^{HH}_{HT.3},
	\end{align}
	where we again exploited a \emph{non-resonant identity} in the second term. 
	For the high-modulation part $I^{HH}_{HT.1}$, as $l+a+1>0$, 
	we derive 
	\begin{align*} 
		\lambda^{l-a+1} \|(\lambda+|\partial_t|)^a I^{HH}_{HT.1} \|_{L_t^1 L_x^2} 
		\lesssim ~& \lambda^{l+a+1} \|h\|_{L_t^\infty} \|C_{\gtrsim \lambda_1^2} \varphi_{\lambda_1}\|_{L_{t,x}^2} \|\psi_{\lambda_2}\|_{L_t^2 L_x^\infty}\\
		\lesssim ~& \lambda_1^{l+a+1} \|h\|_{L_t^\infty} \lambda_1^{-2+a-s+1-b} \|\varphi_{\lambda_1}\|_{S_{\lambda_1}^{s,a,b}} \lambda_2^{\frac{d-2}{2}-s+a} \|\psi_{\lambda_2}\|_{S_{\lambda_2}^{s,a,b}}\\
		\lesssim ~& \lambda_1^{l-2s+3a-b+\frac{d-2}{2}} \|h\|_{L_t^\infty}\|\varphi_{\lambda_1}\|_{S_{\lambda_1}^{s,a,b}} \|\psi_{\lambda_2}\|_{S_{\lambda_2}^{s,a,b}}.
	\end{align*} 
	Then, estimating as in the proof of   $I^{HH}_{LT}$ 
	we obtain the desired bound 
	$\|h\|_{L_t^\infty}\|\varphi\|_{S^{s,a,b}} \|\psi\|_{S^{s,a,b}}$ 
	in \eqref{eq:NoiseRegE41} 
	for the high-modulation part 
	$I^{HH}_{HT.1}$,
	provided that $2s-l-\tfrac{d-2}{2}\geq 3a-b$.
	
	Moreover, 
	for the non-resonant part $I^{HH}_{HT.2}$, as $l+a+1>0$, 
	we have 
	\begin{align*}
		\lambda^{l-a+1} \|(\lambda+|\partial_t|)^a I^{HH}_{HT.2} \|_{L_t^1 L_x^2} 
		&\lesssim  \lambda^{l+a+1} \|h\|_{L_t^\infty}  \|C_{\ll \lambda_1^2} \varphi_{\lambda_1}\|_{L_t^2 L_x^{2^*}} \|P^{(t)}_{\sim \lambda_2^2}\psi_{\lambda_2}\|_{L_t^2 L_x^d}\\
		&\lesssim \lambda_1^{l+a+1} \|h\|_{L_t^\infty} \lambda_1^{-s} \|\varphi_{\lambda_1}\|_{S_{\lambda_1}^{s,a,b}} \lambda_2^{-s+\frac{d-4}{2}} \|\psi_{\lambda_2}\|_{S_{\lambda_2}^{s,a,b}}, 
	\end{align*}
	which yields the bound in~\eqref{eq:NoiseRegE41} for $I^{HH}_{HT.2}$,  
	provided that $2s-l-\tfrac{d-2}{2}\geq a$.
	
	At last, 
	for the remaining part $I^{HH}_{HT.3}$, we estimate  
	\begin{align*} 
		\lambda^{l-a+1} \|(\lambda+|\partial_t|)^a I^{HH}_{HT.3} \|_{L_t^1 L_x^2}
		&\lesssim  \lambda^{l+a+1} \|P_{\gtrsim \lambda_1^2}^{(t)} h \|_{L_t^2} \|C_{\ll \lambda_1^2} \varphi_{\lambda_1}\|_{L_t^2 L_x^{2^*}} \|\psi_{\lambda_2}\|_{L_t^\infty L_x^d}\\
		&\lesssim \lambda_1^{l+a+1-2\alpha} \|h\|_{B_{2,\infty}^\alpha} \lambda_1^{-s} \|\varphi_{\lambda_1}\|_{S_{\lambda_1}^{s,a,b}} \lambda_2^{-s+\frac{d-2}{2}} \|\psi_{\lambda_2}\|_{S_{\lambda_2}^{s,a,b}},
	\end{align*} 
	which gives the bound in~\eqref{eq:NoiseRegE41} for $I^{HH}_{HT.3}$, 
	provided that  $2s-l-\tfrac{d-2}{2}\geq a+2\varepsilon$.
	
	Therefore, combining the above estimates altogether,
	we obtain estimate \eqref{eq:NoiseRegE41}.

	\paragraph{\bf 
		Proof of \eqref{prop:triest2}} 
	We start again by noting that the assumptions imply the existence of $0 < \varepsilon \ll 1$ such that
	\begin{align*}
			2s -l -\tfrac{d-2}{2}> \max\{3a, 2a+2\varepsilon\}, \quad  s-l> \max\{a+2\varepsilon, \min\{2a, \tfrac{7}{4}\}\},\quad \min \{ s, 2s-\tfrac{d-2}{2}-a \}> \beta.
	\end{align*}
	We set $\alpha = \frac12 - \varepsilon$. It suffices to prove that there exist 
	$N, \delta>0$ such that for any $\lambda_1, \lambda_2 \in 2^{\N_0}$,
	\begin{align}
		\|\cJ_0[h|\nabla|(\overline{\varphi}_{\lambda_1}\psi_{\lambda_2})]\|_{W^{l, a, \beta}(I)}&\lesssim \|h\|_{B_{2,\infty}^\alpha\cap L^\infty(I)}\max\{\lambda_1, \lambda_2\}^{-\delta} \|\psi\|_{S^{s,a,0}(I)} \|\varphi\|_{S^{s,a,0}(I)},   \label{prop:trinon1}\\
		\|\cJ_0[h|\nabla|(\overline{\varphi}_{\lambda_1}\psi_{\lambda_2})]\|_{W^{l, a, \beta}(I)} &\lesssim \|h\|_{L^\infty(I)} \max\{\lambda_1, \lambda_2\}^{N} (\|\psi\|_{S^{s,a,0}(I)} \|\varphi\|_{S^{s,a,0}(I)}\|\psi\|_{L^2(I;L_x^{2^*})} \|\varphi\|_{L^2(I;L_x^{2^*})})^{\frac12}.    \label{prop:trinon2}
	\end{align} 
	Assuming the above estimates to hold, 
	we can use interpolation to 
	infer that 
	for some $\theta, \theta'>0$,  
	\begin{align*}
		&\|\cJ_0[h|\nabla|(\overline{\varphi}_{\lambda_1}\psi_{\lambda_2})]\|_{W^{l, a, \beta}(I)} \\
		&\lesssim \max\{\lambda_1, \lambda_2\}^{-\theta'} \|h\|_{B_{2,\infty}^\alpha \cap L^\infty(I)} ( \|\varphi\|_{S^{s,a,0}(I)}\|\psi\|_{S^{s,a,0}(I)})^{1-\theta} (\|\varphi\|_{L^2 (I; L_x^{2^*})} \|\psi\|_{L^2 (I; L_x^{2^*})})^\theta,
	\end{align*} 
	which yields \eqref{prop:triest2} by  summing over $\lambda_1$ and  $\lambda_2$  
	since  
	$\sum_{\lambda_1,\lambda_2} \max\{\lambda_1, \lambda_2\}^{-\theta'} <\infty$.

	\vspace*{4pt plus 2pt minus 2pt}%
	Below we prove \eqref{prop:trinon1} 
	and \eqref{prop:trinon2}. 
	For this purpose, 
	we choose $s'<s$ such that
	\begin{align*}
			2s'-l-\tfrac{d-2}{2}> \max\{3a, 2a+2\varepsilon\}, \quad s'-l>\max\{a+2\varepsilon, \min\{2a,\tfrac{7}{4}\}\}, \quad \min \{s', 2s'-\tfrac{d-2}{2}-a \}> \beta,
	\end{align*}
	and derive from  \eqref{triesti-improve} that
	\begin{align*}
		\|\cJ_0[h|\nabla|(\overline{\varphi}_{\lambda_1}\psi_{\lambda_2})]\|_{W^{l, a, \beta}(I)}&\lesssim \|h\|_{B_{2,\infty}^\alpha\cap L^\infty(I)}  \|\varphi_{\lambda_1}\|_{S^{s',a,0}(I)} \|\psi_{\lambda_2}\|_{S^{s',a,0}(I)}\\
		&\lesssim 
		(\lambda_1 \lambda_2)^{s'-s} \|h\|_{B_{2,\infty}^\alpha\cap L^\infty(I)} \|\varphi_{\lambda_1}\|_{S^{s,a,0}(I)} \|\psi_{\lambda_2}\|_{S^{s,a,0}(I)},
	\end{align*}
	which yields estimate \eqref{prop:trinon1}.
	
	Regarding estimate \eqref{prop:trinon2},
	we let $F_\lambda= \chi_I h|\nabla| P_{ \lambda}(\overline{\varphi}_{\lambda_1}\psi_{\lambda_2})$
	and 
	use the Bernstein inequality to derive 
	\begin{align}\label{prop:BernsteinWaveEstimate}
		\|\cJ_0[F_\lambda]\|_{W^{l,a,\beta}_\lambda}
		&\lesssim \lambda^{l+a} \|\cJ_0[F_\lambda]\|_{L_t^\infty L_x^2} +\lambda^{\beta-1} \|F_\lambda\|_{L_{t,x}^2}\notag\\
		&\lesssim \lambda^{l+a+\beta +\frac{d}{4}} ( \|F_\lambda\|_{L_t^1 L_x^2} + \|F_\lambda\|_{L_t^1 L_x^2}^{\frac12}
		\|F_\lambda\|_{L_t^\infty L_x^1}^{\frac12}).
	\end{align}
	Since $\lambda\lesssim \max\{\lambda_1,\lambda_2\}$ for all non-vanishing terms, 
	we obtain 
	\begin{align}\label{prop:BernsteinWaveEstimate2}
		&\|\cJ_0[h|\nabla|(\overline{\varphi}_{\lambda_1}\psi_{\lambda_2})]\|_{W^{l, a, \beta}(I)}
		\lesssim 	\|\cJ_0[\chi_I h|\nabla|(\overline{\varphi}_{\lambda_1}\psi_{\lambda_2})]\|_{W^{l, a, \beta}} \nonumber \\
		&\lesssim (\max\{\lambda_1,\lambda_2\})^{l+a+\beta +\frac{d}{4}}  
		( \|\chi_I h \overline{\varphi}_{\lambda_1}\psi_{\lambda_2}\|_{L_t^1 L_x^2}+ \|\chi_I h \overline{\varphi}_{\lambda_1}\psi_{\lambda_2}\|_{L_t^1 L_x^2}^{\frac12} \|\chi_I h \overline{\varphi}_{\lambda_1}\psi_{\lambda_2}\|_{L_t^\infty L_x^1}^{\frac12}) \notag\\
		&\lesssim (\max\{\lambda_1,\lambda_2\})^{l+a+\beta +\frac{d}{4}}
		\|h\|_{L^\infty(I)} ( \|\overline{\varphi}_{\lambda_1}\psi_{\lambda_2}\|_{L^1(I; L_x^2)}+ \|\overline{\varphi}_{\lambda_1}\psi_{\lambda_2}\|_{L^1(I; L_x^2)}^{\frac12} \|\overline{\varphi}_{\lambda_1}\psi_{\lambda_2}\|_{L^\infty(I; L_x^1)}^{\frac12}) \notag\\
		&\lesssim \|h\|_{L^\infty(I)} (\max\{\lambda_1, \lambda_2\})^{N} (\|\psi\|_{S^{s,a,0}(I)} \|\varphi\|_{S^{s,a,0}(I)}\|\psi\|_{L^2(I; L_x^{2^*})} \|\varphi\|_{L^2(I; L_x^{2^*})})^{\frac12} 
	\end{align} 
	for some $N>0$, where we used Bernstein estimates and the fact that $d \geq 4$ in the last step.
	This yields estimate \eqref{prop:trinon2}, 
	finishing the proof of~\eqref{prop:triest2}.  
	Therefore, 
	the proof of the trilinear estimates 
	in noise-regularization regime I is complete. 
	\qed

	\subsection{Trilinear estimate in noise-regularization regime II} 
	\label{subsec:TrilinRegimeII}
	We next prove the trilinear estimate 
	\eqref{prop:triestiend} 
	in the noise-regularization regime II. 
	As in the proof of Proposition $6.2$ in \cite{CHN23}, 
	we decompose 
	\begin{align*}
		\overline{\varphi}\psi = \sum_{m\gg 1} \sum_{\lambda\geq m}  \overline{\varphi}_\lambda  \psi_{\frac{\lambda}{m}} +  \sum_\lambda  \overline{\varphi}_\lambda  \psi_{\sim \lambda} + \sum_{m\gg  1}\sum_{\lambda\geq m}  \overline{\varphi}_{\frac{\lambda}{m}}\psi_\lambda. 
	\end{align*} 
	Take $p\in (2,\infty)$ close to $2$ such that $\varepsilon_1: =  -s+\frac{d-4}{2}+\frac{2}{p} \in (0,\frac12)$ and set $\varepsilon_2 := s- \frac{d-4}{2}\in [\frac12, 1)$. We will prove the following estimates
	\begin{align}
		\Big\| \sum_{\lambda\geq m}  \cJ_0( h |\nabla| (\overline{\varphi}_{\lambda} \psi_{\frac{\lambda}{m}})) \Big\|_{W^{l,0,l}(I)}
		&\lesssim m^{\varepsilon_2} \|h\|_{L^\infty(I)} \|\varphi\|_{L^2 (I;W_x^{s,2^*})} \|\psi\|_{L^2(I; W_x^{s,2^*})}\notag\\
		&\quad +m^{-\varepsilon_1} \|h\|_{L^\infty(I)} ( \|\varphi\|_{L^2 (I;W_x^{s,2^*})} \|\psi\|_{L^2(I; W_x^{s,2^*})} \|\varphi\|_{S^{s,0,0}(I)} \|\psi\|_{S^{s,0,0}(I)} )^{\frac12}, \label{prop:triendine3}\\
		\Big\|\sum_{\lambda\in 2^{\N_0}}  \cJ_0( h |\nabla| (\overline{\varphi}_{\lambda} \psi_{\sim \lambda})) \Big\|_{W^{l,0,l}(I)}
		&\lesssim \|h\|_{L^\infty(I)}( \|\varphi\|_{L^2 (I;W_x^{s,2^*})} \|\psi\|_{L^2(I; W_x^{s,2^*})} \|\varphi\|_{S^{s,0,0}(I)} \|\psi\|_{S^{s,0,0}(I)} )^{\frac12}, \label{prop:triendine2}\\
		\Big\| \sum_{\lambda\geq m}  \cJ_0( h |\nabla| (\overline{\varphi}_{\lambda} \psi_{\frac{\lambda}{m}})) \Big\|_{W^{l,0,l}(I)}
		&\lesssim m^{-\varepsilon_1} \|h\|_{B_{p,\infty}^{\frac{1}{p}}\cap L^\infty(I)} \|\varphi\|_{S^{s,0,0}(I)} \|\psi\|_{S^{s,0,0}(I)}.  \label{prop:triendine1} 
	\end{align}
	The invariance by complex conjugation then yields \eqref{prop:triendine3} and  \eqref{prop:triendine1} for the low-high interaction $\sum_{\lambda\geq m} \overline{\varphi}_{\frac{\lambda}{m}}\psi_\lambda$.
	
	Assuming these estimates to hold, 
	letting $M\gg 1$ and applying  \eqref{prop:triendine3} for $m<M$, \eqref{prop:triendine1} for the high-low and low-high interactions when $m\geq M$, and \eqref{prop:triendine2} for the high-high interaction, we get
	\begin{align*}
		&\|\cJ_0[h|\nabla|(\overline{\varphi}\psi)]\|_{W^{l, 0, l}(I)}\\
		&\lesssim \sum_{m\gg 1} \Big( \Big\| \sum_{\lambda\geq m} \cJ_0[h|\nabla|(\overline{\varphi}_\lambda \psi_{\frac{\lambda}{m}})] \Big\|_{W^{l,0,l}(I)} + \Big\| \sum_{\lambda\geq m} \cJ_0[h|\nabla|(\overline{\varphi}_{\frac{\lambda}{m}} \psi_{\lambda})]
		\Big\|_{W^{l,0,l}(I)} \Big) \\
		&\qquad + \Big\| \sum_\lambda \cJ_0[h|\nabla|(\overline{\varphi}_\lambda \psi_{\sim \lambda})] \Big\|_{W^{l,0,l}(I)}\\
		&\lesssim M^{\varepsilon_2} \|h\|_{L^\infty(I)} \|\varphi\|_{L^2(I; W_x^{s,2^*})} \|\psi\|_{L^2(I; W_x^{s,2^*})} +  M^{-\varepsilon_1}  \|h\|_{B_{p,\infty}^{\frac{1}{p}} \cap L^\infty(I)}\|\varphi\|_{S^{s,0,0}(I)} \|\psi\|_{S^{s,0,0}(I)}\\
		&\qquad + \|h\|_{L^\infty(I)} ( \|\varphi\|_{L^2(I; W_x^{s,2^*})} \|\psi\|_{L^2(I; W_x^{s,2^*})}\|\varphi\|_{S^{s,0,0}(I)} \|\psi\|_{S^{s,0,0}(I)})^{\frac12}.
	\end{align*}
	Then, choosing $M$ as $(\|\varphi\|_{L^2(I; W_x^{s,2^*})} \|\psi\|_{L^2(I; W_x^{s,2^*})})^{-1}\|\varphi\|_{S^{s,0,0}(I)} \|\psi\|_{S^{s,0,0}(I)}$ multiplied with a large constant (to ensure $M \gg 1$), we obtain the desired estimate \eqref{prop:triestiend}.
	
	\vspace*{3pt plus 2pt minus 2pt}
	Below we turn to the proof of estimates  \eqref{prop:triendine3},  \eqref{prop:triendine2}, and  \eqref{prop:triendine1}. 
	Note  that estimate \eqref{prop:triendine2} follows directly from  \eqref{prop:triendine3} with $m\sim 1$. 
	Hence, 
	we only need to prove \eqref{prop:triendine3} for any $m \in 2^{\N_0}$
	and \eqref{prop:triendine1}. 
	
	For this purpose, we set 
	$G_\lambda:= h \overline{\varphi}_\lambda 
	\psi_{\frac{\lambda}{m}}$ 
	and use the energy inequality to get for any $m \in 2^{\N_0}$
	\begin{align}\label{prop:triend12}
		\Big\|\sum_{\lambda \geq m} \cJ_0[|\na| G_\lambda] \Big\|_{W^{l,0,l}(I)}
		&\lesssim \Big\| \sum_{\lambda \geq m} \cJ_0[\chi_I |\na| G_\lambda] \Big\|_{W^{l,0,l}} \nonumber \\ 
		&\lesssim \Big(\sum_{\lambda \geq m} \lambda^{2l}\|\cJ_0[\chi_I |\na| G_\lambda]\|^2_{L_t^\infty L_x^2} 
		+ \lambda^{2(l-1)}\|\chi_I |\na| G_\lambda\|_{L_{t,x}^2}^2 \Big)^{\frac12}\notag\\
		&\lesssim \Big(\sum_{\lambda \geq m} \lambda^{2{l+2}}
		\|G_\lambda\|^2_{L^1(I; L_x^2)} + \lambda^{2l} 
		\|G_\lambda\|_{L^2(I; L_x^2)} \Big)^{\frac12}.
	\end{align}
	
	To prove the estimate \eqref{prop:triendine3},  
	by the 
	Minkowski inequality and the Littlewood-Paley theorem, 
	we have
	\begin{align*}
		\Big(\sum_{\lambda\geq m} \lambda^{2(l+1)} \|G_\lambda\|^2_{L^1(I; L_x^2)} \Big)^{\frac12}
		&\lesssim m^{l+1-s} \|h\|_{L^\infty(I)} \Big\| \Big(\sum_{\lambda\in 2^{\N_0}} \lambda^{2s} |\varphi_\lambda|^2\Big)^{\frac12} \Big(\sum_{\lambda\in 2^{\N_0}} \lambda^{2(l+1-s)} |\psi_\lambda|^2\Big)^{\frac12} \Big\|_{L^1(I; L_x^2)}\\
		&\lesssim m^{l+1-s} \|h\|_{L^\infty(I)} \|\varphi\|_{L^2(I; W_x^{s,2^*})} \|\psi\|_{L^2(I; W_x^{l+1-s,d})}\\
		&\lesssim  m^{\varepsilon_2} \|h\|_{L^\infty(I)} \|\varphi\|_{L^2(I; W_x^{s,2^*})} \|\psi\|_{L^2(I; W_x^{s,2^*})},
	\end{align*}
	where we used $l+1-s+\frac{d-4}{2}=s$ and $s-\frac{d-4}{2}=\varepsilon_2$.
	Moreover, 
	for  the $L_{t,x}^2$-norm 
	on the right-hand side of \eqref{prop:triend12}, 
	we estimate 
	\begin{align}\label{prop:Ltx2Esti}
		&\Big( \sum_{\lambda\geq m}  \lambda^{2l} \|G_\lambda \|^2_{L^2(I; L_x^2)}\Big)^{\frac12} \lesssim m^{l-s} \|h\|_{L^\infty(I)} \Big\| \Big(\sum_{\lambda\in 2^{\N_0}} \lambda^{2s} |\varphi_\lambda|^2\Big)^{\frac12} \sup_{\lambda\in 2^{\N_0}} \lambda^{l-s}|\psi_\lambda| \Big\|_{L^2(I; L_x^2)}\notag\\
		&\lesssim m^{l-s} \|h\|_{L^\infty(I)} ( \|\varphi\|_{L^2(I; W_x^{s,2^*})} \|\psi\|_{L^\infty(I; W_x^{l-s,d})} )^{\frac12} 
		( \|\varphi\|_{L^\infty(I;H_x^s)} \|\sup_{\lambda\in 2^{\N_0}} \lambda^{l-s}|\psi_\lambda| \|_{L^2(I; L_x^\infty)} )^{\frac12}\notag \\
		&\lesssim   m^{s-\frac{d-2}{2}} \|h\|_{L^\infty(I)}  ( \|\varphi\|_{L^2(I; W_x^{s,2^*})} \|\psi\|_{L^2(I; W_x^{s,2^*})}\|\varphi\|_{S^{s,0,0}(I)} \|\psi\|_{S^{s,0,0}(I)})^{\frac12},
	\end{align}
	where we used $l-s+\frac{d-4}{2}=s-1<s$ and $l-s + \frac{d-2}{2} = s$. 
	Thus,  
	\eqref{prop:triendine3} follows from the estimates above and that  $s-\frac{d-2}{2}<s-\frac{d-4}{2}-\frac{2}{p} = -\varepsilon_1$.

	Regarding estimate \eqref{prop:triendine1}, we consider arbitrary extensions from $I$ to $\R$ of $h, \varphi$, and $\psi$. Then we decompose by high and low temporal frequency. 
	In view of the estimate (see \cite[Page 3190]{CHN23})
	\begin{align*}
		\|\cJ_0(P^{(t)}_{\gtrsim \lambda^2} |\na|  G_\lambda)\|_{L_t^\infty L_x^2} \lesssim \lambda \sum_{\nu\gtrsim \lambda^2} \nu^{-1} \|P^{(t)}_{\nu}  G_\lambda\|_{L_t^\infty L_x^2} \lesssim   \|G_\lambda\|_{L_{t,x}^2}
	\end{align*}
	for any $\lambda \gg 1$, the $L_t^\infty L_x^2$-component in \eqref{prop:triend12} (with the chosen extensions of $h$, $\varphi$, and $\psi$ instead of the zero extension now) can be bounded by the $L_{t,x}^2$-component, where we assume $m \gg 1$ now.
	Taking into account estimate  \eqref{prop:Ltx2Esti},
	we thus have the bound in \eqref{prop:triendine1} 
	for the high temporal frequency component 
	$\cJ_0(P^{(t)}_{\gtrsim \lambda^2} |\na| G_\lambda)$. 
	
	It remains to prove   \eqref{prop:triendine1}
	for the low temporal component 
	$\cJ_0(P^{(t)}_{\ll \lambda^2} |\na| G_\lambda)$. 
	In view of \eqref{prop:triend12} and \eqref{prop:Ltx2Esti}, it suffices to control the corresponding $L_t^1 L_x^2$-norm in \eqref{prop:triend12}. 
	
	For this purpose, we further decompose
	$$
	P^{(t)}_{\ll \lambda^2} 
	G_\lambda 
	= P^{(t)}_{\ll \lambda^2}(h C_{\gtrsim \lambda^2}\overline{\varphi}_\lambda \psi_{\frac{\lambda}{m}})+ P^{(t)}_{\ll \lambda^2}(h C_{\ll \lambda^2}\overline{\varphi}_\lambda \psi_{\frac{\lambda}{m}})
	=: I_{HM}+ I_{LM}.
	$$
	For the high modulation contribution $I_{HM}$, 
	since $l-s =s-\frac{d-2}{2}<-\varepsilon_1$ and $l-s+\frac{d-2}{2}=s$, we have
	\begin{align}\label{prop:estiILMLT3}
		\big( \sum_{\lambda \geq m} \lambda^{2(l+1)} \| I_{HM} \|^2_{L_t^1 L_x^2} \big)^{\frac12} 
		\lesssim~& \|h\|_{L_t^\infty} m^{l-s} \Big( \sum_{\lambda\geq m} \lambda^{2(s+1)} \|C_{\gtrsim \lambda^2}\overline{\varphi}_\lambda \|^2_{L_{t,x}^2} \Big)^{\frac12}\sup_\lambda \lambda^{l-s} \|\psi_\lambda\|_{L_t^2 L_x^\infty}\notag\\
		\lesssim~& m^{-\varepsilon_1} \|h\|_{L_t^\infty} \|\varphi\|_{S^{s,0,0}} \|\psi\|_{L_t^2 W_x^{s,2^*}}.
	\end{align} 
	Regarding the low modulation $I_{LM}$, we further decompose and exploit a \textit{non-resonant identity} again  
	to obtain
	\begin{align*}
		I_{LM}= P^{(t)}_{\ll \lambda^2}(h C_{\ll \lambda^2}\overline{\varphi}_\lambda P^{(t)}_{\geq (\frac{\lambda}{2^4})^2}\psi_{\frac{\lambda}{m}})+ P^{(t)}_{\ll \lambda^2}( P^{(t)}_{\sim \lambda^2}h C_{\ll \lambda^2}\overline{\varphi}_\lambda P^{(t)}_{< (\frac{\lambda}{2^4})^2}\psi_{\frac{\lambda}{m}})=: I_{LM,HT}+ I_{LM,LT}.
	\end{align*}
	For the first term $I_{LM,HT}$, 
	by H\"older's inequality, 
	\begin{align}\label{prop:estiILMLT2}
		\Big( \sum_{\lambda \geq m} \lambda^{2(l+1)} \| I_{LM,HT} \|^2_{L_t^1 L_x^2} \Big)^{\frac12} &\lesssim \|h\|_{L_t^\infty} \Big( \sum_{\lambda \geq m} \lambda^{2s} \| C_{\ll \lambda^2}\overline{\varphi}_\lambda\|_{L_t^2 L_x^{2^*}}^2 \lambda^{2(l+1-s)}\|P^{(t)}_{\geq (\frac{\lambda}{2^4})^2}\psi_{\frac{\lambda}{m}} \|^2_{L_t^2 L_x^d} \Big)^{\frac12}.
	\end{align}
	Note that $P^{(t)}_{\geq (\frac{\lambda}{2^4})^2} \psi_{\frac{\lambda}{m}} = C_{\gtrsim \lambda^2} P^{(t)}_{\geq (\frac{\lambda}{2^4})^2} \psi_{\frac{\lambda}{m}}$ as $m \gg 1$. Using Bernstein's estimate and $l-s+\frac{d-2}{2}=s$ once again, we deduce
	\begin{align*}
		\lambda^{l+1-s} \|P^{(t)}_{\geq (\frac{\lambda}{2^4})^2} \psi_{\frac{\lambda}{m}}\|_{L_t^2 L_x^d} 
		&\lesssim \lambda^{l+1-s} (\tfrac{\lambda}{m})^{\frac{d-2}{2}}\|C_{\gtrsim \lambda^2} P^{(t)}_{\geq (\frac{\lambda}{2^4})^2} \psi_{\frac{\lambda}{m}}\|_{L_t^2 L_x^2}\\
		&\lesssim \lambda^{s-1} m^{-\frac{d-2}{2}} \|(\imu \partial_t+\Delta) \psi_{\frac{\lambda}{m}}\|_{L_t^2 L_x^2} 
		\lesssim m^{s-1-\frac{d-2}{2}} \|(\imu \partial_t+\Delta) \psi_{\frac{\lambda}{m}}\|_{L_t^2 H_x^{s-1}}.
	\end{align*}  
	In combination with \eqref{prop:estiILMLT2}, we obtain
	\begin{align}
		\label{eq:EstRegIIILMHT}
		\Big( \sum_{\lambda \geq m} \lambda^{2(l+1)} \| I_{LM,HT} \|^2_{L_t^1 L_x^2} \Big)^{\frac12} &\lesssim m^{s-1-\frac{d-2}{2}} \|h\|_{L_t^\infty} \sup_{\lambda} \lambda^s \|C_{\ll \lambda^2}\varphi_\lambda\|_{L_t^2 L_x^{2^*}} \Big( \sum_{\lambda \geq m}   \|(\imu \partial_t+\Delta) \psi_{\frac{\lambda}{m}}\|^2_{L_t^2 H_x^{s-1}} \Big)^{\frac12} \nonumber\\
		&\lesssim m^{-1-\varepsilon_1} \|h\|_{L_t^\infty} \|\varphi\|_{L_t^2 W_x^{s, 2^*}} \|\psi\|_{S^{s,0,0}}.
	\end{align}
	For the second term  $I_{LM,LT}$, we choose $q \in (2,\infty)$ and $\delta \in (0,1)$ such that $\frac{1}{p}+ \frac{1}{q}= \frac12$ and $\frac{1}{q}= \frac{\delta}{2}+ \frac{1-\delta}{\infty}$.
	We then get 
	\begin{align} \label{prop:estiILMLT1} 
		\Big( \sum_{\lambda\geq m} \lambda^{2(l+1)}\|I_{LM,LT}\|_{L_t^1 L_x^2}^2  \Big)^{\frac12}
		&\lesssim \sup_{\lambda} \lambda^s\|\varphi_\lambda\|_{L_t^2L_x^{2^*}}\Big(\sum_{\lambda\geq m}\lambda^{2(l+1-s-\frac{2}{p})}\||\partial_t|^{\frac{1}{p}}P^{(t)}_{\sim \lambda^2} h\|_{L_t^p}^2 \|\psi_{\frac{\lambda}{m}}\|^2_{L_t^q L_x^d}\Big)^{\frac12}\notag\\
		&\lesssim  \|h\|_{B_{p,\infty}^{\frac{1}{p}}} \|\varphi\|_{L_t^2 W_x^{s, 2^*}}  m^{l+1-s-\frac{2}{p}} \Big( \sum_{\lambda\geq 1} \lambda^{2(l+1-s-\frac{2}{p})} \| \psi_{\lambda}\|_{L^q_t L_x^d}^2\Big)^{\frac12} \notag\\
		&\lesssim \|h\|_{B_{p,\infty}^{\frac{1}{p}}} \|\varphi\|_{L_t^2 W_x^{s, 2^*}}  m^{l+1-s-\frac{2}{p}} \Big( \sum_{\lambda\geq 1} \lambda^{2(l+1-s-\frac{2}{p})} \| \psi_{\lambda}\|^{2\delta}_{L^2_t L_x^d} \| \psi_{\lambda}\|^{2-2\delta}_{L^\infty_t L_x^d}\Big)^{\frac12} \notag\\
		&\lesssim  m^{-\varepsilon_1} \|h\|_{B_{p,\infty}^{\frac{1}{p}}} \|\varphi\|_{L_t^2 W_x^{s, 2^*}}  \|\psi\|_{S^{s,0,0}},
	\end{align}		
	where in the last step we 
	also used the estimate
	\begin{align*}
		\| \psi_{\lambda}\|^{2\delta}_{L^2_t L_x^d} \| \psi_{\lambda}\|^{2-2\delta}_{L^\infty_t L_x^d} &\lesssim \Big(\lambda^{\frac{d-4}{2}} \|\psi_\lambda\|_{L_t^2 L_x^{2^*}}\Big)^{2\delta} \Big(\lambda^{\frac{d-2}{2}} \|\psi_\lambda\|_{L_t^\infty L_x^{2}}\Big)^{2-2\delta} \\
		&\lesssim \lambda^{2\delta(\frac{d-4}{2}-s)+(2-2\delta)(\frac{d-2}{2}-s)} \|\psi_\lambda\|^2_{S_\lambda^{s,0,0}},
	\end{align*}
	and the equality 
	\begin{align*}
		2(l+1-s-\tfrac{2}{p})+2\delta(\tfrac{d-4}{2}-s)+ (2-2\delta)(\tfrac{d-2}{2}-s) =0
	\end{align*}
	as $l + 1 - s = s - \frac{d-4}{2}$ in noise-regularization regime II,
	and the definition of $\varepsilon_1$. 
	
	Finally, 
	combining estimates 
	\eqref{prop:estiILMLT3}, \eqref{eq:EstRegIIILMHT}, and \eqref{prop:estiILMLT1}, 
	we control the $L^1_tL^2_x$-norm 
	of the low temporal component 
	in \eqref{prop:triend12} 
	by the right-hand side of~\eqref{prop:triendine1}. 
	
	Therefore, the proof of the trilinear estimate 
	in noise-regularization regime II is complete. 
	\qed
	
	\subsection{Trilinear estimate in noise-regularization regime III}
	It remains to prove the 
	trilinear estimate  \eqref{triesti-LocSmo} in the noise-regularization regime III. We first note that due to the condition $s - l > - \frac12$ there exists $p \in (2,\infty)$ close to $2$ such that
	\begin{align*}
		s - l > 1 - \tfrac{3}{p}.
	\end{align*}
	Then we fix $0 < \varepsilon \ll 1$ such that 
	\begin{equation}
		\label{eq:RegIIIpeps}
		s - l > 1 - \tfrac{3 - 2 \varepsilon}{p}.
	\end{equation}
	Note that the restriction $p > 2$ dictated by the regularity of Brownian motion excludes the endpoint line $s = l - \frac12$ already in the first step of this construction. The $\varepsilon$-loss incurred from using the near-endpoint local smoothing and maximal function spaces (compared to the endpoint spaces) therefore does not restrict the validity of the trilinear estimate -- and thus Theorem~\ref{thm:RegNoise}.
	
	As in the proof of \eqref{triesti-improve}, we consider arbitrary extensions of $h$, $\varphi$, and $\psi$ from $I$ to $\R$ and prove the estimate on $\R$. As in noise-regularization regime I, an application of Lemma~2.6 from~\cite{CHN23} reduces the estimate to proving 
	\eqref{eq:NoiseRegE41}-\eqref{eq:NoiseRegE44} with $a = 0$, the $\Z_\varepsilon^{s,b}$-norm on the right-hand side, and $B^\alpha_{2,\infty}$ replaced with $B^{\frac{1}{p}}_{p,\infty}$. Since $\Z_\varepsilon^{s,b}\hookrightarrow S^{s,0,b}$ in the case $a = 0$, we obtain from the proof of estimate~\eqref{triesti-improve} 
	that \eqref{eq:NoiseRegE42}-\eqref{eq:NoiseRegE44} hold true
	provided condition~\eqref{con:wavetri234} with $a = 0$ is valid, i.e.,
	\begin{align*}
		&s+1\geq l,\quad 2s\geq l+\tfrac{d-2}{2},\quad (s,l)\neq(\tfrac{d}{2},\tfrac{d+2}{2})\notag\\
		& s\geq 0,\quad s\geq \beta,\quad 2s-\beta-\tfrac{d-2}{2}\geq 0,\quad (s,\beta) \neq (\tfrac{d-2}{2},\tfrac{d-2}{2}).
	\end{align*}
	
	Hence, we only need to prove the counterpart of estimate \eqref{eq:NoiseRegE41},  
	i.e., 
	\begin{align}
		\Big( \sum_{\lambda\in 2^\N} \lambda^{2(l+1)}\| P^{(t)}_{\ll \lambda^2} P_{\lambda}(h\overline{\varphi}\psi)\|_{L_t^1L_x^2}^2\Big)^{\frac12}&\lesssim \|h\|_{B^{\frac{1}{p}}_{p,\infty}\cap L_t^\infty} \|\varphi\|_{\Z_\varepsilon^{s,b}}\|\psi\|_{\Z_\varepsilon^{s,b}}.  \label{eq:NoiseRegELocSmo41}
	\end{align}

	For this purpose, 
	we use the same decomposition as in the proof of \eqref{triesti-improve} for the high-low  and high-high interactions  
	(the estimate for the low-high interaction again follows from the invariance of the estimate under complex conjugation):  
	\begin{align*}
		I^{HL} = I^{HL}_{HM} + I^{HL}_{LM} 
		= I^{HL}_{HM,HT} + I^{HL}_{HM,LT} + I^{HL}_{LM,LT} + I^{HL}_{LM,HT},
	\end{align*}
	and 
	\begin{align*}
		I^{HH}=  I^{HH}_{LT} + I^{HH}_{HT} = I^{HH}_{LT} + I^{HH}_{HT.1}+ I^{HH}_{HT.2}+ I^{HH}_{HT.3}.
	\end{align*}
	Since $\Z_\varepsilon^{s,b}\hookrightarrow S^{s,0,b}$, the estimates established in the proof of~\eqref{triesti-improve} immediately yield~\eqref{eq:NoiseRegELocSmo41} for those components which were estimated using only the $L^\infty_t$-norm of $h$.
	Collecting the corresponding conditions with $a = 0$,
	we thus obtain estimate~\eqref{eq:NoiseRegELocSmo41} for 
	all components of the above decompositions of $I^{HL}$ and $I^{HH}$
	except $I^{HL}_{HM,HT}$, $I^{HL}_{LM,HT}$, and $I^{HH}_{HT.3}$, 
	provided that
	\begin{align*}
		s-l\geq -b, \quad 2s-l-\tfrac{d-2}{2}\geq 0,\quad (s,l) \neq 
		(\tfrac{d-2}{2}, \tfrac{d-2}{2}+b ).
	\end{align*}
	It only remains to estimate the three components $I^{HL}_{HM,HT}$, 
	$I^{HL}_{LM,HT}$, and $I^{HH}_{HT.3}$.

	For $I^{HL}_{HM,HT}$ defined in \eqref{HLHM-decom}, 
	we choose $q \in (2,\infty)$ and $\theta \in (0,1)$ such that $\frac{1}{p}+ \frac{1}{q}=\frac{1}{2}$ and $\frac{1}{q} = \frac{\theta}{2} + \frac{1-\theta}{\infty}$. By interpolation, we then infer
	\begin{align*}
		&\mu^{l+1} \| P^{(t)}_{\ll \lambda^2} ( P^{(t)}_{\gtrsim \mu^2} h \overline{C_{\gtrsim \mu^2}\varphi_\mu} \psi_{\ll\mu} )\|_{L_t^1 L_x^2}\\
		&\lesssim  \mu^{l+1} \| P^{(t)}_{\gtrsim \mu^2} h\|_{L_t^p} \|C_{\gtrsim \mu^2}\varphi_\mu\|_{L_{t,x}^2} \|\psi_{\ll\mu}\|_{L_t^q L_x^\infty}\\
		&\lesssim \mu^{l+1-\frac{2}{p}} \|h\|_{B^{\frac{1}{p}}_{p,\infty}} \mu^{-2-s+1-b} \|\varphi_\mu\|_{S_\mu^{s,0,b}} \sum_{\nu \ll \mu} \|\psi_{\nu}\|^\theta_{L_t^2 L_x^\infty} \|\psi_{\nu}\|^{1-\theta}_{L_t^\infty L_x^\infty} \\
		&\lesssim  \mu^{l-\frac{2}{p}-s-b} \|h\|_{B^{\frac{1}{p}}_{p,\infty}} \|\varphi_\mu\|_{\Z^{s,b}_{\varepsilon,\mu}} \sum_{\nu \ll \mu} \nu^{\theta(\frac{d-2}{2}-s)}\|\psi_{\nu}\|_{L_t^2 W_x^{s, 2^*}}^\theta \nu^{(1-\theta)(\frac{d}{2}-s)}\|\psi_{\nu}\|_{L_t^\infty H_x^{s}}^{1-\theta} \\
		&\lesssim  \mu^{l-\frac{2}{p}-s-b} \|h\|_{B^{\frac{1}{p}}_{p,\infty}} \|\varphi_\mu\|_{\Z^{s,b}_{\varepsilon,\mu}} \sum_{\nu \ll \mu}  \nu^{\frac{d}{2}-s - \theta}\|\psi_{\nu}\|_{S^{s,0,b}_\nu},
	\end{align*}
	which is $l^2$-summable and yields estimate \eqref{eq:NoiseRegELocSmo41} for the $I^{HL}_{HM,HT}$ component, provided that
	\begin{align*}
		s - l \geq - b - \tfrac{2}{p}, \quad 2 s - l - \tfrac{d-2}{2} \geq - b, \quad (s,l) \neq (\tfrac{d-2}{2} + \tfrac{2}{p}, \tfrac{d-2}{2} + \tfrac{4}{p} + b),
	\end{align*}
	where we also used $\theta = 1 - \frac{2}{p}$. Note that the last condition is satisfied as $s - l \geq - b$ in regime III.

	Regarding $I^{HL}_{LM,HT}$
	defined in \eqref{HLLM-decom}, 
	recall the definition of $p_\varepsilon$ and $q_\varepsilon$ in Section \ref{Sec-MaximalFunctionEstimate}, 
	and choose $r \in (2,\infty)$ and $\theta \in (0,1)$ such that $\frac{1}{p}+ \frac{1}{r}=1$ and $\frac{1}{r} = \frac{\theta}{2} + \frac{1-\theta}{1}$. We already note that this implies $\theta = \frac{2}{p}$.
	By interpolation, we infer 
	\begin{align}
		\label{eq:RegIIIEstIHLLMHT}
		\mu^{l+1} \| P^{(t)}_{\ll \lambda^2} ( P^{(t)}_{\gtrsim \mu^2} h \overline{C_{\ll\mu^2}\varphi_\mu} \psi_{\ll\mu} )\|_{L_t^1 L_x^2}
		&\lesssim \mu^{l+1} \|P^{(t)}_{\gtrsim \mu^2} h\|_{L_t^p} \sum_{\nu \ll \mu} \| \overline{C_{\ll\mu^2}\varphi_\mu} \psi_{\nu} \|_{L_t^r L_x^2} \nonumber\\
		&\lesssim  \mu^{l+1-\frac{2}{p}} \|h\|_{B^{\frac{1}{p}}_{p,\infty}} \sum_{\nu \ll \mu} \| \overline{C_{\ll\mu^2}\varphi_\mu} \psi_{\nu} \|_{L_t^2 L_x^2}^\theta \| \overline{C_{\ll\mu^2}\varphi_\mu} \psi_{\nu} \|_{L_t^1 L_x^2}^{1-\theta}.
	\end{align}
	We next estimate the two terms we got from interpolating. Using the definition of the $\Z^{s,b}_{\epsilon}$-norm, we get
	\begin{align*}
		\| \overline{C_{\ll\mu^2}\varphi_\mu} \psi_{\nu} \|_{L_t^2 L_x^2} 
		\lesssim \sum_{j = 1}^d \|P_{\mu,\vece_j} C_{\ll\mu^2} \varphi_\mu\|_{L_{\vece_j}^{p_\varepsilon, q_\varepsilon}}  \|\psi_\nu\|_{L_{\vece_j}^{q_\varepsilon,p_\varepsilon}}
		\lesssim \mu^{-s - \frac{1}{2} + \varepsilon} \| \varphi_\mu \|_{\Z^{s,0}_{\epsilon,\mu}}     \nu^{-s + \frac{d-1}{2} - \varepsilon} \| \psi_\nu \|_{\Z^{s,0}_{\epsilon, \nu}}
	\end{align*}
	for the first term and
	\begin{align*}
		\| \overline{C_{\ll\mu^2}\varphi_\mu} \psi_{\nu} \|_{L_t^1 L_x^2}
		\lesssim \| \phi_\mu \|_{L^2_t L^{2^*}_x} \| \psi_\nu \|_{L^2_t L^d_x}
		\lesssim \mu^{-s} \| \varphi_\mu \|_{S^{s,0,0}_\mu} \nu^{\frac{d-4}{2}} \| \psi_\nu \|_{L^2_t L^{2^*}_x} \lesssim \mu^{-s} \| \varphi_\mu \|_{S^{s,0,0}_\mu} \nu^{\frac{d-4}{2} - s} \| \psi_\nu \|_{S^{s,0,0}_\nu}
	\end{align*}
	for the second term. Plugging the last two estimates into~\eqref{eq:RegIIIEstIHLLMHT} and employing the embeddings $\Z^{s,b}_\varepsilon \hookrightarrow S^{s,0,b} \hookrightarrow S^{s,0,0}$, we arrive at
	\begin{align*}
		&\mu^{l+1} \| P^{(t)}_{\ll \lambda^2} ( P^{(t)}_{\gtrsim \mu^2} h \overline{C_{\ll\mu^2}\varphi_\mu} \psi_{\ll\mu} )\|_{L_t^1 L_x^2} \\
		&\lesssim \mu^{l+1-\frac{2}{p}} \|h\|_{B^{\frac{1}{p}}_{p,\infty}} \mu^{(-s - \frac12 + \varepsilon)\theta} \mu^{-s(1-\theta)} \| \varphi_\mu \|_{\Z^{s,b}_{\epsilon,\mu}} \sum_{\nu \ll \mu} \nu^{(-s + \frac{d-1}{2} - \varepsilon) \theta} \nu^{(\frac{d-4}{2} - s)(1-\theta)}  \| \psi_\nu \|_{\Z^{s,b}_{\varepsilon, \nu}} \\
		&\lesssim \mu^{l + 1 - s - \frac{3 - 2 \varepsilon}{p}} \|h\|_{B^{\frac{1}{p}}_{p,\infty}} \| \varphi_\mu \|_{\Z^{s,b}_{\epsilon,\mu}} \sum_{\nu \ll \mu} \nu^{\frac{d-4}{2} - s + \frac{3 - 2 \varepsilon}{p}} \| \psi_\nu \|_{\Z^{s,b}_{\varepsilon, \nu}}.
	\end{align*}
	The above right-hand side is $l^2$-summable and yields estimate \eqref{eq:NoiseRegELocSmo41} for the $I^{HL}_{LM,HT}$ component provided that
	\begin{align*}
		s - l \geq 1 - \tfrac{3 - 2 \varepsilon}{p}, \quad 2 s - l - \tfrac{d-2}{2} \geq 0, \quad (s,l) \neq (\tfrac{d-4}{2} + \tfrac{3 - 2 \varepsilon}{p}, \tfrac{d-6}{2} + \tfrac{6 - 4 \varepsilon}{p})
	\end{align*}
	Note that the last condition is satisfied as $s - l > 1 - \frac{3 - 2 \varepsilon}{p}$ by~\eqref{eq:RegIIIpeps}.

	For the remaining $I^{HH}_{HT.3}$ term, defined in \eqref{HHHT-decom}, 
	we recall that $\lambda_1\sim\lambda_2\gtrsim \lambda$ in this high-high interaction. 
	Let $q \in (2,\infty)$ and $\theta \in (0,1)$ such that $\frac{1}{p} + \frac{1}{q}=\frac{1}{2}$ and $\frac{1}{q} = \frac{\theta}{2} + \frac{1-\theta}{\infty}$. Since $l+1>0$, we use the interpolation and the embedding $\Z_{\varepsilon, \lambda}^{s,b}\hookrightarrow S_{\lambda}^{s,0,b}$ again to  get
	\begin{align*}
		&\lambda^{l+1} \| P^{(t)}_{\ll \lambda^2} ( P^{(t)}_{\gtrsim \lambda_1^2} h P_\lambda (\overline{C_{\ll \lambda_1^2}\varphi_{\lambda_1}} \psi_{\lambda_2}) )\|_{L_t^1 L_x^2}\\
		&\lesssim  \lambda^{l+1} \|P^{(t)}_{\gtrsim \lambda_1^2} h\|_{L_t^p} \|C_{\ll \lambda_1^2}\varphi_{\lambda_1}\|_{L_t^2 L_x^{2^*}} \|\psi_{\lambda_2}\|_{L_t^q L_x^d}\\
		&\lesssim  \lambda_1^{l+1-\frac{2}{p}} \|h\|_{B^{\frac{1}{p}}_{p,\infty}} \lambda_1^{-s} \|\varphi_{\lambda_1}\|_{S_{\lambda_1}^{s,0,0}} \|\psi_{\lambda_2}\|^\theta_{L_t^2 L_x^d} \|\psi_{\lambda_2}\|^{1-\theta}_{L_t^\infty L_x^d}\\
		&\lesssim  \lambda_1^{l+1-\frac{2}{p}} \|h\|_{B^{\frac{1}{p}}_{p,\infty}} \lambda_1^{-s} \|\varphi_{\lambda_1}\|_{\Z^{s,b}_{\varepsilon,\lambda_1}} \lambda_2^{\theta(\frac{d-4}{2}-s)+(1-\theta)(\frac{d-2}{2}-s)} \|\psi_{\lambda_2}\|_{\Z^{s,b}_{\varepsilon,\lambda_2}},
	\end{align*}
	which is $l^2$-summable and gives estimate \eqref{eq:NoiseRegELocSmo41} for the $I^{HH}_{HT.3}$ component, provided that
	\begin{align*}
		2s-l-\tfrac{d-2}{2}\geq 0.
	\end{align*}
	
	We thus showed \eqref{eq:NoiseRegELocSmo41}
	for the remaining terms $I^{HL}_{LM,HL}$,  $I^{HL}_{HM,HT}$, and $I^{HH}_{HT.3}$, completing the proof of~\eqref{eq:NoiseRegELocSmo41}. 
	This concludes the proof of Theorem \ref{Prop-trilinear}. 
	\qed

	\section{Noise-regularization on scattering}
	\label{Sec-Noise-Regular}
	
	In this section we prove the 
	noise-regularization-effect on scattering in Theorem~\ref{thm:RegNoise}. We recall that by the rescaling transform the stochastic Zakharov system~\eqref{eq:StoZak} with nonconservative noise is equivalent to the random system~\eqref{eq:RanZakNoncons-intro}. To prove Theorem~1.5, we thus show that the probability of the event
	\begin{align}
		\Upsilon_{(\Im \phi^{(1)}_k)} := \big\{\omega \in \Omega \colon &(z,v) \text{ is the solution of~\eqref{eq:RanZakNoncons-intro} on } [0,\infty) \text{ and there is } (z_+, v_+) \in H_x^s \times H_x^l \text{ s.t. } \nonumber\\
		&\lim_{t \rightarrow \infty} \|e^{-\imu t \Delta} z(t) - z_+\|_{H_x^s} = 0 \text{ and } \lim_{t \rightarrow \infty} \|e^{-\imu t |\nabla|} v(t) - v_+\|_{H_x^l} = 0 \big\} \label{eq:DefUpsilon}
	\end{align}
	converges to one as $\| \Im \phi_k^{(1)} \| \rightarrow \infty$. Note that this event depends on $(\Im \phi_k^{(1)})$, i.e., on the coefficients of the noise in the Schr\"odinger component.
	To ease the notation, we set $\mathbf{c} := (\Im \phi_k^{(1)})$ in the following.
	
	We will use the following estimate for the Schr\"odinger equation with a free wave potential, which follows from combining Remark \ref{rem:PropagOPWellDef} with the embedding $\X^{s,a} \hookrightarrow S^{s,a,0}$. 
	Recall that $\cU_{v_L}[f]$ and $\mathcal{I}_{v_L}[F]$, respectively, 
	denote the homogeneous and inhomogeneous solutions to equation \eqref{eq:SchrPotentialIntro} with potential $V=v_L$ and initial time $t_0$.

	\begin{lemma}\label{lem:LinSchrPotSmallTimeSsa0}
		Let $0\leq s\leq l+2$, $l\geq \frac{d-4}{2}$ with $(s,l)\neq (\frac{d}{2},\frac{d-4}{2})$, and $a=a^*(s,l)$ given by \eqref{con:ab}.
		Then, there exists $C(v_L) \geq 1$ such that for any interval $I \subset \R$, $t_0 \in I$, $f \in H_x^s$ with $F \in N^{s,a,0}(I)$,   
		\begin{align*}
			\|\cU_{v_L}[f]\|_{S^{s,a,0}(I)} 
			\leq C(v_L) \|f\|_{H^s_x}, 
			\qquad \|\cI_{v_L}[F]\|_{S^{s,a,0}(I)} \leq C(v_L) \|F\|_{N^{s,a,0}(I)},
		\end{align*}
		and  
		\begin{align*}
			\| \cU_{v_L}[u_0]\|_{L^2(I;  L_x^{2^*})}\leq C(v_L) \|u_0\|_{L_x^2},\quad \|\mathcal{I}_{v_L}[F]\|_{L^2(I;  L_x^{2^*})} \leq C(v_L) \|F\|_{L^2(I; L_x^{2_*})}.
		\end{align*}
	\end{lemma} 
	
	We also use the notation $C(v_L)$ 
	for the constants appearing in the 
	estimates in Lemma~\ref{lem:LinSchrPotSmallTimeLocSmo}
	involving the maximal function spaces, i.e., 
	\begin{align*}
		\| \cU_{v_L}[u_0]\|_{\Z_\varepsilon^{s,0}(I)}\leq C(v_L) \|u_0\|_{H_x^s},\quad \|\mathcal{I}_{v_L}[F]\|_{\Z_\varepsilon^{s,0} (I)} \leq C(v_L) \|F\|_{N^{s,0,0}(I)}. 
	\end{align*}

	\subsection{Short-time regime}
	In this subsection, we  consider the following system on $I = [0, T]$,
	\begin{equation}\label{eq:RegulationZak}
		\left\{\aligned
		(\imu \partial_t +\Delta -\Re(v_L))z &= \Re({\rho}) z, &z(0)&=X_0,\\
		(\imu \partial_t +|\nabla|)\rho &=  -h|\nabla||z|^2, &\rho(0)&=0,
		\endaligned
		\right.
	\end{equation}
	where $v_L= e^{\imu t|\nabla|} Y_0$ with $Y_0\in H_x^l$, $X_0\in H_x^s$, and $0<T<\infty$.
	
	We note that solving~\eqref{eq:RegulationZak} is equivalent to solving
	\begin{equation}
		\label{eq:RegNoiseSchrPot}
		(\imu \partial_t +\Delta -\Re(v_L))z = -\Re(\cJ_0(h|\nabla| |z|^2)) z, \qquad z(0) = X_0
	\end{equation}
	and then setting $\rho = - \cJ_0(h|\nabla| |z|^2)$.
	
	In the next proposition, we consider the noise-regularization regimes {\rm I}, {\rm II}, and {\rm III} separately. We recall that these regularity regimes were defined in Table~\ref{table:regularity-regimes}. A direct calculation shows that in the noise-regularization regimes I, II, and III, the assumptions of Theorem \ref{Prop-trilinear} (i), (ii), and (iii), respectively, are all satisfied. Here we take $\beta = s - \frac{1}{2}$ in regimes I and III and $\beta = l$ in regime II. Thus, there exist parameters $\alpha\in (0,\frac{1}{2})$ close to $\frac{1}{2}, \theta\in (0,1), p>2$ close to $2$, and $0<\varepsilon\ll 1$ such that the estimates in Theorem \ref{Prop-trilinear} (i), (ii), and (iii), respectively, are valid. We further note that for every pair $(s,l)$ in regime III, the pairs $(s, s-\frac{1}{2})$ and $(s,\kappa)$ lie in regime I, where $\kappa = s - \frac12 (1 - b)$.
	
	\begin{proposition}[Well-posedness in short-time regime]\label{prop:WPShortTime}
		Let $(s,l)$ satisfy the assumptions of Theorem~\ref{thm:RegNoise}, i.e., condition~\eqref{IniReg-conditionNoiseReg}. Let $(a,b)$ be as in \eqref{con:ab}, $0<T<\infty$, $I = [0,T]$. Let $C \geq 1$ denote a large universal constant and $C(v_L)$ the constant from Lemma~\ref{lem:LinSchrPotSmallTimeSsa0} and Lemma~\ref{lem:LinSchrPotSmallTimeLocSmo} depending on the free wave potential $v_L$. We consider the conditions (i), (ii), and (iii) in noise-regularization regimes I, II, and III separately.
		
		\vspace*{4pt plus 2pt minus 2pt}%
		\paragraph{$(i)$ \textbf{Noise-regularization regime I}} 
		Let $(s,l)$ lie in noise-regularization regime I.
		Let $\alpha \in (0,\frac12)$ and $\theta\in (0,1)$ be as in Theorem~\ref{Prop-trilinear}~(i). Take $A \geq 1$ and $ R>0$ such that
		\begin{align} \label{A-R-I}
			\|h\|_{B^\alpha_{2,\infty}\cap L^\infty(I)}\leq A,
			\qquad
			2CC(v_L) \|X_0\|_{H_x^s} \leq R. 
		\end{align}  
		Let $0 < \delta < R$ be sufficiently small such that
		\begin{align} \label{prop:smallness2}
			 8CC(v_L) A R^{2-\theta} \delta^\theta \leq 1,
		\end{align}
		and $T$ be sufficiently small such that 
		\begin{align}\label{prop:smallnessT1}
			2CC(v_L)(1+\|Y_0\|_{H_x^l})\|e^{\imu t \Delta }X_0\|_{L^2 (0, T; L_x^{2^*})}\leq \delta.
		\end{align}
		
		\paragraph{$(ii)$ \textbf{Noise-regularization regime II}} 
		Let $(s,l)$ lie in noise-regularization regime II of Theorem \ref{thm:RegNoise}. Let $p \in (2,\infty)$ close to $2$ and $\theta\in(0,1)$ be as in Theorem~\ref{Prop-trilinear}~(ii). Take $A\in (0,\infty)$ and $R\geq 1$
		such that
		\begin{align}\label{prop:boundedness1}
			\|h\|_{B_{p,\infty}^{\frac{1}{p}}\cap L^\infty(I)}\leq A, \quad 2C\|X_0\|_{H_x^s}\leq R.
		\end{align}
		Let $0 < \delta < R$ be sufficiently small such that
		\begin{align} \label{prop:smallness1.1}
			8 CC(v_L) AR^{2-\theta} \delta^{\theta}\leq 1,
		\end{align}
		and $T$ small enough such that 
		\begin{align}  \label{prop:smallness1.2}
			4CC(v_L) \|v_L\|_{W^{l,0,l}(I)+L^2(I; W_x^{s,d})}(1+\|X_0\|_{H_x^s}) \leq \delta, \quad 2C\|e^{\imu t\Delta} X_0\|_{L^2(I; W_x^{s,2^*}) }^{\frac{2\theta}{3}} \|X_0\|_{H_x^s}^{1-\frac{2\theta}{3}}\leq\delta. 
		\end{align}  
		
		\paragraph{$(iii)$ \textbf{Noise-regularization regime III}} 
		Let $(s,l)$ lie in noise-regularization regime III. Let $p \in (2,\infty)$ close to $2$ and $0<\varepsilon\ll 1$ be as in Theorem~\ref{Prop-trilinear}~(iii). Take
		$\alpha \in (0, \frac12)$ and $\theta\in (0,1)$ such that estimates~\eqref{triesti-improve} and~\eqref{prop:triest2} in regime I hold for the pairs $(s, s - \frac{1}{2})$ and $(s, \kappa)$, where $\kappa = s - \frac12 (1-b)$. Take $A, R>0$ such that
		\begin{align*}
			\|h\|_{B_{2,\infty}^\alpha\cap L^\infty(I)} + \|h\|_{B_{p,\infty}^{\frac{1}{p}}\cap L^\infty(I)}\leq A,
			\qquad
			2CC(v_L) \|X_0\|_{H_x^s} \leq R. 
		\end{align*}
		Moreover, let $0 < \delta < R$ and $T$ be sufficiently small such
		\begin{align*}
			8CC(v_L) A R^{2-\theta} \delta^\theta \leq 1 \qquad \text{and} \qquad 2CC(v_L)(1+\|Y_0\|_{H_x^l})\|e^{\imu t \Delta }X_0\|_{L^2 (0, T; L_x^{2^*})}\leq \delta.
		\end{align*}
		
		\vspace*{4pt plus 2pt minus 2pt}%
		Then, under the above conditions, 
		system \eqref{eq:RegulationZak} has a unique solution
		$ (z, \rho) \in C([0, T], H_x^s\times H_x^l)$, where uniqueness holds in 
		$S^{s,a,0}(I)\times L^\infty(I; H_x^l), S^{s,0,0}(I)\times L^\infty(I; H_x^l)$, and $\Z_\epsilon^{s,0}(I)\times L^\infty(I; H_x^l)$ in the noise-regularization regimes I, II, and III, respectively. Moreover, the solution satisfies the estimates
		\begin{align*}
			\| z \|_{S^{s,a,0}(I)} \leq R, \qquad \| z \|_{S^{s,0,0}(I)} \leq R, \qquad  \| z \|_{\Z^{s,0}_\varepsilon(I)} \leq R, \ \|z\|_{\Z^{s,b}_\varepsilon(I)}<\infty
		\end{align*}
		and
		\begin{align*}
			\| z \|_{L^2(I; L_x^{2^*})} \leq \delta, \qquad \| z \|_{L^2(I; W_x^{s, 2^*})} \leq \delta, \qquad  \|z\|_{L^2(I; L_x^{2^*})} \leq \delta,
		\end{align*}
		in noise-regularization regimes I, II, and III, respectively.
		
	\end{proposition}
	
	\begin{proof}[Proof of Propostion \ref{prop:WPShortTime}]  
		In the analysis below all the space-time norms are taken over $[0,T]\times \R^d$. 
		For the ease of notation, 
		we drop
		the dependence on $[0,T]\times \R^d$. 
		
		\vspace*{4pt plus 2pt minus 2pt}%
		$(i)$ We start with noise-regularization regime I.   
		We define the fixed point map 
		\begin{align}\label{eq:FixedPointShortTime}
			\Psi(X_0; z):= \cU_{v_L}[X_0] - \mathcal{I}_{v_L}[\Re(\cJ_0 (h|\nabla| |z|^2))z]
		\end{align}
		on the complete metric space
		\begin{align*}
			B_{R, \delta}:=\{ z\in S^{s,a,0}(I):  \|z\|_{S^{s,a,0}(I)}\leq R,\  \|z\|_{L^2(I; L_x^{2^*}) }\leq \delta \} 
		\end{align*}
		equipped with the metric induced by the norm $\| \cdot \|_{S^{s,a,0}(I)}$.
		
		Using Lemma~\ref{lem:LinSchrPotSmallTimeSsa0}, Lemma \ref{lem:BilinearEstimates} \ref{it:BilinearEstNonendpointbNzero}, and the trilinear estimate \eqref{prop:triest2}, we have
		\begin{align}\label{prop:SelfmapSNEP}
			\|\Psi(X_0; z)\|_{S^{s,a,0}} &\leq C(v_L) \|X_0\|_{H_x^s}+ CC(v_L) \| \cJ_0(h |\nabla||z|^2)\|_{W^{l, a, s-\frac12}} \|z\|_{S^{s,a,0}}  \notag \\
			& \leq C(v_L) \|X_0\|_{H_x^s}+ CC(v_L) \|h\|_{B^\alpha_{2,\infty}\cap L_t^\infty} \|z\|_{S^{s,a,0}}^{3-2\theta} \|z\|_{L_t^2 L_x^{2^*}}^{2\theta}.
		\end{align}

		To estimate the $L_t^2 L_x^{2^*}$-norm, we write
		\begin{align*}
			\cU_{v_L}[X_0](t)= e^{\imu t\Delta}X_0 + \mathcal{I}_{v_L}[\Re(v_L)e^{\imu(\cdot)\Delta}X_0].
		\end{align*}
		Since $l \geq \frac{d-4}{2}$, by Lemma~\ref{lem:LinSchrPotSmallTimeSsa0} and the Sobolev embedding $H_x^{\frac{d-4}{2}}\hookrightarrow L_x^{\frac{d}{2}}$, 
		we have
		\begin{align}
			\|\cU_{v_L}[X_0]\|_{L_t^2 L_x^{2^*}}&\leq \|e^{\imu t\Delta}X_0\|_{L_t^2 L_x^{2^*}} + C(v_L)\|\Re(v_L) e^{\imu t\Delta} X_0\|_{L_t^2 L_x^{2_*}}\notag\\
			&\leq \|e^{\imu t\Delta}X_0\|_{L_t^2 L_x^{2^*}} +C(v_L) \|v_L\|_{L_t^\infty L_x^{\frac{d}{2}}} \|e^{\imu t\Delta} X_0\|_{L_t^2 L_x^{2^*}}\notag\\
			&\leq CC(v_L) (1+ \|Y_0\|_{H_x^{l}}) \|e^{\imu t\Delta} X_0\|_{L_t^2 L_x^{2^*}}.\label{eq:HomoDisperEstiPertByFreeWave}
		\end{align}
		Moreover, using the Strichartz estimate and the H\"older inequality, we infer
		\begin{align}
			\|\mathcal{I}_{v_L}[\Re(\cJ_0(h |\nabla||z|^2))z] \|_{L_t^2 L_x^{2^*}}
			&\leq C(v_L) \|\Re(\cJ_0(h |\nabla||z|^2))z\|_{L_t^2 L_x^{2_*}}\notag\\
			&\leq C(v_L) \|\cJ_0(h |\nabla||z|^2)\|_{L_t^\infty L_x^{\frac{d}{2}}} \|z\|_{L_t^2 L_x^{2^*}} \notag\\
			&\leq CC(v_L) \|h\|_{B^\alpha_{2,\infty}\cap L_t^\infty} \|z\|_{S^{s,a,0}}^{2-2\theta} \|z\|_{L_t^2 L_x^{2^*}}^{1+2\theta}\label{eq:inhomoDisperEstiPertByFreeWave},
		\end{align}
		where we also used the embeddings  $W^{l,a,s-\frac12} \hookrightarrow L_t^\infty H_x^{l}\hookrightarrow L_t^\infty L_x^{\frac{d}{2}}$ and 
		the trilinear estimate \eqref{prop:triest2} in the last inequality.
		It follows that
		\begin{align}\label{prop:SelfmapLtLxNEP}
			\| \Psi( X_0; z)\|_{L_t^2 L_x^{2^*}}
			\leq CC(v_L) (1+ \|Y_0\|_{H_x^{l}}) \|e^{\imu t\Delta} X_0\|_{L_t^2 L_x^{2^*}}+ CC(v_L)\|h\|_{B^\alpha_{2,\infty}\cap L_t^\infty} \|z\|_{S^{s,a,0}}^{2-2\theta} \|z\|_{L_t^2 L_x^{2^*}}^{1+2\theta}.
		\end{align}
		
		Hence, combining estimates \eqref{prop:SelfmapSNEP} and \eqref{prop:SelfmapLtLxNEP}, we obtain that for any $z\in B_{R,\delta}$,
		\begin{align*}
			&\| \Psi(X_0 ; z)\|_{S^{s,a,0}}\leq \frac{R}{2}+ CC(v_L) A R^{3-2\theta} \delta^{2\theta} \leq R, \\
			& \| \Psi( X_0; z)\|_{L_t^2 L_x^{2^*}} \leq CC(v_L) (1+ \|Y_0\|_{H_x^l}) \|e^{\imu t\Delta} X_0\|_{L_t^2 L_x^{2^*}}+ CC(v_L) A R^{2-2\theta} \delta^{1+2\theta} \leq \delta, 
		\end{align*}
		which yields that $\Psi(X_0 ; \cdot)$ is a self-mapping on $B_{R,\delta}$.
		
		Regarding the contraction property of the map $\Psi$, 
		we apply Lemma~\ref{lem:LinSchrPotSmallTimeSsa0}, Lemma \ref{lem:BilinearEstimates} \ref{it:BilinearEstNonendpointbNzero}, and the trilinear estimate \eqref{prop:triest2} to obtain
		\begin{align}
			&\|\Psi(X_0; z)- \Psi(X_0; w)\|_{S^{s,a,0}}\notag\\
			&\leq  C(v_L) \|\Re(\cJ_0 [h|\nabla||z|^2])z - \Re(\cJ_0[h|\nabla||w|^2])w\|_{N^{s,a,0}}\notag\\
			&\leq CC(v_L) \|\Re(\cJ_0 [h|\nabla||z|^2])\|_{W^{l,a, s-\frac{1}{2}}}  \|z-w\|_{S^{s,a,0}} \notag\\
			&\qquad + CC(v_L) \| \Re(\cJ_0 [h|\nabla|( (w-z)\overline{w} + z(\overline{w}-\overline{z}))]) \|_{W^{l,a,s-\frac{1}{2}}} \| w \|_{S^{s,a,0}}\notag\\
			&\leq  CC(v_L)  \|h\|_{B^\alpha_{2,\infty}\cap L_t^\infty}\|z\|_{L_t^2 L_x^{2^*}}^{2\theta} \|z\|_{S^{s,a,0}}^{2-2\theta} \|z-w\|_{S^{s,a,0}}\notag\\
			&\qquad+ CC(v_L)  \|h\|_{B^\alpha_{2,\infty}\cap L_t^\infty}\|z-w\|_{L_t^2 L_x^{2^*}}^{\theta} \|w\|_{L_t^2 L_x^{2^*}}^{\theta} \|z-w\|_{S^{s,a,0}}^{1-\theta} \|w\|_{S^{s,a,0}}^{1-\theta} \|w\|_{S^{s,a,0}}\notag\\
			&\qquad+ CC(v_L)  \|h\|_{B^\alpha_{2,\infty}\cap L_t^\infty} \|z\|_{L_t^2 L_x^{2^*}}^{\theta} \|w-z\|_{L_t^2 L_x^{2^*}}^{\theta} \|z\|_{S^{s,a,0}}^{1-\theta} \|z-w\|_{S^{s,a,0}}^{1-\theta} \|w\|_{S^{s,a,0}}.\label{prop:NRGContra1}
		\end{align}
		Then, the embedding $S^{s,a,0}\hookrightarrow L_t^2 L_x^{2^*}$ and the Young inequality   $x^{1-\theta}y\lesssim_{\theta} x^{2-\theta}+ y^{2-\theta}$ yield that
		\begin{align}
			\text{R.H.S. of } \eqref{prop:NRGContra1} &\leq CC(v_L) \|h\|_{ B^\alpha_{2,\infty}\cap L_t^\infty}\|z-w\|_{S^{s,a,0}} (\|z\|_{L_t^2 L_x^{2^*}}^\theta \|z\|_{S^{s,a,0}}^{2-\theta} + \|w\|_{L_t^2 L_x^{2^*}}^\theta \|w\|_{S^{s,a,0}}^{2-\theta}\notag\\
			&\hspace{16em}+\|z\|_{L_t^2 L_x^{2^*}}^\theta \|z\|_{S^{s,a,0}}^{1-\theta}\|w\|_{S^{s,a,0}})\notag\\
			&\leq CC(v_L) \|h\|_{ B^\alpha_{2,\infty}\cap L_t^\infty}\|z-w\|_{S^{s,a,0}}(\|z\|_{L_t^2 L_x^{2^*}}^\theta + \|w\|_{L_t^2 L_x^{2^*}}^\theta)(\|z\|_{S^{s,a,0}}^{2-\theta} + \|w\|_{S^{s,a,0}}^{2-\theta}).\label{prop:NRGContra2}
		\end{align}
		Hence, from the estimates \eqref{prop:NRGContra1}, \eqref{prop:NRGContra2} and the smallness condition \eqref{prop:smallness2} 
		we derive 
		\begin{align*}
			\|\Psi(X_0; z)- \Psi(X_0; w)\|_{S^{s,a,0}}
			\leq CC(v_L)A  2 \delta^\theta  2 R^{2 - \theta} \|z-w\|_{S^{s,a,0}} \leq \frac12 \|z-w\|_{S^{s,a,0}}.
		\end{align*}
		This yields that $\Psi(X_0,\cdot)$ is a contraction on $B_{R, \delta}$. 
		Hence,  there exists a  unique 
		fixed point $z\in S^{s,a,0}(I)$ 
		such that 
		$z= \Psi(X_0; z)$, 
		which implies that 
		$z$ solves  \eqref{eq:RegNoiseSchrPot}. The uniqueness in $S^{s,a,0}(I)$ follows by standard methods.
		
		Regarding the wave component 
		$\rho := - \cJ_0[h|\nabla||z|^2]$, the trilinear estimate \eqref{triesti-improve} implies that
		\begin{align*}
			\|\rho\|_{W^{l,a,s-\frac12}} \lesssim  \|h\|_{B^\alpha_{2,\infty}\cap L_t^\infty} \|z\|_{S^{s,a,0}}^2<\infty,
		\end{align*}
		which yields that $\rho \in C([0,T];H_x^{l})$. Uniqueness of the wave component simply follows from the uniqueness of the Schr{\"o}dinger component.
		
		This concludes the proof of  Propostion \ref{prop:WPShortTime} in noise-regularization regime I.
		
		\smallskip
		$(ii)$ In noise-regularization regime II, we cannot recover a full power of the dispersive norm. Instead, we employ an interpolation type intermediate norm which was also used in the proof of Theorem~7.6 in (the corrigendum of) \cite{CHN23}. We first define the constant $\Lambda:= \|e^{\imu t\Delta}X_0\|_{L^2(I; W_x^{s,2^*})} \|X_0\|_{H_x^s}^{-1}$ if $X_0\neq 0$ and $\Lambda:=1$ otherwise. For any $\gamma>0$, we then define the weighted norm
		\begin{align*}
			\|z\|_{\Y_{\gamma}}:= \inf_{z=z_1+z_2} \{\Lambda^\gamma\|z_1\|_{S^{s,0,0}(I)} + \Lambda^{\gamma-1} \|z_1\|_{L^2(I; W_x^{s,2^*})} + \|z_2\|_{S^{s,0,0}(I)}\},
		\end{align*}
		and 
		\begin{align*}
			\|z\|_{\Y}:= \|z\|_{\Y_{\theta}} + \|z\|_{\Y_{\frac{2\theta}{3}}}.
		\end{align*}
		We note that $(1 + \Lambda^{-\gamma})^{-1} \| z \|_{S^{s,0,0}(I)} \leq \| z \|_{\Y_\gamma} \leq \| z \|_{S^{s,0,0}(I)}$ so that $\| \cdot \|_{\Y}$ defines an equivalent norm to $\|\cdot\|_{S^{s,0,0}(I)}$. Moreover, an application of Young's inequality yields
		\begin{equation}
		\label{eq:InterpolNormYoung}
			x^{1-\gamma} y^\gamma \leq \Lambda^\gamma x + \Lambda^{\gamma - 1} y
		\end{equation}
		for any $x,y \geq 0$ and $\gamma \in (0,1)$.
		By the embedding $S^{s,0,0}\hookrightarrow L_t^2 W_x^{s,2^*}$, we see that $\| \cdot \|_{\Y_\gamma}$ also controls the dispersive norm $\| \cdot \|_{L^2(I; W_x^{s,2^*})}$. Collecting the estimates we will employ below, we have
		\begin{align}\label{prop:YnormEsti}
			 \|z\|_{L^2(I; W_x^{s,2^*})} \leq \|z\|_{\Y} \leq 2  \|z\|_{S^{s,0,0}(I)}.
		\end{align}  
		
		We rewrite the fixed point operator $\Psi$ as
		\begin{align*}
			\Psi(X_0; z) = e^{\imu t \Delta} X_0+ \mathcal{I}_{v_L} [\Re(v_L)e^{\imu t\Delta }X_0] - \mathcal{I}_{v_L}[\Re(\cJ_0(h|\nabla||z|^2))z]
		\end{align*}
		by expanding the linear propagator $\cU_{v_L}$ and then consider $\Psi$ as a map
		on the closed ball
		\begin{align*}
			B_{R,\delta}:=\{z\in S^{s,0,0}(I): \|z\|_{S^{s,0,0}(I)}\leq R,\  \|z\|_{\Y}\leq \delta\}
		\end{align*}
		equipped with the metric induced by $\| \cdot \|_{S^{s,0,0}(I)}$.
		For the ease of notation, we omit the interval $I$ in the following.

		Applying Lemma~\ref{lem:LinSchrPotSmallTimeSsa0}, the bilinear estimate in Lemma \ref{it:BilinearEstNonendpoint} with $\beta = l$, and the trilinear estimate \eqref{prop:triestiend}, we obtain
		\begin{align}
			\|\mathcal{I}_{v_L}[\Re(\cJ_0(h|\nabla||z|^2))z]\|_{S^{s,0,0}} &\leq CC(v_L)\|h\|_{B_{p,\infty}^{\frac{1}{p}}\cap L_t^\infty} \|z\|_{S^{s,0,0}}^{3-2\theta} \|z\|_{L_t^2W_x^{s,2^*}}^{2\theta},\label{prop:TriEsti}\\
			\|\cI_{v_L}[\Re(v_L) e^{\imu t\Delta}X_0]\|_{S^{s,0,0}} &	\leq CC(v_L)\|v_L\|_{W^{l,0,l}+L_t^2 W_x^{s,d}} \|X_0\|_{H_x^s}.\label{prop:BilEsti}
		\end{align}
		Applying~\eqref{eq:InterpolNormYoung} with $\gamma = \frac{2\theta}{3}$, we deduce from~\eqref{prop:TriEsti} that
		\begin{align}\label{prop:TriEsti2}
			\|\mathcal{I}_{v_L}[\Re(\cJ_0(h|\nabla||z|^2))z]\|_{S^{s,0,0}} \leq CC(v_L)\|h\|_{B_{p,\infty}^{\frac{1}{p}}\cap L_t^\infty} \|z\|_{\Y_{\frac{2\theta}{3}}}^3\leq CC(v_L)\|h\|_{B_{p,\infty}^{\frac{1}{p}}\cap L_t^\infty} \|z\|_{\Y}^3.
		\end{align}
		Hence, combining Strichartz estimates for the linear term with~\eqref{prop:TriEsti} and~\eqref{prop:BilEsti}, we infer
		\begin{align*}
			\| \Psi(X_0; z)\|_{S^{s,0,0}} 
			&\leq C\|X_0\|_{H_x^s} + CC(v_L) (\|v_L\|_{W^{l,0,l}+L_t^2W_x^{s,d}} \|X_0\|_{H_x^s} + \|h\|_{B_{p,\infty}^{\frac{1}{p}}\cap L_t^\infty} \|z\|^3_{\Y})\\
			&\leq \frac{R}{2}+ \frac{R}{4} + \frac{R}{8} \leq R, 
		\end{align*}
		where we used conditions \eqref{prop:boundedness1}, \eqref{prop:smallness1.1}, and \eqref{prop:smallness1.2}.
		
		Concerning the $\Y$-norm, we first observe that the definition of the norm and $\Lambda$ combined with Lemma~\ref{lem:LinEstimates} imply that
		\begin{align*}
			\|e^{\imu t\Delta}X_0\|_{\Y}&\leq C\|e^{\imu t\Delta}X_0\|^{\frac{2\theta}{3}}_{L_t^2 W_x^{s,2^*}}\|X_0\|^{1-\frac{2\theta}{3}}_{H_x^s}.
		\end{align*}
		Here we also used that Strichartz estimates allow us to bound the contribution from $\gamma = \theta$ by the one from $\gamma = \frac{2\theta}{3}$.
		
		Combining the previous estimate with \eqref{prop:BilEsti}, \eqref{prop:TriEsti2}, and the embedding $S^{s,0,0}\hookrightarrow \Y$, we get 
		\begin{align*}
			\| \Psi(X_0; z)\|_{\Y} &\leq C\|e^{\imu t\Delta}X_0\|^{\frac{2\theta}{3}}_{L_t^2 W_x^{s,2^*}}\|X_0\|^{1-\frac{2\theta}{3}}_{H_x^s} +CC(v_L)(\|v_L\|_{W^{l,0,l}+L_t^2 W_x^{s,d}} \|X_0\|_{H_x^s} +\|h\|_{B_{p,\infty}^{\frac{1}{p}}\cap L_t^\infty} \|z\|_{\Y}^3)\\
			&\leq \frac{\delta}{2}+ \frac{\delta}{4}+ \frac{\delta}{8} =\delta,
		\end{align*}
		where we used the conditions \eqref{prop:boundedness1}, \eqref{prop:smallness1.1}, and \eqref{prop:smallness1.2}. We conclude that $\Psi$ is a self-map on $B_{R,\delta}$.
	 
		Moreover, using  Lemma~\ref{lem:LinSchrPotSmallTimeSsa0}, the bilinear estimate in Lemma~\ref{lem:BilinearEstimates} \ref{it:BilinearEstNonendpointbNzero} with $\beta = l$, the trilinear estimate \eqref{prop:triestiend}, and the embedding $S^{s,0,0}\hookrightarrow L_t^2W_x^{s,2^*}$, 
		we get 
		\begin{align*}
			&\| \Psi(X_0; z_1)- \Psi(X_0; z_2) \|_{S^{s,0,0}} \\
			&\leq C(v_L) \|\Re(\cJ_0(h|\nabla||z_1|^2)) (z_1-z_2) + \cJ_0(h|\nabla|(\bar{z}_1(z_1-z_2)+z_2(\bar{z}_1-\bar{z}_2)))z_2\|_{N^{s,0,0}}\\
			&\leq C C(v_L) \|\Re(\cJ_0(h|\nabla||z_1|^2)) \|_{W^{l,0,l}} \| z_1-z_2 \|_{S^{s,0,0}} + \|\cJ_0(h|\nabla|(\bar{z}_1(z_1-z_2)+z_2(\bar{z}_1-\bar{z}_2))) \|_{W^{l,0,l}} \| z_2\|_{S^{s,0,0}}\\
			&\leq CC(v_L) \|h\|_{B^{\frac{1}{p}}_{p,\infty} \cap L^\infty_t} \|z_1-z_2\|_{S^{s,0,0}}(\|z_1\|^{2-2\theta}_{S^{s,0,0}}\|z_1\|^{2\theta}_{L_t^2 W_x^{s,2^*}}   + \|z_1\|^{1-\theta}_{S^{s,0,0}}\|z_2\|_{S^{s,0,0}}\|z_1\|_{L_t^2 W_x^{s,2^*}}^{\theta} \\
			&\hspace{18em} + \|z_2\|^{2-\theta}_{S^{s,0,0}}\|z_2\|^{\theta}_{L_t^2 W_x^{s,2^*}})\\
			&\leq CC(v_L) \|h\|_{B^{\frac{1}{p}}_{p,\infty} \cap L^\infty_t} \|z_1-z_2\|_{S^{s,0,0}} (\|z_1\|^{\theta}_{L_t^2 W_x^{s,2^*}}\|z_1\|^{1-\theta}_{S^{s,0,0}} + \|z_2\|^{\theta}_{L_t^2 W_x^{s,2^*}} \|z_2\|^{1-\theta}_{S^{s,0,0}}) (\|z_1\|_{S^{s,0,0}} + \|z_2\|_{S^{s,0,0}}) .
		\end{align*}
		Applying~\eqref{eq:InterpolNormYoung} with $\gamma = \theta$ and the embedding $\Y_{\theta}\hookrightarrow \Y$, we thus arrive at
		\begin{align*}
			\| \Psi(X_0; z_1)- \Psi(X_0; z_2) \|_{S^{s,0,0}} &\leq CC(v_L) \|h\|_{B^{\frac{1}{p}}_{p,\infty} \cap L^\infty_t} \|z_1-z_2\|_{S^{s,0,0}} (\|z_1\|_{S^{s,0,0}} + \|z_2\|_{S^{s,0,0}}) (\|z_1\|_{\Y} + \|z_2\|_{\Y})\\
			&\leq \frac{1}{2}\|z_1-z_2\|_{S^{s,0,0}},
		\end{align*}
		where we used the conditions \eqref{prop:boundedness1} and \eqref{prop:smallness1.1}. Hence, $\Psi(X_0;\cdot)$ is a contractive self-mapping on $B_{R,\delta}$ and there exists a unique solution $z\in B_{R,\delta}$ to~\eqref{eq:RegNoiseSchrPot}. Uniqueness in $S^{s,0,0}(I)$ follows again by standard methods. Moreover, by estimate \eqref{prop:YnormEsti} we have  
		\begin{align*}
			\|z\|_{L_t^2 W_x^{s,2^*}}\leq \|z\|_{\Y}\leq\delta.
		\end{align*}
		
		Regarding the wave component,  
		using the trilinear estimate \eqref{prop:triestiend}  again, 
		we obtain
		\begin{align*}
			\|\rho\|_{L_t^\infty H_x^l}\leq \|\rho\|_{W^{l,0,l}}\leq CC(v_L) A \|z\|^{2-2\theta}_{S^{s,0,0}} \|z\|_{L_t^2 W_x^{s,2^*}}^{2\theta}<\infty.
		\end{align*}
		Uniqueness of the wave component follows from the uniqueness of the Schr{\"o}dinger component.
		Hence, the Proposition is proved in noise-regularization regime II.

		$(iii)$ Regarding noise-regularization regime III,  
		we first note that
		$a=0$ 
		in this regime due to \eqref{con:ab}. 
		Consider the fixed point operator $\Psi$ from  \eqref{eq:FixedPointShortTime}, 
		but now on the complete metric space
		\begin{align*}
			{B}_{R, \delta} :=\{ z\in \Z_{\varepsilon}^{s,0}(I):  \|z\|_{\Z_{\varepsilon}^{s,0}(I)}\leq R,\  \|z\|_{L^2(I; L_x^{2^*}) }\leq \delta \} 
		\end{align*}
		equipped with the metric induced by the norm $\| \cdot \|_{\Z_{\varepsilon}^{s,b}(I)}$.  
		
		We first use Lemma~\ref{lem:LinSchrPotSmallTimeLocSmo}, Lemma~\ref{lem:BilinearEstimates} \ref{it:BilinearEstNonendpointbNzero}, the trilinear estimate~\eqref{prop:triest2} with pair $(s, s-\frac12)$ in regime I, and the embedding $\Z_\varepsilon^{s,0}\hookrightarrow S^{s,0,0}$ 
		to infer 
		\begin{align}\label{prop:SelfmapSNEPLoc}
			\|\Psi(X_0; z)\|_{\Z_\varepsilon^{s,0}} &\leq C(v_L) \|X_0\|_{H_x^s}+  C(v_L) \| \Re(\cJ_0(h|\nabla||z|^2))z\|_{N^{s, 0, 0}} \notag \\
			&\leq C(v_L) \|X_0\|_{H_x^s}+ C C(v_L) \| \cJ_0(h|\nabla||z|^2)\|_{W^{s-\frac{1}{2}, 0, s-\frac{1}{2}}} \|z\|_{\Z_\varepsilon^{s,0}}  \notag \\
			& \leq C(v_L) \|X_0\|_{H_x^s}+ C C(v_L) \|h\|_{B^\alpha_{2,\infty}\cap L_t^\infty} \|z\|_{\Z_\varepsilon^{s,0}}^{3-2\theta} \|z\|_{L_t^2 L_x^{2^*}}^{2\theta}.
		\end{align}
		Regarding the $L_t^2 L_x^{2^*}$-norm, we proceed as in \eqref{eq:HomoDisperEstiPertByFreeWave} and \eqref{eq:inhomoDisperEstiPertByFreeWave} to get
		\begin{align}
			\|\cU_{v_L}[X_0]\|_{L_t^2 L_x^{2^*}}&\leq C C(v_L) (1+ \|Y_0\|_{H_x^{s-\frac{1}{2}}}) \|e^{\imu t\Delta} X_0\|_{L_t^2 L_x^{2^*}},\label{eq:HomoDisperEstiPertByFreeWaveLoc}\\
			\|\mathcal{I}_{v_L}[\Re(\cJ_0(h |\nabla||z|^2))z] \|_{L_t^2 L_x^{2^*}}&\leq C C(v_L) \|h\|_{B^\alpha_{2,\infty}\cap L_t^\infty} \|z\|_{S^{s,0,0}}^{2-2\theta} \|z\|_{L_t^2 L_x^{2^*}}^{1+2\theta}\label{eq:inhomoDisperEstiPertByFreeWaveLoc}, 
		\end{align}
		which together with the embedding $\Z_\varepsilon^{s,0}\hookrightarrow S^{s,0,0}$ 
		yields that 
		\begin{align}\label{prop:SelfmapLtLxNEPLoc}
			\| \Psi( X_0; z)\|_{L_t^2 L_x^{2^*}}
			\leq C  C(v_L)(1+ \|Y_0\|_{H_x^{s-\frac{1}{2}}}) \|e^{\imu t\Delta} X_0\|_{L_t^2 L_x^{2^*}}+ C C(v_L)\|h\|_{B^\alpha_{2,\infty}\cap L_t^\infty} \|z\|_{\Z_\varepsilon^{s,0}}^{2-2\theta} \|z\|_{L_t^2 L_x^{2^*}}^{1+2\theta}.
		\end{align} 
		Hence, combining estimates \eqref{prop:SelfmapSNEPLoc} and \eqref{prop:SelfmapLtLxNEPLoc},
		we obtain that for any $z\in B_{R,\delta}$
		\begin{align*}
			&\| \Psi(X_0 ; z)\|_{\Z_\varepsilon^{s,0}}\leq \frac{R}{2}+ C C(v_L) A R^{3-2\theta} \delta^{2\theta} \leq R, \\
			& \| \Psi( X_0; z)\|_{L_t^2 L_x^{2^*}} \leq C C(v_L) (1+ \|Y_0\|_{H_x^l}) \|e^{\imu t\Delta} X_0\|_{L_t^2 L_x^{2^*}}+ C C(v_L) A R^{2-2\theta} \delta^{1+2\theta} \leq \delta,
		\end{align*}
		which yields that $\Psi(X_0 ; \cdot)$ is a self-mapping on $B_{R,\delta}$.
		
		Furthermore, 
		treating the difference as in the previous regimes and combining Lemma~\ref{lem:LinSchrPotSmallTimeLocSmo}, Lemma~\ref{lem:BilinearEstimates}~\ref{it:BilinearEstNonendpointbNzero}, the trilinear estimate~\eqref{prop:triest2} with $(s,s-\frac12)$,
		and the embedding $\Z_\varepsilon^{s,0}\hookrightarrow S^{s,0,0}$ as above, we obtain
		\begin{align*}
			\|\Psi(X_0; z)- \Psi(X_0; w)\|_{\Z_\varepsilon^{s,0}}
			&\leq  C(v_L)  \|h\|_{B^\alpha_{2,\infty}\cap L_t^\infty}\|z\|_{L_t^2 L_x^{2^*}}^{2\theta} \|z\|_{\Z_\varepsilon^{s,0}}^{2-2\theta} \|z-w\|_{\Z_\varepsilon^{s,0}}\notag\\
			&\qquad + C(v_L)  \|h\|_{B^\alpha_{2,\infty}\cap L_t^\infty}\|z-w\|_{L_t^2 L_x^{2^*}}^{\theta} \|w\|_{L_t^2 L_x^{2^*}}^{\theta} \|z-w\|_{\Z_\varepsilon^{s,0}}^{1-\theta} \|w\|_{\Z_\varepsilon^{s,0}}^{1-\theta} \|w\|_{\Z_\varepsilon^{s,0}}\notag\\
			&\qquad + C(v_L)  \|h\|_{B^\alpha_{2,\infty}\cap L_t^\infty} \|z\|_{L_t^2 L_x^{2^*}}^{\theta} \|w-z\|_{L_t^2 L_x^{2^*}}^{\theta} \|z\|_{\Z_\varepsilon^{s,0}}^{1-\theta} \|z-w\|_{\Z_\varepsilon^{s,0}}^{1-\theta} \|w\|_{\Z_\varepsilon^{s,0}},
		\end{align*}
		and hence, by the smallness condition \eqref{prop:smallness2} and the argument in~\eqref{prop:NRGContra2},  
		\begin{align*}
			\|\Psi(X_0; z)- \Psi(X_0; w)\|_{\Z_\varepsilon^{s,0}} &\leq CC(v_L) \|h\|_{ B^\alpha_{2,\infty}\cap L_t^\infty}\|z-w\|_{\Z_\varepsilon^{s,0}}(\|z\|_{L_t^2 L_x^{2^*}}^\theta + \|w\|_{L_t^2 L_x^{2^*}}^\theta)(\|z\|_{\Z_\varepsilon^{s,0}}^{2-\theta} + \|w\|_{\Z_\varepsilon^{s,0}}^{2-\theta})\notag\\
			&\leq \frac12 \|z-w\|_{\Z_\varepsilon^{s,0}}.
		\end{align*} 
		Hence,  $\Psi(z;X_0)$ is a contractive self-map on $B_{R, \delta}$, 
		which yields a unique solution $z\in \Z_\varepsilon^{s,0}$ 
		of~\eqref{eq:RegNoiseSchrPot}.

		For the wave component, 
		using the trilinear estimate \eqref{triesti-improve} for the pair $(s, s-\frac12)$, 
		we get 
		\begin{align*}
			\|\rho \|_{W^{s-\frac12,0,s-\frac12}} \lesssim C \|h\|_{B^\alpha_{2,\infty}\cap L_t^\infty} \|z\|_{\Z_\varepsilon^{s,0}}^2<\infty, 
		\end{align*}
		implying that $\rho \in {W^{s-\frac12,0,s-\frac12}}$.
		
		We are left to improve the wave regularity to $H_x^l$. 
		
		For this purpose, 
		we first show that 
		$z\in \Z_\varepsilon^{s,b}$. Let $\kappa = s-\frac{1}{2} (1-b)$ 
		and rewrite system \eqref{eq:RegulationZak} as
		\begin{align*}
			z(t)= e^{\imu t\Delta}X_0 + \mathcal{I}_0[\Re(e^{\imu t|\nabla|}Y_0)z] -\mathcal{I}_0[\Re(\cJ_0(h|\nabla||z|^2))z].
		\end{align*}
		Then, 
		applying Lemma \ref{lem:LinFlowAdaptedSpaces}, the bilinear estimate in Lemma \ref{lem:BilinearEstimates} \ref{it:BilinearEstNonendpointbNzero}, the trilinear estimate \eqref{triesti-improve} with pair $(s,\kappa)$ in regime I, and the embedding $\Z_\varepsilon^{s,0}\hookrightarrow S^{s,0,0}$, we get
		\begin{align}
			\| \mathcal{I}_0(\Re(\cJ_0[h|\nabla||z|^2])z)\|_{S^{s,0,b}}
			\lesssim\|\cJ_0(h|\nabla||z|^2)\|_{W^{\kappa,0,s-\frac12}} \|z\|_{S^{s,0,0}} \lesssim \|h\|_{B^\alpha_{2,\infty}\cap L_t^\infty} \|z\|_{\Z_\varepsilon^{s,0}}^3. \label{prop:DuhamelNonWavez}
		\end{align}
		Similarly, since $l\geq \kappa$, we have
		\begin{align}
			\|\cI_0[\Re(e^{\imu t|\nabla|}Y_0)z]\|_{S^{s,0,b}}&\lesssim \|\Re(e^{\imu t|\nabla|}Y_0)z\|_{N^{s,0,b}}\notag\\
			&\lesssim \|e^{\imu t|\nabla|}Y_0\|_{W^{\kappa,0,s-\frac12}} \|z\|_{S^{s,0,0}}\lesssim \|Y_0\|_{H_x^l} \|z\|_{\Z_\varepsilon^{s,0}}.\label{prop:DuhamelFreeWavez}
		\end{align}
		Hence, Lemma \ref{lem:LinEstimates},  estimates \eqref{prop:DuhamelNonWavez} and \eqref{prop:DuhamelFreeWavez} imply that
		\begin{align}\label{prop:ZInSszerob}
			\|z\|_{S^{s,0,b}}\lesssim \|X_0\|_{H_x^s} + \|Y_0\|_{H_x^l} \|z\|_{\Z_\varepsilon^{s,0}} + \|h\|_{B^\alpha_{2,\infty}\cap L_t^\infty} \|z\|_{\Z_\varepsilon^{s,0}}^3<\infty,
		\end{align}
		which yields that $z\in S^{s,0,b}$. Since we already showed $z \in \Z_\varepsilon^{s,0}$, 
		we obtain that 
		$z\in \Z_\varepsilon^{s,b}$ by the definition \eqref{def:Z-norm}, 
		as claimed. 
		
		Finally, we apply the trilinear estimate \eqref{triesti-LocSmo} to obtain
		\begin{align*}
			\|\rho\|_{L_t^\infty H_x^l} \leq \|\cJ_0[h|\nabla||z|^2]\|_{W^{l,0,s-\frac12}} \lesssim   \|h\|_{B^{\frac{1}{p}}_{p,\infty}\cap L_t^\infty}\|z\|^2_{\Z_\varepsilon^{s,b}} <\infty,
		\end{align*}
		which yields that $\rho\in L_t^\infty H_x^l$. The uniqueness of the wave component is again a consequence of the uniqueness of $z$.
		
		The proof of Proposition \ref{prop:WPShortTime} is complete. 
	\end{proof}

	\subsection{Large-time regime}
	In this subsection, for any given $0<T<\infty$, we consider the system in the large time regime $I=[T,\infty)$,
	\begin{equation}\label{eq:RegulationZakGWP}
		\left\{\aligned
		(\imu \partial_t +\Delta -\Re(v_L))z &= \Re(V) z + \Re({\rho}) z, &z(T)&=z_T,\\
		(\imu \partial_t +|\nabla|)\rho &=  -h|\nabla||z|^2, &\rho(T)&=0,
		\endaligned
		\right.
	\end{equation}
	where $z_T\in H_x^s$, $V$ is a potential, and $v_L= e^{\imu t|\nabla|} Y_0$ is the same free wave potential as in the short-time regime. The reason to introduce the potential $V$ here is that for solving~\eqref{eq:RegulationZak} on the time interval $[T,\infty)$ with initial data $(z(T), v(T))$, the right free wave potential to extract is $v_{L,T} = e^{\imu (t-T)|\nabla|} v(T)$. If one works with $v_{L,T}$, however, the constant $C(v_{L,T})$ resulting from applications of Lemma~\ref{lem:LinSchrPotSmallTimeSsa0} and Lemma~\ref{lem:LinSchrPotSmallTimeLocSmo}, depend on $T$ and we do not have any uniform control of these constants in $T$. We overcome this problem by keeping the wave potential $v_L$ fixed, absorbing the error in the potential $V$, and controlling $V$ in suitable norms.
	
	We will show global well-posedness and scattering for system \eqref{eq:RegulationZakGWP} provided the noise term $h$ and the potential $V$ are sufficiently small. We recall that~\eqref{eq:RegulationZakGWP} is again equivalent to~\eqref{eq:RegNoiseSchrPot} with adjusted initial condition, i.e.
	\begin{equation}
		\label{eq:RegNoiseSchrPotLongTime}
		(\imu \partial_t +\Delta -\Re(v_L))z = \Re(V) z -\Re(\cJ_0(h|\nabla| |z|^2)) z, \qquad z(T) = z_T.
	\end{equation}
	
	We also recall that the initial time $t_0$ is implicit in our notation for the Duhamel operators $\cU_{v_L}$, $\cI_{v_L}$, $\cJ_0$, etc. In what follows, we take $t_0 = T$. We also recall that the assumptions of Theorem~\ref{Prop-trilinear} are satisfied in the corresponding noise-regularization regimes, see the comment in front of Proposition~\ref{prop:WPShortTime}.
	
	\vspace*{4pt plus 2pt minus 2pt}%
	
	\begin{proposition}[Well-posedness in  large-time regime and scattering]\label{prop:WPLargeTime}
		Let $(s,l)$ be as in Theorem \ref{thm:RegNoise}, $(a,b)$ as in \eqref{con:ab},  $0<T<\infty$, and $I = [T,\infty)$. Let $C\geq 1$ denote a large universal constant and $C(v_L)$ the constant from Lemma~\ref{lem:LinSchrPotSmallTimeSsa0} and Lemma~\ref{lem:LinSchrPotSmallTimeLocSmo} depending on the free wave potential $v_L$.
		
		Take $E,\delta_0>0$ such that
		\begin{align}\label{prop:smallness3}
			2CC(v_L) \|z_T\|_{H_x^s}\leq E, \quad
			16CC(v_L)E^2 \delta_0 \leq 1. 
		\end{align}
		We assume the following conditions to hold, 
		where $C>0$ is a large universal constant: 
		
		\vspace*{4pt plus 2pt minus 2pt}%
		\paragraph{$(i)$ \textbf{Noise-regularization regime I}} 
		Let $(s,l)$ lie in noise-regularization regime I of Theorem \ref{thm:RegNoise}. Let $\alpha \in (0,\frac12)$ be as in Theorem~\ref{Prop-trilinear}~(i).
		Assume that $h$ and $V$ satisfy 
		\begin{align*}
			\|h\|_{B^\alpha_{2,\infty}\cap L^\infty(I)}\leq \delta_0, \qquad 
			4 C C(v_L) \|V\|_{W^{l,a,s - \frac12}(I)} \leq 1.
		\end{align*}
		where $p>2$ is as in \eqref{prop:triestiend}. 
		
		\vspace*{4pt plus 2pt minus 2pt}%
		\paragraph{$(ii)$ \textbf{Noise-regularization regime II}} 
		Let $(s,l)$ lie in noise-regularization regime II of Theorem \ref{thm:RegNoise}. Let $p \in (2,\infty)$ be as in Theorem~\ref{Prop-trilinear}~(ii).
		Assume that $h$ and $V$ satisfy
		\begin{align*}
			\|h\|_{B^{\frac{1}{p}}_{p,\infty}\cap L^\infty(I)}\leq \delta_0, \qquad 4 C C(v_L) \|V\|_{W^{l,0,l}(I)} \leq 1.
		\end{align*}
		
		\vspace*{4pt plus 2pt minus 2pt}%
		\paragraph{$(iii)$ \textbf{Noise-regularization regime III}} 
		Let $(s,l)$ lie in noise-regularization regime III of Theorem \ref{thm:RegNoise} and $s\leq l$. Let $p>2$ close to $2$ and $0<\varepsilon\ll 1$ be as in Theorem~\ref{Prop-trilinear}~(iii). Take
		$\alpha \in (0, \frac12)$ such that estimate~\eqref{triesti-improve} in regime I holds for the pairs $(s, s - \frac{1}{2})$ and $(s, \kappa)$, where $\kappa = s - \frac12 (1-b)$.
		
		Assume that 
		$h$ and $V$ satisfy the following: There are 
		$A, C'>0$ and  $T' \in [T,\infty)$ large enough such that
		\begin{align}\label{prop:exponentialDeacy}
			\|h\|_{B^\alpha_{2,\infty}\cap L^\infty(I)} \leq \delta_0,\quad \|h\|_{B^{\frac{1}{p}}_{p,\infty}\cap L^\infty([t_1,t_2])}\lesssim e^{-C't_1},\quad \|h\|_{B^{\frac{1}{p}}_{p,\infty}\cap L^\infty(I)}\leq A,
		\end{align}
		for all $T'< t_1<t_2$, and
		\begin{align}\label{prop:SmallPotentialV}
			4 C C(v_L) \| V \|_{W^{s-\frac12,0,s-\frac12}(I)} \leq 1 \qquad \text{and} \qquad \| V \|_{W^{l,0,s-\frac12}(I)} < \infty.
		\end{align}
		
		\vspace*{3pt plus 1pt minus 1pt} 
		Then
		system \eqref{eq:RegulationZakGWP} has a unique solution $(z, \rho) \in C(I, H_x^s\times H_x^l)$, where the uniqueness holds in 
		$S^{s,a,0}(I)\times L^\infty(I; H_x^l), S^{s,0,0}(I)\times L^\infty(I; H_x^l)$, and $\Z^{s,0}_\varepsilon(I)\times L^\infty(I; H_x^l)$ in the noise-regularization regimes I, II, and III, respectively. Moreover, $(z,\rho)$ scatters at infinity 
		in the sense that there exists some $(z_+, \rho_+) \in H_x^s\times H_x^l$ such that
		\begin{align}  \label{zV-scatter}
			\lim_{t\to\infty} \|e^{-\imu t\Delta}z(t)-z_+\|_{H_x^s} = 0 \qquad \text{and} \qquad  \lim_{t \rightarrow \infty}\|e^{-\imu t |\nabla|} \rho(t) - \rho_+\|_{H_x^l} =0.
		\end{align}	
	\end{proposition}
	
	\begin{remark}
		We note that the exponential decay of $h$ in \eqref{prop:exponentialDeacy} is used to derive scattering of the wave component in noise-regularization regime III. In noise-regularization regime I and II, the dispersive norm provided by the trilinear estimates \eqref{prop:triest2} and \eqref{prop:triestiend}, respectively, suffices to obtain the scattering of the wave part, see \eqref{prop:CauchyWaveI} and \eqref{prop:CauchyWaveII} below.
	\end{remark}

\begin{proof}[Proof of Proposition \ref{prop:WPLargeTime}] 
		As before, 
		we omit the dependence on $I\times \R^d$ to 
		ease the notation. 
		
		$(i)$ and $(ii)$
		As in the short time regime, we first consider $(s,l)$ in noise-regularization regime I and solve system \eqref{eq:RegulationZakGWP} with $(z_T, \rho_T)\in H_x^s\times H_x^{l}$. 
		Define the map
		\begin{align}\label{eq:FixedPointLargeTime}
			\Phi(z_T; z):= \cU_{v_L}[z_T] + \cI_{v_L}[\Re(V)z] - \cI_{v_L}[\Re(\cJ_0( h|\nabla||z|^2))z]
		\end{align}
		on the ball
		\begin{align*}
			B_E:=\{z\in S^{s,a,0}(I) \colon \|z\|_{S^{s,a,0}(I)}\leq E\}.
		\end{align*}
		
		From Lemma~\ref{lem:LinSchrPotSmallTimeSsa0}, Lemma \ref{lem:BilinearEstimates} \ref{it:BilinearEstNonendpointbNzero} and the trilinear estimate \eqref{triesti-improve}, we infer  that for any $z\in B_E$,
		\begin{align}
			\|\cI_{v_L}[\Re(\cJ_0( h|\nabla||z|^2))z]\|_{S^{s,a,0}}&\leq C(v_L) \|\Re(\cJ_0( h|\nabla||z|^2)z]\|_{N^{s,a,0}}\leq CC(v_L)  \|\cJ_0( h|\nabla||z|^2)\|_{W^{l,a,s-\frac12}} \|z\|_{S^{s,a,0}}\notag\\
			&\leq CC(v_L) \|h\|_{B^\alpha_{2,\infty}\cap L_t^\infty} \|z\|_{S^{s,a,0}}^3.\label{prop:largetimeDuhamelSchr}
		\end{align}
		
		Analogously, we get
		\begin{align}\label{prop:largetimeDuhamelSchr2}
			\|\cI_{v_L}[\Re(V)z]\|_{S^{s,a,0}}\leq C(v_L) \|\Re(V)z\|_{N^{s,a,0}} \leq CC(v_L)\|V\|_{W^{l,a,s-\frac12}} \|z\|_{S^{s,a,0}}.
		\end{align}
		
		Combining the estimates above, Lemma~\ref{lem:LinSchrPotSmallTimeSsa0} and the smallness condition \eqref{prop:smallness3}, 	we obtain that
		\begin{align*}
			\|\Phi(z_T;z)\|_{S^{s,a,0}}&\leq CC(v_L) \|z_T\|_{H_x^s} + CC(v_L) \|h\|_{B^\alpha_{2,\infty}\cap L_t^\infty} \|z\|^3_{S^{s,a,0}} + CC(v_L)\|V\|_{W^{l,a,s-\frac12}} \|z\|_{S^{s,a,0}}\\
			&\leq \frac{E}{2} + \frac{E}{16} +\frac{E}{4}\leq E,
		\end{align*}
		which implies that $\Phi(z_T;z)$ is a self-map on $B_E$. Moreover, 
		arguing as in~\eqref{prop:largetimeDuhamelSchr} and \eqref{prop:largetimeDuhamelSchr2} again, we infer
		\begin{align*}
			&\|\Phi(z_T;z)-\Phi(z_T;w)\|_{S^{s,a,0}} \\
			&\leq CC(v_L) \|h\|_{B^\alpha_{2,\infty}\cap L_t^\infty} (\|z\|_{S^{s,a,0}}+ \|w\|_{S^{s,a,0}})^2 \|z-w\|_{S^{s,a,0}}+ CC(v_L)\|V\|_{W^{l,a,s-\frac{1}{2}}}\|z-w\|_{S^{s,a,0}}\\
			&\leq \frac14 \|z-w\|_{S^{s,a,0}} + \frac14 \|z-w\|_{S^{s,a,0}}  = \frac12 \|z-w\|_{S^{s,a,0}},
		\end{align*}
		where we also used the smallness condition \eqref{prop:smallness3} in the last step. Hence, $\Phi(z_T;\cdot)$ is a contractive self-map on $B_E$,
		which implies that there exists a unique fixed point $z\in B_E$. Uniqueness in $S^{s,a,0}$ then follows via standard arguments.
		
		For the wave component  
		$\rho = - \Re(\cJ_0[h|\nabla||z|^2])$, 
		by the trilinear estimate \eqref{triesti-improve}, we have
		\begin{align}\label{prop:waveWtildel}
			\|\rho\|_{W^{l,a,s-\frac12}}\lesssim  \|h\|_{B^\alpha_{2,\infty}\cap L_t^\infty}  \|z\|_{S^{s,a,0}}^2<\infty.
		\end{align}

		In noise-regularization regime II, we argue again as in~\eqref{prop:largetimeDuhamelSchr} but use \eqref{prop:triestiend} instead of~\eqref{triesti-improve} and the embedding $S^{s,0,0} \hookrightarrow L_t^2 W_x^{s,2^*}$, to derive
		\begin{align*}
			\|\cI_{v_L}[\Re(V)z]\|_{S^{s,0,0}}&\leq C(v_L) \|\Re(V)z\|_{N^{s,0,0}}\leq CC(v_L)\|V\|_{W^{l,0,l}}\|z\|_{S^{s,0,0}},\\
			\|\mathcal{I}_{v_L}[\Re(\cJ_0(h|\nabla||z|^2))z]\|_{S^{s,0,0}}&\leq C(v_L) \|\Re(\cJ_0(h|\nabla||z|^2))z\|_{N^{s,0,0}}\\
			&\leq CC(v_L)  \|\cJ_0( h|\nabla||z|^2)\|_{W^{l,0,l}} \|z\|_{S^{s,0,0}}
			\leq CC(v_L) \|h\|_{B^{\frac{1}{p}}_{p,\infty}\cap L_t^\infty} \|z\|_{S^{s,0,0}}^3.
		\end{align*}
		Hence, we obtain 
		\begin{align*}
			\|\Phi(z_T;z)\|_{S^{s,0,0}}\leq CC(v_L) \|z_T\|_{H_x^s} + CC(v_L) \|h\|_{B^{\frac{1}{p}}_{p,\infty}\cap L_t^\infty} \|z\|^3_{S^{s,0,0}}+ CC(v_L)\|V\|_{W^{l,0,l}}\|z\|_{S^{s,0,0}}.
		\end{align*}
		Moreover, we have
		\begin{align*}
			&\|\Phi(z_T;z)-\Phi(z_T;w)\|_{S^{s,0,0}} \\
			&\leq CC(v_L) \|h\|_{B^{\frac{1}{p}}_{p,\infty}\cap L_t^\infty} (\|z\|_{S^{s,0,0}}+ \|w\|_{S^{s,0,0}})^2 \|z-w\|_{S^{s,0,0}}+ CC(v_L)\|V\|_{W^{l,0,l}}\|z-w\|_{S^{s,0,0}}.
		\end{align*}
		Using~\eqref{prop:smallness3}, we conclude that $\Phi(z_T;\cdot)$ is again a contractive self-map on $B_E$ (recall that $a = 0$ in regime~II),
		which yields a unique fixed point $z\in B_E$. Uniqueness in $S^{s,0,0}$ follows from standard arguments. For the wave component we get
		\begin{align}\label{prop:waveWtildel2}
			\|\rho\|_{L_t^\infty H_x^l} \leq \|\Re(\cJ_0[h|\nabla||z|^2])\|_{W^{l,0,l}} \lesssim \|h\|_{B^{\frac{1}{p}}_{p,\infty}\cap L_t^\infty} \|z\|^2_{S^{s,0,0}}<\infty,
		\end{align}
		which shows that $\rho\in L_t^\infty H_x^l$. 
		
		Uniqueness of the wave component in $L_t^\infty H_x^{l}$ follows from the uniqueness of the Schr{\"o}dinger component both in regime I and II.
		
		Regarding the scattering property of the Schr\"odinger component $z$, in noise-regularization regime I, using Lemma \ref{lem:BilinearEstimates} \ref{it:BilinearEstNonendpointbNzero} and Lemma \ref{lem:LinearEstimateHalfWave} we get
		\begin{align*}
			\|\Re(e^{\imu (t-T)|\nabla|}v(T))z\|_{N^{s,a,0}} &\lesssim \|e^{\imu (t-T)|\nabla|}v(T)\|_{W^{l,a,s-\frac12}} \|z\|_{S^{s,a,0}} \lesssim \|v(T)\|_{H_x^l} \|z\|_{S^{s,a,0}}<\infty,\\
			\|\Re(V)z\|_{N^{s,a,0}} &\lesssim \|V\|_{W^{l,a,s-\frac12}} \|z\|_{S^{s,a,0}}<\infty.
		\end{align*}
		Moreover, Lemma \ref{lem:BilinearEstimates} \ref{it:BilinearEstNonendpointbNzero} and \eqref{prop:waveWtildel} imply that
		\begin{align*}
			\|\Re(\rho)z\|_{N^{s,a,0}}\lesssim \|\rho\|_{W^{l,a,s-\frac12}} \|z\|_{S^{s,a,0}}
			\lesssim \|h\|_{B^\alpha_{2,\infty}\cap L_t^\infty} \|z\|_{S^{s,a,0}}^3<\infty.
		\end{align*}
		
		Similarly, in noise-regularization regime II we have
		\begin{align*}
			\|\Re(e^{\imu (t-T)|\nabla|}v(T))z\|_{N^{s,0,0}} &\lesssim \|e^{\imu (t-T)|\nabla|}v(T)\|_{W^{l,0,l}} \|z\|_{S^{s,0,0}} \lesssim \|v(T)\|_{H_x^l} \|z\|_{S^{s,0,0}}<\infty,\\
			\|\Re(V)z\|_{N^{s,0,0}} &\lesssim \|V\|_{W^{l,0,l}} \|z\|_{S^{s,0,0}}<\infty.
		\end{align*}
		Moreover, Lemma \ref{lem:BilinearEstimates} \ref{it:BilinearEstNonendpointbNzero} and \eqref{prop:waveWtildel2} imply that
		\begin{align*}
			\|\Re(\rho)z\|_{N^{s,0,0}}\lesssim \|\rho\|_{W^{l,0,l}} \|z\|_{S^{s,0,0}}
			\lesssim \|h\|_{B^{\frac{1}{p}}_{p,\infty}\cap L_t^\infty} \|z\|_{S^{s,0,0}}^3<\infty.
		\end{align*}
		
		For any $t_1,t_2\in I$ and $t_1<t_2$, 
		applying Lemma $2.5$ from \cite{CHN23},  we thus obtain both in regime I and II 
		\begin{align}\label{prop:z_scatter}
			\|e^{-\imu t_2\Delta}z(t_2) - e^{-\imu t_1\Delta}z(t_1)\|_{H_x^s} = \Big\|\int_{t_1}^{t_2} e^{-\imu s\Delta} (\Re(v_L)z+ \Re(V)z-\Re(\cJ_0[h|\nabla||z|^2])z) \dd s \Big\|_{H_x^s}\longrightarrow 0
		\end{align}
		as $t_1, t_2\to \infty$, which yields that $\{ e^{-\imu t \Delta} z(t)\}$ is Cauchy in $H^s_x$ as $t\to\infty$.

		It remains to prove the scattering property of the wave component. In noise-regularization regime I, for any $T<t_1<t_2<\infty$ large enough, an application of \eqref{prop:triest2} yields that for some $\theta \in (0,1)$
		\begin{align}\label{prop:CauchyWaveI}
			\|e^{-\imu t_1|\nabla|}\rho(t_1)-e^{-\imu t_2|\nabla|}\rho(t_2)\|_{H_x^l}&\lesssim  \|\Re(\cJ_0(h|\nabla||z|^2))\|_{W^{l,a,s-\frac12}([t_1,t_2])}\notag\\
			&\lesssim  \delta_0\|z\|_{S^{s,a,0}}^{2-2\theta} \|z\|^{2\theta}_{L^2(t_1,t_2; L_x^{2^*})} \longrightarrow 0\quad \text{as } t_1, t_2 \to \infty,
		\end{align}
		where we exploited the global bound $\|z\|_{L^2(I;  L_x^{2^*})}\lesssim \|z\|_{S^{s,a,0}(I)}<\infty$ in the last step. 
		Similarly, 
		in noise-regularization regime II, we use \eqref{prop:triestiend} to get 
		that for some $\theta \in (0,1)$ 
		\begin{align}\label{prop:CauchyWaveII}
			\|e^{-\imu t_1|\nabla|}\rho(t_1)-e^{-\imu t_2|\nabla|}\rho(t_2)\|_{H_x^l}&\lesssim  \|\Re(\cJ_0(h|\nabla||z|^2))\|_{W^{l,0,l}([t_1,t_2])}\notag\\
			&\lesssim  \delta_0 \|z\|_{S^{s,0,0}}^{2-2\theta} \|z\|_{L^2(t_1,t_2; W_x^{s,2^*})}^{2\theta}   \longrightarrow 0\quad \text{as } t_1, t_2 \to\infty,
		\end{align}
		exploiting
		the global bound $\|z\|_{L^2(I; W_x^{s,2^*})}\lesssim \|z\|_{S^{s,0,0}(I)}<\infty$ in the last step.
		
		Hence, in both noise-regularization regimes I and II, 
		$\{e^{-\imu t|\nabla|}\rho(t)\}$ 
		is Cauchy in $H_x^l$, and thus the wave component 
		scatters at infinity. 
		The statement of Proposition 
		\ref{prop:WPLargeTime} is proved in noise-regularization regimes I and II.

		$(iii)$ 
		We first solve system \eqref{eq:RegulationZakGWP} in $H_x^s\times H_x^{s-\frac{1}{2}}$ and recall that the pair $(s,s-\frac12)$ lies in regime I.
		Consider the solution map $\Phi$ as in \eqref{eq:FixedPointLargeTime} on the closed ball
		\begin{align*}
			B_E:=\{z\in \Z_\varepsilon^{s,0}(I): \|z\|_{\Z_\varepsilon^{s,0}(I)}\leq E\}.
		\end{align*} 
		Applying  Lemma \ref{lem:LinSchrPotSmallTimeLocSmo}, Lemma \ref{lem:BilinearEstimates} \ref{it:BilinearEstNonendpointbNzero}, and the trilinear estimate \eqref{triesti-improve} for the pair $(s, s-\frac12)$, we derive that
		\begin{align*}
			\|\mathcal{I}_{v_L}[\Re(\cJ_0( h|\nabla||z|^2))z]\|_{\Z_\varepsilon^{s,0}}&\leq C(v_L) \|\Re(\cJ_0( h|\nabla||z|^2))z\|_{N^{s,0,0}} \notag\\
			&\leq CC(v_L)  \|\Re(\cJ_0( h|\nabla||z|^2))\|_{W^{s-\frac{1}{2},0,s-\frac12}} \|z\|_{\Z_\varepsilon^{s,0}} \leq CC(v_L) \|h\|_{B^\alpha_{2,\infty}\cap L_t^\infty} \|z\|_{\Z_\varepsilon^{s,0}}^3
		\end{align*}
		and analogously
		\begin{align*}
			\|\cI_{v_L}[\Re(V)z]\|_{\Z_\varepsilon^{s,0}} \leq C(v_L) \|\Re(V)z\|_{N^{s,0,0}}\leq CC(v_L) \|V\|_{W^{s-\frac12,0,s-\frac12}} \|z\|_{S^{s,0,0}}.
		\end{align*}
		The combination of this estimate with Lemma~\ref{lem:LinSchrPotSmallTimeLocSmo} for the linear propagator and the smallness condition in \eqref{prop:exponentialDeacy} and \eqref{prop:SmallPotentialV} implies
		\begin{align*}
			\|\Phi(z_T;z)\|_{\Z_\varepsilon^{s,0}}&\leq CC(v_L) \|z_T\|_{H_x^s} + CC(v_L) \|h\|_{B^\alpha_{2,\infty}\cap L_t^\infty} \|z\|^3_{\Z_\varepsilon^{s,0}} + CC(v_L) \|V\|_{W^{s-\frac12,0,s-\frac12}} \|z\|_{S^{s,0,0}}\\
			&\leq \frac{E}{2} + \frac{E}{16} + \frac{E}{4}\leq E
		\end{align*}
		for any $z \in B_E$,
		which shows that $\Phi(z_T;\cdot)$ is a self-map on $B_E$. The same arguments also yield
		\begin{align*}
			&\|\Phi(z_T;z)-\Phi(z_T;w)\|_{\Z_\varepsilon^{s,0}} \\
			&\leq CC(v_L) \|h\|_{B^\alpha_{2,\infty}\cap L_t^\infty} (\|z\|_{\Z_\varepsilon^{s,0}}+ \|w\|_{\Z_\varepsilon^{s,0}})^2 \|z-w\|_{\Z_\varepsilon^{s,0}} + CC(v_L)\|V\|_{W^{s-\frac12,0,s-\frac12}} \|z-w\|_{\Z_\varepsilon^{s,0}}\\
			&\leq \frac14 \|z-w\|_{\Z_\varepsilon^{s,0}} + \frac14 \|z-w\|_{\Z_\varepsilon^{s,0}} =\frac12 \|z-w\|_{\Z_\varepsilon^{s,0}}
		\end{align*}
		for any $z,w \in B_E$.
		Hence, $\Phi(z_T;\cdot)$ is a contractive self-map on $B_E$ and we deduce that there exists a unique solution $z\in \Z_\varepsilon^{s,0}$ 
		of~\eqref{eq:RegNoiseSchrPotLongTime}. 
		
		For the wave component $\rho$, 
		we note that the estimate \eqref{triesti-improve} implies that $\rho \in W^{s-\frac{1}{2},0,s-\frac12}$. We next improve the wave regularity to $H^l_x$. 
		As in the proof of \eqref{prop:ZInSszerob}, we use the estimates \eqref{prop:DuhamelNonWavez}, \eqref{prop:DuhamelFreeWavez}, and 
		\begin{align*}
			\|\cI_0[\Re(V)z]\|_{S^{s,0,b}}\lesssim \|V\|_{W^{\kappa,0,s-\frac12}} \|z\|_{S^{s,0,0}}\lesssim \|V\|_{W^{l,0,s-\frac12}} \|z\|_{S^{s,0,0}}<\infty,
		\end{align*}
		together with the condition \eqref{prop:SmallPotentialV} to derive that $z\in S^{s,0,b}$ and hence $z\in \Z_\varepsilon^{s,b}$. 
		Combining this with \eqref{triesti-LocSmo} we derive that
		\begin{align*}
			\|\rho\|_{L_t^\infty H_x^l} \lesssim  \|\cJ_0(h|\nabla||z|^2)\|_{W^{l,0,s-\frac12}} \lesssim \| h \|_{B^{\frac{1}{p}}_{p,\infty} \cap L^\infty_t} \|z\|^2_{\Z_\varepsilon^{s,b}} \lesssim A \|z\|^2_{\Z_\varepsilon^{s,b}} <\infty,
		\end{align*}
		which shows that $\rho \in L_t^\infty H_x^l$. The uniqueness of $\rho$ in $L_t^\infty H_x^{l}$ follows from the uniqueness of the Schr{\"o}dinger component.
		
		\smallskip
		
		The scattering property of the Schr\"odinger component $z$  
		can be proved analogously as in noise-regularization regime I. Replacing the $S^{s,a,0}$-norm on the right-hand side with the $\Z_\varepsilon^{s,0}$-norm, one shows that $\Re(v_L) z, \Re(\rho) z$ and $\Re(V)z$ belong to $N^{s,0,0}$. Applying Lemma~2.5 from~\cite{CHN23} as in~\eqref{prop:z_scatter}, we obtain that $\{e^{- \imu t \Delta} z(t)\}$ is Cauchy as $t \rightarrow \infty$. Regarding the scattering property of the wave component, for any $T<t_1<t_2<\infty$ large enough, an application of the trilinear estimate \eqref{triesti-LocSmo} and the exponential decay \eqref{prop:exponentialDeacy} yield that
		\begin{align*}
			\|e^{-\imu t_1|\nabla|}\rho(t_1)-e^{-\imu t_2|\nabla|}\rho(t_2)\|_{H_x^l}&\lesssim  \|\cJ_0(h|\nabla||z|^2)\|_{W^{l,0,s-\frac12}([t_1,t_2])}\\
			&\lesssim   \|h\|_{B^{\frac{1}{p}}_{p,\infty}\cap L^\infty([t_1,t_2])} \|z\|_{\Z_\varepsilon^{s,b}}^2 \\
			&\lesssim  e^{-C't_1} \|z\|_{\Z_\varepsilon^{s,b}}^2	\longrightarrow 0\quad \text{as } t_1, t_2 \to\infty.
		\end{align*}
		Hence, we  infer that  $\{e^{-\imu t|\nabla|}\rho(t)\}$ is Cauchy in $H_x^l$, implying the scattering of the wave component.
		
		Consequently, 
		the proof of Proposition \ref{prop:WPLargeTime} is complete. 
	\end{proof}

	\subsection{Proof of Theorem \ref{thm:RegNoise}}
		Now, we are ready to prove Theorem \ref{thm:RegNoise}. Recall the random system \eqref{eq:RanZakNoncons-intro} and the definition of $h_{\mathbf{c}}$ in \eqref{h-W1-def}. Thanks to Corollary \ref{prop:ExtendToInfiniteNoise}, we can apply Lemma \ref{eq:DefHoelderNorm}, Lemma \ref{le:GBMGloBesov}, and Lemma \ref{prop:GeoBMFastDecay} for $h_{\mathbf{c}}$ with $c$ replaced by $\|\mathbf{c}\|$. 
	
	We first consider noise-regularization regime I. Take $\alpha \in (0, \frac12)$ as in Theorem~\ref{Prop-trilinear}~(i). 
	Lemma \ref{prop:GeoBMFastDecay}~(ii) and Remark~\ref{rem:BesovIntersection} show that $\sup_{\|\mathbf{c}\| \geq 1}\|h_{\mathbf{c}}\|_{B_{2,\infty}^\alpha\cap L^\infty([0,\infty)}  <\infty$ $\PP$-a.s., implying that
	\begin{align*}
		\PP\Big(\bigcup_{n \in \N} \{\|h_{\mathbf{c}}(\cdot, \omega)\|_{B_{2,\infty}^\alpha\cap L^\infty([0,\infty))} \leq n \}  \Big) = 1.
	\end{align*}
	Hence, for every $\eta > 0$ there exists $A > 0$ such that
	\begin{align*}
		\PP( \{\omega \in \Omega \colon \|h_{\mathbf{c}}(\cdot, \omega)\|_{B_{2,\infty}^\alpha\cap L^\infty([0,\infty))} \leq A \} ) \geq 1 - \frac{\eta}{2}.
	\end{align*}
	Let $R= 2 CC(v_L)\|X_0\|_{H_x^s}$ and $\delta>0$ and $C(v_L)$ as in noise-regularization regime I of Proposition \ref{prop:WPShortTime}, where $v_L = e^{\imu t |\nabla|} Y_0$. 
	Then take $E\geq 2CC(v_L) R$
	and $\delta_0>0$ as in Proposition~\ref{prop:WPLargeTime}.
	By Lemma \ref{prop:GeoBMFastDecay}~(i) and Remark~\ref{rem:BesovIntersection}, for any $\eta>0$, there exists $c_0 > 0$ such that
	\begin{align*}
		\PP(\{\omega \in \Omega \colon \|h_{\mathbf{c}}(\cdot, \omega)\|_{B_{2,\infty}^\alpha\cap L^\infty([\|\mathbf{c}\|^{-1},\infty))} > \delta_0 \}) \leq \frac{\eta}{2}
	\end{align*}
	for all $\| \mathbf{c} \| \geq c_0$. Consequently, the event
	\begin{align*}
		\Omega_{\mathbf{c}} := \{\omega \in \Omega \colon \|h_{\mathbf{c}}(\cdot, \omega)\|_{B_{2,\infty}^\alpha\cap L^\infty([0,\infty))} \leq A\} \cap \{ \|h_{\mathbf{c}}(\cdot, \omega)\|_{B_{2,\infty}^\alpha\cap L^\infty([\|\mathbf{c}\|^{-1},\infty))} \leq \delta_0\}
	\end{align*}
	has high probability
	\begin{align*}
		\PP(\Omega_{\mathbf{c}}) \geq 1 - \eta
	\end{align*}
	for all $\| \mathbf{c} \| \geq c_0$. Fix $\omega \in \Omega$ in the following.
	
	Since $\|e^{\imu (\cdot)\Delta}X_0\|_{L^2(0,t;L_x^{2^*})}\to 0 $ as $t\to 0^+$, we can choose $c_1 \geq c_0$ such that for all $\|\mathbf{c}\| \geq c_1$
	\begin{align*}
		2CC(v_L)(1+\|Y_0\|_{H_x^l})\|e^{\imu t \Delta }X_0\|_{L^2 (0, 2 \|\mathbf{c}\|^{-1}; L_x^{2^*})}\leq \delta.
	\end{align*}
	Hence, the assumptions of Proposition~\ref{prop:WPShortTime} in noise-regularization regime I with $h = h_{\mathbf{c}}(\cdot, \omega)$ are satisfied on the interval $I = [0, 2 \|\mathbf{c}\|^{-1}]$ so that Proposition~\ref{prop:WPShortTime} yields a unique solution $(z_{1,\mathbf{c}}(\cdot, \omega), v_{1,\mathbf{c}}(\cdot, \omega))$ in $C([0,2|\mathbf{c}\|^{-1}]; H^s_x \times H_x^l)$ to system \eqref{eq:RegulationZak} with 
	\begin{align}
		\label{eq:Propertiesz1c}
		\|z_{1,\mathbf{c}}(\cdot, \omega)\|_{S^{s,a,0}([0,2\|\mathbf{c}\|^{-1}])}\leq R \qquad \text{and} \qquad \|z_{1,\mathbf{c}}(\cdot, \omega)\|_{L^2([0,2\|\mathbf{c}\|^{-1}]; L^{2^*}_x)}\leq \delta
	\end{align}	
	for all $\|\mathbf{c}\| \geq c_1$. In particular, $\|z_{1,\mathbf{c}}(\|\mathbf{c}\|^{-1},\omega)\|_{H_x^s}\leq R$. We only consider $\|\mathbf{c}\| \geq c_1$ in the following.
	
	Next, we extend the solution beyond $\|\mathbf{c}\|^{-1}$ (solving on the extended interval $[0,2\|\mathbf{c}\|^{-1}]$ provides the necessary overlap to ensure the combined solution belongs to $S^{s,a,0}([0,\infty))$, cf. the proof of Theorem~\ref{thm:LocalWP}). To that purpose, we have to solve \eqref{eq:RegulationZak} on $[\|\mathbf{c}\|^{-1}, \infty)$ with initial data $z_{1,\mathbf{c}}(\|\mathbf{c}\|^{-1}, \omega)$ and free wave potential $v_{L, \mathbf{c}} := e^{\imu (t - \|\mathbf{c}\|^{-1}) |\nabla|} v_{1,\mathbf{c}}(\|\mathbf{c}\|^{-1}, \omega)$. The problem with this approach is that the constant $C(v_{L,\mathbf{c}})$ depending on the free wave potential, arising in Lemma~\ref{lem:LinSchrPotSmallTimeSsa0}, now depends on $\mathbf{c}$ and we do not have any uniform control over it. To overcome this problem, we write
	\begin{align*}
		v_{L, \mathbf{c}} = v_L + V_{\mathbf{c}} \qquad \text{with } V_{\mathbf{c}} = v_{L,\mathbf{c}} - v_{L} = e^{\imu (t - \|\mathbf{c}\|^{-1}) |\nabla|} v_{1,\mathbf{c}}(\|\mathbf{c}\|^{-1},\omega) - e^{\imu t |\nabla|} Y_0
	\end{align*}
	with the same free wave potential $v_L = e^{\imu t |\nabla|} Y_0$ as in the short-time regime and exploit that $V_{\mathbf{c}}$ is small in appropriate norms if $\|\mathbf{c}\|$ is sufficiently large. Consequently, we have to solve
	\begin{equation*}
		\left\{\aligned
		(\imu \partial_t +\Delta -\Re(v_L))z &= \Re(V_{\mathbf{c}}(\cdot, \omega)) z + \Re({\rho}) z, &z(\|\mathbf{c}\|^{-1})&=z_{1, \mathbf{c}}(\|\mathbf{c}\|^{-1},\omega),\\
		(\imu \partial_t +|\nabla|)\rho &=  -h_{\mathbf{c}}(\cdot, \omega)|\nabla||z|^2, &\rho(\|\mathbf{c}\|^{-1})&=0,
		\endaligned
		\right.
	\end{equation*}
	i.e.~\eqref{eq:RegulationZakGWP} with $T = \|\mathbf{c}\|^{-1}$, $h = h_{\mathbf{c}}(\cdot, \omega)$, and potential $V = V_{\mathbf{c}}(\cdot, \omega)$.
	
	Since $\omega \in \Omega_{\mathbf{c}}$, the condition $\|h_{\mathbf{c}}(\cdot, \omega)\|_{B_{2,\infty}^\alpha\cap L^\infty([\|\mathbf{c}\|^{-1},\infty))} \leq \delta_0$ of Proposition~\ref{prop:WPLargeTime}~(i) is satisfied. To check the assumption on $V_{\mathbf{c}}(\cdot, \omega)$ we note (dropping the $\omega$ for the ease of notation) that
	\begin{align*}
		V_{\mathbf{c}} = e^{\imu (t - \|\mathbf{c}\|^{-1}) |\nabla|} (e^{\imu \|\mathbf{c}\|^{-1} |\nabla|}Y_0 - \cJ_0[h_{\mathbf{c}} |\nabla| |z_{1,\mathbf{c}}|^2](\|\mathbf{c}\|^{-1}) - e^{\imu t |\nabla|} Y_0 = \imu e^{\imu t |\nabla|} \int_0^{\|\mathbf{c}\|^{-1}} e^{- \imu t' |\nabla|} (h_{\mathbf{c}} |\nabla| |z_{1,\mathbf{c}}|^2) \dd t'.
	\end{align*}
	Using Lemma~\ref{lem:LinearEstimateHalfWave} and the embedding $W^{l,a,s-\frac12} \hookrightarrow L^\infty_t H^l_x$, we infer
	\begin{align*}
		\| V_{\mathbf{c}} \|_{W^{l,a,s-\frac12}([\|\mathbf{c}\|^{-1},\infty))} &\leq C \Big\|  \int_0^{\|\mathbf{c}\|^{-1}} e^{- \imu t' |\nabla|} (h_{\mathbf{c}} |\nabla| |z_{1,\mathbf{c}}|^2) \dd t' \Big\|_{H^l_x} \leq C \| \cJ_0 [h_{\mathbf{c}} |\nabla| |z_{1,\mathbf{c}}|^2]\|_{L^\infty([0,\|\mathbf{c}\|^{-1}]; H^l_x)} \\
		& \leq C \|\cJ_0 [h_{\mathbf{c}} |\nabla| |z_{1,\mathbf{c}}|^2]\|_{W^{l,a,s-\frac12}([0,\|\mathbf{c}\|^{-1}])}.
	\end{align*}
	An application of Proposition~\ref{Prop-trilinear} in noise-regularization regime I thus yields
	\begin{align*}
		\| V_{\mathbf{c}} \|_{W^{l,a,s-\frac12}([\|\mathbf{c}\|^{-1},\infty))} &\leq C \|h_{\mathbf{c}}\|_{B^\alpha_{2,\infty} \cap L^\infty([0,\|\mathbf{c}\|^{-1}])} \| z_{1,\mathbf{c}} \|_{L^2([0,\|\mathbf{c}\|^{-1}];L^{2^*}_x)}^{2\theta} \| z_{1,\mathbf{c}} \|_{S^{s,a,0}([0,\|\mathbf{c}\|^{-1}])}^{2 - 2\theta} \\
		&\leq C A \delta^{2\theta} R^{2 - 2\theta},
	\end{align*}
	where we used that $\omega \in \Omega_{\mathbf{c}}$ and the bounds for $z_{1,\mathbf{c}}$ in~\eqref{eq:Propertiesz1c} in the last step. Using condition~\eqref{prop:smallness2} on $\delta$ in noise-regularization regime I, we thus arrive at
	\begin{align*}
		4 C C(v_L) \| V_{\mathbf{c}}(\cdot, \omega) \|_{W^{l,a,s-\frac12}([\|\mathbf{c}\|^{-1},\infty))} \leq 4 C C(v_L) A R^{2 - 2\theta} \delta^{2 \theta} \leq 1.
	\end{align*}
	
	Hence, the conditions of Proposition~\ref{prop:WPLargeTime} in regime I are satisfied so that Proposition~\ref{prop:WPLargeTime} yields a unique solution $(z_{2,\mathbf{c}}(\cdot, \omega), v_{2,\mathbf{c}}(\cdot, \omega))\in C([\|\mathbf{c}\|^{-1},\infty)); H^s_x \times H_x^l)$ to system \eqref{eq:RegulationZakGWP} $T = \|\mathbf{c}\|^{-1}$, $h = h_{\mathbf{c}}(\cdot, \omega)$, and potential $V = V_{\mathbf{c}}(\cdot, \omega)$. Moreover, $(z_{2,\mathbf{c}}(\cdot, \omega), v_{2,\mathbf{c}}(\cdot, \omega))$ scatters at infinity by~\eqref{zV-scatter}.
	
	Finally, we concatenate the solutions from the short-time and the large-time regimes by setting
	\begin{align*}
		z_{\mathbf{c}}(t,\omega)&= \chi_{[0,\|\mathbf{c}\|^{-1}]}(t)z_{1,\mathbf{c}}(t,\omega)+ \chi_{(\|\mathbf{c}\|^{-1},\infty)}(t)z_{2,\mathbf{c}}(t,\omega), \\
		v_{\mathbf{c}}(t,\omega)&= \chi_{[0,\|\mathbf{c}\|^{-1}]}(t)v_{1,\mathbf{c}}(t,\omega)+ \chi_{(\|\mathbf{c}\|^{-1},\infty)}(t)v_{2,\mathbf{c}}(t,\omega)
	\end{align*}
	Then $(z_{\mathbf{c}},v_{\mathbf{c}})$ solves the system \eqref{eq:RanZakNoncons-intro} in $C([0,\infty); H_x^s\times H_x^l)$ and scatters at infinity. Moreover, since $z_{1,\mathbf{c}}$ and $z_{2,\mathbf{c}}$ both solve~\eqref{eq:RegNoiseSchrPotLongTime} with the same potential and initial value on $[\|\mathbf{c}\|^{-1}, 2 \|\mathbf{c}\|^{-1}]$, the uniqueness part of Proposition~\ref{prop:WPLargeTime} implies that $z_{1, \mathbf{c}}$ and $z_{2, \mathbf{c}}$ coincide on that interval. In view of Lemma~2.8 in~\cite{CHN23} (the analogue of the decomposability result in Lemma~\ref{lem:DecompX} for the $S^{s,a,b}$-spaces), we also have that $z_{\mathbf{c}} \in S^{s,a,0}([0,\infty))$. Combining the uniqueness results for $z_{1,\mathbf{c}}$ and $z_{2,\mathbf{c}}$, we conclude that $(z_{\mathbf{c}},v_{\mathbf{c}})$ is the unique solution of \eqref{eq:RanZakNoncons-intro} in $S^{s,a,0}([0,\infty)) \times L^\infty([0,\infty);H^l_x)$.
	
	Since $\omega \in \Omega_{\mathbf{c}}$ was arbitrary, we infer that 	
	the probability of the event $\Upsilon_{\mathbf{c}}$ defined in \eqref{eq:DefUpsilon} satisfies
	\begin{align*}
		\PP(\Upsilon_{\mathbf{c}})\geq \PP(\Omega_{\mathbf{c}})\geq 1-\eta
	\end{align*}
	for all $\| \mathbf{c} \| \geq c_1$. As $\eta > 0$ was arbitrary, we conclude that
	\begin{align*}
		\PP(\Upsilon_{\mathbf{c}}) \longrightarrow 1
	\end{align*}
	as $\| \mathbf{c} \| \rightarrow \infty$. In view of the rescaling transform, this proves Theorem~\ref{thm:RegNoise} in noise-regularization regime I.
	
	The proof in regimes II and III proceeds analogously, requiring only straightforward adaptations. We define $\Omega_{\mathbf{c}}$ to accommodate the assumptions on $h$ in Propositions~\ref{prop:WPShortTime} and~\ref{prop:WPLargeTime}. Note that Lemma~\ref{prop:GeoBMFastDecay} guarantees that for any $\eta > 0$, the event $\Omega_\mathbf{c}$ has high probability $\PP(\Omega_{\mathbf{c}}) \geq 1 - \eta$ for sufficiently large $\| \mathbf{c} \|$ in regimes II and III as well. The assumptions of Proposition~\ref{prop:WPShortTime} are then satisfied for any $\omega \in \Omega_{\mathbf{c}}$ if we choose $\|c\|$ large enough. To apply Proposition~\ref{prop:WPLargeTime}, it remains to check the assumptions for $V_{\mathbf{c}}$. As in regime I, we first reduce this to the corresponding estimate for the Duhamel integral on $[0,\|\mathbf{c}\|^{-1}]$. Combining the trilinear estimates from Theorem~\ref{Prop-trilinear} with the bounds on $z_{1,\mathbf{c}}$ from Proposition~\ref{prop:WPShortTime} in regimes II and III, respectively, shows that these assumptions are satisfied. Applying Proposition~\ref{prop:WPLargeTime} and proceeding as in regime I, we obtain the unique global scattering solution of \eqref{eq:RanZakNoncons-intro} for these remaining regimes. In view of the rescaling transform, this completes the proof of Theorem~\ref{thm:RegNoise}.	\qed

	\appendix
	
	\section{Refined rescaling transforms} \label{Sec-Rescaling}
	
	In this subsection,
	we recall the refined rescaling transforms
	which are important to solve the stochastic Zakharov system
	up to the maximal existence time.
	They were first introduced in \cite{Zh20} to solve critical stochastic Schr\"odinger equations,
	and later in \cite{HRSZ24, CHN23}
	to address the global well-posedness below the ground state
	for the stochastic Zakharov sytstem in dimensions three and four.

	\begin{proposition} [Refined rescaling transformations] \label{prop:RefinedRescaling}
		Let $\sigma, \tau \colon \Omega \rightarrow [0,T]$ be such that $\sigma+\tau\leq T$.
		\begin{enumerate}
			\item \label{it:RefResc0ToSig}
			Let $(u_\sigma, v_\sigma) \in C([0,\tau], H_x^s \times H_x^l)$
			be an analytically weak solution of the system
			\begin{equation}   \label{eq:RanZakSigma}
				\left\{\aligned
				\partial_t u_\sigma(t) &=  \imu e^{-W_{1,\sigma}(t)} \Delta (e^{W_{1,\sigma}(t)} u_\sigma(t))
				-   \imu  \Re (v_\sigma(t))\, u_\sigma(t)
				- \imu \Re (\cT_{\sigma+t, \sigma}(W_2)) u_\sigma(t) ,   \\
				\partial_t v_\sigma(t) &=   \imu  |\na| v_\sigma(t) + \imu |\na| |u_\sigma(t)|^2,
				\endaligned
				\right.
			\end{equation}
			as equations in $H_x^{s-2} \times H_x^{l-1}$,
			where the increments of noise $W_{1,\sigma}$ and $ \cT_{\sigma+t, \sigma} (W_2)$ are defined by
			\begin{align}
				& W_{1,\sigma}(t):= W_1(\sigma+t) - W_1(\sigma),   \label{eq:DefWsigma}  \\
				& \cT_{\sigma+t, \sigma} (W_2)
				:= -\imu \int_\sigma^{\sigma+t} e^{\imu (\sigma+t-s)|\na|} \dd W_2(s)  \label{eq:DefTsigmaW2}
			\end{align}
			for all $t\in [0,\tau]$.
			For any $t\in [\sigma, \sigma + \tau]$, we set
			\begin{align}
				u(t)
				&:= e^{-W_1(\sigma)} u _\sigma(t-\sigma),    \label{eq:DefusigmaRescal} \\
				v(t) &:= v_\sigma(t-\sigma)- e^{\imu (t-\sigma) |\na|}\cT_\sigma(W_2).    \label{eq:DefvsigmaRescal}
			\end{align}
			Then, $(u, v)$ is an analytically weak solution of system \eqref{eq:RanZakW1W2}
			on $[\sigma, \sigma+\tau]$
			with
			\begin{align}
				& u(\sigma) =  e^{-W_1(\sigma)} u_\sigma(0),   \label{usigma-vsigma-initial.1}\\
				& v(\sigma) =   v_\sigma(0)-\cT_\sigma(W_2).    \label{usigma-vsigma-initial.2}
			\end{align}
			
			\item \label{it:RefRescSigTo0} If $(u,v) \in C([\sigma, \sigma+\tau], H_x^s\times H_x^l)$
			is an analytically weak solution of system \eqref{eq:RanZakW1W2} on $[\sigma, \sigma+\tau]$
			as equations in $H_x^{s-2}\times H_x^{l-1}$,
			then
			\begin{align}
				& u_\sigma(t)
				:=  e^{W_1(\sigma)} u(\sigma+t),    \label{eq:u-sigma-v-sigma-rescal.1} \\
				& v_\sigma(t) := v(\sigma+t) + e^{\imu t|\na|}\cT_\sigma(W_2)),   \label{eq:u-sigma-v-sigma-rescal.2}
				\ \ t\in [0,\tau],
			\end{align}
			is an analytically weak solution of system \eqref{eq:RanZakSigma} on $[0,\tau]$.
		\end{enumerate}
	\end{proposition}
	
	\begin{remark}
		\label{rem:RanZakRefinedRescaling}
		Letting
		\begin{align}
			\label{eq:Defbsigmacsigma}
			b_\sigma := 2 \nabla W_{1,\sigma}, \qquad
			c_\sigma := |\nabla W_{1,\sigma}|^2 + \Delta W_{1,\sigma},
		\end{align}
		in the setting of Proposition \ref{prop:RefinedRescaling},
		we note that~\eqref{eq:RanZakSigma} is equivalent to
		\begin{equation}   \label{eq:RanZakbsigmacsigma}
			\left\{\aligned
			\imu \partial_t u_\sigma + \Delta u_\sigma
			&= \Re(v_\sigma) u_\sigma - b_\sigma \cdot \nabla u_\sigma - c_\sigma u_\sigma + \Re(\cT_{\sigma + \cdot, \sigma}(W_2)) u_\sigma,  \\
			\imu \partial_t v_\sigma + |\nabla |v_\sigma  &= - |\nabla||u_\sigma|^2.
			\endaligned
			\right.
		\end{equation}
	\end{remark}

	The following result allows to glue two solutions on different intervals.
	It can be proved in an analogous manner as Proposition~3.3 in \cite{HRSZ25}.
	
	\begin{proposition} [Gluing solutions] \label{prop:GluingSolutions}
		Let $(u_1,v_1)\in C([0,\sigma], H_x^s \times H_x^l)$
		be an analytically weak solution of~\eqref{eq:RanZakW1W2} on $[0,\sigma]$,
		and let $(u_\sigma,v_\sigma)\in C([0,\tau], H_x^s \times H_x^l)$
		be an analytically weak solution of the refined Zakharov system
		\eqref{eq:RanZakSigma} on $[0,\tau]$
		with the initial condition
		\begin{align*}
			(u_\sigma(0), v_\sigma(0))
			:= (e^{W_1(\sigma)} u_1(\sigma), v_1(\sigma) + \cT_\sigma(W_2)).
		\end{align*}
		For every $t\in [0,\sigma+\tau]$,
		we set
		\begin{align*}
			u(t):= \begin{cases}
				u_1(t), \quad &\text{if } t \in [0,\sigma), \\
				e^{-W_1(\sigma)}u_\sigma(t-\sigma), &\text{if } t \in [\sigma, \sigma + \tau], 						\end{cases} \qquad
			v(t):= \begin{cases}
				v_1(t), \quad &\text{if } t \in [0,\sigma), \\
				v_\sigma(t-\sigma) - e^{\imu (t-\sigma)|\na|}\cT_\sigma(W_2)), &\text{if } t \in [\sigma, \sigma + \tau].
			\end{cases}
		\end{align*}
		Then, $(u,v) \in C([0,\sigma+\tau], H_x^s\times H_x^l)$
		is an analytically weak solution of  \eqref{eq:RanZakW1W2} on the larger
		interval $[0,\sigma+\tau]$.
	\end{proposition}

	We also have the following product estimate for rescaling transforms,
	which shows the important compatability between
	the function space $\XS(I)$
	and the rescaling transforms. This estimate can be proved in an analogous manner as \cite[Lemma 4.1]{HRSZ24}.
	
	\begin{lemma}   [Product estimate for rescaling transforms]
		\label{lem:ProductNoiseInX}
		Let $ s\in \R, 0\leq a\leq 1, \sigma \in [0,\infty)$, $I \subseteq \R$ be a bounded interval, and $u \in \X^{s,a}(I)$.
		Then, $e^{\pm W_1(\sigma)} u$ belongs to $\X^{s,a}(I)$ and we have
		\begin{align}  \label{Esti-prod-res}
			\|e^{\pm W_1(\sigma)} u\|_{\XS(I)}
			\lesssim (1 +  \|e^{\pm W_1(\sigma)} - 1\|_{H_x^{\frac{d}{2}+2+(s-1)_+}}(1 + |I|^{\frac{1}{2}})) \|u\|_{\X^{s,a}(I)}.
		\end{align}
	\end{lemma}

	As a consequence of Lemma~\ref{lem:ProductNoiseInX},
	the following result shows that the $\X^{s,a}$-space is invariant under the refined rescaling transforms.
	
	\begin{corollary}
		\label{cor:RefinedRescalingInX}
		Let $\sigma \geq 0$ and $\tau > 0$.
		\begin{enumerate}
			\item \label{it:RefinedRescalingInX1} If $u_\sigma \in \XS([0,\tau])$, then $u$ defined by $u(t) = e^{-W_1(\sigma)} u_\sigma(t - \sigma)$ for $t \in [\sigma, \sigma + \tau]$ belongs to \linebreak $\XS([\sigma, \sigma + \tau])$.
			\item \label{it:RefinedRescalingInX2} If $u \in \XS([\sigma, \sigma + \tau])$, then $u_\sigma$ defined by $u_\sigma(t) = e^{W_1(\sigma)}u(t + \sigma)$ for $t \in [0,\tau]$ belongs to $\XS([0,\tau])$.
		\end{enumerate}
		The above statements remain true if we replace $[0,\tau]$ and $[\sigma, \sigma + \tau]$ by $[0,\tau)$ and $[\sigma, \sigma + \tau)$, respectively.
	\end{corollary}

	\section{Estimates in Fourier restriction norms}
	
	We first collect several estimates in Fourier restriction norms from~\cite{CHN23}. We begin with the lemma containing the control of the linear Schr\"odinger flow in the adapted spaces $S^{s,a,b}$.
	\begin{lemma} [Lemma $2.4$, \cite{CHN23}] 	\label{lem:LinFlowAdaptedSpaces}
		Let $s\in \R$, $0\leq a, b\leq 1$. For any $\lambda \in 2^{\N_0}$ we have
		\begin{equation*}
			\|e^{\imu t \Delta} f_\lambda\|_{S^{s,a,b}_\lambda} \lesssim \lambda^s \|f_\lambda\|_{L^2_x}, \qquad
			\Big\|\int_{t_0}^t e^{\imu (t-t') \Delta} g_\lambda(t') \dd t' \Big\|_{S^{s,a,b}_\lambda} \lesssim \|g_\lambda\|_{N^{s,a,b}_\lambda}.
		\end{equation*}
	\end{lemma}

	The next lemma gives a product estimate for weighted fractional derivatives,
	which is useful in the bilinear and trilinear estimates.
	
	\begin{lemma}[Lemma $2.7$, \cite{CHN23}]
		\label{le:27zak4} 	
		Let $a\in \R, \mu>0$, and $1\leq \tilde{p}, \tilde{q}, \tilde{r}, p, q, r\leq \infty$ with $\frac{1}{p}=\frac{1}{q}+\frac{1}{r}$ and $\frac{1}{\tilde{p}}= \frac{1}{\tilde{q}} + \frac{1}{\tilde{r}}$. Then
		\begin{align*}
			\|(\mu+|\partial_t|)^a(vu)\|_{L_t^{\tilde{p}}L_x^p} \lesssim \mu^{-|a|} \|(\mu+|\partial_t|)^{|a|} v\|_{L_t^{\tilde{r}}L_x^r} \|(\mu+|\partial_t|)^a u\|_{L_t^{\tilde{q}}L_x^q}.
		\end{align*}
	\end{lemma}

	Next, we provide the bilinear estimates for the nonlinearities of the Zakharov system.

	\begin{lemma} [Bilinear estimates] 	\label{lem:BilinearEstimates}
		\begin{enumerate}
			\item[]
			
			\item (Theorem $3.1$,  \cite{CHN23}) \label{it:BilinearEstNonendpointbNzero}  Let $d\geq 4, 0\leq s\leq l+2, \beta\geq 0$ and $0\leq a,b\leq 1$ such that
			\begin{align*}
				l\geq b+\tfrac{d-4}{2},\quad s-l\leq a+1-b,\quad s+l\geq 2a, \quad \beta \geq \max\{s-1, \tfrac{d-4}{2}+b\}
			\end{align*}
			and
			\begin{align*}
				(s,l)\neq (\tfrac{d-2}{2}+a, \tfrac{d-4}{2}+b),\quad (\beta,b) \neq (\tfrac{d-2}{2}, 1).
			\end{align*}
			Then
			\begin{align}\label{BilinearEstNonendpointbNzero}
				\|\Re(v)u\|_{N^{s,a,b}}\lesssim \|v\|_{W^{l,a,\beta}} \|u\|_{S^{s,a,0}}.
			\end{align}
			
			\item (Corollary $4.2$, \cite{CHN23}) \label{it:BilinearEstNonendpointW} Let $d\geq 4, s,l,\beta\geq 0$, and $0\leq a\leq 1$ satisfy
			\begin{align*}
				\beta<\min \{s, 2s-\tfrac{d-2}{2}-a \},\quad 2a<2s-l-\tfrac{d-2}{2},\quad a<s-l.
			\end{align*}
			There exist $\theta \in (0,1)$ and $C > 0$ such that for any interval $I \subseteq \R$ we have
			\begin{align}
				\|\cJ_0[|\nabla| ({u} w)]\|_{W^{l,a,\beta}(I)}
				&\leq C (\|u\|_{\SOne(I)} \|w\|_{\SOne(I)})^{1-\theta} (\|u\|_{L^2 (I ; L_x^{2^*})} \|w\|_{L^2 (I ; L_x^{2^*})})^{\theta}. \label{eq:Bilinnablauu}
			\end{align}
			\item \label{it:BilinearEstEndpoint1}(Proposition $6.1$, \cite{CHN23}) For $d\geq 4, (s,l)=(\frac{d-3}{2},\frac{d-4}{2})$, there exists a constant $C > 0$ such that for any interval $ I \subseteq \R$ we have
			\begin{align}
				\|v u\|_{N^{s,0,0}(I)} &\leq C \| v \|_{W_w^{l,0,l}(I)} \|u\|_{L^2(I; W_x^{s,2^*})}^{\frac{1}{2}} \|u\|_{S_w^{s,0,0}(I)}^{\frac{1}{2}}. \label{eq:BilinvuEndpoint}
			\end{align}
			
			\item \label{it:BilinearEstEndpoint2} (Proposition $6.2$, \cite{CHN23}) For $d\geq 4, (s,l)=(\frac{d-3}{2},\frac{d-4}{2})$, there exists a constant $C > 0$ such that for any interval $I \subseteq \R$ we have
			\begin{align}
				\|\cJ_0[|\nabla| (u w)]\|_{W^{l,0,l}(I)} &\leq C ( \|u\|_{L^2(I; W^{s, 2^*}_x)} \|w\|_{L^2(I; W^{s, 2^*}_x)} )^{\frac{1}{2}} ( \|u\|_{S_w^{s,0,0}(I)} \|w\|_{S_w^{s,0,0}(I)} )^{\frac{1}{2}}. \label{eq:BilinnablauwEndpoint}
			\end{align}
		\end{enumerate}
	\end{lemma}
	
	\begin{remark}\label{re:extendsl}
		The proof of Proposition $6.2$ from \cite{CHN23} shows that estimate \eqref{eq:BilinnablauwEndpoint} actually holds for 
		\begin{align*}
			s=l+\tfrac{1}{2}\geq \tfrac{d-3}{2}.
		\end{align*}
	\end{remark}

	The next bilinear estimate requires control of the wave component only in the sum space $W^{l,a,\beta}(I)+L^2(I; W_x^{s,d})$, which allows to obtain smallness on short time intervals.
	
	\begin{lemma}
		\label{it:BilinearEstNonendpoint} Let $d\geq 4$. Assume that $\beta\geq \max\{\frac{d-4}{2}, s-1\}$ and
		\begin{align*}
			0\leq a\leq 1,\quad 0\leq s\leq l+2,\quad l\geq \tfrac{d-4}{2},\quad s-l\leq a+1,\quad s+l\geq 2a, 
		\end{align*}
		with $(s,l)\neq (\frac{d-2}{2}+a,\frac{d-4}{2})$. There exists $C>0$ such that for any interval $I\subseteq \R$ we have
		\begin{align}\label{prop:BilinearEstNonendpoint}
			\|\Re{(v)}u\|_{N^{s,a,0}(I)}\leq C\|v\|_{W^{l,a,\beta}(I)+L^2(I; W_x^{s,d})}\|u\|_{S^{s,a,0}(I)}.
		\end{align}
	\end{lemma}
	
	\begin{proof}
		By Bernstein and classical product estimates, we have
		\begin{align*}
			\|\Re(v)u\|_{N^{s,a,0}(I)} \lesssim \|\Re(v)u\|_{L^2(I; W_x^{s,d})} \lesssim \|v\|_{L^2(I; W_x^{s,d})} \|u\|_{L_t^\infty H_x^s},
		\end{align*}
		see the proof of Corollary~3.2 in~\cite{CHN23}.
		Then \eqref{prop:BilinearEstNonendpoint} follows from this estimate and  Lemma \ref{lem:BilinearEstimates} \ref{it:BilinearEstNonendpointbNzero}.
	\end{proof}
	
	We finally collect Strichartz and local smoothing estimates for the Schr{\"o}dinger flow.
	\begin{lemma} [Strichartz and local smoothing estimates \cite{HRSZ24}]  \label{lem:StrichartzLocalSmooth}
		Let $\lambda, \mu \in 2^{\N_0}$ with $|\log_2 (\mu/\lambda)| \leq 4$, $\vece \in \Sp^{d-1}$, and $(q,p)$, $(\tilde{q}, \tilde{p})$ be Schr\"odinger admissible, i.e., $\frac{2}{q} + \frac{d}{p} = \frac{d}{2} = \frac{2}{\tilde{q}} + \frac{d}{\tilde{p}}$. We then have the following estimates.
		\begin{enumerate}
			\item \label{it:HomStrichartz} Homogeneous Strichartz estimate:
			\begin{align*}
				\|e^{\imu t \Delta}  f_\lambda\|_{L^q_t L^p_x} \lesssim \|f_\lambda\|_{L^2_x}.
			\end{align*}
			\item \label{it:HomLocalSmooth} Homogeneous local smoothing estimate:
			\begin{align*}
				\|e^{\imu t \Delta} P_{\mu,\vece} f\|_{L^{\infty,2}_{\vece}} \lesssim \mu^{-\frac{1}{2}} \|f\|_{L^2_x}, \qquad \mu > 1.
			\end{align*}
			\item \label{it:InhomStrichartz} Inhomogeneous Strichartz estimate:
			\begin{align*}
				\Big\| \int_{t' < t} e^{\imu (t-t') \Delta} g_\lambda(t') \dd t' \Big\|_{L^q_t L^p_x} \lesssim \|g_\lambda\|_{L^{\tilde{q}'}_t L^{\tilde{p}'}_x},
			\end{align*}
			where $\tilde{p}'$ and $\tilde{q}'$ are the conjugate numbers of $\tilde{p}$ and $\tilde{q}$, respectively, i.e., $\frac{1}{\tilde{p}'} + \frac{1}{\tilde{p}}=1$ and $\frac{1}{\tilde{q}'} + \frac{1}{\tilde{q}}=1$.
			
			\item \label{it:InhomLocalSmooth} Inhomogeneous local smoothing estimate:
			\begin{align*}
				\Big\| \int_{t' < t} e^{\imu (t-t') \Delta} P_\lambda P_{\mu,\vece} g(t') \dd t' \Big\|_{L^{\infty,2}_\vece} \lesssim \lambda^{-1} \|g\|_{L^{1,2}_\vece}, \qquad \mu,\lambda > 1.
			\end{align*}
			\item \label{it:InhomStrichartzLocalSmooth} Inhomogeneous Strichartz to local smoothing estimate:
			\begin{align*}
				\Big\| \int_{t' < t} e^{\imu (t-t') \Delta} P_\lambda P_{\mu,\vece} g(t') \dd t' \Big\|_{L^q_t L^p_x} \lesssim \lambda^{-\frac{1}{2}} \|g\|_{L^{1,2}_\vece}, \qquad \mu,\lambda > 1.
			\end{align*}
			\item \label{it:InhomLocalSmoothStrichartz} Inhomogeneous local smoothing to Strichartz estimate:
			\begin{align*}
				\Big\| \int_{t' < t} e^{\imu (t-t') \Delta} P_\lambda P_{\mu,\vece} g(t') \dd t' \Big\| _{L^{\infty,2}_\vece} \lesssim \lambda^{-\frac{1}{2}} \|g\|_{L^{\tilde{q}'}_t L^{\tilde{p}'}_x}, \qquad \mu,\lambda > 1.
			\end{align*}
		\end{enumerate}
		\vspace*{-\topsep}
	\end{lemma}

	\section*{Acknowledgements} 
	Funded by the Deutsche Forschungsgemeinschaft (DFG, German Research Foundation) -- Project-ID 317210226 -- SFB 1283. 
	Part of this research was performed while Deng Zhang was visiting the Simons Laufer Mathematical Sciences Institute in Berkeley, California, during the fall semester of 2025, supported by the National Science Foundation under Grant No. DMS-1928930. 
	D.\ Zhang is also grateful for the NSFC grants (No. 12271352, 12322108)
	and Shanghai Frontiers Science Center of Modern Analysis.

	\bibliographystyle{abbrv}
	\bibliography{StZakharovdgeq4}
	
\end{document}